\documentclass[a4paper,
oneside, 				 
numbers=noenddot,
DIV=12,
headinclude=false, 
bibtotoc, 
fontsize=10pt, 
headsepline=true, 
open=right,
abstracton,
cleardoublepage=empty, 
footinclude=false,
]{scrartcl}

\usepackage{setspace}
\usepackage{standalone}

\usepackage{authblk}

\usepackage{mathtools}

\usepackage{stmaryrd}

\usepackage{pdflscape}

\DeclareSymbolFont{extraup}{U}{zavm}{m}{n}
\DeclareMathSymbol{\varclub}{\mathalpha}{extraup}{84} 
\DeclareMathSymbol{\varspade}{\mathalpha}{extraup}{85}

\usepackage{url}

\usepackage{upgreek}

\usepackage{sidecap}

\usepackage{xcolor}
\definecolor{darkred}{rgb}{0.9,0.1,0.1}
\definecolor{darkblue}{rgb}{0,0,0.7}
	\definecolor{darkgreen}{rgb}{0,0.5,0}
\definecolor{bluegray}{rgb}{0.4, 0.6, 0.8}
\definecolor{cadmiumorange}{rgb}{0.93, 0.53, 0.18}
\definecolor{darkcerulean}{rgb}{0.03, 0.27, 0.49}
\usepackage[pagebackref, colorlinks=false, linktocpage=true, linkbordercolor=red]{hyperref} %

\usepackage[english]{babel}
\usepackage[utf8]{inputenc}
\usepackage[T1]{fontenc} 
\usepackage{lmodern} 
\usepackage{amsmath, amsfonts, amssymb}
\usepackage{mathtools}
\usepackage{mathrsfs}
\usepackage{verbatim}
\usepackage{stmaryrd}
\usepackage[framed,amsmath,thmmarks,thref,hyperref]{ntheorem}
\usepackage{commath}

\usepackage[labelfont=bf]{caption}

\allowdisplaybreaks

\usepackage[english=american]{csquotes}

\usepackage{graphicx}

\usepackage{comment}

\usepackage{framed}

\usepackage{pdflscape}

\usepackage[]{listofsymbols}

\setkomafont{sectioning}{\bfseries} 

\setkomafont{pagenumber}{\bfseries}
\setkomafont{pageheadfoot}{\normalfont}
\setkomafont{pagenumber}{\normalfont}

\usepackage{enumitem} 
\usepackage{cleveref}
\usepackage{ifthen}

\usepackage{tikz,pgfplots}
\pgfplotsset{compat=newest}
\usetikzlibrary{cd}
\usetikzlibrary{calc}
\usetikzlibrary{arrows, scopes}
\usetikzlibrary{positioning}
\usetikzlibrary{shapes.geometric}
\usetikzlibrary{backgrounds}
\usetikzlibrary{arrows,chains,matrix,positioning,scopes}

\usepackage[multiple,hang,flushmargin]{footmisc}

\usepackage{graphicx}

\usepackage[colorinlistoftodos,disable]{todonotes} 

\usepackage{marginnote}

\usepackage[framemethod=TikZ]{mdframed}

\mdfdefinestyle{temporary}{linewidth=3pt, linecolor=red}
\mdfdefinestyle{onlyframe}{middlelinecolor=black!100,roundcorner=1ex, skipbelow=15pt}
\mdfdefinestyle{theorem}{middlelinecolor=black!25, backgroundcolor=black!15,roundcorner=1ex, skipbelow=15pt}
\mdfdefinestyle{onlyback}{middlelinecolor=white!100, backgroundcolor=black!10,roundcorner=1ex, skipbelow=15pt}
\mdfdefinestyle{main}{middlelinecolor=darkblue!25, backgroundcolor=darkblue!15,roundcorner=1ex, skipbelow=15pt}

\usepackage{commath}

\usepackage{enumitem}

\usepackage{nicefrac}

\usepackage{rotating}

\usepackage{relsize}

\usepackage{ctable}
\usepackage{multicol}
\usepackage{multirow}

\usepackage{accents}

\usepackage{ucs}
\usepackage{fixltx2e} 
\usepackage{lmodern} 
\usepackage{setspace,graphicx,tikz,tabularx} 
\usepackage[draft=false,babel,tracking=true,kerning=true,spacing=true]{microtype} 

\input{private.sty}

\usepackage{tikz,pgfplots}
\pgfplotsset{width=\textwidth,height=5cm}
\usepackage{mhequ}
\usepackage{mhsymb}

\usepackage{mathrsfs}

\usepackage{multirow}

\usepackage[framemethod=TikZ]{mdframed}

\mdfdefinestyle{temporary}{linewidth=3pt, linecolor=red}
\mdfdefinestyle{theorem}{middlelinecolor=black!25, backgroundcolor=black!15,roundcorner=1ex, skipbelow=15pt}
\mdfdefinestyle{main}{middlelinecolor=darkblue!25, backgroundcolor=darkblue!15,roundcorner=1ex, skipbelow=15pt}

\def\CA{\mathcal{A}}
\def\CB{\mathcal{B}}
\def\CC{\mathcal{C}}
\def\CD{\mathcal{D}}
\def\CE{\mathcal{E}}
\def\CF{\mathcal{F}}
\def\CG{\mathcal{G}}
\def\CH{\mathcal{H}}
\def\CI{\mathcal{I}}
\def\CJ{\mathcal{J}}

\def\CL{\mathcal{L}}
\def\CM{\mathcal{M}}
\def\CN{\mathcal{N}}
\def\CO{\mathcal{O}}
\def\CP{\mathcal{P}}
\def\CQ{\mathcal{Q}}
\def\CR{\mathcal{R}}
\def\CS{\mathcal{S}}
\def\CT{\mathcal{T}}
\def\CU{\mathcal{U}}
\def\CV{\mathcal{V}}
\def\CW{\mathcal{W}}
\def\CX{\mathcal{X}}
\def\CY{\mathcal{Y}}

\def\DD{\CD}
\def\EE{\mathscr{E}}
\def\FF{\mathscr{F}}

\def\II{\mathscr{I}}
\def\JJ{\mathscr{J}}

\def\LL{\mathscr{L}}
\def\MM{\boldsymbol{\CM}}

\def\TT{\mathscr{T}}

\def\fI{\mathfrak{I}}

\def\fK{\mathfrak{K}}

\def\fR{\mathfrak{R}}
\def\fS{\mathfrak{S}}

\def\fc{\mathfrak{c}} 
\def\fd{\mathfrak{d}}

\def\fs{\mathfrak{s}}
\def\ft{\mathfrak{t}}

\def\sh{\mathsf{h}}

\newcommand{\esh}{\operatorname{e}_{\mathsf{h}}}
\newcommand{\death}{\raisebox{-0.17em}{\includegraphics[width=3.3mm]{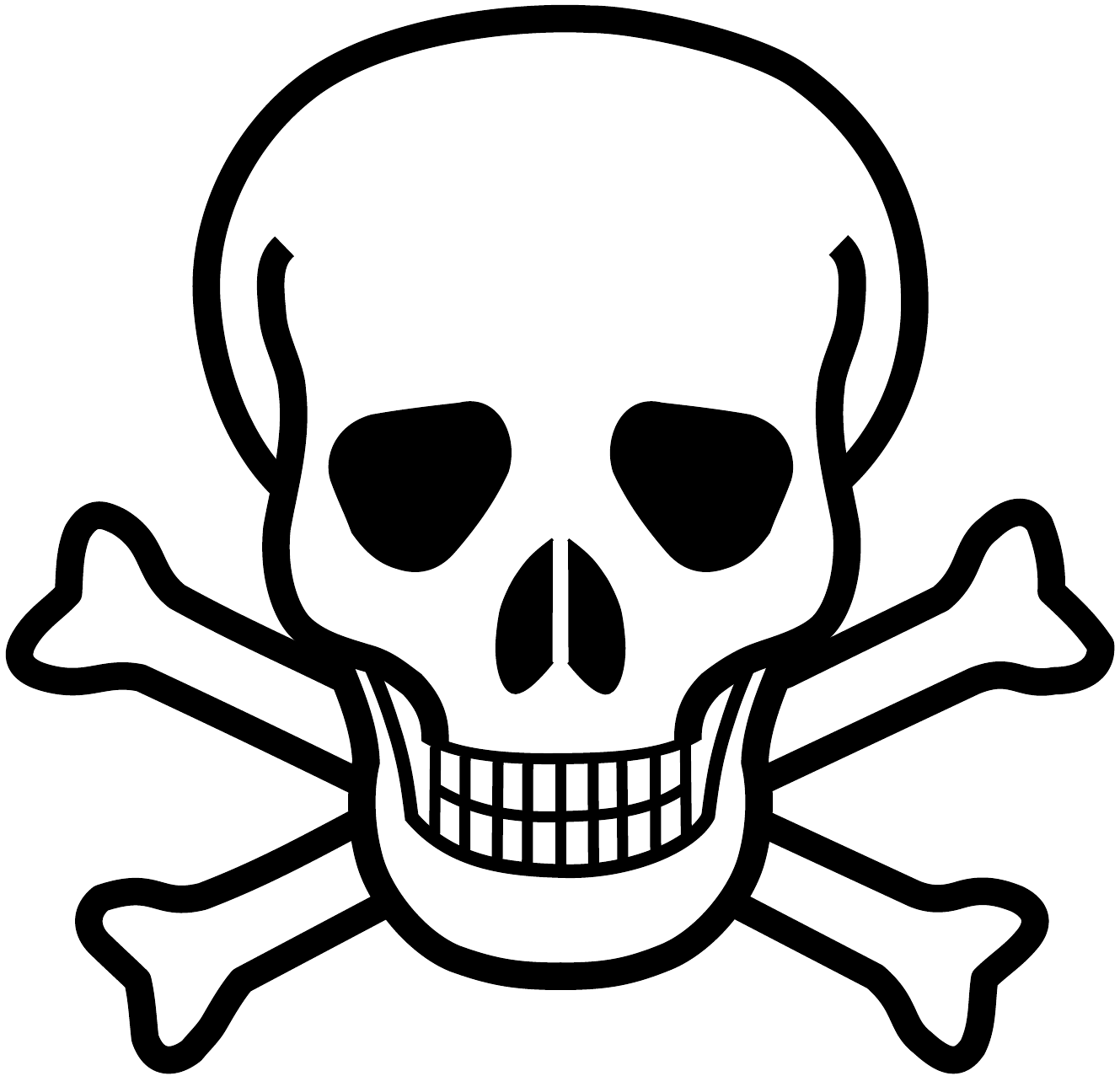}}}
\newcommand{\deathsmall}{\raisebox{-0.1em}{\includegraphics[width=2.0mm]{skull}}}

\colorlet{testcolor}{green!60!black}
\colorlet{testcolor2}{blue!100!black}
\colorlet{lightblue}{darkblue!80}
\colorlet{grayblue}{bluegray!100}
\colorlet{orange}{cadmiumorange!100}

\usetikzlibrary{calc}
\usetikzlibrary{shapes.misc}
\usetikzlibrary{shapes.symbols}
\usetikzlibrary{shapes.geometric}
\usetikzlibrary{snakes}
\usetikzlibrary{decorations}
\usetikzlibrary{decorations.markings}

\makeatletter
\pgfdeclareshape{crosscircle}
{
	\inheritsavedanchors[from=circle] 
	\inheritanchorborder[from=circle]
	\inheritanchor[from=circle]{north}
	\inheritanchor[from=circle]{north west}
	\inheritanchor[from=circle]{north east}
	\inheritanchor[from=circle]{center}
	\inheritanchor[from=circle]{west}
	\inheritanchor[from=circle]{east}
	\inheritanchor[from=circle]{mid}
	\inheritanchor[from=circle]{mid west}
	\inheritanchor[from=circle]{mid east}
	\inheritanchor[from=circle]{base}
	\inheritanchor[from=circle]{base west}
	\inheritanchor[from=circle]{base east}
	\inheritanchor[from=circle]{south}
	\inheritanchor[from=circle]{south west}
	\inheritanchor[from=circle]{south east}
	\inheritbackgroundpath[from=circle]
	\foregroundpath{
		\centerpoint%
		\pgf@xc=\pgf@x%
		\pgf@yc=\pgf@y%
		\pgfutil@tempdima=\radius%
		\pgfmathsetlength{\pgf@xb}{\pgfkeysvalueof{/pgf/outer xsep}}%
		\pgfmathsetlength{\pgf@yb}{\pgfkeysvalueof{/pgf/outer ysep}}%
		\ifdim\pgf@xb<\pgf@yb%
		\advance\pgfutil@tempdima by-\pgf@yb%
		\else%
		\advance\pgfutil@tempdima by-\pgf@xb%
		\fi%
		\pgfpathmoveto{\pgfpointadd{\pgfqpoint{\pgf@xc}{\pgf@yc}}{\pgfqpoint{-0.707107\pgfutil@tempdima}{-0.707107\pgfutil@tempdima}}}
		\pgfpathlineto{\pgfpointadd{\pgfqpoint{\pgf@xc}{\pgf@yc}}{\pgfqpoint{0.707107\pgfutil@tempdima}{0.707107\pgfutil@tempdima}}}
		\pgfpathmoveto{\pgfpointadd{\pgfqpoint{\pgf@xc}{\pgf@yc}}{\pgfqpoint{-0.707107\pgfutil@tempdima}{0.707107\pgfutil@tempdima}}}
		\pgfpathlineto{\pgfpointadd{\pgfqpoint{\pgf@xc}{\pgf@yc}}{\pgfqpoint{0.707107\pgfutil@tempdima}{-0.707107\pgfutil@tempdima}}}
	}
}
\makeatother

\colorlet{symbols}{black!90!black}
\colorlet{cameron}{darkgreen!90!black}
\colorlet{testcolor}{green!60!black}
\colorlet{darkblue}{blue!60!black}
\colorlet{darkgreen}{green!60!black}

\def\symbol#1{\textcolor{symbols}{#1}}
\def\1{\mathbf{\symbol{1}}}
\def\X{\symbol{X}}

\usetikzlibrary{calc}
\usetikzlibrary{shapes.misc}
\usetikzlibrary{shapes.symbols}
\usetikzlibrary{shapes.geometric}
\usetikzlibrary{snakes}
\usetikzlibrary{decorations}
\usetikzlibrary{decorations.markings}

\def\drawx{\draw[-,solid] (-3pt,-3pt) -- (3pt,3pt);\draw[-,solid] (-3pt,3pt) -- (3pt,-3pt);}
\tikzset{
	root/.style={draw=symbols,circle,inner sep=0pt,minimum size=0.5mm,fill=white},	
	smalldot/.style={circle,fill=symbols,draw=symbols,inner sep=0pt,minimum size=0.5mm},
	dot/.style={circle,fill=black,inner sep=0pt,minimum size=1mm},
	bigdot/.style={circle,fill=black,inner sep=0pt,minimum size=4mm},
	noiseedge/.style={black,semithick,decorate, decoration={snake,segment length=4pt,amplitude=1pt}},
	smallestdot1/.style={circle,fill=symbols,draw=symbols,inner sep=0pt,minimum size=0.2mm},
	smallestdot2/.style={circle,fill=cameron,draw=cameron,inner sep=0pt,minimum size=0.2mm},
	smallnoiseedge1/.style={draw=symbols,decorate, decoration={snake,segment length=1.75pt,amplitude=0.55pt}},
	smallnoiseedge2/.style={draw=cameron,decorate, decoration={snake,segment length=1.75pt,amplitude=0.55pt}},
	noiseedge1/.style={draw=symbols,decorate, decoration={snake,segment length=0.9pt,amplitude=0.55pt}},
	noiseedge2/.style={draw=cameron,decorate, decoration={snake,segment length=0.9pt,amplitude=0.55pt}},
	smallnoisenode1/.style={draw=symbols,circle,inner sep=0pt,minimum size=0.5mm,fill=white},	
	smallnoisenode2/.style={draw=cameron,circle,inner sep=0pt,minimum size=0.5mm,fill=white},
	poly/.style={draw=symbols,circle,inner sep=0pt,minimum size=0.5mm,fill=black},
	dot/.style={circle,fill=black,inner sep=0pt,minimum size=1mm},
	int/.style={circle,fill=black,draw=black,inner sep=0pt,minimum size=0.7mm},
	circ/.style={circle,draw=black,inner sep=0pt, minimum size=1mm},
	var/.style={circle,fill=black!10,draw=black,inner sep=0pt, minimum size=2mm},
	dotred/.style={circle,fill=black!50,inner sep=0pt, minimum size=2mm},
	generic/.style={semithick,shorten >=1pt,shorten <=1pt},
	oddfunc/.style={generic, dotted},
	dist/.style={ultra thick,draw=testcolor,shorten >=1pt,shorten <=1pt},
	testfcn/.style={ultra thick,testcolor,shorten >=1pt,shorten <=1pt,<-},
	testfunction/.style={ultra thick,testcolor,shorten >=1pt,shorten <=1pt},
	testfcnx/.style={ultra thick,testcolor,shorten >=1pt,shorten <=1pt,<-,
		postaction={decorate,decoration={markings,mark=at position 0.6 with {\drawx}}}},
	kprime/.style={semithick,shorten >=1pt,shorten <=1pt,densely dashed,->},
	kprimex/.style={semithick,shorten >=1pt,shorten <=1pt,densely dashed,->,
		postaction={decorate,decoration={markings,mark=at position 0.4 with {\drawx}}}},
	kernel/.style={semithick,shorten >=1pt,shorten <=1pt,->,draw=black},
	Pkernel/.style={ultra thick,shorten >=1pt,shorten <=1pt,->,draw=blue},
	PkernelBig/.style={very thick,shorten >=1pt,shorten <=1pt,decorate, draw=blue, decoration={zigzag,amplitude=1.5pt,segment length = 3pt,pre length=2pt,post length=2pt}},
	multx/.style={shorten >=1pt,shorten <=1pt,
		postaction={decorate,decoration={markings,mark=at position 0.5 with {\drawx}}}},
	kernelx/.style={semithick,shorten >=1pt,shorten <=1pt,->,
		postaction={decorate,decoration={markings,mark=at position 0.4 with {\drawx}}}},
	kernel1/.style={->,semithick,shorten >=1pt,shorten <=1pt,postaction={decorate,decoration={markings,mark=at position 0.45 with {\draw[-] (0,-0.1) -- (0,0.1);}}}},
	kernel2/.style={->,semithick,shorten >=1pt,shorten <=1pt,postaction={decorate,decoration={markings,mark=at position 0.45 with {\draw[-] (0.05,-0.1) -- (0.05,0.1);\draw[-] (-0.05,-0.1) -- (-0.05,0.1);}}}},
	kernelBig/.style={semithick,shorten >=1pt,shorten <=1pt,decorate, decoration={zigzag,amplitude=1.5pt,segment length = 3pt,pre length=2pt,post length=2pt}},
	kernelBigg/.style={thick,shorten >=1pt,shorten <=1pt,decorate, decoration={zigzag,amplitude=3.5pt,segment length = 7pt,pre length=2pt,post length=2pt}},
	kernelBigg1/.style={thick,shorten >=1pt,shorten <=1pt,decorate, decoration={saw,amplitude=3.5pt,segment length = 7pt,pre length=2pt,post length=2pt}},
	kernelBigg2/.style={thick,shorten >=1pt,shorten <=1pt,decorate, decoration={bumps,amplitude=3.5pt,segment length = 7pt,pre length=2pt,post length=2pt}},
	rho/.style={dotted,semithick,shorten >=1pt,shorten <=1pt},
	renorm/.style={shape=circle,fill=white,inner sep=1pt},
	labl/.style={shape=rectangle,fill=white,inner sep=1pt},
	cumu2n/.style={inner sep=3pt},
	cumu2/.style={draw=red!80,fill=red!40},
	cumu3/.style={regular polygon, regular polygon sides=3,draw=red!80,rounded corners=3pt,fill=red!40,minimum size=5mm},
	cumu4/.style={regular polygon, regular polygon sides=4,draw=red!80,rounded corners=3pt,fill=red!40,minimum size=7mm},
	cumu5/.style={regular polygon, regular polygon sides=5,draw=red!80,rounded corners=3pt,fill=red!40,minimum size=7mm},
	bcumu2n/.style={inner sep=3pt},
	bcumu2/.style={draw=blue!80,fill=blue!40},
	bcumu3/.style={regular polygon, regular polygon sides=3,draw=blue!80,rounded corners=3pt,fill=blue!40,minimum size=5mm},
	bcumu4/.style={regular polygon, regular polygon sides=4,draw=blue!80,rounded corners=3pt,fill=blue!40,minimum size=7mm},
	bcumu5/.style={regular polygon, regular polygon sides=5,draw=blue!80,rounded corners=3pt,fill=blue!40,minimum size=7mm},
	xi/.style={circle,fill=symbols!30,draw=symbols,inner sep=0pt,minimum size=1.2mm},
	cmn/.style={circle,fill=cameron!50,draw=symbols,inner sep=0pt,minimum size=1.2mm},
	xix/.style={crosscircle,fill=symbols!10,draw=symbols,inner sep=0pt,minimum size=1.2mm},
	xib/.style={circle,fill=symbols!10,draw=symbols,inner sep=0pt,minimum size=1.6mm},
	xibx/.style={crosscircle,fill=symbols!10,draw=symbols,inner sep=0pt,minimum size=1.6mm},
	not/.style={circle,fill=symbols,draw=symbols,inner sep=0pt,minimum size=0.5mm},
	>=stealth,
}
\makeatletter
\def\DeclareSymbol#1#2#3{\expandafter\gdef\csname MH@symb@#1\endcsname{\tikz[baseline=#2,scale=0.13,draw=symbols]{#3}}\expandafter\gdef\csname MH@symb@#1s\endcsname{\scalebox{0.7}{\tikz[baseline=#2,scale=0.15,draw=symbols]{#3}}}}
\def\<#1>{\csname MH@symb@#1\endcsname}
\makeatother

\DeclareSymbol{Xi2}{-2}{\draw (0,-0.25) node[xi] {} -- (-1,1) node[xi] {};}

\DeclareSymbol{Xi}{0.5}{\draw (0,0) node[xi]}

\DeclareSymbol{Xi22}{0.5}{\draw (0,0) node[xi] {} -- (-1,1) node[not] {} -- (0,2) node[xi] {};}

\DeclareSymbol{Xi1}{-2}{\draw (0,-0.75) -- (0,0.75) node[xi] {};}
\DeclareSymbol{Xi2}{-2}{\draw (0,-0.25) node[xi] {} -- (-1,1) node[xi] {};}
\DeclareSymbol{Xi3}{0}{\draw (0,0) node[xi] {} -- (-1,1) node[xi] {} -- (0,2) node[xi] {};}
\DeclareSymbol{Xi4}{2}{\draw (0,0) node[xi] {} -- (-1,1) node[xi] {} -- (0,2) node[xi] {} -- (-1,3) node[xi] {};}
\DeclareSymbol{Xi2X}{-2}{\draw (0,-0.25) node[xi] {} -- (-1,1) node[xix] {};}
\DeclareSymbol{XXi2}{-2}{\draw (0,-0.25) node[xix] {} -- (-1,1) node[xi] {};}

\DeclareSymbol{IXi2}{0}{\draw (0,-0.25) node[not] {} -- (-1,1) node[xi] {} -- (0,2) node[xi] {};}
\DeclareSymbol{IXi^2}{-1}{\draw (-1,1) node[xi] {} -- (0,0) node[not] {} -- (1,1) node[xi] {};}

\DeclareSymbol{XiX}{-2.8}{\node[xibx] {};}
\DeclareSymbol{Xi}{-2.8}{\node[xib] {};}
\DeclareSymbol{IXiX}{-1}{\draw (0,-0.25) node[not] {} -- (-1,1) node[xix] {};}

\DeclareSymbol{Xi3b}{-1}{\draw (-1,1) node[xi] {} -- (0,0) node[xi] {} -- (1,1) node[xi] {};}

\DeclareSymbol{IXi3}{2}{\draw (0,-0.25) node[not] {} -- (-1,1) node[xi] {} -- (0,2) node[xi] {} -- (-1,3) node[xi] {};}
\DeclareSymbol{IXi}{-2}{\draw (0,-0.25) node[not] {} -- (-1,1) node[xi] {};}
\DeclareSymbol{XiI}{-2}{\draw (0,-0.25) node[xi] {} -- (-1,1) node[not] {};}

\DeclareSymbol{Xi4b}{-1}{\draw(0,1.5) node[xi] {} -- (0,0); \draw (-1.2,1) node[xi] {} -- (0,0) node[xi] {} -- (1.2,1) node[xi] {};}
\DeclareSymbol{Xi4b'}{-1}{\draw(0,1.5) node[xi] {} -- (0,-0.2); \draw (-1,1) node[xi] {} -- (0,-0.2) node[not] {} -- (1,1) node[xi] {};}
\DeclareSymbol{Xi4c}{0}{\draw (0,1) -- (0.8,2.2) node[xi] {};\draw (0,-0.25) node[xi] {} -- (0,1) node[xi] {} -- (-0.8,2.2) node[xi] {};}
\DeclareSymbol{Xi4d}{-4.5}{\draw (0,-1.5) node[not] {} -- (0,0); \draw (-1,1) node[xi] {} -- (0,0) node[xi] {} -- (1,1) node[xi] {};}
\DeclareSymbol{Xi4e}{0}{\draw (0,2) node[xi] {} -- (-1,1) node[xi] {} -- (0,0) node[xi] {} -- (1,1) node[xi] {};}
\DeclareSymbol{Xi4e'}{0}{\draw (0,2) node[xi] {} -- (-1,1) node[xi] {} -- (0,-0.2) node[not] {} -- (1,1) node[xi] {};}

\DeclareSymbol{X}{0}{
	\node[poly] at (0,0.5) {};
}

\DeclareSymbol{i}{0}{
	\node[poly] at (0,0.5) {};
	\node[label={[xshift=0.1cm, yshift=-0.3cm] \tiny $i$}] at (0,0) {};
}

\DeclareSymbol{wn}{0}{
	\draw[white] (-.5,0) -- (.5,0); \draw[noiseedge1] (0,0)  -- (0,1.5) {};
}

\DeclareSymbol{cm}{0}{
	\draw[white] (-.5,0) -- (.5,0); \draw[noiseedge2] (0,0)  -- (0,1.5) {};
}

\DeclareSymbol{1}{0}{
	\draw (0,0) -- (.7,1.3) node[smallestdot1] {};
	\draw[noiseedge1] (.7,1.3) -- (-.4,2.1) {};
	\node[root] (root) at (0,0) {};	
}

\DeclareSymbol{1g}{0}{
	\draw[noiseedge2] (.7,1.3) -- (-.4,2.1) {};
	\draw (0,0) -- (.7,1.3) node[smallestdot1] {};	
	\node[root] (root) at (0,0) {};	
}

\DeclareSymbol{11}{0}{
	\draw[noiseedge1] (0,0) -- (-1.1,0.8)  {};
	\draw[noiseedge1] (.7,1.3) -- (-.4,2.1) {};
	\draw (0,0) -- (.7,1.3) node[smallestdot1] {};
	\node[root] (root) at (0,0) {};	
}

\DeclareSymbol{1g1}{0}{
	\draw[noiseedge1] (0,0) -- (-1.1,0.8)  {};
	\draw[noiseedge2] (.7,1.3) -- (-.4,2.1)  {};
	\draw (0,0) -- (.7,1.3) node[smallestdot1] {};
	\node[root] (root) at (0,0) {};	
}

\DeclareSymbol{1g1g}{0}{
	\draw[noiseedge2] (0,0) -- (-1.1,0.8) {};
	\draw[noiseedge2] (.7,1.3) -- (-.4,2.1) {};
	\draw (0,0) -- (.7,1.3) node[smallestdot1] {};
	\node[root] (root) at (0,0) {};	
}

\DeclareSymbol{11g}{0}{
	\draw[noiseedge2] (0,0) -- (-1.1,0.8) {};
	\draw[noiseedge1] (.7,1.3) -- (-.4,2.1)  {};
	\draw (0,0) -- (.7,1.3) node[smallestdot1] {};
	\node[root] (root) at (0,0) {};	
}

\DeclareSymbol{phi4}{0}{
	\draw (0,0) -- (1,1) node[smallestdot1] {};
	\draw[noiseedge2] (0,1.2) -- (0.7,1.9) {};
	\draw (0,0) -- (0,1.2) node[smallestdot1] {};
	\draw[noiseedge1] (1.2,0) -- (1.9,0.7) {};
	\draw (0,0) -- (1.2,0) node[smallestdot1] {};
	\draw[noiseedge1] (2,2) -- (2.7,2.7) {};
	\draw (1,1) -- (2,2) node[smallestdot1] {};
	\draw[noiseedge1] (1,2.2) -- (1.7,2.9) {};
	\draw (1,1) -- (1,2.2) node[smallestdot1] {};
	\draw[noiseedge2] (2.2,1) -- (2.9,1.7) {};	
	\draw (1,1) -- (2.2,1) node[smallestdot1] {};
	\node[root] (root) at (0,0) {};	
}

\DeclareSymbol{kpz}{0}{
	\draw[noiseedge1] (0,1.9) -- (0.7,2.6) {};
	\draw[gray] (0,1.2) -- (0,1.9) node[smallestdot1] {};
	\draw[noiseedge2] (0.7,1.2) -- (1.4,1.9) {};
	\draw[gray] (0,1.2) -- (0.7,1.2) node[smallestdot1] {};
	\draw[gray] (0,0) -- (0,1.2) node[smallestdot1] {};
	\draw[noiseedge2] (1.2,0.7) -- (1.9,1.4) {};
	\draw[gray] (1.2,0) -- (1.2,0.7) node[smallestdot1] {};
	\draw[noiseedge1] (1.9,0) -- (2.6,0.7) {};
	\draw[gray] (1.2,0) -- (1.9,0) node[smallestdot1] {};
	\draw[gray] (0,0) -- (1.2,0) node[smallestdot1] {};
	\node[root] (root) at (0,0) {};	
}

\newcommand{\gr}[1]{\textcolor{cameron}{#1}}

\usepackage{caption}
\usepackage{subcaption}

\usepackage{lipsum}

\usepackage{cite}

\setkomafont{dictumauthor}{\normalfont}
\setkomafont{dictumtext}{\normalfont}

\usepackage{colortbl}

\allowdisplaybreaks[1]

\usepackage{scalerel}

\makeatletter
\newcommand\bigwp{\mathop{\mathpalette\bigDi@mond\relax}}
\newcommand\bigDi@mond[2]{%
	\vcenter{\hbox{\m@th
			\scalebox{\ifx#1\displaystyle 2\else1.2\fi}{$#1\diamond$}%
	}}%
}
\makeatother

\author[1,2]{Peter K. Friz}
\author[1]{Tom Klose}
\affil[1]{Technische Universit\"at Berlin, Berlin, Germany.}
\affil[2]{Weierstraß–Institut f\"ur Angewandte Analysis und Stochastik, Berlin, Germany.}

\date{March 19, 2022}

\title{Precise Laplace asymptotics for singular stochastic PDEs: The case of 2D gPAM}

\numberwithin{equation}{section}

\begin{document}

\maketitle

\thispagestyle{plain}

\newcolumntype{R}[1]{>{\raggedright\arraybackslash}p{#1}}

\vspace{2em}

\noindent

\begin{abstract}
	We implement a Laplace method for the renormalised solution to the generalised 2D Parabolic Anderson Model (gPAM) driven by a small spatial white noise.
	Our work rests upon Hairer's theory of regularity structures which allows to generalise classical ideas of Azencott and Ben Arous on path space as well as Aida and Inahama and Kawabi on rough path space to the space of models.
	The technical cornerstone of our argument is a Taylor expansion of the solution in the noise intensity parameter: We prove precise bounds for its terms and the remainder and use them to estimate asymptotically irrevelant terms to arbitrary order.
	While most of our arguments are not specific to~gPAM, we also outline how to adapt those that are.
\end{abstract}

\setcounter{tocdepth}{4}
\tableofcontents

\renewcommand\thesection{I}

\section{Introduction}

We start this introduction with a review of the history of (precise) Laplace asymptotics and mention some of their applications~(subsection~\ref{sec:history_applications}). Our contribution is the development of the Laplace method for the generalised Parabolic Anderson Model (gPAM), a singular stochastic PDE, so we collect its well-posedness and large deviation results in subsection~\ref{intro:eq_ldp}. We then state our main results in full detail in subsection~\ref{intro:sec:results} and indicate how they may be generalised to cover other singular stochastic PDEs~(subsection~\ref{sec:intro_generalisation}). Finally, we discuss the organisation of this article in subsection~\ref{sec:intro_organisation}. 

\subsection{History, relation to other work, and applications}  \label{sec:history_applications}

In 1774, Laplace~\cite{laplace}\footnote{The quoted article is an English translation of the original 1774 article in French. The latter is properly referenced in the article we quoted.} himself devised the method that now bears his name to calculate asymptotics as $\eps \to 0$ of integrals
	\begin{equation*}
		\int_D f(x) \exp\del[3]{-\frac{F(x)}{\eps^2}} \dif x
	\end{equation*}
for intervals $D = [a,b]$ and functions $F$ with a unique minimum inside~$D$. It was generalised later to cover (some) domains~$D \subseteq \R^d$ for $d > 1$ as well and quickly found its way into the canon of classical asymptotic analysis. In parallel to the rapid developments in mathematical physics, Laplace's method gained further traction: it inspired asymptotic approximation techniques that were successfully applied in functional space. More precisely, they allowed to study newly emerging \emph{path integrals} of the form
	\begin{equation*}
		\mathfrak{J}_D(\eps) := \int_D f(x) \exp\del[3]{-\frac{F(x)}{\eps^2}} \dif \P\del[3]{\frac{x}{\eps}}
	\end{equation*}
as popularised by Glimm and Jaffe~\cite{glimm_jaffe} when $\eps \to 0$, where $D$ is a Borel subset of a Banach space~$\CX$, $f$ and $F$ are real-valued functions on~$\CX$, and $\P$ is a probability measure on~$\CB(\CX)$.\footnote{One also requires sufficient regularity of $f$, $F$, and $\partial D$, the boundary of~$D$.} In this direction, pioneering work was done by Feynman~\cite{feynman-hibbs} in his research on quantum mechanics in the 1960's.

Probabilists joined the venture in the 1960's. Building on earlier works of Cramer, Varadhan's theory of \emph{large deviations}~\cite{varadhan66} provides a powerful framework to set up Laplace's method in infinite-dimensional spaces and, in turn, allowed to calculate $\log$-asymptotics of $\mathfrak{J}_\CX(\eps)$ as $\eps \to 0$. 
A large deviation principle for rescaled Wiener measure was given in \cite{MR201892}, as were precise asymptotic results for Wiener integrals of the form \eqref{intro:functional_general}, with $X^\eps = \eps w$ in notation introduced below. 
We fix the following notation: For a family~$(X^\eps: \eps \in [0,1])$ of $\CX$-valued random variables that satisfies a large deviation principle~(LDP), we set\footnote{The restriction to $f \equiv 1$ is for simplicity only. See also~\thref{intro:rmk:f} below.}
		\begin{equation}
			J(\eps) 
			:= \E\sbr[3]{\exp\del[3]{-\frac{F(X^\eps)}{\eps^2}}} 
			\label{intro:functional_general}
		\end{equation}
where $F: \CX \to \R$ is assumed to have $N+3$ Fr\'{e}chet derivatives for some $N \geq 0$. \emph{Precise} Laplace asymptotics, eponymous for this article, assert that~$J$ has an expansion of type
	\begin{equation}
		J(\eps)
		=
		\exp(-c\eps^{-2})(a_0 + \eps a_1 + \ldots + \eps^N a_N + o(\eps^N)) \quad \text{as} \quad \eps \to 0
		\label{intro:asymptotic_expansion}
	\end{equation}
with explicitly given coefficients~$c, a_0, \ldots ,a_N \in \R$. 
The hypotheses~\ref{ass:h1} to~\ref{ass:h4} on~$F$ -- detailed in~\thref{thm:laplace_asymp} below -- under which this statement is true have found their definite form in the works of Ben Arous~\cite{ben_arous_laplace} on path space~$\CX = \CC([0,1];\R^d)$: Building on foundational work by Azencott~\cite{azencott} (in the elliptic case), he establishes~\eqref{intro:asymptotic_expansion} for finite-dimensional diffusions
	\begin{equation}
		\dif X_t^\eps = \eps \sigma(X_t^\eps) \dif w_t + b(\eps,X_t^\eps) \dif t, \quad X^\eps_0 = x, 
		\label{intro:sde_rde}
	\end{equation}
where $w$ is an $r$-dimensional Brownian motion (BM) and $\sigma$ and $b$ satisfy suitable boundedness and differentiability assumptions.
All of the previous statements, and many more, are surveyed in an extensive review article by Piterbarg and Fatalov~\cite{piterbarg_vatalov}.
We further note that sharp Laplace asymptotics~\eqref{intro:asymptotic_expansion} for some classical stochastic PDEs have been established by Rovira and Tindel~\cite{rovira_tindel_parabolic,rovira_tindel_hyperbolic}. See also Albeverio et al.~\cite{albeverio_small_noise_1,albeverio_small_noise_2}. 

The advent of pathwise solution techniques initiated by Lyons's theory of~\emph{rough paths}~\cite{lyons} has brought new momentum to the Laplace method. Starting with Aida~\cite{aida2007}, the Japanese school of stochastic analysis has spearheaded these developments: Inahama~\cite{inahama_taylor, inahama06} and then Inahama and Kawabi~\cite{inahama_kawabi_abel,inahama_kawabi_ldp,inahama_kawabi} consider the Stratonovich version of~\eqref{intro:sde_rde} in infinite-dimensional Banach spaces in which the classical theory of stochastic differential equations (SDEs) breaks down. They bypass this shortcoming by referring to rough paths theory -- which \emph{does} work in infinite dimensions -- and obtain the expansion~\eqref{intro:asymptotic_expansion} in this context. 
In another article, Inahama~\cite{inahama_rde_fbm} has considered~\eqref{intro:sde_rde} for $w = w^H$, an $r$-dimensional fractional Brownian motion (fBM) of Hurst parameter~$H \in (\nicefrac{1}{4},\nicefrac{1}{2}]$ with rough path lift~$\mathbf{W}^H$. More precisely, he considers the solution~$X^\eps$ of the rough differential equation~(RDE)
	\begin{equation}
		\dif X_t^\eps = \eps \sigma(X_t^\eps) \dif \mathbf{W}^H_t + b(\eps,X_t^\eps) \dif t, \quad X_0^\eps = 0.
		\label{inahama:fBM_eq}
	\end{equation} 
When $b$ and $\sigma$ are sufficiently smooth with bounded derivatives, he establishes~\eqref{intro:asymptotic_expansion} under (essentially) the same hypotheses~\ref{ass:h1}~to~\ref{ass:h4} for~$\CX = \CC^{q-\operatorname{var}}([0,1];\R^d)$ and $q > H^{-1}$. 

In recent years, insights from rough paths theory have contributed to important breakthroughs in stochastic PDEs: In the seminal article~\cite{hairer_kpz}, Hairer derives an \emph{intrinsic} local solution to the KPZ equation and then developed his ideas into the theory of \emph{regularity structures}~\cite{hairer_rs}, a far-reaching generalisation of rough paths (RP) theory. --- In this article, we consider a popular example that Hairer's theory allows to treat -- the generalised 2D Parabolic Anderson Model (gPAM) formally given by the equation
	\begin{equation*}
		\partial_t \hat{u}^{(\eps)} = \Delta \hat{u}^{(\eps)} + g(\hat{u}^{(\eps)})\del[1]{\eps \xi  - \eps^2 \glqq \infty \grqq g'(\hat{u}^{(\eps)})}, \quad \hat{u}(0,\cdot) = u_0,
	\end{equation*}
where~$\xi$ is a spatial white noise (SWN) -- and establish the expansion~\eqref{intro:asymptotic_expansion} for $X^\eps = \hat{u}^{(\eps)}$. See subsection~\ref{intro:eq_ldp} for a detailed discussion of~gPAM and its large deviation behaviour and subsection~\ref{intro:sec:results} for a precise statement of our results. 
In subsection~\ref{sec:intro_generalisation}, we will explain our focus on gPAM, point out the difficulties in taking the solutions of other singular stochastic PDEs for~$X^\eps$, and indicate how to overcome them.
 
We emphasise that the arguments presented in this article immediately generalise Inahama's result~\cite{inahama_rde_fbm} for the specific driver~$\mathbf{W}^H$ in~\eqref{inahama:fBM_eq} to \emph{any} Gaussian rough path~\cite{GaussI,friz-victoir}. 
On the surface, the restriction in Inahama's work is due to its heavy reliance on a variation embedding of $\CH^H$, the Cameron-Martin space of $w^H$, due to Friz and Victoir~\cite{friz_victoir_variation_embedding} -- but ultimately can be traced further: the embedding allows him to identify a certain operator (below denoted by $D^2(F \circ \Phi \circ \LL)(\sh)$) as Hilbert-Schmidt and then resort to the classical correspondence
	\begin{center}
		Hilbert-Schmidt operators $\longleftrightarrow$ Carleman-Fredholm determinants $\longleftrightarrow$ exponential integrability	
	\end{center}
for elements in the second Wiener chaos, see for example Janson~\cite[Chap. 6]{janson}. In contrast, our argument is solely based on large deviations\footnote{Ben Arous~\cite{ben_arous_laplace} also used a large deviation argument. See~\thref{rmk:ben_arous_ldp} for further comments in this direction.} and therefore bypasses this correspondence completely, see subsection~\ref{sec:exp_integr_quadr} for details and~subsection~\ref{sec:intro_organisation} for an outline of the argument. In turn, we do not have to rely on the variation embedding specific to~$\CH^H$.

A Laplace method, in spirit of Azencott~\cite{azencott1985}, think asymptotic evaluation of $\P(X_1^\eps \in A)$, rather than $\E\sbr[1]{e^{-F(X^\eps)/\eps^2}}$ as in ~\cite{azencott, ben_arous_laplace},  on rough path type spaces has also been used in ~\cite{friz_gassiat_pigato} to analyze option price asymptotics of \emph{rough volatility}, cf.~\cite{rough_vol_rs} and also Section 14.5 in \cite{friz-hairer} for context. Related ``pre-rough'' works, also with applications in finance, include Kusuoka and Osajima \cite{kusuoka2008remark, osajima2015general}.

\paragraph*{Applications.}

Beside being interesting in their own right, precise Laplace asymptotics have applications in various fields of stochastic analysis and mathematical physics. 

In that regard, we mention a series of articles by Aida~\cite{aida2003,aida2007,aida2009_1,aida2009_2,aida2012} concerned with~\emph{semi-classical analysis}.
More specifically, he studies the lowest eigenvalue $E_0(\lambda)$, or ground state, of a Schr\"odinger-type operator $-L_{\l,V} = -L + V_\l$ on an abstract Wiener space~$(B,H,\mu)$. When $L$ is the Ornstein-Uhlenbeck operator on $L^2(B,\mu)$ and~$V_\l = \l V(\l^{-\nicefrac{1}{2}}\cdot)$ a rescaled potential on~$B$, asymptotics like~\eqref{intro:asymptotic_expansion} are instrumental in his work for proving a lower bound for the so-called semi-classical limit\footnote{In general, semi-classical physics is concerned with the behaviour of quantum mechanical equations as Planck's constant $\hbar$ goes to~$0$. This can be phrased as an equivalent problem for which a parameter~$\l$ tends to~$\infty$.} of $E_0(\l)$ as~$\l \to \infty$. 

Another major field of application are precise \emph{heat kernel asymptotics} in the small time limit. Consider the Stratonovich SDE given by
	\begin{equation}
		\dif X_t^\eps = \eps \sum_{i=1}^r V_i(X_t^\eps) \circ \dif w_t^i + \eps^2 V_0(X_t^\eps)\dif t, \quad X_0^\eps = z \in \R^d,
		\label{intro:sde_heat_kernel}
	\end{equation}
where $w$ is an $r$-dim. BM and $V_0, \ldots, V_r$ are smooth vector fields (VF) on $\R^d$ satisfying Hörmander's bracket condition. Under these assumptions, it is well-known that~$X^1_t$ has a smooth density~$p_t(z,\cdot)$ w.r.t.~$\lambda^d$, the $d$-dim. Lebesgue measure, for each $t > 0$. From an analytical viewpoint, the generator $\eps^2 \CL$ of $X^\eps$ (where $\CL = V_0 + \nicefrac{1}{2} \sum_{i=1}^r V_i^2$) gives rise to the parabolic PDE~$\partial_t f = \CL f$ of which $p_\cdot(z,\cdot)$ is a fundamental solution -- hence the terminology~\emph{heat kernel}. Ben Arous~\cite{ben_arous_heat_kernel}, building on earlier works by Azencott, Bismut and other, obtains the asymptotic expansion, outside the cutlocus, 
	\begin{equation}
		p_t(z,\bar{z}) 
		=
		t^{-\nicefrac{d}{2}} \exp\del[3]{-\frac{D(z,\bar{z})}{2t}}\del[1]{a_0(z,\bar{z}) + t a_1(z,\bar{z}) + \ldots + t^N a_N(z,\bar{z}) + o(t^N)} \quad \text{as} \quad t \to 0
		\label{eq:heat_kernel_asymptotics}
	\end{equation}
with the following rough strategy of proof: first, Brownian scaling implies that~$X^{\eps}_1 = X^1_{\eps^2}$ in distribution, so the small time problem can be transformed into a small noise problem. Secondly, the \emph{Laplace method} is employed to expand the Fourier transform of~$p_{\eps^2}(z,\cdot)$. Fourier inversion then leads to the expansion~\eqref{eq:heat_kernel_asymptotics}, which is where one uses non-degeneracy of the Malliavin matrix. See also Watanabe, Takanobu \cite{MR877589,MR1354169} and Kusuoka--Stroock \cite{MR1120913} for some representative works on the applications of Malliavin calculus to precise asymptotics of large classes of Wiener functionals. 
 
In the Young regime, more precisely when~$w = w^H$ in~\eqref{intro:sde_heat_kernel} is a fBM with Hurst index~$H > \nicefrac{1}{2}$, Baudoin--Hairer~\cite{MR2322701} established the existence of a smooth density. Baudoin--Ouyang~\cite{baudoin_ouyang} as well as I\-na\-ha\-ma~\cite{inahama_heat_kernel_young} then obtained precise heat kernel asymptotics for said density.

These results have subsequently been generalised to the rough regime by  I\-na\-ha\-ma, first for $H \in (\nicefrac{1}{3},\nicefrac{1}{2}]$ under ellipticity~\cite{inahama_kernel_asymptotics_1}, followed by Inahama and Naganuma~\cite{inahama_kernel_asymptotics_2} under Hörmander's condition and for~$H \in (\nicefrac{1}{4},\nicefrac{1}{2}]$, which is the decisive regime for which this can be expected \cite{MR2680405}.

In the case of singular stochastic PDEs, the existence of densities has been settled by Cannizzaro, Friz, and Gassiat~\cite{cfg} for gPAM and by Gassiat and Labb\'{e}~\cite{gassiat_labbe} for the~$\Phi^4_3$ equation. Sch\"onbauer~\cite{schoenbauer} has then generalised these results in the most general regularity structures framework. 
In all these singular SPDE works, the obstacle in obtaining smoothness is the lack of moment estimates for the inverse of the Malliavin matrix. (In a rough paths / SDE context, this was overcome in \cite{CLL,cass2015smoothness}; a meaningful extension to the singular SPDE setting remains an open problem.) 

Our work, in a similar setting as~\cite{cfg}, constitutes a first example of a family of measures on infinite-dimensional space, well-defined only through singular SPDE theory, for which large deviations can be extended to full Laplace asymptotics. Techniques from singular stochastic PDEs have recently been used to construct the~$\Phi^4_3$ measure~\cite{10.1214/17-AOP1212, barashkov_gubinelli} and the Sine-Gordon measure, for which large deviation results (in the semi-classical limit) were recently shown in~\cite{barashkov_phd_thesis}; it is likely that the methods of this paper will also be of relevance in establishing precise Laplace asymptotics in these models.

\subsection{The generalised Parabolic Anderson Model and its large deviations} \label{intro:eq_ldp}

\emph{Formally}, 2D gPAM\footnote{For further motivations of linear PAM as well as its relation to (spatially) discrete models, the reader is directed to the monograph of K\"onig~\cite{koenig}.} is given by the stochastic~PDE
	\begin{equation}
		\begin{cases}
		\partial_t u^{\eps}(t,x) & = \Delta u^{\eps}(t,x) + g(u^{\eps}(t,x))\eps \xi(x) \\
		u^{\eps}(0,x) & = u_0(x)
		\end{cases}
		, \quad (t,x) \in [0,\infty) \x \T^2
		\tag{gPAM$_\eps$},
		\label{gpam}
	\end{equation}
where the initial datum $u_0 \in \CC^\eta(\T^2)$ for $\eta \in (\nicefrac{1}{2},1)$ is fixed, $g \in \CC_b^{N + 7}(\R,\R)$ with $N \geq 0$\footnote{For Laplace asymptotics, assuming differentiability of order~$N+3$ is commonplace. The additional~\enquote{$+4$} is due to regularity structure (and reminiscient of rough paths) theory, see~\thref{rmk:order_of_diffb} for details.}, 
$\eps \in I := [0,1]$, and $\xi$ is the centred Gaussian field with Cameron-Martin space $\CH := L^2(\T^2)$ and formal covariance structure $\rho(x,y) = \d(x-y)$ known as spatial white noise~(SWN). The covariance correctly suggests that SWN is a highly singular object: it can be proved that $\xi \in \CC^{-\frac{d}{2}-\kappa}(\T^d)$ a.s. for each $\kappa > 0$ and spatial dimension $d$. From Schauder theory, one then expects $u^{\eps}$ (and a fortiori $g(u^{\eps})$) to be in $\CC^{2-\frac{d}{2}-\kappa}$. As is well-known from harmonic analysis, however, the product $g(u^{\eps}) \xi$ is well-defined if and only if 
	\begin{equation*}
		\del[2]{2-\frac{d}{2}-\kappa} + \del[2]{-\frac{d}{2}-\kappa}  = 2 - d -2\kappa > 0,
	\end{equation*}
which fails in case $d = 2$ and thus proves \eqref{gpam} ill-posed. At the expense of modifying the equation~\eqref{gpam} in a \emph{renormalisation procedure}, this malaise has been overcome with Hairer's advent of \emph{regularity structures}.\footnote{\emph{Paracontrolled calculus} developed by Gubinelli, Imkeller, and Perkowski \cite{gip} provides an alternative solution theory.} In our context, \cite[Thm.~1.11]{hairer_rs} reads:

\begin{theorem} \label{intro:ex_sol_hairer}
	Fix~$\eps > 0$ and let $\xi_\d := \xi * \rho_\d$, where $\rho_\d := \d^{-2} \rho(\d^{-1} \cdot)$ is a spatially rescaled version of a fixed mollifier~$\rho$. Then, there exists a choice of constants~$(\fc_\d: \d > 0)$ and a stopping time~$T^\eps > 0$ such that for each $u_0 \in \CC^{\eta}(\T^2)$ and each~$T > 0$, the classical solution~$\hat{u}^{\eps}_{\xi_\d}$ to the \emph{renormalised} equation
	\begin{equation}
		\begin{cases}
			\partial_t \hat{u}^{\eps}_{\xi_\d} & = \Delta \hat{u}^{\eps}_{\xi_\d} + g(\hat{u}_{\xi_\d}^{\eps})\del[1]{\eps \xi_\d - \eps^2\fc_{\d} g'(\hat{u}_{\xi_\d}^{\eps})} \\
			\hat{u}_{\xi_\d}^{\eps}(0,\cdot) & = u_0
		\end{cases}
		\tag{$\widehat{\text{gPAM}}_\eps$},
		\label{rgpam}
	\end{equation}
	converges on $\{T^\eps > T\}$ in probability, as $\d \to 0$, ~to a unique limit
	\begin{equation*}
		\hat{u}^\eps \in \CX_T := \{u \in \CC([0,T],\CC^\eta(\T^2)): u(0,\cdot) = u_0\}
	\end{equation*}
	which is independent of the mollifier~$\rho$.
\end{theorem}
The regularity structures set-up to be used in this article will be introduced in Appendix~\ref{app:background_pam}. Its main idea is the following lifting procedure: at first, one constructs a finite number of explicit stochastic objects taylored to the structure of \eqref{gpam}; their collection is called the \emph{canonical lift} $\bz^{\eps\xi_\d} = \LL(\eps\xi_\d)$ of the driving noise $\eps \xi_\d$ to the \emph{model space}~$\MM$. Secondly, one establishes the continuity of the solution map $\Phi: \MM \to \CX_T$ and, thirdly, builds a renormalised version $\hbz^{\eps\xi_\d} \in \MM$ from~$\bz^{\eps\xi_\d}$ such that $\hat{u}^{\eps}_{\xi_\d} = \Phi(\hbz^{\eps\xi_\d})$. At last, one proves that the limit
	\begin{equation*}
		\d_\eps\hbz 
		:= \lim_{\d \to 0} \hbz^{\eps\xi_\d} 
		= \lim_{\d \to 0} \d_\eps\hbz^{\xi_\d} 
	\end{equation*}
exists in probability in $\MM$ -- and finally \emph{defines} $\hat{u}^{\eps} := \Phi(\d_\eps\hbz)$. Hairer carried out this programme for the equation~\eqref{gpam}, amongst others, in the seminal article~\cite{hairer_rs}. In a series of papers~\cite{bhz, chandra-hairer,rs_renorm}, it has subsequently been developed into a black-box theory to cover a large class of singular stochastic PDEs. 

With the (local in time) existence of solutions secured, Hairer and Weber have subsequently investigated the large deviation behaviour of~$(\hat{u}^{\eps}: \eps \in I)$. More precisely, they obtained the following result~\cite[Thm.~4.4]{hairer_weber_ldp}, where~\enquote{$u^{\deathsmall}$} denotes a \enquote{graveyard path}:

\begin{theorem}\label{thm:ldp}
	The family $(\hat{u}^{\eps}: \eps \in I)$ satisfies a large deviation principle (LDP) in $\CX_T \cup \{u^{\deathsmall}\}$ with good rate function (RF)~$\JJ$, where~$\hat{u}^\eps := u^{\deathsmall}$ if~$T^\eps < T$.\footnote{In case~$T > T^\eps$, it seems more natural to solve the stochastic PDE up to time~$T^\eps$ and only set~$\hat{u}^\eps(t,\cdot) := \death$ for~$t \in (T^\eps,T]$, where \enquote{$\death$} is a graveyard \emph{state}. However, we follow Hairer and Weber~\cite{hairer_weber_ldp} in discarding all such trajectories in the first place by sending them to the graveyard path~$u^{\deathsmall}$.}
\end{theorem}

\begin{remark} \label{intro:footnote:ldp}
	Actually, Hairer and Weber have dealt with the $\Phi^4_3$ equation rather than \eqref{gpam} but their methods generalise to all equations that fall into the framework of Hairer's earlier work~\cite{hairer_rs}. For the reader's convenience, we will spell out their argument for the case of \eqref{gpam} in Appendix~\ref{app:ldp_pam}. 
	A more detailed formulation of~\thref{thm:ldp} is provided in~\thref{app:thm_ldp_gpam}.
	Regarding explosion times, see section~\ref{sec:explosion} for details. 
\end{remark}
 
Large deviations provide a first answer concerning the asymptotics of $J$ in~\eqref{intro:functional_general} on \emph{logarithmic} scale: for $X^\eps := \hat{u}^{\eps}$ and $F: \CX_T \to \R$ continuous and bounded, Varadhan's lemma~\cite[Thm.~2.1.10]{deuschel-stroock} implies that
	\begin{equation}
		J(\eps) = \exp\del[3]{-\frac{\inf_{\CX_T} (F + \JJ) + o(1)}{\eps^2}}, \quad \eps \to 0.
		\label{eq:varadhan:2}
	\end{equation}
We will significantly sharpen~\eqref{eq:varadhan:2} by obtaining the \emph{full expansion}. 	

\noindent

\newpage
\subsection{Main results} \label{intro:sec:results}

Let $h \in \CH$ and consider the equation
	\begin{equation}
		(\partial_t - \Delta) \hat{u}^{\eps}_{\xi_\d;h} 
		= g\del[1]{\hat{u}^{\eps}_{\xi_\d;h}}\del[1]{\eps \xi_\d  + h - \eps^2 \fc_{\d} g'\del[1]{\hat{u}^{\eps}_{\xi_\d;h}}}, 
		\quad 
		\hat{u}^{\eps}_{\xi_\d;h}(0,\cdot) = u_0,
		\label{eq:approx_renorm_shifted}
	\end{equation}
a shifted version of~\eqref{rgpam}. In analogy with \thref{intro:ex_sol_hairer}, we set $\hat{u}^{\eps}_h := \lim_{\d \to 0} \hat{u}^{\eps}_{\xi_\d;h}$ (in prob.~in~$\CX_T$). 

Our main result, eponymous for this work, reads as follows:

\begin{thm}[Precise Laplace asymptotics]\label{thm:laplace_asymp}
	Let~$g \in \CC_b^{N + 7}$. For~$T > 0$, we assume that ~$F: \CX_T \to \R$ satisfies the following hypotheses:
	\begin{enumerate}[label={(H\arabic*)}]
		\item \label{ass:h1} $F \in \C_{b}(\CX_T;\R)$, i.e. $F$ is continuous and bounded from $\CX_T$ into $\R$. 
		\item \label{ass:h2} The functional $\FF$ given by 
			\begin{equation}
				\FF := (F \circ \Phi \circ \LL) + \II: \ \CH \to (-\infty,+\infty]
			\end{equation}
		attains its \emph{unique} minimum\footnote{In analogy with previous works on Laplace asymptotics, for example~Ben~Arous~\cite[sec.~1.4]{ben_arous_laplace}, we emphasise that the theorem remains valid if there are \emph{finitely many} minimisers: One simply localises around each of them and proceeds with the local analysis. The value of the coefficients~$a_0, \ldots, a_N$, however, will of course change in that case.} at $\sh \in \CH$. Here, $\II$ denotes Schilder's rate function, i.e. 
			\begin{equation}
			\II(h) = \frac{\norm{h}^2_\CH}{2} \thinspace \mathbf{1}_{h \in \CH}  + \infty \thinspace \mathbf{1}_{h \notin \CH}.
			\label{eq:schilders_rf}
			\end{equation}
			\item \label{ass:h3} The functional $F$ is $N + 3$ times Fr\'{e}chet differentiable in a neighbourhood $\CN$ of $(\Phi \circ \LL)(\sh)$. 
			In addition, there exist constants $M_1,\ldots,M_{N+3}$ such that
			\begin{equation}
				\abs[0]{D^{(k)}F(v)[y,\ldots,y]} \leq M_k \norm{y}_{\CX_T}^k, \quad k=1,\ldots,N+3,
				\label{thm:laplace_asymp:ass:h3:boundedness}
			\end{equation}
		holds for all~$v \in \CN$ and $y \in \CX_T$.
		\item The minimiser $\sh$ is \emph{non-denegerate}. More precisely, $D^2\FF\sVert_\sh > 0$ in the sense that for all $h \in \CH\setminus\{0\}$ we have
			\begin{equation}
				D^2\FF\sVert_{\sh}[h,h] > 0.
			\end{equation}
		 \label{ass:h4}
	\end{enumerate}	
	Then, we have the precise Laplace asymptotics	
	\begin{equation}
	J(\eps)
	= 
	\E\sbr{\exp\del{-\frac{F(\hat{u}^{\eps})}{\eps^2}}\1_{T^\eps > T}} = \exp\del[3]{-\frac{\FF(\sh)}{\eps^2}}(a_0 + \eps a_1 + \ldots + \eps^N a_N + o(\eps^N)) \quad \text{as} \; \, \eps  \to 0,
	\label{thm:laplace_asymp:expansion}
	\end{equation}
where 
	\begin{equation*}
		a_0 = \E\sbr[2]{\exp\del[2]{-\frac{1}{2} \hat{Q}_\sh}}, 
		\quad \hat{Q}_\sh := \partial_\eps^2\sVert[0]_{\eps = 0} F(\hat{u}^\eps_\sh).
	\end{equation*} 	
The other coefficients $(a_k)_{k=1}^N \in \R$ are given explicitly in~section~\ref{sec:asymp_coeff}, more specifically in~\thref{prop:coeff_asymp_exp}.
\end{thm}

\begin{remark}\label{rmk:ld_factor}
	Observe that, under our hypothesis~\ref{ass:h2}, the factor~$\exp\del[1]{-\eps^{-2} \FF(\sh)}$ in formula~\eqref{thm:laplace_asymp:expansion} coincides with the RHS of \eqref{eq:varadhan:2}. In other words, we recover the previously known \emph{large deviation factor}.
\end{remark}

\begin{remark} \label{intro:rmk:f}
	In many works on precise Laplace asymptotics, the functional $J$ is slightly more general than~\eqref{intro:functional_general}, namely 
	\begin{equation*}
	J(\eps) = \E\sbr[3]{f(X^\eps)\exp\del[3]{-\frac{F(X^\eps)}{\eps^2}}\1_{T^\eps > T}},
	\end{equation*}
	where $f: \CX_T \to \R$ is continuous, bounded, as well as $(N + 1)$ times Fr\'{e}chet differentiable and positive in a neighbourhood of $(\Phi \circ \LL)(\sh)$. Since the presence of $f$ does not add conceptual difficulty, we work with $f \equiv 1$ to streamline the presentation. However, our work can easily be amended for general~$f$.
\end{remark}

\begin{remark} \label{rmk:constant_coeff}
	Note that~$\hat{Q}_\sh$ is in the second inhomogeneous Wiener chaos with non-trivial components only in~$\CH_k$ for~$k \in \{0,2\}$. As such, it admits the representation~$\hat{Q}_\sh = \E\sbr[1]{\hat{Q}_\sh} + I_2(A)$ for some Hilbert-Schmidt operator~$A$ and~$I_2$ the second Wiener-It\^{o} isometry. From general principles (see~\cite[Thm.~$6.2$]{janson} and~\cite[Thm.~$9.2$]{simon_trace_ideals}), this immediately implies that
	\begin{equation}
		a_0 
		= e^{-\nicefrac{1}{2} \E\sbr[0]{\hat{Q}_\sh}} \operatorname{det}_2(\operatorname{Id}_\CH + A)^{-\nicefrac{1}{2}}
		\label{eq:repr_a0_Carleman_Fredholm}
	\end{equation}
	where~$\operatorname{det}_2$ denotes the Carleman-Fredholm determinant.
	In recent works on precise Laplace asymptotics such as~\cite{inahama_kawabi}, it is shown (in our notation) that
	\begin{equation*}
		A[k,k] = Q_\sh(\LL(k)), \quad
		\E\sbr[1]{\hat{Q}_\sh} = \operatorname{Tr}(A-\tilde{A}) + \lambda
	\end{equation*}
	for any~$k \in \CH$, some explicit~$\lambda \in \R$, and some explicit Hilbert-Schmidt operator~$\tilde{A}$ that renders~$A-\tilde{A}$ trace-class.
	Such information is interesting from the viewpoint of semi-classical analysis but requires considerably more effort
	\begin{enumerate}[label=(\arabic*)]
		\item due to the need for renormalisation (recall that~$\hat{Q}_\sh$ is \emph{quadratic} in~$\xi$) and \item since one needs more than just Cameron-Martin regularity of the minimiser~$\sh$ (see our sketch in sec.~\ref{sec:intro_generalisation} below) to show that~$A - \tilde{A}$ is trace-class.
	\end{enumerate}
	These problems are addressed in great detail in a recent follow-up article~\cite{klose_arxiv} of the second named author.
\end{remark}

Along the way, and of paramount importance in the proof of~\thref{thm:laplace_asymp}, we obtain a formula that is commonly referred to as \emph{(stochastic) Taylor expansion}.\footnote{Actually, the expansion is fully deterministic and the term \enquote{stochastic} only historic: It was accurate in the work of Azencott~\cite{azencott} but became outdated by the emergence of pathwise techniques. This was already remarked by Inahama and Kawabi~\cite{inahama_kawabi} in the context of rough paths. Henceforth we refrain from using the now misleading term~\enquote{stochastic}.} 
We believe that this result is of independent interest and therefore consider it the second contribution of this paper:

\begin{thm}[Taylor expansion in the noise intensity] \label{thm:stoch_taylor_gpam} 
	Let $h \in \CH$ and $g \in \CC_b^{\ell + 4}$ for some $\ell \geq 1$. Then, for each $T > 0$ there exists an~$\eps_0 = \eps_0(T) > 0$ such that with $I_0 := [0,\eps_0)$, the function
	\begin{equation}
		\hat{u}_h^{\bullet}: I_0 \to \CX_T, \quad \eps \mapsto \hat{u}_h^{\eps}
		\label{thm:stoch_taylor_gpam:diffb} 
	\end{equation}
	is $\ell$ times Fr\'{e}chet differentiable. By Taylor's formula, the expansion
	\begin{equation}
		\hat{u}^{\eps}_h
		=
		w_h
		+
		\sum_{m=1}^{\ell - 1} \frac{\eps^m}{m!} \hat{u}^{(m)}_{h}
		+ \hat{R}^{(\ell)}_{h,\eps}, \quad \eps \in I_0,
		\label{thm:stoch_taylor_gpam:exp}
	\end{equation}
	 then holds with the following summands: 
	\begin{enumerate}[label=(\arabic*)]
		\item \label{thm:stoch_taylor_gpam:1} The term $w_h := \Phi(\LL(h))$ is the unique solution to the \emph{deterministic} PDE
		\begin{equation}
		(\partial_t - \Delta) w_h = g(w_h) h, \quad w_h(0,\cdot) = u_0.
		\label{eq:wh}
		\end{equation}
		\item \label{thm:stoch_taylor_gpam:2} The terms~$\hat{u}^{(m)}_{h}$, $m \in [\ell - 1]$,\footnote{We use the notation~$[n] := \{1,\ldots,n\}$ for any~$n \in \N$ throughout the article.} and $\hat{R}^{(\ell)}_{h,\eps}$ are given by
		\begin{equation*}
			\hat{u}^{(m)}_{h}
			:=
			\partial_\eps^{m}\sVert_{\eps = 0} \hat{u}^{\eps}_h, \qquad
			\hat{R}^{(\ell)}_{h,\eps}
			:=
			\int_0^1 \frac{(1-s)^{\ell-1}}{(\ell-1)!} \partial_r^{\ell}\sVert[0]_{r = s\eps} \sbr[1]{\hat{u}^r_h} \dif s,
		\end{equation*}
	\end{enumerate}
	In addition, the terms in the expansion~\eqref{thm:stoch_taylor_gpam:exp} have the following properties:\footnote{We formulate the estimates w.r.t. a generic model~$\bz$ and the corresponding quantities~$u_h^\eps(\bz)$ and~$u_h^{(m)}(\bz)$: see~Def.~\ref{def:terms_taylor_exp} on p.~\pageref{def:terms_taylor_exp} below for their precise definition. The expansion in~\eqref{thm:stoch_taylor_gpam:exp} then holds for $\eps < \eps_0(T,\bz)$. Since $\hat{u}_h^\eps = u_h^\eps(\hbz)$, $\hat{u}_h^{(m)} = u_h^{(m)}(\hbz)$, and~$\eps_0(T) = \eps_0(T,\hbz)$, this is consistent with the first part of the theorem where~$\bz = \hbz$.}
	\begin{enumerate}[label=(\roman*)]
		\item \label{thm:stoch_taylor_gpam:i} The terms~$u^{(m)}_{h}(\bz)$ and the remainder~$R^{(\ell)}_{h,\eps}(\bz)$ are continuous functions of the model~$\bz$. In addition, the estimates\footnotemark
			\begin{equation}
				\norm[0]{u^{(m)}_{h}(\bz)}_{\CX_T} \aac (1 + \barnorm{\bz^{\minus}})^m, \quad
				\norm[0]{R^{(\ell)}_{h,\eps}(\bz)}_{\CX_T} \aac \eps^\ell(1 + \barnorm{\bz^{\minus}})^\ell 
				\label{thm:stoch_taylor_gpam:estimate}
			\end{equation}
		hold: The first one uniformly over all~$h \in \CH$ with~$\norm[0]{h}_\CH \leq \rho$, the second \emph{additionally} over all~$(\eps,\bz^{\minus}) \in I \x \MM_-$ with~$\eps \barnorm{\bz^{\minus}} < \rho$, both for each~$\rho > 0$.
		\item \label{thm:stoch_taylor_gpam:ii} The terms~$u^{(m)}_{h}(\bz)$ are \emph{$m$-homogeneous} w.r.t. dilation in the model~$\bz$, that is
		\begin{equation*}
			\eps^m u^{(m)}_{h}(\bz) = u^{(m)}_{h}(\d_\eps\bz).
		\end{equation*}
	\end{enumerate}
	In addition, the terms~$\hat{u}_h^{(m)}$ are limits of~$\hat{u}_{\xi_\d,h}^{(m)}$ as~$\d \to 0$ in the spirit of~\thref{intro:ex_sol_hairer} above. 
	See~\thref{prop:expl_eq} below for the \emph{linear} stochastic PDEs that~$\hat{u}_{\xi_\d,h}^{(m)}$ satisfy.
\end{thm}
\footnotetext{By \enquote{$\aac$} we denote estimates up to a constant. If we wish to specify its dependence on a parameter~$\a$, we write~\enquote{$\aac_\a$}.
The \enquote{norm} $\barnorm{\cdot}$ appearing on the RHS is the \emph{homogeneous} norm for a minimal admissible model~$\bz^{\minus} \in \MM_-$, see subsection~\ref{app:sec:adm_models} in the appendix for details.
For now, one may think of the homogeneous rough path norm~$\threebars \boldsymbol{X} \threebars_\a = \norm[0]{X}_\a + \norm[0]{\mathbb{X}}_{2\a;2}^{\nicefrac{1}{2}}$ for an $\a$-Hölder,~$2$-step rough path~$\boldsymbol{X} = (X,\mathbb{X})$.
}

\subsection{Generalisation to other singular stochastic PDEs} \label{sec:intro_generalisation}

In a series of major recent works, Hairer and co-authors~\cite{bhz, chandra-hairer,rs_renorm} have developed regularity structures into a black-box theory that encompasses a large class of stochastic PDEs. 
In principle, the arguments which we will use in the proofs of theorems~\ref{thm:laplace_asymp} and~\ref{thm:stoch_taylor_gpam} apply to any stochastic PDE that is covered by the afore-mentioned articles.

One reason we concentrate on gPAM in this article is that it allows for a relatively simple analysis via regularity structures, so we can emphasise the main principles that our arguments are based upon.

Another reason is that Cameron-Martin shifts~$\xi \mapsto \xi + h$ of the driving noise of gPAM, as needed in~\eqref{eq:approx_renorm_shifted}, have already been implemented on model space~$\MM$ via the translation operator~$T_h$ in~\cite[sec.~$3.4$]{cfg}.
Its well-posedness is a complicated \emph{analytical} problem, in general, because it requires products of elements in Besov spaces of different $L^p$-scale to exist. 
Even in the case of gPAM, building a solution theory for the shifted equation~\eqref{eq:approx_renorm_shifted} poses the problem to define
	\begin{equation*}
		\mathbf{\Pi}\<11g> := (G\xi) h, \quad G\xi \in \CC^{1-\kappa} = B_{\infty,\infty}^{1-\kappa}, \quad h \in \CH = B_{2,2}^0 \embed \CC^{-1}
	\end{equation*}  
where~$G\xi$ denotes the solution to the linearised equation~($g \equiv 1$);
Cannizzaro, Friz, and Gassiat~\cite[Lem.~A.$2$]{cfg} have managed to obtain the necessary estimates~\enquote{by hand}. 
However, more complicated trees that already appear for the shifted $\Phi^4_3$ or KPZ equation, for example
	\begin{equation*}
		\tau := \scalebox{1.1}{\<phi4>} \rightsquigarrow \mathbf{\Pi}\tau := (Gh)(G\xi)G\del[1]{(G\xi)(G\xi)(Gh)}, \quad 
		\sigma := \scalebox{1.1}{\<kpz>} \rightsquigarrow \mathbf{\Pi}\sigma := \sbr[1]{G_x\del[1]{(G_x \xi)(G_x h)}}^2, 
	\end{equation*}
render ad hoc estimates unwieldy.
Instead, analytic arguments need to be complemented by insights into the algebraic make-up of the underlying regularity structure.

Schönbauer~\cite{schoenbauer}, in contrast, approached the problem from a different perspective. By an ingeneous \enquote{doubling trick}, similar in spirit to previous arguments by Inahama~\cite{inahama_malliavin_diffb} in the rough paths context, he is able to leverage the \emph{probabilistic} estimates of Chandra and Hairer~\cite{chandra-hairer}; thereby, he manages to prove~\cite[Thm.~$3.11$]{schoenbauer} that the action of~$T_h$ on BPHZ lifts of smooth functions is well-defined and continuous almost surely.
While this argument applies in the most general setting we referenced in the beginning of this section, it falls short of proving $T_h$ well-defined and continuous on \emph{generic} models. The latter is crucial for the localisation procedure, see~sec.~\ref{sec:localisation_cameron_martin} below, to work beyond gPAM.

We propose to follow a completely different route using the \emph{structural assumptions on the Laplace functionals} that we consider. As the reader will observe later, we only ever want to translate~$\xi$ in direction~$\sh$, the minimiser of the functional~$\FF$ given in~\ref{ass:h2} above. A priori, it is clear that~$\sh$ is in $\CH$ (otherwise~$\II(\sh) = +\infty$) --- but it is known from variational calculus that minimisers often enjoy better regularity properties, see for example~\cite[Thm.~$14.4.1$]{jost_pde}.
The strategy there -- establish some initial regularity 
and then bootstrap using Schauder estimates -- informs our reasoning. Under the premise that~$g \in \CC_b^\infty(\R)$ and~$u_0 \in \CC^{\infty}(\T^2)$, we sketch the agument that~$\sh \in \CC^\infty(\T^2)$ for gPAM, see~\cite[Thm.~$2$]{klose_arxiv} for a precise statement. 
We stress that this renders the translation operator~$T_\sh$ well-defined \emph{immediately} without any appeal to advanced analytic estimates.

\begin{enumerate}[label=(\arabic*)]
	\item By first-order optimality, $D\FF\sVert[0]_{\sh} \equiv 0$. Equivalently, for all~$k \in \CH$ we have
		\begin{equation}
			\scal{\sh,k}_\CH = DF\sVert[0]_{w_\sh}(v_{\sh,k}), \quad v_{\sh,k} := \partial_\eps\sVert[0]_{\eps=0} w_{\sh + \eps k}
			\label{eq:first_order_optimality}
		\end{equation}
	where~$w_\sh$ is the solution to~\eqref{eq:wh}.
	\item The directional derivative~$v_{\sh,k}$ satisfies the \emph{linear}, inhomogeneous PDE
		\begin{equation}
			(\partial_t - \Delta) v_{\sh,k} = g'(w_\sh) v_{\sh,k} \sh + g(w_\sh) k, \quad v_{\sh,k}(0,\cdot) = 0.
			\label{eq:dir_der_h_k}
		\end{equation}
	If~$\sh$ actually were in~$H^\gamma$, $\gamma \geq 0$\footnote{Recall that~$\sh \in \CH \equiv H^0$, so the statement \emph{is} true when~$\gamma = 0$.}, we expect~$w_\sh \in \CC_T H^{\gamma+1^{-}}$ by the regularising properties of the heat semigroup. In turn, we can then take~$k \in H^{-\theta}$ for some~$\theta > \gamma$ and still make sense of~\eqref{eq:dir_der_h_k} and its solution~$v_{\sh,k}$. Even more so, we can establish the estimate
		\begin{equation}
			\norm[0]{v_{\sh,k}}_{\CC_T^\gamma} \aac_T \norm[0]{k}_{H^{-\theta}}, \quad \CC_T^\gamma := \CC_T H^{-\gamma + \frac{1}{2}}
			\label{eq:dir_der_h_k_est}
		\end{equation}
	via Gronwall's inequality.
	\item We extend the linear, bounded operator~$DF\sVert[0]_{w_\sh}$ from~$\CX_T$ to~$\CC_T^\gamma$ by a density argument, preserving its norm. Altogether, combining~\eqref{eq:first_order_optimality} with the estimate~\eqref{eq:dir_der_h_k_est}, we obtain the estimate
		\begin{equation*}
			\abs[0]{\scal{\sh,k}_\CH} \aac_T \norm[0]{DF\sVert[0]_{w_\sh}}_{\CL(\CX_T,\R)} \norm[0]{k}_{H^{-\theta}} 
		\end{equation*}
	for~$k \in \CC^\infty$ and infer that~$\sh \in H^\theta$ by duality. 
	\item Since~$\theta > \gamma$, the minimiser~$\sh$ has gained regularity in this procedure. Bootstrapping the previous steps implies that~$\sh \in \CC^{\infty}$.
\end{enumerate}
We expect that this argument can be adapted to cover the time-dependency of the Cameron-Martin space~$L^2([0,T] \x \T^3)$ for the $\Phi^4_3$ equation and, similarly, for other singular stochastic PDEs.
Investigating this strategy further is left for future work.

Once $T_\sh$ is made sense of as outlined, the arguments presented in the article at hand can easily be generalised beyond gPAM.

\subsection{Organisation of the article} \label{sec:intro_organisation}

In this subsection, we discuss the organisation of the article at hand and indicate by triangles~\enquote{$\triangleright$} in which (sub-)sections the hypotheses from~\thref{thm:laplace_asymp} are used. We also highlight where we use the large deviation principle (LDP) for $(\d_\eps\hbz^{\minus}: \eps \in I)$ formulated in~\thref{app:thm_ldp_models}.
Note that the boldface terms can be read as an outline for the proof of~\thref{thm:laplace_asymp}. %
	\begin{enumerate}[label=$\bullet$ \textbf{(\arabic*)}, wide=0pt, parsep = 0.5pt, listparindent=0pt]
		\item \label{overview:1} \textbf{Localisation.} In section~\ref{sec:localisation_cameron_martin}, we localise the functional $J$ to a $\rho$-neighbourhood of $\LL_-(\sh)$, the lift of the minimiser~$\sh$ to the space of minimal models~$\MM_-$. By the Cameron-Martin theorem, we then recenter the localised functional~$J_\rho$ to the neighbourhood $\{\eps \barnorm{\hbz^{\minus}} < \rho\}$, at the expense of changing $(F \circ \Phi^{\minus})(\d_\eps\hbz^{\minus})$ to the function ~$\tilde{F}^{\minus}_\Phi(\sh,\eps) := (F \circ \Phi^{\minus})(T_\sh \d_\eps\hbz^{\minus}) - F(w_\sh) + \eps\xi(\sh)$.\footnote{Note that~$(F \circ \Phi^{\minus})(\d_\eps\hbz^{\minus}) = F(\hat{u}^\eps)$ and~$(F \circ \Phi^{\minus})(T_\sh \d_\eps\hbz^{\minus}) = F(\hat{u}_\sh^\eps)$.} \ $\triangleright$ \ref{ass:h1}, (LDP) 
		\item \label{overview:2} \textbf{Taylor expansion in the noise intensity.}\footnote{For a graphical overview of this section's content -- including technical details -- see also~\thref{figure:overview} on p.~\pageref{figure:overview}.} In~\thref{coro:stoch_taylor_gpam_functional}, we obtain a Taylor expansion of the functional~$F(\hat{u}_\sh^\eps) = (F \circ \Phi^{\minus})(T_\sh \d_\eps \hbz^{\minus})$ up to order $N+2$ for small $\eps$.
		In~\thref{coro:stoch_taylor_gpam_functional:properties}, we then see that the Taylor terms and the remainder depend continuously on~$\hbz$ and provide estimates for both. In addition, we show that the former are homogeneous in~$\eps$.
		This result follows easily from~\thref{thm:stoch_taylor_gpam} whose proof we prepare throughout the whole section~\ref{sec:stoch_taylor}. \thinspace $\triangleright$~\ref{ass:h3}
		
		\emph{The proof of~\thref{thm:laplace_asymp} continues in section~\ref{sec:local_analysis}. In order to better grasp its structure, the reader might want to skip the rest of section~\ref{sec:stoch_taylor} in a first reading.}
		\vspace{0.5em}
			\begin{enumerate}[label=$\bullet$ \textbf{(2.\arabic*)}, wide=0pt, parsep = 7pt, listparindent=0pt]
				\item \label{overview:2.1}	\textbf{Abstract fixed-point problem.} In subsection~\ref{sec:abstr_fp_prob}, we set up the fixed-point problem for~$U^\eps$, the abstract representation of~$\hat{u}^{\eps}_\sh$ in a suitable space of modelled distributions. The choice of the latter requires care: although the natural candidate~$\DD^{\gamma,\eta}(T_\sh \d_\eps\hbz)$ has a Banach structure, it is unsuited for taking derivatives in~$\eps$, for it is a different space for each~$\eps \in I$. Instead, we study a family of fixed point (FP) equations parametrised by $\eps$ whose solutions\footnote{Green colour indicates that we are working in the extended gPAM regularity structure, see sec.~\ref{app:rs_background}.}~$\gr{U^\eps}$ live in one fixed Banach space~$\DD^{\gamma,\eta}_{\gr{\CU}}(E_\sh\hbz)$. We then prove that this choice is consistent: both $U^\eps$ and $\gr{U^\eps}$ are reconstructed to~$\hat{u}^{\eps}_\sh$ w.r.t. their underlying models~(\thref{prop:fp_choice}).
				\item \label{overview:2.2}	\textbf{Fr\'{e}chet differentiability of the fixed-point map.} The main result of subsection~\ref{sec:diff_fp_map} is~\thref{prop:der_fp_map}, in which we prove Fr\'{e}chet differentiability of order~$N+3$ for the abstract fixed-point map~$\fI$ that sends $\eps$ to~$\gr{U^\eps}$. In the same theorem, we derive explicit FP equations for the derivatives~$\fI^{(m)} := \partial_\eps^m \fI$ and establish their continuous model dependence. Our main tool here is the Implicit Function Theorem~(IFT), the applicability of which is verified in lemmas~\ref{thm:frechet:gpam} and~\ref{lem:sec_cond_ift}. 
				\item \label{overview:2.3}	\textbf{Taylor expansion in the modelled distribution space.} 
				In~\thref{thm:abstract_taylor}, we obtain a Taylor expansion of~$\gr{U^\eps}$ in the Banach space~$\DD^{\gamma,\eta}_{\gr{\CU}}(E_\sh\hbz)$.
				In subsection~\ref{sec:taylor_exp_mod_distr}, we analyse its terms and connect them to their counterparts in~\thref{thm:stoch_taylor_gpam}. 
				In particular, we obtain estimates for the Taylor terms~(subsection~\ref{sec:estimates_taylor_terms}) and the remainder~(subsection~\ref{sec:estimates_taylor_remainder}) by means of an abstract version of Duhamel's principle~(\thref{lem:inv_theta_op}).
				\item \label{overview:2.4} \textbf{Proof of Taylor  expansion.} In subsection~\ref{sec:pf_stoch_taylor} we synthesise the results of subsections~\ref{sec:abstr_fp_prob} to~\ref{sec:taylor_exp_mod_distr} to give a proof of~\thref{thm:stoch_taylor_gpam}.
				\item \label{overview:2.5} \textbf{Stochastic PDEs for the Taylor terms.}
				We derive stochastic PDEs for the Taylor terms~$\hat{u}^{(m)}_{\xi_\d,\sh}$ 
				from~\thref{thm:stoch_taylor_gpam}. 
				Their limits~$\hat{u}_\sh^{(m)}$ as~$\d \to 0$ appear in the formulas for the coefficients~$a_0, \ldots, a_N$.
			\end{enumerate} \vspace{0.5em}
		\item \label{overview:3} \textbf{Local analysis in the vicinity of the minimiser.} Section~\ref{sec:local_analysis} revolves around the study of~$\tilde{F}^{\minus}_{\Phi}(\sh,\eps)$. The constant term in the Taylor-like expansion for~$(F \circ \Phi^{\minus})(T_\sh \d_\eps\hbz^{\minus})$ compensates the term $-F(w_\sh)$. \ $\triangleright$~\ref{ass:h2},~\ref{ass:h4} \vspace{0.5em}
			\begin{enumerate}[label=$\bullet$ \textbf{(3.\arabic*)}, wide=0pt, parsep = 7pt, listparindent=0pt]
				\item \label{overview:3.1} \textbf{First-order optimality annihilates the linear term.} In subsection~\ref{sec:first_order_vanish}, we prove that the sum of $\eps \xi(\sh)$ and the linear term in the expansion for~$(F \circ \Phi^{\minus})(T_\sh \d_\eps\hbz^{\minus})$ vanishes because $\sh$ is a mi\-ni\-mi\-ser of~$\FF$~(\thref{prop:first_order_optimality}).	\ $\triangleright$~\ref{ass:h2}
				\item \label{overview:3.2} \textbf{Exponential integrability of the quadratic term.} In subsection~\ref{sec:exp_integr_quadr}, we establish the exponential integrability of the quadratic term~$-\nicefrac{1}{2}\hat{Q}_\sh$ in the afore-mentioned expansion of~$(F \circ \Phi^{\minus})(T_\sh \d_\eps\hbz^{\minus})$ (\thref{prop:exp_integr_Q}). We stress that, in contrast to previous works, we do not have to resort to the connection of second Wiener chaos, trace-class operators, and Carleman-Fredholm determinants. Instead, we use a large deviations argument for~$(-\nicefrac{1}{2}Q_\sh(\d_\eps\hbz): \eps \in I)$, itself valid by continuity of~$Q_\sh$ in the model, and then conclude by non-degeneracy of~$\sh$. \ $\triangleright$~\ref{ass:h4}, (LDP)
				\item \label{overview:3.3} \textbf{The coefficients in the asymptotic expansion.} In subsection~\ref{sec:asymp_coeff}, we decompose~$J_\rho$ in four steps that finally lead to the asymptotic expansion claimed in eq.~\eqref{thm:laplace_asymp:expansion}: Its coefficients $a_0, \ldots, a_N$ are finite thanks to the exponential integrability of $-\nicefrac{1}{2}\hat{Q}_\sh$~(\thref{prop:coeff_asymp_exp}). We close our proof of~\thref{thm:laplace_asymp} by showing that the four remainder terms created in the decomposition procedure are of order~$o(\eps^N)$. Our main tool here is a Fernique-type theorem for the minimal BPHZ model~$\hbz^{\minus}$.
			\end{enumerate} \vspace{0.5em}
	\end{enumerate}
	\vspace{-1.5em}
	\begin{enumerate}[label=$\bullet$ \textbf{(\Alph*)}, wide=0pt, parsep = 0.5pt, listparindent=0pt]
		\item \textbf{Background knowledge on regularity structures.} In appendix~\ref{app:background_pam} we collect the prerequisites of Hairer's theory of regularity structures that are necessary to follow the presentation in this article. \vspace{0.5em}
			\begin{enumerate}[label=$\bullet$ \textbf{(A.\arabic*)}, wide=0pt, parsep = 7pt, listparindent=0pt]
				\item \textbf{The original and extended gPAM regularity structure.} In subsection~\ref{app:rs_background}, we briefly recapitulate the algebraic structure necessary to treat~\ref{gpam}, that is: the original regularity structure, built by Hairer, along with its extended counterpart that was introduced to accommodate Cameron-Martin shifts.
				\item \textbf{Admissible models.} Section~\ref{app:sec:adm_models} deals with (minimal) admissible models and the analytical operations that act on them. After repeating the notions of extension and translation, we introduce the dilation operator and study its properties. In particular, we see that it dovetails nicely with the homogeneous norm on minimal model space. In the end, a short summary of BPHZ renormalisation on model space is presented.
				\item \textbf{Modelled distributions.} In section~\ref{sec:modelled_distributions}, we recapitulate the notion of a modelled distribution and how (the lifts of) non-linear functions act upon them. We then study Fr\'{e}chet differentiability of such lifts, obtain a version of Taylor's theorem, and see how modelled distributions behave w.r.t. dilation. 
			\end{enumerate} \vspace{0.5em}
			\item \textbf{Deterministic gPAM and explosion times.} In section~\ref{sec:explosion}, we prove that the solutions to deterministic gPAM~$w_h$ from \eqref{eq:wh} exist globally when the driver~$h$ is a Cameron-Martin function.
			In addition, we provide information on the explosion time~$T^\eps$ of~$\hat{u}^\eps$ from~\thref{intro:ex_sol_hairer} and that of~$\hat{u}_h^\eps$ from~\eqref{thm:stoch_taylor_gpam:diffb}. 
			To quantify the latter, we assume boundedness of~$\eps \barnorm{\hbz^{\minus}}$, a consequence of the localisation procedure in section~\ref{sec:localisation_cameron_martin}.  
			\item \textbf{Large Deviations.} 
			Section~\ref{app:ldp_pam} is, in essence, a recapitulation of Hairer's and Weber's work for the stochastic Allen-Cahn equation in the case of gPAM.
			Most importantly, in \thref{app:thm_ldp_models} we see that the family~$(\d_\eps\hbz^{\minus}: \eps \in I)$ satisfies a LDP on the space~$\MM_-$ of minimal admissible models.
			This allows us to give a complete proof of~\thref{thm:ldp}, formulated with all the details in~\thref{app:thm_ldp_gpam}.
			\item \textbf{Fernique's theorem for the minimal BPHZ model.}  In section~\ref{app:fernique}, we prove that $\barnorm{\hbz^{\minus}}$, the homogeneous norm of the minimal BPHZ model $\hbz^{\minus}$, has Gaussian tails~(\thref{thm:fernique}) -- as needed for our estimates in subsection~\ref{sec:asymp_coeff} to work.
			\item \textbf{Proof of Duhamel's Formula.} In section~\ref{sec:duhamel_proof}, we provide a proof of~\thref{lem:inv_theta_op}, the abstract version of Duhamel's Formula we use to estimate the Taylor terms and remainder in section~\ref{sec:taylor_exp_mod_distr}.
	\end{enumerate}

	\emph{From here on, we assume that the reader is familiar with the content of appendix~\ref{app:background_pam}.}

\paragraph*{Acknowledgements.}

The authors acknowledge funding from the European Research Council through Consolidator Grant 683164. TK thanks the Berlin Mathematical School for institutional support and the Hausdorff Research Institute for Mathematics in Bonn for its warm hospitality and financial support during the Junior Research Trimester \enquote{Randomness, PDEs, and Nonlinear Fluctuations} where parts of this paper were completed. He is grateful to Giuseppe Cannizzaro, Antoine Hocquet, and, in particular, Carlo Bellingeri and Paul Gassiat for stimulating discussions on the subject of this article. Both authors thank the anonymous referees for their careful reading of the manuscript and their helpful suggestions to improve the presentation in this article. 

\listoftodos

\renewcommand\thesection{\arabic{section}}
\setcounter{section}{0}

\section{Localisation around the minimiser and Cameron-Martin}  \label{sec:localisation_cameron_martin}
We set $F^{\minus}_\Phi := F \circ \Phi \circ \EE$ and denote by~$\LL_-(\sh)$ the lift of the minimiser $\sh \in \CH$ to $\MM_-$. In addition, for an $\barnorm{\cdot}$-open neighbourhood $\CO$ of $\LL_-(\sh)$, we decompose
\begin{equs}[][rp:eq:localisation]
	J(\eps) 
	\equiv 
	\E\sbr{\exp\del{-\frac{F^{\minus}_\Phi(\d_\eps \hbz^{\minus})}{\eps^2}}; \ T^\eps > T}
	& = 
	J_{\CO}(\eps) 
	+
	J_{\CO^c}(\eps) 
\end{equs}	
with
	\begin{equation}
		J_A(\eps) 
		:=
		\E\sbr[4]{\exp\del{-\frac{F^{\minus}_\Phi(\d_\eps \hbz^{\minus})}{\eps^2}}; \ T^\eps > T, \ \d_\eps \hbz^{\minus} \in A}, \quad A \in \{\CO, \CO^c\}.
	\end{equation}
Our aim in this section is to prove that the essential contribution to $J(\eps)$ comes from the vicinity~$\CO$ of~$\LL_-(\sh)$. In other words, we want to qualify the contribution of~$J_{\CO^c}$ to~$J$ as asymptotically irrelevant. To this end, we appropriately adapt the proof of \cite[Lem.~$1.32$]{ben_arous_laplace} and \cite[Lem.~$8.2$]{inahama06}. However, these authors rule out explosion in finite time. In contrast, Azencott~\cite{azencott} does account for this phenomenon, so the proof of the proposition that follows is also informed by his work.

\begin{proposition}\label{rp:prop:localisation}
	For each $\barnorm{\cdot}$-open neighbourhood $\CO$ of $\LL_-(\sh)$, there exists some $d > a := \FF(\sh)$ and $\eps_0 > 0$ such that for all $\eps \in (0,\eps_0)$, we have 
	\begin{equation}
		J_{\CO^c}(\eps) 
		\leq \exp\del[3]{-\frac{d}{\eps^2}}.                                                                                                                                                             	
		\label{rp_prop:localisation:eq}                              
	\end{equation}	
\end{proposition}
                   
\begin{proof}
	Since $\JJ_{\MM_-}$ is a \emph{good} rate function (RF), it has compact sublevel sets, i.e. for each $k >0$, the set 
	\begin{equation}
	L_k := \{\bz^{\minus} \in \MM_-: \thinspace \JJ_{\MM_-}(\bz^{\minus}) \leq k\}
	\label{rp:prop:localisation:pf:eq1}	
	\end{equation}                                                                                                                      
	is compact in $\MM_-$. For $\d > 0$, we further denote by                                  
	$L_k(\d) := \{\bz^{\minus} \in \MM_-: \ \barnorm{\bz^{\minus} \thinspace;L_k} < \d\}$
	the open $\d$-fattening of $L_k$.
	
	Since the set $L_k \cap \CO^c$ is compact in $\MM_-$ as well, the function $\FF^{\minus} := \FF \circ \EE = F_\Phi^{\minus} + \JJ_{\MM_-}$ attains its minimum $a_1 > a$ on this set by continuity of $F$, $\Phi$, and~$\EE$ and lower semi-continuity (l.s.c.) of $\JJ_{\MM_-}$ as a RF. Let now $a_2 \in (a,a_1)$; then, there exists $\d_0 > 0$ such that
	\begin{equation}
		\FF^{\minus}(\bz^{\minus}) > a_2 \quad \text{for} \quad \bz^{\minus} \in \overline{L_k(\d_0)} \cap \CO^c
		\label{rp:prop:localisation:pf:eq2}                                                      
	\end{equation} 
	and thus for each $\bz^{\minus} \in \overline{L_k(\d)} \cap \CO^c$ with $\d \in(0,\d_0)$. 
	
	$\triangleright$ The inequality in \eqref{rp:prop:localisation:pf:eq2} can be seen as follows: By l.s.c. of $\FF^{\minus}$, the set $\{\FF^{\minus} > a_2\}$ is open and since $a_2 < a_1$, we have $L_k \cap \CO^c \subseteq \{\FF^{\minus} > a_2\}$. Hence, $L_k \subseteq \{\FF^{\minus} > a_2\} \cup \CO$ and since $\C := \sbr[1]{\{\FF^{\minus} > a_2\} \cup \CO}^c$ is closed, $L_k$ is compact, and $L_k \cap \C = \emptyset$, the function
	\begin{equation*}
	\operatorname{dist}(\cdot\thinspace;\C): \ 
	L_k \ni \bz^{\minus} \mapsto \operatorname{dist}(\bz^{\minus};\C) := \barnorm{\bz^{\minus}\thinspace;\C} > 0
	\end{equation*}
	attains its strictly positive minimum on $L_k$. Therefore, we may find $\d_0 > 0$ such that 
	\begin{equation*}
	\overline{L_k(\d_0)} \subseteq \{\FF^{\minus} > a_2\} \cup \CO,
	\end{equation*}
	which implies that $\overline{L_k(\d_0)} \cap \CO^c \subseteq \{\FF^{\minus} > a_2\}$. $\triangleleft$

	We now fix some $\d \in (0,\d_0)$ and split the analysis of~$J_{\CO^c}(\eps) $ into two parts:
	\begin{enumerate}[label=\textbf{Part \arabic*:}, align=left, leftmargin=0pt, labelindent=0pt,listparindent=0pt, itemindent=!]
		\item \textbf{Controlling $\boldsymbol{J_{\CO^c}(\eps) }$ on $\boldsymbol{L_k(\d)^c \cap \CO^c}$.}
		
		By assumption~\ref{ass:h1}, the estimate
		\begin{equation}
			\E\sbr[3]{\exp\del[3]{-\frac{F_\Phi^{\minus}(\d_\eps \hbz^{\minus})}{\eps^2}}; \thinspace T^\eps > T, \ \d_\eps\hbz^{\minus} \in L_k(\d)^c \cap \CO^c}
			\leq \exp\del[3]{\frac{\norm{F}_{\infty}}{\eps^2}} \P\del[1]{\d_\eps \hbz^{\minus} \in L_k(\d)^c}
			\label{rp:prop:localisation:pf:eq3}
		\end{equation}
		holds. From \thref{app:thm_ldp_models}, we know that the family $\{\P_\eps := \P \circ (\d_\eps \hbz^{\minus})^{-1}: \thinspace \eps > 0\}$ satisfies a LDP on~$\MM_-$, so we apply the LDP upper bound to the closed set $L_k(\d)^c$ to obtain
		\begin{equation*}
		\limsup_{\eps \to 0} \eps^2 \log \P_\eps\del[1]{L_k(\d)^c} 
		\leq -\inf\{\JJ_{\MM_-}(\bz^{\minus}): \thinspace \bz^{\minus} \in L_k(\d)^c\}
		\leq -k
		\end{equation*}
		by definition of $L_k \subseteq L_k(\d)$. We may thus find $\eps_1 > 0$ such that for all $\eps \in (0,\eps_1)$ we have
		\begin{equation*}
		\P_\eps\del[1]{L_k(\d)^c} \leq \exp\del[3]{-\frac{k}{2\eps^2}}.
		\end{equation*}
		Choosing $k > 0$ such that $-d_1 \equiv -d_1(k) := \norm{F}_\infty - \nicefrac{k}{2} < 0$, we further estimate \eqref{rp:prop:localisation:pf:eq3} by
		\begin{equation}
			(\text{RHS}) 
			\leq \exp\del[3]{-\frac{d_1}{\eps^2}}
			\label{rp:prop:localisation:pf:eq4}                                                
		\end{equation}
		for $\eps \in (0,\eps_1)$. This same $k$ is fixed for the rest of the proof.
		\item \textbf{Controlling $\boldsymbol{J_{\CO^c}(\eps)}$ on $\boldsymbol{L_k(\d) \cap \CO^c}$.}
		
		We now define the function $F_1^{\minus}(\bz^{\minus}) := F_\Phi^{\minus}(\bz^{\minus})$ for $\bz^{\minus} \in \overline{L_k(\d)} \cap \CO^c$ and $F_1^{\minus}(\bz^{\minus}) := c > a_2$ otherwise. This function is l.s.c. and bounded, for $F$ is bounded by assumption~\ref{ass:h1}. We may thus apply \cite[Lem.~$2.1.8$]{deuschel-stroock} in conjunction with~\thref{sec:explosion:time_0_lem} to infer that
		\begin{equation*}
			\limsup_{\eps \to 0} \eps^2 \log \E\sbr[3]{\exp\del[3]{-\frac{F_1^{\minus}(\d_\eps\hbz^{\minus})}{\eps^2}}; \thinspace T^\eps > T}
			\leq                                              
			-\inf (F_1 + \JJ_{\MM})
			\overset{\eqref{rp:prop:localisation:pf:eq2}}{\leq}
			-a_2.
		\end{equation*}
		Thus, there exists $a_3 \in(a,a_2)$ and $\eps_2 > 0$ such that for all $\eps \in (0,\eps_2)$, we have
		\begin{equation}
			\E\sbr[3]{\exp\del[3]{-\frac{F_1^{\minus}(\d_\eps\hbz^{\minus})}{\eps^2}}; \thinspace T^\eps > T}
			\leq \exp\del[3]{-\frac{a_3}{\eps^2}}.                                                                 
			\label{rp:prop:localisation:pf:eq5}
		\end{equation}                                 
		Now observe that
		\begin{equs}
			\thinspace
			&                                                                                          
			\E\sbr[3]{\exp\del[3]{-\frac{F^{\minus}_\Phi(\d_\eps\hbz^{\minus})}{\eps^2}}; \thinspace T^\eps > T, \ \d_\eps\hbz^{\minus} \in L_k(\d) \cap \CO^c}
			\leq 
			\E\sbr[3]{\ldots; \thinspace T^\eps > T, \ \d_\eps\hbz^{\minus} \in \overline{L_k(\d)} \cap \CO^c} \\
			\equiv 
			\
			&
			\E\sbr[3]{\exp\del[3]{-\frac{F_1^{\minus}(\d_\eps\hbz^{\minus})}{\eps^2}}; \thinspace T^\eps > T, \ \d_\eps\hbz^{\minus} \in \overline{L_k(\d)} \cap \CO^c}
			\overset{\eqref{rp:prop:localisation:pf:eq5}}{\leq}
			\exp\del[3]{-\frac{a_3}{\eps^2}}.
		\end{equs}
	\end{enumerate}
	Combining these two steps, we find that $\abs[0]{J_{\CO^c}(\eps)} \leq \exp\del[1]{-\eps^{-2}d}$ for all $\eps \in (0,\eps_0)$, where $\eps_0 := \eps_1 \wedge \eps_2$ and $d := d_1 + a_3 > a_3 > a$. This completes the proof.
\end{proof}

By the previous proposition, it suffices to consider the functional
	\begin{equation}
		J_\rho(\eps) 
		:= 
		\E\sbr[4]{\exp\del[4]{-\frac{F^{\minus}_\Phi\del[0]{\d_\eps \hbz^{\minus}}}{\eps^2}}; \thinspace T^\eps > T, \ 
			\barnorm{T_{-\sh} \d_\eps \hbz^{\minus}} < \rho}.
	\end{equation} 
for some~$\rho > 0$ which we will later choose sufficiently small. 
Note that $J_\rho(\eps) \equiv J_\CO(\eps)$ for the specific choice 
	\begin{equation*}
		\CO \equiv \CO(\sh,\rho) := \{\bz^{\minus} \in \MM_-: \ \barnorm{T_{-\sh} \bz^{\minus}} < \rho\}.\footnote{Note that this set is open by continuity of the translation operator $T_{-\sh}$ on $(\MM_-, \barnorm{\cdot})$, a consequence of~\thref{app:prop:translation} and~\thref{lem:estimate_min_norm_vs_hom_norm}.}
	\end{equation*}
Instead of tracking the $\sh$-dependency of $J_\rho(\eps)$ via the set $\CO$, we want to encode it by appropriately changing the functional $F_\Phi \rightsquigarrow \tilde{F}_\Phi(\sh,\cdot)$. As usual, this amounts to a straightforward application of the Cameron-Martin theorem.
The setup is that of the abstract Wiener space~$(B,\CH,\mu)$ one also encounters when proving large deviation principles, see sec.~\ref{app:ldp_pam} in the appendix.

\begin{proposition}\label{prop:cameron_martin}
	In the same setting as before, there exists~$\rho_0 = \rho_0(T) > 0$ such that
		\begin{equation}
			J_\rho(\eps) 
			=                                                  
			\exp\del[3]{-\frac{\FF(\sh)}{\eps^2}} \E\sbr[3]{\exp\del[3]{-\frac{\tilde{F}^{\minus}_\Phi(\sh,\eps)}{\eps^2}}; \thinspace \eps\barnorm{\hbz^{\minus}} < \rho}
			\label{eq:prop:cameron_martin:functional}
		\end{equation}
	for all $\rho \in (0,\rho_0)$, where 
	\begin{equation}
		\tilde{F}^{\minus}_\Phi(\sh,\eps) 
		:= 
		F^{\minus}_\Phi\del[1]{T_\sh \d_\eps \hbz^{\minus}}
		- F^{\minus}_\Phi(\LL_-(\sh))
		+ \eps \xi(\sh).
		\label{eq:prop:cameron_martin:FPhi}
	\end{equation}
\end{proposition}
                                                          
\begin{proof} 
	Recall that white noise $\xi$ is the identity on the Gaussian measure space $(B,\mu)$, that is $\xi(\omega) = \omega$ for $\omega \in B$. Also recall that $T^\eps = T_\infty^{\minus}(\d_\eps \hbz^{\minus})$. From the Cameron-Martin theorem \cite[Prop.~2.26]{daprato-zabczyk} on the abstract Wiener space~$(B,\CH,\mu)$, we know that the push-forward measure~$\mu_\sh^\eps := (T_{-\nicefrac{\sh}{\eps}})_*\mu$ has density
		\begin{equation*}
			\frac{\dif \mu_\sh^\eps}{\dif \mu}(\omega)
			=
			\exp\del[3]{-\frac{1}{\eps^2}\sbr[1]{\II(\sh) + \eps \omega(\sh)}}.
		\end{equation*}
	where $\omega(\sh) = I(\sh)(\omega)$ with the Paley-Wiener map $I$. We then find that	
	\begin{align*}
		J_\rho(\eps)
		&
		=
		\E\sbr[3]{\exp\del[3]{-\frac{F^{\minus}_\Phi\del[1]{\d_\eps\hbz^{\minus}}}{\eps^2}}; \ 
			T_\infty^{\minus}(\d_\eps \hbz^{\minus}) > T, \
			\barnorm{T_{-\sh} \d_\eps \hbz^{\minus}} < \rho} \\
		& 
		=                                                                                                     
		\int_B
		\exp\del[3]{-\frac{F^{\minus}_\Phi \del[1]{T_\sh \d_\eps\hbz^{\minus}(\omega - \nicefrac{\sh}{\eps})}}{\eps^2}}
		\mathbf{1}_{T_\infty^{\minus}(T_\sh\d_\eps \hbz^{\minus}(\omega - \nicefrac{\sh}{\eps})) > T} \mathbf{1}_{\barnorm{\d_\eps\hbz^{\minus}(\omega-\nicefrac{\sh}{\eps})} < \rho} \thinspace \mu(\dif \omega) 
		\\
		&                                                                                                                                              
		=
		\exp\del[3]{-\frac{F^{\minus}_\Phi(\LL_-(\sh)) + \II(\sh)}{\eps^2}} 
		\E\sbr[3]{\exp\del[3]{-\frac{\tilde{F}^{\minus}_\Phi(\sh,\eps)}{\eps^2}}; \  T_\infty^{\minus}(T_\sh \d_\eps \hbz^{\minus}) > T, \ \barnorm{\d_\eps\hbz^{\minus}} < \rho} 
		\\
		&
		=
		\exp\del[3]{-\frac{\FF(\sh)}{\eps^2}} \E\sbr[3]{\exp\del[3]{-\frac{\tilde{F}_\Phi(\sh,\eps)}{\eps^2}}; \  T_\infty^{\minus}(T_\sh \d_\eps \hbz^{\minus}) > T, \ \eps \barnorm{\hbz^{\minus}} < \rho}.
	\end{align*}
	Note that we have used~\thref{lem:transop_cm_shift} (with $h := -\sh$) for the second equality.
	Since $T_\infty^\sh = + \infty$~(see~\thref{rmk:explosion_gpam}), any $T > 0$ satisfies~$T < T_\infty^\sh$. Choosing $\rho_0 > 0$ as in~\thref{sec:explosion:cm_lem}, we have
		\begin{equation}
			\{\eps \barnorm{\hbz^{\minus}} < \rho\} \subseteq \{T_\infty^{\minus}(T_\sh \d_\eps \hbz^{\minus}) > T\} 
		\end{equation}	
	for all~$\rho \in (0,\rho_0)$, which establishes the claim. 
\end{proof}

\section{Taylor expansion in the noise intensity parameter} \label{sec:stoch_taylor} 

The application of the Cameron-Martin theorem in the last section prompts us to analyse the function~$\tilde{F}^{\minus}_\Phi(\sh,\cdot)$ further. In this direction, the first and most important step is the local expansion of
	\begin{equation*}
		\Phi^{\minus}(T_\sh \d_\eps\hbz^{\minus}) 
		\equiv
		(\Phi \circ \EE)(T_\sh \d_\eps\hbz^{\minus}) 
		\equiv
		\Phi(T_\sh \d_\eps\hbz) 
		=
		\Phi(T_\sh \d_\eps\hbz) = \hat{u}^{(\eps)}_\sh(\hbz)
	\end{equation*}
 in~$\eps$. This is the content of~\thref{thm:stoch_taylor_gpam} given in subsection~\ref{intro:sec:results} of the introduction.

The local expansion of~$F^{\minus}_\Phi(T_\sh \d_\eps \hbz^{\minus})$ in $\eps$ that we need for the analysis of~$\tilde{F}^{\minus}_\Phi(\sh,\cdot)$ is an easy corollary of that theorem. While~\thref{thm:stoch_taylor_gpam} is stated for general~$g \in \CC^{\ell+4}$, $\ell \geq 1$, the specific regularity assumptions on $g$ (and $F$) for the study of precise Laplace asymptotics impose~$\ell = N + 3$. 

\begin{corollary}[Taylor expansion of~{{$F\del[1]{\hat{u}^{\bullet}_\sh}$}}]\label{coro:stoch_taylor_gpam_functional}
	For each $\eps \in I_0$ we have
	\begin{equation}
		F_\Phi(T_\sh \d_\eps \hbz)
		\equiv 
		F\del[1]{\hat{u}^{\eps}_\sh}
		=
		F_\sh^{(0)}
		+
		\sum_{m=1}^{N + 2} \frac{\eps^m}{m!} \hat{F}^{(m)}_{\sh}
		+ \hat{R}^{F;\eps;N + 3}_{\sh}
		\label{thm:stoch_taylor_gpam:exp_F}
	\end{equation}
	where $F_\sh^{(0)} := F(w_\sh) \equiv F_\Phi\del[0]{\LL(\sh)}$, 
	\begin{equs}
		\hat{F}^{(m)}_{\sh}
		\equiv
		F^{(m)}_{\sh}(\hbz) 
		& := \partial_\eps^{m}\sVert_{\eps = 0} F\del[1]{u^{\eps}_\sh(\hbz)},  \quad m \in [N+2],
	\end{equs}
	and~$\hat{R}^{F;\eps;N + 3}_{\sh} \equiv R^{F;\eps;N + 3}_{\sh}(\hbz)$ is implicitly defined from~\eqref{thm:stoch_taylor_gpam:exp_F}. 
	We also set~$\hat{Q}_\sh := \hat{F}_\sh^{(2)}$.
\end{corollary}

\begin{remark}
	Observe that the preceding corollary was formulated for the specific element~$\sh \in \CH$. However, $\sh$ does not enter the proof by its property~\ref{ass:h2} of being a minimiser of~$\FF$. Instead, we will need the assumptions in~\ref{ass:h3}, that is: the differentiability of~$F$ in the neighbourhood~$\CN$ of $w_\sh \equiv (\Phi \circ \LL)(\sh)$ and the bounds~\eqref{thm:laplace_asymp:ass:h3:boundedness} for its derivatives.
\end{remark}

\begin{proof}
	By~\thref{thm:stoch_taylor_gpam}, $\eps \mapsto \hat{u}^{\eps}_\sh$ is differentiable on~$I_0 = [0,\eps_0)$ and thus continuous. Therefore, with possibly smaller~$\eps_0 > 0$, we have $\hat{u}_\sh^{\eps} \in \CN$ for all $\eps \in  I_0$. Assumption~\ref{ass:h3} and~\thref{thm:stoch_taylor_gpam} then immediately imply that~$F \circ \hat{u}_\sh^{\bullet} \in \CC^{N+3}(I_0,\R)$, so the expansion in~\eqref{thm:stoch_taylor_gpam:exp_F} follows from Taylor's theorem.
\end{proof}

We collect the properties of the Taylor terms and the remainder in the corollary that follows.

\begin{corollary} \label{coro:stoch_taylor_gpam_functional:properties}
	In the setting of~\thref{coro:stoch_taylor_gpam_functional}, the following assertions hold for $m = 1, \ldots, N + 2$:
	\begin{enumerate}[label=(\roman*)]
		\item \label{coro:stoch_taylor_gpam_functional:i} The terms~$F^{(m)}_{\sh}(\bz)$ and the remainder~$R^{F;\eps;N + 3}_{\sh}(\bz)$ are continuous functions of the model~$\bz$. In addition, we have the estimates
		\begin{equation}
			\abs[1]{\hat{F}^{(m)}_{\sh}} \aac (1 + \barnorm{\hbz^{\minus}})^m, \quad
			\abs[1]{\hat{R}^{F;\eps;(N + 3)}_{\sh}} \aac_\rho \eps^{N+3}(1 + \barnorm{\hbz^{\minus}})^{N+3} 
			\label{coro:stoch_taylor_gpam_functional:remainder_estimate}
		\end{equation}
		where the second bound requires $\eps \barnorm{\hbz^{\minus}} < \rho$ for some $\rho > 0$~to be valid.
		\item \label{coro:stoch_taylor_gpam_functional:ii}
		The terms~$\hat{F}^{(m)}_{\sh} \equiv F^{(m)}_{\sh}(\hbz)$ are \emph{$m$-homogeneous}  
		w.r.t. dilation of the model~$\hbz$, that is
		\begin{equation}
			\eps^m F^{(m)}_{\sh}(\hbz) = F^{(m)}_{\sh}(\d_\eps\hbz) 
			\label{coro:stoch_taylor_gpam_functional:m_homog}
		\end{equation}
		for $\eps \in I_0$. In particular, $\eps^2 Q_\sh(\hbz) = Q_\sh(\d_\eps \hbz)$.
	\end{enumerate}	
\end{corollary}
For the proof of the corollary, in particular part~\ref{coro:stoch_taylor_gpam_functional:i}, we need some auxiliary formulas. They have already appeared in the literature before, cf.~\cite[pp.~306/307]{inahama_kawabi}.  

\begin{lemma} \label{lem:coro_stoch_taylor}
	For $k,n \in \N$, let $S_k^n := \{\boldsymbol{i} \in \N_{\geq 1}^k: \thinspace \abs{\bi} = n\}$. Then, for all~$\eps \in I_0$ the equalities
		\begin{align}
			\hat{F}^{(m)}_\sh 
			& = 
			m! \sum_{k=1}^m \frac{1}{k!} \sum_{\bi \in S_k^m} D^{(k)}F\sVert[0]_{w_\sh} \del[3]{\frac{1}{i_1!} \hat{u}^{(i_1)}_\sh,\ldots,\frac{1}{i_k!} \hat{u}^{(i_k)}_\sh}, \label{lem:coro_stoch_taylor:eq1}\\
			\hat{R}^{F;\eps;(N+3)}_\sh 
			& =
			\sum_{k=1}^{N+2} \frac{1}{k!} \sum_{\substack{\bi \in [N+3]^k \\ \abs{\bi} \geq N+3}} D^{(k)}F\sVert[0]_{w_\sh} \del[3]{\frac{1}{i_1!} \hat{R}^{(N+3)}_{\sh,\eps}(i_1),\ldots,\hat{R}^{(N+3)}_{\sh,\eps}(i_k)} \label{lem:coro_stoch_taylor:eq2} \\
			& +
			\int_0^1 \frac{(1-s)^{N+2}}{(N+2)!} D^{(N+3)}F\del[0]{w_\sh + s \hat{R}_{\sh,\eps}^{(1)}} \del[1]{\hat{R}_{\sh,\eps}^{(1)},\ldots,\hat{R}_{\sh,\eps}^{(1)}} \dif s \notag
		\end{align}
	hold with
		\begin{equation}
		 \hat{R}^{(N+3)}_{\sh,\eps}(m) 
		 :=
		 \begin{cases}
			\eps^m \hat{u}_{\sh}^{(m)} & \quad \text{if} \ \ m \in [N+2], \\
			\hat{R}_{\sh,\eps}^{(N+3)} & \quad \text{if} \ \ m = N+3, 
		 \end{cases}
		 \label{lem:coro_stoch_taylor:eq3}
		\end{equation}
	the terms on the RHS of which are defined in~\thref{thm:stoch_taylor_gpam}. 
	In particular, we find $\hat{F}_\sh^{(1)} = DF\sVert_{w_\sh}\del[1]{\hat{u}^{(1)}_\sh}$ and
	\begin{equation} 
		\hat{Q}_\sh
		=
		DF\sVert_{w_\sh}\del[1]{\hat{u}^{(2)}_\sh}
		+ D^2F\sVert_{w_\sh}\del[1]{\hat{u}^{(1)}_\sh,\hat{u}^{(1)}_\sh}.	
	\end{equation}
\end{lemma}

\begin{proof}
	Recall that~$\hat{u}^{\eps}_\sh \in \CN$ for all $\eps \in I_0$. By hypothesis~\ref{ass:h3}, $F$ is $C^{N+3}$ in $\CN$, so we expand it around~$w_\sh$ to get 
		\begin{equation}
			F(\hat{u}_\sh^\eps) - F(w_\sh)
			=
			\sum_{k=1}^{N+2} \frac{1}{k!} D^{(k)}F(w_\sh)\sbr[1]{R^{(1)}_{\sh,\eps}}^k 
			+
			\int_0^1 \frac{(1-s)^{N+2}}{(N+2)!} D^{(N+3)}F(w_\sh + s \hat{R}_{\sh,\eps}^{(1)}) \sbr[1]{\hat{R}_{\sh,\eps}^{(1)}}^{N+3} \dif s 
			\label{lem:coro_stoch_taylor:pf_aux0}
		\end{equation}
	where 
		\begin{equation*}
			\hat{R}_{\sh,\eps}^{(1)}
			=
			\hat{u}_\sh^\eps - w_\sh
			=
			\sum_{\ell=1}^{N+2} \frac{\eps^\ell}{\ell!} \hat{u}^{(\ell)}_\sh + \hat{R}_{\sh,\eps}^{(N+3)}
			=
			\sum_{\ell=1}^{N+3} \frac{1}{\ell!} \hat{R}_{\sh,\eps}^{(N+3)}(\ell)
		\end{equation*}
	by~\thref{thm:stoch_taylor_gpam} and eq.~\eqref{lem:coro_stoch_taylor:eq3}. This identity enables us to write
		\begin{equation}
			\sum_{k=1}^{N+2} \frac{1}{k!} D^{(k)}F\sVert[0]_{w_\sh}\sbr[1]{R^{(1)}_{\sh,\eps}}^k 
			=
			\sum_{k=1}^{N+2} \frac{1}{k!} \sum_{\bi \in [N+3]^k} D^{(k)}F\sVert[0]_{w_\sh}\sbr[3]{\frac{1}{i_1!}\hat{R}_{\sh,\eps}^{(N+3)}(i_1),\ldots,\frac{1}{i_k!}\hat{R}_{\sh,\eps}^{(N+3)}(i_k)}.
			\label{lem:coro_stoch_taylor:pf_aux}
		\end{equation}
	Next, we observe the trivial equalities
		\begin{equs}
			\thinspace
			[N+3]^k 
			& =
			\{ \bi \in [N+3]^k: \ \abs{\bi} \leq N+2 \} \ \sqcup \ \{ \bi \in [N+3]^k: \ \abs{\bi} \geq N+3 \} \\
			& =
			\bigcup_{m=1}^{N+2} S_k^m \ \sqcup \ \{ \bi \in [N+3]^k: \ \abs{\bi} \geq N+3 \}
		\end{equs}
	by which we may split up the sum in eq.~\eqref{lem:coro_stoch_taylor:pf_aux}. As for the first part involving the union of the~$S_k^m$'s, that leads to the sum
		\begin{align}
			\thinspace &
			\sum_{k=1}^{N+2} \frac{1}{k!} \sum_{m=1}^{N+2} \sum_{\bi \in S_k^m} D^{(k)}F\sVert[0]_{w_\sh}\sbr[3]{\frac{\eps^{i_1}}{i_1!}\hat{u}_{\sh}^{(i_1)},\ldots,\frac{\eps^{i_k}}{i_k!}\hat{u}_{\sh}^{(i_k)}} 
			= 
			\sum_{m=1}^{N+2} \eps^m \sum_{k=1}^{m} \frac{1}{k!}  \sum_{\bi \in S_k^m} D^{(k)}F\sVert[0]_{w_\sh}\sbr[3]{\frac{\hat{u}_{\sh}^{(i_1)}}{i_1!},\ldots,\frac{\hat{u}_{\sh}^{(i_k)}}{i_k!}} \notag \\
			= \ &  
			\sum_{m=1}^{N+2} \frac{\eps^m}{m!} \partial_\eps^m\sVert[1]_{\eps=0} (F \circ \hat{u}^\eps_\sh)
			=
			\sum_{m=1}^{N+2} \frac{\eps^m}{m!} \hat{F}_\sh^{(m)}
			\label{lem:coro_stoch_taylor:pf_aux_1}
		\end{align}
	In the first equality, we have just used the fact that~$S_k^m = \emptyset$ for~$k > m$ as well as the multi-linearity of~$D^{(k)}F(w_\sh)$. The second equality is due to Riordan's Formula, a version of the well-known Faa di Bruno formula (see the review paper by Johnson~\cite{faa_di_bruno}), which states that
		\begin{equation}
			\partial_\eps^m (f \circ g)(\eps) = m! \sum_{k=1}^m \frac{1}{k!} \sum_{\bi \in S_k^m} f^{(k)}(g(\eps))\sbr[3]{\frac{1}{i_1!} g^{(i_1)},\ldots,\frac{1}{i_k!} g^{(i_k)}}
			\label{eq:riordans_formula}
		\end{equation}
	for $f$ and $g$ sufficiently smooth.
	
	Comparing prefactors of~$\eps^m$ in~\eqref{lem:coro_stoch_taylor:pf_aux_1}, the claim for $\hat{F}^{(m)}_\sh$ in~\eqref{lem:coro_stoch_taylor:eq1} follows.
	A fortiori, that implies the claim~\eqref{lem:coro_stoch_taylor:eq2} for $\hat{R}^{F;\eps;(N+3)}_\sh$ by comparing the terms in~\eqref{lem:coro_stoch_taylor:pf_aux0} to those in~\eqref{thm:stoch_taylor_gpam:exp_F}. 
\end{proof}

\begin{proof}[of~\thref{coro:stoch_taylor_gpam_functional:properties}]
	Assertion~\ref{coro:stoch_taylor_gpam_functional:ii} is a direct consequence of~\thref{thm:stoch_taylor_gpam}\ref{thm:stoch_taylor_gpam:ii} combined with~\eqref{lem:coro_stoch_taylor:eq1}.
	
	We turn to assertion~\ref{coro:stoch_taylor_gpam_functional:i}. In conjunction with assumption~\ref{ass:h3}, eq.~\eqref{thm:laplace_asymp:ass:h3:boundedness}, the explicit representation for~$\hat{F}_\sh^{(m)}$ derived in~\thref{lem:coro_stoch_taylor} allows for the estimate
	\begin{equs}[][coro:stoch_taylor_gpam_functional:properties:pf_aux1]
		\abs[1]{\hat{F}^{(m)}_{\sh}}
		&
		\aac_m 
		\sum_{k=1}^m M_k \sum_{\bi \in S_k^m} \prod_{\ell=1}^k  \norm[1]{\hat{u}^{(i_\ell)}_\sh}_{\CX_T}
		\aac_T 
		\sum_{k=1}^m M_k \#[S_k^m] (1 + \barnorm{\hbz^{\minus}})^m
		\aac_m
		(1 + \barnorm{\hbz^{\minus}})^m.
	\end{equs}
	In the penultimate estimate, we have applied the estimates from~\thref{thm:stoch_taylor_gpam}\ref{thm:stoch_taylor_gpam:i}. As for the remainder, note that $\{w_\sh + s \hat{R}_{\sh,\eps}^{(1)}: s \in [0,1]\} \subseteq \CN$ for $\eps \in I_0$, for it is the line segment connecting~$w_\sh$ and~$\hat{u}^{\eps}_\sh$. Hence, the integral term in~\eqref{lem:coro_stoch_taylor:eq2} may be estimated analogously to~\eqref{coro:stoch_taylor_gpam_functional:properties:pf_aux1}. The same is true for the sum in~\eqref{lem:coro_stoch_taylor:eq2}, except that one uses the additional estimate
		\begin{equation*}
			\eps^{\abs[0]{\bi}} (1 + \barnorm{\hbz^{\minus}})^{\abs[0]{\bi}} 
			\leq
			(1 + \rho)^{\abs[0]{\bi} - (N+3)} \eps^{N+3} (1 + \barnorm{\hbz^{\minus}})^{N+3} 
		\end{equation*}
	for multi-indices $\bi$ with $\abs[0]{\bi} \geq N+3$, as readily implied by $\eps \in [0,1]$ and $\eps \barnorm{\hbz^{\minus}} < \rho$.
\end{proof}	
In the remainder of this section, we will prepare the proof of~\thref{thm:stoch_taylor_gpam} to be given in subsection~\ref{sec:pf_stoch_taylor}.  If the reader is willing to momentarily believe the validity of that theorem, she may want to proceed with section~\ref{sec:local_analysis} on page~\pageref{sec:local_analysis} in a first reading. In that section, we use the corollaries~\ref{coro:stoch_taylor_gpam_functional} and~\ref{coro:stoch_taylor_gpam_functional:properties} to continue the proof of the precise Laplace asymptotics as claimed in~\thref{thm:laplace_asymp}.

\subsection{The abstract fixed-point problem} \label{sec:abstr_fp_prob}

Consider the following abstract fixed-point problems for $\bz \in \MM$ and $\gr{\bz} \in \gr{\MM}$:

\begin{align}
	U & = \CP^{\bz} \del[1]{G (U)\<wn>} + \TT Pu_0 &&  \hspace{-7em} \rightarrowtriangle \quad U \in \DD^{\gamma,\eta}_\CU(\bz), \label{eq:fp_normal} \\
	\gr{U} & = \CP^{\gr{\bz}}  \del[1]{G(\gr{U}) [\<wn> + \<cm>]} + \TT Pu_0 &&  \hspace{-7em}\rightarrowtriangle \quad \gr{U} \in \DD^{\gamma,\eta}_{\gr{\CU}}(\gr{\bz}), \label{eq:fp_extended} \\
	\gr{U^\eps} & = \CP^{\gr{\bz}} \del[1]{G(\gr{U^\eps}) [\eps \<wn> + \<cm>]} + \TT Pu_0 &&  \hspace{-7em}\rightarrowtriangle \quad \gr{U^\eps} \in \DD^{\gamma,\eta}_{\gr{\CU}}(\gr{\bz}), \label{eq:fp_extended_eps} 
\end{align}
\vspace{-2em}
\begin{remark} \label{rmk:dep_fp_model}
	Note that the model enters these equations implicitly via the abstract convolution kernels $\CP^{\bz}$ and $\CP^{\gr{\bz}}$, respectively, and thereby characterises the right space of modelled distributions we seek solutions in.
	Our main case of interest is $\gr{\bz} := E_{\sh} \hbz \in \gr{\MM}$ where $\sh \in \CH$ is the minimiser from assumption~\ref{ass:h2} and $\hbz \in \MM$ the BPHZ model. For $h \in \CH$ we suppress the dependence on the specific model~$\bz \in \MM$ and simply write
	\begin{equation*}
		\CP^h = \CP(T_h\bz;\thinspace \cdot \thinspace), \quad \CP^{e_h} = \CP(E_h\bz;\thinspace \cdot \thinspace).
	\end{equation*}
\end{remark}
The solvability of the first two fixed-point equations was proved by Cannizzaro, Friz, and Gassiat in~\cite[Prop.~3.23 \& 3.25]{cfg}. 
Re\-gar\-ding~\eqref{eq:fp_extended_eps}, we define the abstract, parameter-dependent FP map 
\begin{equation}
	\CM^{\gr{\bz}}: I \x \DD^{\gamma,\eta}_{\gr{\CU}}(\gr{\bz}) \to \DD^{\gamma,\eta}_{\gr{\CU}}(\gr{\bz}), \quad I := [0,1],
\end{equation} 
for $\gr{\bz} \in \gr{\MM}$ by 
\begin{equation}
	\CM^{\gr{\bz}}(\eps,\gr{Y}) 
	:= \CP^{\gr{\bz}}\del[1]{G^\eps(\gr{Y})} + \TT Pu_0,
	\quad
	G^\eps(\gr{Y}) := G(\gr{Y}) [\eps\<wn> + \<cm>].
	\label{eq:def_fp_map}
\end{equation}
Next, we revisit and amend some variants of the notion of~\emph{local Lipschitz continuity} as introduced by Hairer in~\cite[sec.~7.3]{hairer_rs} and slightly rephrased by Sch\"onbauer~\cite[Def.~3.18]{schoenbauer}.

\begin{definition} \label{def:unif_strong_local_Lipschitz}
Let $\gr{\CV}, \gr{\tilde{\CV}}$ be two sectors\footnote{See \cite[Def. 2.5]{hairer_rs}. In addtion, the domain~$D$ is given as~$D := I \x \T^2$ for some interval~$I \subseteq \R^+$, cf.~\thref{rmk:domain_D} in the appendix.} in~$\gr{\TT}$ and fix~$\gamma \geq \bar{\gamma} > 0$ as well as~$\eta, \bar{\eta} \in \R$.
	\begin{itemize}
		\item Given a model~$\gr{\bz} \in \gr{\MM}$, we say that a map~$F: \DD^{\gamma,\eta}_{\gr{\CV}}(\gr{\bz}) \to \DD^{\bar{\gamma},\bar{\eta}}_{\gr{\tilde{\CV}}}(\gr{\bz})$ is \emph{locally Lipschitz continuous} if for any compact set $K \subset D$ and any $R > 0$ one has the bound
			\begin{equation}
				\threebars F(\gr{Y}) - F(\gr{\bar{Y}}) \threebars_{\bar{\gamma},\bar{\eta},K} 
				\aac \threebars \gr{Y} - \gr{\bar{Y}} \threebars_{\gamma,\eta,K} 
				\label{def:unif_strong_local_Lipschitz:eq1}
			\end{equation}
		uniformly over all $\gr{Y}, \gr{\bar{Y}} \in \DD^{\gamma,\eta}_{\gr{\CV}}(\gr{\bz})$ with $\threebars \gr{Y} \threebars_{\gamma,\eta, K} \vee \threebars \gr{\bar{Y}} \threebars_{\gamma,\eta, K} \leq R$.
		\item Let $F^{\gr{\bz}}: \DD^{\gamma,\eta}_{\gr{\CV}}(\gr{\bz}) \to \DD^{\bar{\gamma},\bar{\eta}}_{\gr{\tilde{\CV}}}(\gr{\bz})$ be locally Lipschitz for any~$\gr{\bz} \in \gr{\MM}$. We say that $F = (F^{\gr{\bz}})_{\gr{\bz} \in \gr{\MM}}$ is \emph{stronly locally Lipschitz continuous} if for any $\gr{\bz} \in \gr{\MM}$ there exists a neighbourhood $\gr{\boldsymbol{\CB}}$ of $\gr{\bz}$ in $\gr{\MM}$ such that for any compact set $K \subseteq D$ and $R > 0$ one has the bound
			\begin{equation}
				\threebars F^{\gr{\bz}}(\gr{Y}) ; F^{\gr{\bbz}}(\gr{\bar{Y}}) \threebars_{\bar{\gamma},\bar{\eta},K} 
				\aac 
				\threebars \gr{Y} \thinspace ; \gr{\bar{Y}} \threebars_{\gamma,\eta,K} 
				+ 
				\threebars \gr{\bz} \thinspace ; \gr{\bbz} \threebars_{\gamma,\bar{K}}
				\label{def:unif_strong_local_Lipschitz:eq2}
			\end{equation}
		uniformly over all $\gr{\bbz} \in \gr{\boldsymbol{\CB}}$ and $\gr{Y} \in \DD^{\gamma,\eta}_{\gr{\CV}}(\gr{\bz}), \gr{\bar{Y}} \in \DD^{\gamma,\eta}_{\gr{\CV}}(\gr{\bbz})$ such that $\threebars \gr{Y} \threebars_{\gamma,\eta, K,\gr{\bz}} \vee \threebars \gr{\bar{Y}} \threebars_{\gamma,\eta, K, \gr{\bbz}} \leq R$. Here, $\bar{K}$ denotes the $1$-fattening of $K$, i.e. $\bar{K} = \{z \in D: \thinspace \operatorname{dist}(z,K) \leq 1\}$.
		\item A family $F = (F^{\eps})_{\eps \in I}$ of stronly locally Lipschitz maps $F^{\eps} = (F^{\eps,\gr{\bz}})_{\gr{\bz} \in \gr{\MM}}$ is called \emph{uniformly stronly locally Lipschitz continuous} if the bound~\eqref{def:unif_strong_local_Lipschitz:eq2} holds uniformly in $\eps$ for each $F^{\eps}$.
	\end{itemize}
\end{definition}
The following lemma characterises the Lipschitz properties of the functions~$(G^\eps)_{\eps \in I}$ introduced in~\eqref{eq:def_fp_map}.

\begin{lemma} \label{prop:a_priori_estimate}
	The family $(G^\eps)_{\eps \in I}$ is uniformly strongly locally Lipschitz. 
\end{lemma}

\begin{proof}
	Let $\alpha := \reg{\<wn>} = -1 - \kappa$. For any $\zeta > 0$, we have $\<wn>, \<cm> \in \DD^{\zeta,\zeta}_{\gr{\CV}}$ where $\gr{\CV} = \scal{\<wn>,\<cm>}$ is a sector of reg.~$\alpha$. Then the same is true for $f_\eps := \eps \<wn> + \<cm>$ for any $\eps \in I$. 
	Recall that for $(\gr{\bz}, \gr{Y}) \in \gr{\MM} \ltimes \DD^{\gamma,\eta}_{\gr{\CU}}$ we have $G(\gr{Y}) \in \DD^{\gamma,\eta}_{\gr{\CU}}(\gr{\bz})$ by \cite[Prop.~6.13]{hairer_rs} because the sector $\gr{\CU}$ is function-like. Since the pair $(\gr{\CV},\gr{\CU})$ is $\bar{\gamma}$-regular (see \cite[Def. 4.6]{hairer_rs}) with $\bar{\gamma} := \gamma + \alpha$, with $\bar{\eta} := \eta + \alpha$ it follows from \cite[Prop.\thinspace6.12]{hairer_rs} that $G^\eps(\DD^{\gamma,\eta}_{\gr{\CU}}(\gr{\bz})) \subseteq \DD^{\bar{\gamma},\bar{\eta}}(\gr{\bz})$ for each $\gr{\bz} \in \gr{\MM}$.
	
	In order to check the estimate~\eqref{def:unif_strong_local_Lipschitz:eq2}, let $\gr{\bz}, \gr{\bbz}$, and $\gr{\bar{Y}}$ be as required. Note that~$\a = \min \gr{\CA}$, so by def. of the structure group~$\gr{\CG}$, we have $\gr{\Gamma} f_\eps = f_\eps = \gr{\bar{\Gamma}} f_\eps$ for any $\gr{\Gamma},\gr{\bar{\Gamma}} \in \gr{\CG}$. 
	Hence $\threebars f_\eps^{\gr{\bz}};f_\eps^{\gr{\bbz}} \threebars_{\zeta,\zeta;K} = 0$. Again, \cite[Prop.\thinspace6.12]{hairer_rs} then implies that
	\begin{equs}[][prop:a_priori_estimate:pf:eq]
		\threebars G^{\eps,\gr{\bz}}(\gr{Y});G^{\eps,\gr{\bbz}}(\gr{\bar{Y}}) \threebars_{\bar{\gamma},\bar{\eta},K}
		&
		\aac 
		\threebars G^{\gr{\bz}}(\gr{Y});G^{\gr{\bbz}}(\gr{\bar{Y}}) \threebars_{\gamma,\eta,K} 
		+
		\threebars f^{\gr{\bz}}_\eps;f^{\gr{\bbz}}_\eps \threebars_{\zeta,\zeta, K}
		+
		\threebars \gr{\bz}; \gr{\bbz} \threebars_{\gamma,\bar{K}} \\
		&
		\aac_{K,R} 
		\threebars \gr{Y},\gr{\bar{Y}} \threebars_{\gamma,\eta,K}
		+
		\threebars \gr{\bz}; \gr{\bbz} \threebars_{\gamma;\bar{K}}, 
	\end{equs}
	where the last estimate is a consequence of \cite[Prop.~3.11]{hairer_pardoux} and the boundedness assumption on~$\gr{Y}, \gr{\bar{Y}}$. 
\end{proof}

We can now answer the question of solvability of \eqref{eq:fp_extended_eps} affirmatively. The following corollary is a simple consequence of~\cite[Thm.~7.8]{hairer_rs} which applies by the previous lemma.

\begin{corollary}\label{coro:solvability:fp_prob}
	For each $\gr{\bz} \in \gr{\MM}$ there exists a time $\bar{T} > 0$ such that for each $\eps \in I$ there exists a unique fixed point $\gr{U^\eps}(\gr{\bz}) \in \DD^{\gamma,\eta;\bar{T}}_{\gr{\CU}}(\gr{\bz})$ to \eqref{eq:fp_extended_eps}. 
	In addition, the map 
	\begin{equation}
		\fI_{\gr{\bz}}: I \to \DD^{\gamma,\eta;\bar{T}}_{\gr{\CU}}(\gr{\bz}), \quad
		\fI_{\gr{\bz}}(\eps) := \gr{U^\eps}(\gr{\bz})
		\label{eq:fI}
	\end{equation}
	satisfies the bound
	\begin{equation}
	\threebars \fI_{\gr{\bz}}(\eps);\fI_{\gr{\bbz}}(\eps) \threebars_{\gamma,\eta,K} \aac_{\bar{T}} \threebars \gr{\bz}; \gr{\bbz} \threebars_{\gamma;\bar{K}}, 
	\label{prop:a_priori_estimate:eq}
	\end{equation} 
	uniformly over all $\eps \in I$ and all $\gr{\bz}, \gr{\bbz} \in \gr{\MM}$ with $\threebars \gr{\bz} \threebars_{\gamma;\fK} \vee \threebars \gr{\bbz} \threebars_{\gamma;\fK} \leq R$, for each $R > 0$ and all compact sets $K \subseteq (0,\bar{T}) \x \T^2$. In particular,~$\fI_{\gr{\bz}}$ depends continuously on~$\gr{\bz}$.
\end{corollary}

\begin{remark}
	If it is clear from the context what the underlying model $\gr{\bz}$ is, we will just write $\gr{U^\eps}$ and~$\fI$, respectively. 
\end{remark}

\begin{remark}
	One might be inclined to think that continuity of $\fI_{\gr{\bz}}$ in $\gr{\bz}$ is automatic from the Implicit Function Theorem (IFT). We want to emphasise that this is \emph{not} the case due to the structure of the \enquote{fibred} space $\gr{\MM} \ltimes \DD^{\gamma,\eta}_{\gr{\CU}}$~(see~\thref{figure:model_space_fibres}). To the best of the authors' knowledge,  the \emph{geometry} of it has not yet been fully understood. However, we point to the work of Ghani Varzaneh, Riedel, and Scheutzow~\cite{ghani_riedel_scheutzow} who introduce the notion of a \emph{measurable field of Banach spaces} to characterise the (in spirit related) space of controlled rough paths.
	\begin{figure}[h]
		\centering
		\begin{tikzpicture}[scale=0.25]
		\node  (0) at (1.5, -17) {};
		\node  (1) at (11, -15) {};
		\node  (2) at (12.75, -10.5) {};
		\node  (3) at (2, -11.25) {};
		\node  (4) at (3.25, -26.25) {};
		\node  (5) at (10.25, -26.75) {};
		\node  (6) at (18.5, -27.25) {};
		\node  (7) at (21, -25) {};
		\node  (8) at (17, -21.75) {};
		\node  (9) at (9.75, -20.25) {};
		\node  (10) at (2.25, -24.25) {};
		\node  (11) at (2.25, -14.5) {};
		\node [bigdot, scale=0.3, label={[shift={(0.2,-0.6)}]$\gr{W}$}] (12) at (5.5, -12) {};
		\node [label={[shift={(0.0,-0.2)}]$\fI_{\gr{\bz}}(\eps)$}] (13) at (9.25, -11.75) {};
		\node  (14) at (15.75, -15) {};
		\node  (15) at (24.5, -17.25) {};
		\node  (16) at (28.25, -12.5) {};
		\node  (17) at (19.5, -10.25) {};
		\node  (21) at (5, -15.75) {};
		\node [bigdot, scale=0.3, label={below:$\gr{\bz}$}] (22) at (7, -24.75) {};
		\node [bigdot, scale=0.3, label={below:$\gr{\bbz}$}] (23) at (15.25, -24) {};
		\node  (24) at (7, -24.25) {};
		\node  (25) at (18.75, -13.5) {};
		\node  (26) at (21.75, -12.75) {};
		\node [label={[shift={(0.0,-0.2)}]$\fI_{\gr{\bbz}}(\eps)$}] (27) at (25.25, -13) {};
		\node  (28) at (20.25, -15.5) {};
		\node  (29) at (15.25, -23.75) {};
		\node [scale=0.3, label={below:$\DD_{\gr{\CU}}^{\gamma,\eta}(\gr{\bz})$}] (32) at (1, -17) {};
		\node [scale=0.3, label={below:$\DD_{\gr{\CU}}^{\gamma,\eta}(\gr{\bbz})$}] (33) at (24, -17.5) {};
		\node [scale=0.3, label={below:$\gr{\MM}$}] (34) at (13, -20) {};
		\node  (35) at (7.5, -24.75) {};
		\node  (36) at (14.75, -24) {};
		\node [bigdot, scale=0.3, label={[shift={(0.0,-0.6)}]$\gr{W}$}] (37) at (21.75, -12.75) {};
		\node  (38) at (8.5, -12.25) {};
		\node  (39) at (19, -12.5) {};
		\draw [in=165, out=30] (0.center) to (1.center);
		\draw [in=225, out=45] (1.center) to (2.center);
		\draw [bend right=15, looseness=1.25] (2.center) to (3.center);
		\draw [in=135, out=-105] (3.center) to (0.center);
		\draw [in=210, out=75] (10.center) to (9.center);
		\draw [in=135, out=30] (9.center) to (8.center);
		\draw [in=90, out=-45, looseness=1.25] (8.center) to (7.center);
		\draw [in=15, out=-105] (7.center) to (6.center);
		\draw [in=15, out=-180, looseness=1.25] (6.center) to (5.center);
		\draw [in=-60, out=-165, looseness=0.75] (5.center) to (4.center);
		\draw [in=255, out=105, looseness=1.25] (4.center) to (10.center);
		\draw [in=-165, out=45, looseness=0.75] (11.center) to (12);
		\draw [in=-165, out=15, looseness=0.75] (12) to (13.center);
		\draw [in=-165, out=60] (14.center) to (17.center);
		\draw [in=135, out=-15] (17.center) to (16.center);
		\draw [in=90, out=-165, looseness=0.50] (16.center) to (15.center);
		\draw [in=0, out=135] (15.center) to (14.center);
		\draw [<-, >=stealth, dashed] (21.center) to (24.center);
		\draw [in=150, out=45] (25.center) to (26.center);
		\draw [in=-150, out=-30] (26.center) to (27.center);
		\draw [->, >=stealth, dashed] (29.center) to (28.center);
		\draw [->, >=stealth, in=165, out=-30] (35.center) to (36.center);
		\draw [->, >=stealth, red, in=150, out=-30, looseness=1.25] (38.center) to (39.center);
		\end{tikzpicture}
		\caption{A schematic depiction of the space $\gr{\MM} \ltimes \DD^{\gamma,\eta}_{\gr{\CU}}$. The \enquote{fibres} $\DD^{\gamma,\eta}_{\gr{\CU}}(\gr{\bz})$ are Banach spaces that vary when the underlying model $\gr{\bz} \in \gr{\MM}$ changes.}
		\label{figure:model_space_fibres}
	\end{figure}
\end{remark}

\begin{notation}
	We will use the following notations from \cite[Prop.~3.25]{cfg}. Note that we suppress the dependence on the initial condition~$u_0 \in \CC^{\eta}(\T^2)$ that is fixed throughout the paper.
	\begin{itemize}
		\item We denote the solution map to \eqref{eq:fp_normal} by $\CS$, that is:
		\begin{equation}
			\CS: \MM \to \DD^{\gamma,\eta}_\CU(\bz), \quad
			\bz \mapsto U = U(\bz),
		\end{equation}
		\item We introduce the map $\CS_{tr}$ as 
		\begin{equation}
			\CS_{tr}: \CH \x \MM \to \DD^{\gamma,\eta}_{\CU}(T_h \bz), \quad
			\CS_{tr} (h, \bz) := \CS(T_h \bz).
		\end{equation}
		\item We denote by $\gr{\CS}$ the solution map to \eqref{eq:fp_extended}, that is:
		\begin{equation}
			\gr{\CS}: \gr{\MM} \to \DD^{\gamma,\eta}_{\gr{\CU}}(\gr{\bz}), \quad
			\gr{\bz} \mapsto \gr{U} = \gr{U}(\gr{\bz}),
		\end{equation}
		\item We introduce the map $\gr{\CS_{ex}}$ as 
		\begin{equation}
			\gr{\CS_{ex}}: \CH \x \MM \to \DD^{\gamma,\eta}_{\gr{\CU}}(E_h \bz), \quad
			\gr{\CS_{ex}} (h, \bz) := \gr{\CS}(E_h \bz).
		\end{equation}
	\end{itemize}
	For fixed $\eps \in I$, we add two more definitions to this list. 
	\begin{itemize}
		\item We denote by $\gr{\CS^{\eps}}$ the solution map to \eqref{eq:fp_extended_eps}, that is:
		\begin{equation}
			\gr{\CS^{\eps}}: \gr{\MM} \to \DD^{\gamma,\eta}_{\gr{\CU}}(\gr{\bz}), \quad
			\gr{\bz} \mapsto \gr{U^\eps} = \gr{U^\eps}(\gr{\bz}),
		\end{equation}	
		\item We introduce the map $\gr{\CS^{\eps}_{ex}}$ as 
		\begin{equation}
			\gr{\CS^{\eps}_{ex}}: \CH \x \MM \to \DD^{\gamma,\eta}_{\gr{\CU}}(E_h \bz), \quad
			\gr{\CS^{\eps}_{ex}} (h, \bz) := \gr{\CS^{\eps}}(E_h \bz).
		\end{equation}	
	\end{itemize}
By definition, $\gr{\CS^\eps}(\gr{\bz}) = \fI_{\gr{\bz}}(\eps)$: The map $\gr{\CS^\eps}$ just emphasises the dependence on the model~$\gr{\bz}$.
\end{notation}

\subsection*{A family of fixed-point equations}

Recall that our object of interest, introduced below \thref{eq:approx_renorm_shifted} in the introduction, can also be written as
\begin{equation}
	\hat{u}^{\eps}_\sh(\hbz)
	=
	\Phi(T_\sh \d_\eps\hbz) 
	= 
	\CR\del[1]{\CS_{tr}(\sh,\d_\eps\hbz)}
	=
	\lim_{\d \to 0} \CR\del[1]{\CS_{tr}(\sh,M_\d \bz^{\eps\xi_\d}}. 
\end{equation}
In order to expand this quantity in $\eps$ we would like to differentiate the map~$\eps \mapsto \CS_{tr}(\sh,\d_\eps\hbz)$ but immediately run into a problem: the space $\DD^{\gamma,\eta}_\CU(T_\sh \d_\eps\hbz)$ of modelled distributions that $\CS_{tr}(\sh,\d_\eps\hbz)$ lives in depends on $\eps$ itself. This is due to the \enquote{fibred} structure of $\gr{\MM} \ltimes \DD^{\gamma,\eta}_{\gr{\CU}}$, see~\thref{figure:model_space_fibres}.

We need the following lemma that is inspired by~\cite[Prop.~3.25]{cfg} to remediate this problem. 
\begin{lemma} \label{lem:dil_transl_rel}
	Let $\bz \in \MM$ be an admissible model, $h \in \CH$, and $\eps \in I$. Then, 
	\begin{equation}
		(\fd_\eps \circ \ft_{\<cms>}) \sbr[1]{\CS_{tr}(h,\d_\eps \bz)}
		=
		\gr{\CS^{\eps}_{ex}} (h, \bz).
	\end{equation}	
\end{lemma}

\begin{proof}
	From~\cite[Prop.~3.25]{cfg}, we know that
	\begin{equation*}
		\gr{\CS_{ex}}(h,\bbz) = \ft_{\<cms>} \del[1]{\CS_{tr}(h,\bbz)},
	\end{equation*}	
	for each~$\bbz \in \MM$, so by choosing $\bbz = \d_\eps \bz$ we only need to prove that 
	\begin{equation*}
		\fd_\eps \sbr[1]{\gr{\CS_{ex}}(h,\d_\eps\bz)} = \gr{\CS^\eps_{ex}}(h,\bz).
	\end{equation*}
	We denote the RHS by $\gr{U^\eps}$ and the LHS by $\fd_\eps(\gr{U_\eps})$, where~$\gr{U_\eps} := \gr{\CS_{ex}}(h,\d_\eps\bz)$. We then observe that
	\begin{equs}
		\thinspace
		&
		\CP^{E_h \bz} (G(\fd_\eps \gr{U_\eps})[\eps \<wn> + \<cm>])
		= \CP^{E_h \bz} (\fd_\eps(G(\gr{U_\eps})[\<wn> + \<cm>])) \\
		= \
		&
		\fd_\eps\sbr[1]{\CP^{\d_\eps E_h \bz} \del[0]{G(\gr{U_\eps})[\<wn> + \<cm>]}} 
		= \fd_\eps\sbr[1]{\CP^{E_h \d_\eps \bz} \del[0]{G(\gr{U_\eps})[\<wn> + \<cm>]}},
	\end{equs} 
	where the last and the penultimate  equality are due to~\thref{lem:ext_dil_commute} and~\thref{lem:consistency_dilation}\ref{lem:consistency_dilation:iii}, respectively. Note that we have also used that $\fd_\eps$ commutes with multiplication and application of~$G$. It follows that
	\begin{equs}
		\fd_\eps \gr{U_\eps}
		& =
		\fd_\eps\sbr[1]{\CP^{E_h \d_\eps \bz} (G(\gr{U_\eps})[\<wn> + \<cm>]) + \TT Pu_0} \\
		& =
		\CP^{E_h \bz} (G(\fd_\eps \gr{U_\eps})[\eps \<wn> + \<cm>]) + \TT Pu_0,
	\end{equs}
	from which we infer that $\fd_\eps \gr{U_\eps}$ solves~\eqref{eq:fp_extended_eps} (w.r.t. $\gr{\bz} := E_h \bz$). Hence, $\gr{U^\eps} = \fd_\eps \gr{U_\eps}$ by uniqueness.
\end{proof}

The following proposition links the fixed-point equation~\eqref{eq:fp_normal} with driving model $T_\sh \d_\eps \hbz \in \MM$ to the familiy of fixed-point problems in~\eqref{eq:fp_extended_eps} where the driving model now is $E_\sh \hbz \in \gr{\MM}$. Note that the latter does not suffer from $\eps$-dependency of the solution space~$\DD^{\gamma,\eta}(E_\sh \hbz)$, the problem we set out to solve.

\begin{proposition}\label{prop:fp_choice} 
	Let $\eps \in I$, $h \in \CH$, and $\bz \in \MM$. Then, 
	\begin{equation}
		\CR^{T_h \d_\eps \bz} \CS_{tr}(h,\d_\eps \bz)
		=
		\CR^{E_h \bz} \gr{\CS^{\eps}_{ex}} (h, \bz).
		\label{prop:fp_choice:eq1}
	\end{equation}
	In particular, choosing the BPHZ model $\hbz \in \MM$ leads to
	\begin{equation*}
		\hat{u}^{\eps}_\sh(\hbz) 
		= \CR^{E_\sh \hbz} \gr{\CS^{\eps}_{ex}} (\sh, \hbz)
		\equiv \CR^{E_\sh \hbz} \fI^{E_\sh \hbz}(\eps).
	\end{equation*}
\end{proposition}

\begin{proof}
	We combine~\cite[Lem.~3.22(ii)]{cfg}, which states that $\CR^{T_h\bbz} = \CR^{E_h\bbz} \circ \thinspace \ft_{\<cms>}$ for each $\bbz \in \MM$, and~\thref{lem:consistency_dilation}\ref{lem:consistency_dilation:ii} to obtain the equality
	\begin{equs}
		\CR^{T_h \d_\eps \bz} \CS_{tr}(h,\d_\eps \bz)
		=
		\CR^{E_h \d_\eps \bz} \thinspace \ft_{\<cms>}\del[1]{\CS_{tr}(h,\d_\eps \bz)} \\
		= 
		\CR^{E_h \bz} (\fd_\eps \circ \ft_{\<cms>}) \del[1]{\CS_{tr}(h,\d_\eps \bz)}
	\end{equs}
	where we have also used that $E_h$ and $\d_\eps$ commute, see~\thref{lem:ext_dil_commute}. The claim now follows from~\thref{lem:dil_transl_rel}.
\end{proof}

In~\thref{coro:solvability:fp_prob}, we have proved \emph{local} existence of the FP eq.~\eqref{eq:fp_extended_eps} for each $\gr{\bz} \in \gr{\MM}$. More precisely, we have obtained $\bar{T} = \bar{T}(\gr{\bz}) > 0$ such that $\gr{U^{\eps}}(\gr{\bz})$ exists up to time $\bar{T}$, uniformly over $\eps \in I$.

In general, however, the existence time $T_\infty(\eps,\gr{\bz})$ of $\gr{U^\eps}(\gr{\bz})$ does depend on~$\eps$.\footnote{Think of the deterministic ODE $\dot{x}_t = \eps x_t^2$ with I.C. $x_0 > 0$. Since $x_t = x_0 \del[0]{1-\eps x_0 t}^{-1}$, it has explosion time $T(\eps) = \del[0]{\eps x_0}^{-1}$, which is uniformly bounded from below by~$x_0^{-1}$ for $\eps \in [0,1]$. Hence, any $\bar{T} \leq x_0^{-1}$ would do.} 
Recall that we have started with a fixed time horizon $T \in (0, T_\infty^\sh \wedge T_0)$, so we need to find the right subset of~$I \x \gr{\MM}$ for which existence up to time~$T$ is ensured. We need the following corollary towards this end.

\begin{corollary} \label{coro:ex_time_h_tr_vs_ex}
	For each $\eps \in I$ and $\bz \in \MM$, we have $T_\infty(\eps,E_\sh\bz) = T_\infty(T_\sh\d_\eps\bz)$. 
\end{corollary}

\begin{proof}
	Since $\gr{U^\eps}(E_\sh\bz) = \gr{\CS_{ex}^\eps}(\sh,\bz)$ and $T_\infty(T_\sh \d_\eps \bz)$ is the explosion time of $\CS_{tr}(\sh,\d_\eps\bz)$, the claim follows im\-me\-dia\-te\-ly from~\thref{lem:dil_transl_rel}.
\end{proof}

With this result, the following corollary was proved in the last part of the proof of~\thref{prop:cameron_martin}.
\begin{corollary}\label{coro:expl_rho}
	For $T > 0$ as before, choose $\rho_0 = \rho_0(T) > 0$ as in~\thref{sec:explosion:cm_lem}. Then, for any $\rho \in (0,\rho_0]$ and $(\eps,\bz^{\minus}) \in I \x \MM_-$ with $\eps \barnorm{\bz^{\minus}} < \rho$, we have $T_\infty(\eps,E_\sh\bz) > T$ for the extension $\bz = \EE(\bz^{\minus})$ of~$\bz^{\minus}$.
\end{corollary}
For fixed $T$ and $\bz$, it makes sense to consider the \emph{maximal} $\eps \in I$ that satisfies the conditions in the previous corollary.

\begin{definition}
	For~$T > 0$ and~$\rho_0 = \rho_0(T) > 0$ as in~\thref{sec:explosion:cm_lem}, we define
		\begin{equation}
			\eps_0(T,\bz) := \frac{\rho_0(T)}{\barnorm{\bz^{\minus}}} \wedge 1 \qquad \del[3]{\frac{1}{0} := \infty}
			\label{eq:def_eps_zero}
		\end{equation} 
	Note that~$\eps_0(T,\bz) = \sup \{\eps \in I: \thinspace \eps \barnorm{\bz^{\minus}} < \rho_0(T)\}$.
\end{definition}

A few remarks before we summarise our findings thus far. 

\begin{remark}
	It is immediate to see that~$\eps_0(T,\cdot) \circ \EE$ is continuous on~$(\MM_-,\barnorm{\cdot})$ and, by~\thref{lem:estimate_min_norm_vs_hom_norm}, on~$(\MM_-,\threebars \cdot \threebars)$. Therefore,~$\eps_0(T,\cdot)$ is continuous on~$(\MM,\threebars \cdot \threebars)$ by continuity of~$\EE$, cf.~\thref{app:thm:extension};
	in particular, we have
		\begin{equation*}
			\liminf_{\d \to 0} \thinspace \eps_0(T,\hbz^{\xi_\d}) = \eps_0(T,\hbz) \quad \text{as} \quad \lim_{\d \to 0} \hbz^{\xi_\d} = \hbz.
		\end{equation*}
	In light of~\thref{coro:expl_rho}, this means that the renormalised solutions~$\hat{u}_{\xi_\d,\sh}^{(\eps)}$ and their limit~$\hat{u}_{\sh}^{(\eps)}$ as $\d \to 0$ exist up to time~$T$, uniformly over all~$\eps < \eps_0(T,\hbz)$.
	Since~$T > 0$ is fixed, the previous observation justifies to fix~$\eps_0 := \eps_0(T,\hbz)$ and~$I_0 := [0,\eps_0)$ for the rest of the article. From here on, we only consider~$\eps \in I_0$, so non-explosion is guaranteed both for the approximate and the limiting equations.
\end{remark}

\begin{remark} \label{not:green_models}
	Aligning with the nature of our problem, we will only consider models~$\gr{\bz} \in \gr{\MM}$ that satisfy~$\gr{\bz} = E_\sh \bz$ for some $\bz \in \MM$ from here on.
	In light of the previous remark, the solution~$\gr{U^\eps}$ w.r.t. $\gr{\bz} = E_\sh \hbz$ exists up to time $T$ for all $\eps \in I_0$. 
	Disregarding such questions of existence, one could work with any $\gr{\bz} \in \gr{\MM}$ at the expense of replacing $T \rightsquigarrow \bar{T} = \bar{T}(\gr{\bz})$ again, cf.~\thref{coro:solvability:fp_prob}.
\end{remark} 

\begin{remark}[The case~$\eps = 0$.] \label{rmk:W_indep_model_not}
	Recall that our problem reduces to the \emph{deterministic} PDE~\eqref{eq:wh}, namely
		\begin{equation*}
			(\partial_t - \Delta) w_\sh = g(w_\sh) \sh, \quad w_\sh(0,\cdot) = u_0,
		\end{equation*}
	in case~$\eps = 0$. This equation is represented by~\eqref{eq:fp_extended_eps} in case~$\eps_0$, so that $\CR(E_\sh \hbz, \thinspace \gr{U^0}) = w_\sh$ as one might reasonably expect. Now recall that $\gr{U^0}(E_\sh \hbz)$ takes values in the sector~$\gr{\CU} = \scal{\1,\<1>,\<1g>,X}$ of which only one symbol (namely $\<1>$) will be mapped to a non-deterministic distribution via $\hat{\Pi}$ (namely, a re-centered version of $G * \xi$). Compared with general ansatz~\eqref{eq:ansatz_U} for $\gr{Y} \in \DD^{\gamma,\eta}_{\gr{\CU}}(E_\sh\bz)$, one then quickly finds from~\eqref{eq:fp_extended_eps} with $\eps = 0$ that~$\scal{\gr{U^0}(E_\sh \hbz),\<1>} = 0$, which is consistent with our intuition. We gain even more information: since for any $\bz \in \MM$, 
		\begin{itemize}
			\item the interpretation of $\1$ and $X$ is determined by admissibility of $\bz$ and
			\item that of $\<1g>$ by the extension operator $E_\sh$
		\end{itemize}
	 this means that~$\gr{U^0}(E_\sh \bz)$ is actually \emph{independent} of~$\bz$.\footnote{Of course, this is only half the truth: We have~$\gr{U^0}(E_\sh \bz) \in \DD^{\gamma,\eta}(E_\sh\bz)$ and the RHS \emph{does} depend on~$\bz$. The way that this is meant is that~$\gr{U^0}(E_\sh \bz) \in \DD^{\gamma,\eta}(E_\sh\bbz)$ for~$\bz,\bbz \in \MM$.} Henceforth, we will therefore use the symbol~$\gr{W}$ for any $\gr{U^0}(E_\sh\bz)$ with explicit representation given by 
		\begin{equation}
			\gr{W}(z) = w_\sh(z) \1 + \phi_{\<1gs>}^{\gr{W}}(z) \<1g> + \scal{\phi_X^{\gr{W}}(z),X}. 
			\label{rmk:structure_W:eq}
		\end{equation}  
	Note that $\gr{W} = \fI_{E_\sh\hbz}(0)$ is exactly the constant term in the Taylor expansion of~$\fI_{E_\sh\hbz}(\eps)$ that we are after.
	With $W := U(\LL(\sh))$ as in~\eqref{eq:fp_normal}, it also satisfies
	\begin{equation*}
		\CR^{\LL(\sh)}  W = w_\sh = \CR^{E_\sh \bz} \gr{W}
	\end{equation*}
	for any~$\bz \in \MM$.	
	Finally, we emphasise that it is not important that $\sh$ (in $E_\sh$) is the minimiser of~$\FF$, at least at this stage. However, this property of~$\sh$ will be used in the sequel and thus makes it the natural choice.
\end{remark}

\subsection*{Summary of this section's findings}

For $T \in (0,T_\infty^\sh \wedge T_0)$, \thref{prop:fp_choice} justifies to abandon the map 
\begin{equation*}
I_0 \to \DD^{\gamma,\eta,T}_{\CU}(T_\sh \d_\eps \hbz), \quad \eps \mapsto \CS_{tr} (\sh, \d_\eps \hbz)
\end{equation*}
and instead focus on 
\begin{equation}
\fI_{E_\sh \hbz}: I_0 \to \DD^{\gamma,\eta,T}_{\gr{\CU}}(E_\sh \hbz), \quad
\eps \mapsto \gr{\CS^{\eps}_{ex}} (\sh, \hbz),
\end{equation}
the solution map to~\eqref{eq:fp_extended_eps} (for $\gr{\bz} := E_\sh\hbz$) which we have already encountered in~\thref{prop:a_priori_estimate}, eq.~\eqref{eq:fI}. Note, however, that we have replaced $\bar{T} \rightsquigarrow T$ and, accordingly, $I \rightsquigarrow I_0$.
The target Banach space of this map does not depend on $\eps$ anymore, so we can study its Fr\'{e}chet differentiability in the next section.

\begin{landscape}
	\begin{figure}[h]
		\subsection*{A schematic overview of our strategy}
		\centering
		\includegraphics[scale=0.7]{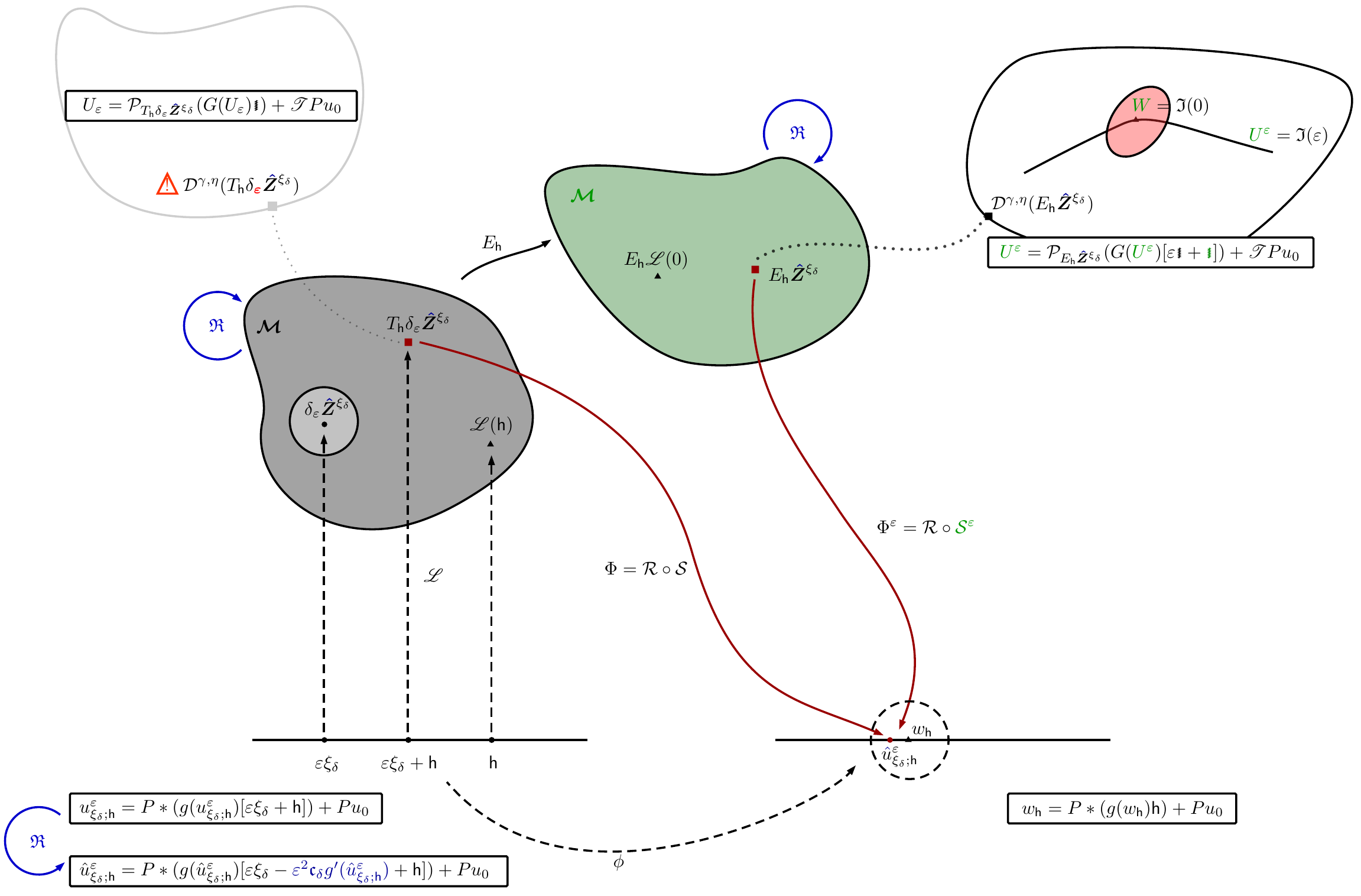}
		\captionsetup{font=small,format=plain}
		\caption{
			Fix~$\d > 0$. On the LHS, we see the \enquote{noises}~$\sh$ and~$\eps\xi_\d + \sh$ which are mapped by the solution map~$\phy$ to~$w_\sh$ and~$u_{\xi_\d;\sh}^\eps$, respectively.
			Dashed lines such as that for~$\phy$ represent \emph{discontinuous} operations: we thus \emph{lift}~$\sh$ and~$\eps\xi_\d + \sh$ (and also~$\eps \xi_\d$) to the model space~$\MM$.
			Attached to the model~$T_\sh\d_\eps\bz^{\xi_\d}$ is a space~$\DD^{\gamma,\eta}(T_\sh\d_\eps \bz^{\xi_\d})$ of modelled distributions: it hosts an abstract FP problem for~$U_\eps$ which encodes~$u_{\xi_\d;\sh}^\eps$. The solution map~$\Phi$ that sends $T_\sh \d_\eps \bz^{\xi_\d}$ to $U_\eps$ via~$\CS$ and then to~$u_{\xi_\d;\sh}^\eps$ via~$\CR$ is~\emph{continuous} now.
			Unfortunately,~$\bz^{\xi_\d}$ does not converge as~$\d \to 0$: a shortcoming cured by the \emph{renormalisation group}~$\fR$ which acts on~$\MM$ to change~$\bz^{\xi_\d}$ into~$\hbz^{\xi_\d}$ and,~\enquote{dually},~$u_{\xi_\d;\sh}$ into~$\hat{u}_{\xi_\d;\sh}$.
			Indeed,~$\hbz^{\xi_\d}$ does have a limit~$\hbz$ as~$\d \to 0$ and so one can~\emph{define}~$\hat{u}_\sh^\eps := \Phi(T_\sh\d_\eps\hbz)$.
			The lightgray ball around~$\d_\eps \hbz^{\xi_\d}$ symbolises the localisation procedure of sec.~\ref{sec:localisation_cameron_martin} (for technical reasons, it was carried out in~$\MM_-$ rather than~$\MM$).
			We run into problems once~$\eps$ starts to vary: Because the Banach space~$\DD^{\gamma,\eta}(T_\sh\d_\eps \hbz^{\xi_\d})$ depends on~$\eps$ itself, we cannot differentiate~$U_\eps$ w.r.t. that parameter.
			The remedy is to work on the extended model space~$\gr{\MM}$ depicted on the RHS: In the Banach space~$\DD^{\gamma,\eta}(E_\sh\hbz^{\xi_\d})$ attached to~$E_\sh\hbz^{\xi_\d}$, we set up a FP problem with~$\eps$ as a parameter.
			Its solution~$\gr{U^\eps} = \gr{\CS^\eps}(E_\sh\hbz^{\xi_\d})$ is reconstructed to~$\hat{u}_{\xi_\d;\sh}^\eps$ by the operator~$\CR$ just as~$U_\eps$ is~(\thref{prop:fp_choice}), as symbolised by the two red arrows.
			However, the map~$\fI: [0,1] \ni \eps \mapsto \gr{U^\eps} \in \DD^{\gamma,\eta}(E_\sh \hbz^{\xi_\d})$ is well-defined and may be investigated for its differentiability properties.
			The shape of this figure is inspired by L. Zambotti's slides on \emph{SPDEs and regularity structures} from the Second Haifa Probability School 2017~\cite{zambotti_slides}.
		}
		\label{figure:overview}
	\end{figure}
\end{landscape}

\subsection{Differentiability of the fixed-point map} \label{sec:diff_fp_map}

Our goal in this section is to establish differentiability of $\fI_{E_{\mathsf{h}} \hbz}$ via the Implicit Function Theorem in Banach spaces~\cite[Thm.~4.E.]{zeidler_applied_fa}. 
Throughout this section, $\gr{\bz} := E_\sh \hbz$ will be fixed; see, however,~\thref{not:green_models}. We then introduce the space 
	\begin{equation*}
		\CY_T^{\gr{\bz}} := I_0 \x \DD^{\gamma,\eta,T}_{\gr{\CU}}(\gr{\bz}), \quad 
		\norm[0]{(\eps,\gr{Y})}_{\CY_T^{\gr{\bz}}} := \del[1]{\eps^2 + \threebars \gr{Y} \threebars_{\gamma,\eta;T}^2}^{\nicefrac{1}{2}}
	\end{equation*}
and consider the map
\begin{equation}
	\Psi^{\gr{\bz}}_T: \CY_T^{\gr{\bz}} \to \DD^{\gamma,\eta,T}_{\gr{\CU}}(\gr{\bz}), \quad
	(\eps,\gr{Y}) \mapsto \gr{Y} - \CM_{\gr{\bz}}(\eps,\gr{Y}),
	\label{eq:Psi_impl_fct}
\end{equation}
with~$\CM^{\gr{\bz}}(\eps,\gr{Y}) = \CP^{\gr{\bz}}\del[1]{G(\gr{Y}) [\eps\<wn> + \<cm>]} + \TT Pu_0$ as introduced in~\eqref{eq:def_fp_map}. Whenever convenient and without risk of confusion, we will drop the dependence on~$T$ and~$\gr{\bz}$ for all involved quantities.

Proposition~\ref{lem:diffb_G_from_g} allows to calculate the candidates for the derivatives~$D^{(k)} \Psi$.

\begin{lemma}\label{lem:k_der_Psi}
	Let $k \in \N$ and suppose the function $\Psi$ is $\CC^k$ F-differentiable in $(\eps,\gr{Y}) \in \CY$. Then, its $k$-th F-derivative $D^{(k)} \Psi(\eps,\gr{Y}) \in \CL^{(k)}(\DD^{\gamma,\eta;T}_{\gr{\CU}}(\gr{\bz}))$\footnote{For a normed linear space~$V$ and $k \in \N$, we denote by~$\CL^{(k)}(V)$ the $k$-linear, bounded maps from~$V$ into itself.} at $(\eps,\gr{Y})$ is given by
	\begin{equs}[][lem:k_der_Psi:eq]
		\thinspace
		&
		D^{(k)} \Psi(\eps,\gr{Y})\sbr[1]{(\eps_1,\gr{Y_1}), \ldots, (\eps_k,\gr{Y_k})} \\
		&
		=
		\gr{Y_1} \mathbf{1}_{k=1} 
		- \CP\del[2]{G^{(k)}(\gr{Y}) \prod_{m=1}^k \gr{Y_m} \thinspace [\eps\<wn> + \<cm>]}
		- \sum_{\ell=1}^{k} \eps_\ell \CP\del[2]{G^{(k-1)}(\gr{Y}) \prod_{\substack{m=1, \\ m \neq \ell}}^k \gr{Y_m} \thinspace \<wn>}
	\end{equs}	
\end{lemma}

\begin{proof}
	Using \thref{lem:diffb_G_from_g}, one computes
	\begin{equs}
		D \Psi(\eps,\gr{Y})(\eps_1,\gr{Y_1}) 
		&
		= \frac{\dif}{\dif t}\sVert[2]_{t = 0} \Psi(\eps + t \eps_1, \gr{Y} + t \gr{Y_1})
		= \gr{Y_1} - \CP\del[1]{G'(\gr{Y}) \gr{Y_1} [\eps\<wn> + \<cm>]} - \eps_1 \CP(G(\gr{Y})\thinspace\<wn>),
	\end{equs}	
	so the claim is true for $k=1$. Analogously, we find
	\begin{equs}
		\thinspace
		&
		D^{(2)} \Psi(\eps,\gr{Y})[(\eps_1,\gr{Y_1}),(\eps_2,\gr{Y_2})] 
		= \frac{\dif}{\dif t}\sVert[2]_{t = 0} D\Psi(\eps + t \eps_1, \gr{Y} + t \gr{Y_1})(\eps_2,\gr{Y_2}) \\
		&
		=
		- \CP\del[1]{G''(\gr{Y})\gr{Y_1 Y_2} [\eps \<wn> + \<cm>]} -\eps_1 \CP(G'(\gr{Y})\gr{Y_2}\<wn>) -\eps_2 \CP(G'(\gr{Y})\gr{Y_1}\<wn>)
	\end{equs}
	which settles the case $k=2$. We proceed by induction, $k \mapsto k+1$:
	\begin{equs}
		\thinspace
		&
		D^{(k+1)} \Psi(\eps,\gr{Y})\sbr[1]{(\eps_1,\gr{Y_1}), \ldots, (\eps_{k+1},\gr{Y_{k+1}})} \\
		&
		=
		\frac{\dif}{\dif t}\sVert[2]_{t = 0}
		D^{(k)} \Psi(\eps + t\eps_{k+1},\gr{Y} + t\gr{Y_{k+1}})\sbr[1]{(\eps_1,\gr{Y_1}), \ldots, (\eps_k,\gr{Y_k})} \\
		&
		=
		-\CP\del[1]{G^{(k+1)}(\gr{Y}) \gr{Y_1} \cdot \ldots \cdot \gr{Y_k} \gr{Y_{k+1}} [\eps\<wn> + \<cm>]}
		- \eps_{k+1} \CP\del[1]{G^{(k)}(\gr{Y})\gr{Y_1} \cdot \ldots \cdot \gr{Y_k} \<wn>} \\
		&
		- \sum_{\ell=1}^{k} \eps_\ell \CP\del[1]{G^{k}(\gr{Y}) \gr{Y_{k+1}} \prod_{\substack{m=1, \\ m \neq \ell}}^k \gr{Y_m} \thinspace \<wn>}.
	\end{equs}
	The last expression can be reorganised as claimed.
\end{proof}

Note that the premise of the previous proposition is yet to be validated. In any case, however, the expressions on the RHS of~\eqref{lem:k_der_Psi:eq} make perfect sense: We will use this fact to our advantage when verifying the criteria for the Implicit Function Theorem~(IFT) to apply. 

We start with differentiability of the map $\Psi$ from~$\eqref{eq:Psi_impl_fct}$. For the corresponding statement in the rough paths setting (without dilation and encoded Cameron-Martin shifts), see Bailleul~\cite{bailleul_reg_ito_lyons}.

\begin{lemma}[IFT,~$\boldsymbol{1}$st condition] \label{thm:frechet:gpam} 
	Suppose $g \in \CC^{\ell + 4}$, $\ell \geq 1$, and construct the lift $G^{(m)}$ of $g^{(m)}$, $m = 1, \ldots, \ell$, by means of~\eqref{eq:G_ext_rs}.
	Then, there exists a neighbourhood $N$ of $(0,\gr{W}) \in \CY$ on which the map $\Psi$ is $\CC^\ell$ Fr\'{e}chet with derivatives given in~\eqref{lem:k_der_Psi:eq}. 
\end{lemma}

\begin{remark}
	It has proved futile to use the IFT to prove Malliavin differentiability of the solution map to (singular) stochastic (P)DEs. For gPAM, this is formulated in~\cite[Prop.\thinspace 4.7]{cfg}. In the proof of that proposition, the authors establish the $\CC^1$ F-differentiability of some $\widetilde{\Psi}$ (there called~$F_\gamma$) which coincides with $\Psi$ up to the roles of $\<wn>$ (there denoted by~$\textcolor{darkblue}{\Xi}$) and $\<cm>$ (there denoted by~$\textcolor{darkblue}{H}$) reversed.  
	Sch\"onbauer~\cite{schoenbauer} has used that idea in full generality, covering any stochastic PDE that can be solved with the theory of regularity structures.
	
	Prior to these recent developments,~Nualart and Saussereau~\cite{nualart_saussereau} have employed the same argument in the context of SDEs. So have Deya and Tindel~\cite{deya_tindel}, albeit for stochastic PDEs driven by the fractional BM with Hurst index $H > \frac{1}{2}$.
\end{remark}

In the following proof, we will occasionally refer to the proof of~\thref{lem:diffb_G_from_g}, where similar arguments have been used. We recommend the reader to check that proposition before proceeding.

\begin{proof} 
	Let $M > 0$, denote by $\CB_T^{\gamma,\eta}(\gr{W},M)$ a closed ball in $\DD^{\gamma,\eta;T}_{\gr{\CU}}(\gr{\bz})$ of radius $M$, centered at $\gr{W}$, and set $B_T^{\gamma,\eta}(\gr{W},M) := I_0 \x \CB_T^{\gamma,\eta}(\gr{W},M) \subseteq \CY$.
	We fix $\alpha := \reg{\<wn>} = \reg{\<cm>} = - 1 - \kappa$.
	
	Suppose we had already established that $\Psi \in \CC^k$ at $(\eps,\gr{Y}) \in B_T^{\gamma,\eta}(\gr{W},M)$ for some $k \in \{1,\ldots,\ell-1\}$; its derivatives up to order $k$ are then given by~\eqref{lem:k_der_Psi:eq}. 
	We want to do the induction step $k \mapsto k+1$ and prove that $\Psi \in \CC^{k+1}$ at~$(\eps,\gr{Y})$. As $D^{(k)}\Psi(\eps,\gr{Y})$ is linear in $\eps$, it suffices to consider the case $\eps = 0$. We let $(\eps_{k+1},\gr{Y_{k+1}}) \in I_0 \x \CB_T^{\gamma,\eta}(\gr{0},M)$ and study the $k$-th order remainder~(see~\cite[p.~$229$, eq.~($7$)]{zeidler_applied_fa})
	\begin{align}
		\thinspace
		&
		\Delta^{(k)}\sbr[1]{(\eps_i,\gr{Y_i})_{i=1}^{k+1}}
		\equiv 
		\Delta^{(k)}\sbr[1]{(0,\gr{Y});(\eps_i,\gr{Y_i})_{i=1}^{k+1}} \label{thm:frechet:gpam:pf_remainder_k} \\
		&
		:=
		D^{(k)}\Psi(\eps_{k+1},\gr{Y} + \gr{Y_{k+1}})\sbr[1]{(\eps_i,\gr{Y_i})_{i=1}^k}
		-
		D^{(k)}\Psi(0,\gr{Y})\sbr[1]{(\eps_i,\gr{Y_i})_{i=1}^k}
		-
		D^{(k+1)}\Psi(0,\gr{Y})\sbr[1]{(\eps_i,\gr{Y_i})_{i=1}^{k+1}} \notag
	\end{align}
	where $D^{(k+1)}\Psi(0,\gr{Y})\sbr[1]{(\eps_i,\gr{Y_i})_{i=1}^{k+1}}$ is \emph{defined} by the corresponding quantity in \eqref{lem:k_der_Psi:eq}. At this stage, this notation is only suggestive. However, we will prove that
	\begin{equation}
		\sup_{\substack{\norm[0]{(\eps_m,\gr{Y_m})}_{\CY} \leq 1, \\ 1 \leq m \leq k}} \threebars \Delta^{(k)}\sbr[1]{(0,\gr{Y});(\eps_i,\gr{Y_i})_{i=1}^{k+1}} \threebars_{\gamma,\eta;T} = o\del[1]{\norm[0]{(\eps_{k+1},\gr{Y_{k+1}})}_{\CY}}, \quad (\eps_{k+1},\gr{Y_{k+1}}) \to 0,
		\label{thm:frechet:gpam:pf_remainder_k_est}
	\end{equation}
	which indeed qualifies~$D^{(k+1)}\Psi(0,\gr{Y})$ as the $(k+1)$-th derivative of $\Psi$ at $(0,\gr{Y})$. An elementary calculation gives
	\begin{equation}
		\Delta^{(k)}\sbr[1]{(0,\gr{Y});(\eps_i,\gr{Y_i})_{i=1}^{k+1}} 
		=
		\Delta^{(k)}_{\<wns>}\sbr[1]{\gr{Y};(\eps_i,\gr{Y_i})_{i=1}^{k+1}} 
		+ 
		\Delta^{(k)}_{\<cms>}\sbr[1]{\gr{Y};(\gr{Y_i})_{i=1}^{k+1}}
		\label{thm:frechet:gpam:pf_remainder_k_decomp}
	\end{equation}
	where one summand collects the terms with the noise symbol $\<wn>$\thinspace, namely
	\begin{equs}[][thm:frechet:gpam:pf_remainder_k_wn]
		\Delta^{(k)}_{\<wns>}\sbr[1]{\gr{Y};(\eps_i,\gr{Y_i})_{i=1}^{k+1}}
		& :=
		\Delta^{(k);1}_{\<wns>}\sbr[1]{\gr{Y};(\eps_i,\gr{Y_i})_{i=1}^{k+1}} 
		+ 
		\Delta^{(k);2}_{\<wns>}\sbr[1]{\gr{Y};(\eps_i,\gr{Y_i})_{i=1}^{k+1}}, \\
		\Delta^{(k);1}_{\<wns>}\sbr[1]{\gr{Y};(\eps_i,\gr{Y_i})_{i=1}^{k+1}}
		& :=
		- \eps_{k+1}\CP\del[2]{\sbr[1]{G^{(k)}(\gr{Y} + \gr{Y_{k+1}}) - G^{(k)}(\gr{Y})}\prod_{m=1}^k \gr{Y_m} \thinspace \<wn>}, \\
		\Delta^{(k);2}_{\<wns>}\sbr[1]{\gr{Y};(\eps_i,\gr{Y_i})_{i=1}^{k+1}}
		& :=
		- \sum_{\ell=1}^k \eps_\ell \CP\del[2]{\sbr[1]{G^{(k-1)}(\gr{Y} + \gr{Y_{k+1}}) - G^{(k-1)}(\gr{Y}) - G^{(k)}(\gr{Y})\gr{Y_{k+1}}} \prod_{\substack{m=1, \\ m \neq \ell}}^k \gr{Y_m} \thinspace \<wn>},
	\end{equs}
	and the other one collects those with the Cameron-Martin symbol $\<cm>$\thinspace, namely
	\begin{equation}
		\Delta^{(k)}_{\<cms>}\sbr[1]{\gr{Y};(\gr{Y_i})_{i=1}^{k+1}}
		:=
		- \CP\del[2]{\sbr[1]{G^{(k)}(\gr{Y} + \gr{Y_{k+1}}) - G^{(k)}(\gr{Y}) - G^{(k+1)}(\gr{Y})\gr{Y_{k+1}}} \prod_{m=1}^k \gr{Y_m} \thinspace \<cm>}.
		\label{thm:frechet:gpam:pf_remainder_k_cm}
	\end{equation}
	\begin{enumerate}[label=\textbf{Part \arabic*:}, align=left, leftmargin=0pt, labelindent=0pt,listparindent=0pt, itemindent=!]
		\item \textbf{Analysis of the Cameron-Martin part~$\boldsymbol{\Delta^{(k)}_{\<cms>}}$.}
		At first, we write
		\begin{equation}
			\Delta^{(k)}_{\<cms>}\sbr[1]{\gr{Y};(\gr{Y_i})_{i=1}^{k+1}} = -\CP\del[1]{R^{(k)}\sbr[1]{\gr{Y};(\gr{Y_i})_{i=1}^{k+1}} \<cm>}
		\end{equation}
		with $R^{(k)}\sbr[1]{\gr{Y};(\gr{Y_i})_{i=1}^{k+1}}$ implicitly defined from~\eqref{thm:frechet:gpam:pf_remainder_k_cm} and explicitly given in~\eqref{lem:diffb_G_from_g:pf:remainder_k} in the proof of~\thref{lem:diffb_G_from_g}, where it is also identified as an element of~$\DD^{\gamma,\eta;T}_{\gr{\CU}}(\gr{\bz})$. 
		 
		As in the proof of \cite[Lem.~9.1]{hairer_rs}, the map~$z \mapsto \<cm>$ is easily seen to be an element in $\DD^{\zeta,\zeta}(\gr{\bz})$, $\zeta > 0$, with values in a sector of reg.~$\alpha$. From \cite[Prop.~6.12]{hairer_rs}, we thus infer that
		$R^{(k)}\sbr[1]{\gr{Y};(\gr{Y_i})_{i=1}^{k+1}} \<cm> \in \DD^{\gamma+\alpha,\eta+\alpha;T}(\gr{\bz})$
		with values in another sector of reg.~$\alpha$.
		Since $\eta < \bar{\a} := \alpha+2$ and $\bar{\a} > 0$\footnote{See \thref{rmk:diff_fp_map_phi43} below for a comment on the implications of this constraint for adaptions of our argument to cover the $\Phi^4_3$ equation where $\bar{\a} < 0$.}, we have
		\begin{equation}
			\threebars \Delta^{(k)}_{\<cms>}\sbr[1]{\gr{Y};(\gr{Y_i})_{i=1}^{k+1}} \threebars_{\gamma,\eta;T}
			\aac
			\threebars \Delta^{(k)}_{\<cms>}\sbr[1]{\gr{Y};(\gr{Y_i})_{i=1}^{k+1}} \threebars_{\gamma,\bar{\a};T}
			\aac
			\threebars \Delta^{(k)}_{\<cms>}\sbr[1]{\gr{Y};(\gr{Y_i})_{i=1}^{k+1}} \threebars_{\gamma + \bar{\a},\bar{\a};T}
			\label{pf:thm:frechet:gpam:est_cm:aux}
		\end{equation}
		and now apply~\cite[Prop.'s~6.16 \& 6.12]{hairer_rs}\footnote{Actually, instead of \cite[Prop.~6.12]{hairer_rs}, we apply an analogue of \cite[Thm.~4.7]{hairer_rs} (valid for $\DD^{\gamma}$ spaces \emph{without} blow-ups) in the setting of $\DD^{\gamma,\eta}$ spaces \emph{with} blow-ups.} to obtain that
		\begin{equation} 
			\threebars \Delta^{(k)}_{\<cms>}\sbr[1]{\gr{Y};(\gr{Y_i})_{i=1}^{k+1}} \threebars_{\gamma+\bar{\a},\bar{\a};T} 
			\aac
			\threebars R^{(k)}\sbr[1]{\gr{Y};(\gr{Y_i})_{i=1}^{k+1}} \<cm> \thinspace \threebars_{\gamma + \a,\eta + \a;T}
			\aac
			\threebars R^{(k)}\sbr[1]{\gr{Y};(\gr{Y_i})_{i=1}^{k+1}} \threebars_{\gamma,\eta;T}.
			\label{pf:thm:frechet:gpam:est_cm}
		\end{equation}
		The argument is now identical to that in~\eqref{lem:diffb_G_from_g:pf:remainder_k_est_prod_est}: relying on Taylor's theorem up to order one (i.e.~\thref{prop:taylor_modelled_distrib}), eq.~\eqref{prop:taylor_modelled_distrib_eq}) we have
		\begin{equs}[][estimate_unif]
			\threebars R^{(k)}\sbr[1]{\gr{Y};(\gr{Y_i})_{i=1}^{k+1}} \threebars_{\gamma,\eta;T}
			& 
			\aac
			\del[1]{\threebars \gr{Y} \threebars_{\gamma,\eta;T} + \threebars \gr{Y_{k+1}} \threebars_{\gamma,\eta;T}} \threebars \gr{Y_{k+1}} \threebars_{\gamma,\eta;T}^2 \prod_{n=1}^k \threebars \gr{Y_n} \threebars_{\gamma,\eta;T} \\
			&
			\leq
			C(\threebars \gr{W} \threebars_{\gamma,\eta;T},M) \threebars \gr{Y_{k+1}} \threebars_{\gamma,\eta;T}^2, 
		\end{equs}
		where we have used that $\gr{Y} \in \CB^T_{\gamma,\eta}(\gr{W},M)$, $\gr{Y_{k+1}} \in \CB^T_{\gamma,\eta}(\gr{0},M)$,~and that $\threebars \gr{Y_m} \threebars_{\gamma,\eta;T} \leq 1$ for $1 \leq m \leq k$ as assumed in~\eqref{thm:frechet:gpam:pf_remainder_k_est}. 
		We then infer that
		\begin{equation}
			\sup_{\substack{\threebars \gr{Y_m} \threebars_{\gamma,\eta;T} \leq 1, \\ 1 \leq m \leq k}} \threebars \Delta^{(k)}_{\<cms>}\sbr[1]{\gr{Y};(\gr{Y_i})_{i=1}^{k+1}} \threebars_{\gamma,\eta;T}
			= O\del[1]{\threebars \gr{Y_{k+1}} \threebars_{\gamma,\eta;T}^2}, \quad \gr{Y_{k+1}} \to 0.
			\label{eq:cm_est}
		\end{equation}
		\item \textbf{Analysis of the noise part~$\boldsymbol{\Delta^{(k)}_{\<wns>} = \Delta^{(k);1}_{\<wns>} + \Delta^{(k);2}_{\<wns>}}$.}
		Modulo minor changes, the arguments are the same as before. For completeness, we still provide all the estimates, starting with $\Delta^{(k);2}_{\<wns>}$. 
		
		We define $P_k^r := \prod_{\substack{m=1, \\ m \neq r}}^k \gr{Y_m}$ for $r = 1,\ldots,k$ and $R^{(k);r}$ like the RHS of~\eqref{lem:diffb_G_from_g:pf:remainder_k_est_prod}, only with $P_k$ replaced by~$P_k^r$. Exactly the same argument as in~\eqref{pf:thm:frechet:gpam:est_cm:aux} and \eqref{pf:thm:frechet:gpam:est_cm} then gives
		\begin{equation}
			\threebars \Delta^{(k);2}_{\<wns>}\sbr[1]{\gr{Y};(\eps_i,\gr{Y_i})_{i=1}^{k+1}} \threebars_{\gamma,\eta;T}
			\aac
			\sum_{r=1}^k \eps_r \threebars R^{(k);r}\sbr[1]{\gr{Y};(\gr{Y_i})_{i=1}^{k+1}} \threebars_{\gamma,\eta;T},
			\label{eq:rem_wn_2nd_term_est}
		\end{equation}
		which by the same arguments as for~\eqref{lem:diffb_G_from_g:pf:remainder_k_est_prod_est} and then \eqref{estimate_unif} results in the estimate
		\begin{equs}
			\threebars \Delta^{(k);2}_{\<wns>}\sbr[1]{\gr{Y};(\eps_i,\gr{Y_i})_{i=1}^{k+1}} \threebars_{\gamma,\eta;T}
			& 
			\aac
			\sum_{r=1}^k \eps_r \del[1]{\threebars \gr{Y} \threebars_{\gamma,\eta;T} + \threebars \gr{Y_{k+1}} \threebars_{\gamma,\eta;T}} \threebars \gr{Y_{k+1}} \threebars_{\gamma,\eta;T}^2 \prod_{\substack{m=1, \\ m \neq r}}^k \threebars \gr{Y_m}  \threebars_{\gamma,\eta;T} \\
			&
			\leq
			C(\threebars \gr{W} \threebars_{\gamma,\eta;T},M) \threebars \gr{Y_{k+1}} \threebars_{\gamma,\eta;T}^2,
		\end{equs}
		where we have additionally used that $\eps_r \leq 1$ for $1\leq r \leq k$. 
		In turn, we obtain
		\begin{equation}
			\sup_{\substack{\norm[0]{(\eps_m,\gr{Y_m})}_{\CY} \leq 1, \\ 1 \leq m \leq k}} 
			\threebars \Delta^{(k);2}_{\<wns>}\sbr[1]{\gr{Y};(\eps_i,\gr{Y_i})_{i=1}^{k+1}} \threebars_{\gamma,\eta;T}
			= O\del[1]{\threebars \gr{Y_{k+1}} \threebars_{\gamma,\eta;T}^2}, \quad \gr{Y_{k+1}} \to 0.
			\label{eq:wn_est_2}
		\end{equation}
	\end{enumerate}	
	The term $\Delta_{\<wns>}^{(k),;1}\sbr[1]{\gr{Y};(\eps_i,\gr{Y_i})_{i=1}^{k+1}}$ gets treated in the same way, only this time we use Taylor's formula only up to $0$-th order (\thref{prop:taylor_modelled_distrib}, eq.~\eqref{prop:taylor_modelled_distrib_eq}). The same arguments as before then lead to
	\begin{equation}
		\sup_{\substack{\norm[0]{(\eps_m,\gr{Y_m})}_{\CY} \leq 1, \\ 1 \leq m \leq k}} \threebars \Delta^{(k);1}_{\<wns>}\sbr[1]{\gr{Y};\eps_{k+1},(\gr{Y_i})_{i=1}^{k+1}} \threebars_{\gamma,\eta;T}
		= O\del[1]{\eps_{k+1} \threebars \gr{Y_{k+1}} \threebars_{\gamma,\eta;T}}, \quad (\eps_{k+1},\gr{Y_{k+1}}) \to 0.
		\label{eq:wn_est_1}
	\end{equation}
	Combining~\eqref{eq:cm_est},~\eqref{eq:wn_est_2}, and \eqref{eq:wn_est_1} gives 
	\begin{equation*}
		\sup_{\substack{\norm[0]{(\eps_m,\gr{Y_m})}_{\CY} \leq 1, \\ 1 \leq m \leq k}} \threebars \Delta^{(k)}\sbr[1]{(0,\gr{Y});(\eps_i,\gr{Y_i})_{i=1}^{k+1}} \threebars_{\gamma,\eta;T} = O\del[1]{\norm[0]{(\eps_{k+1},\gr{Y_{k+1}})^2}_{\CY}}, \quad (\eps_{k+1},\gr{Y_{k+1}}) \to 0,
	\end{equation*}
	so the claim~\eqref{thm:frechet:gpam:pf_remainder_k_est} follows. 
	All the estimates are also valid for $k = 0$, which establishes $\Psi \in \CC^1$ at $(\eps,\gr{Y}) \in B_T^{\gamma,\eta}(\gr{W},M)$, so we can close the inductive proof.
\end{proof}

\begin{remark}\label{rmk:diff_fp_map_phi43}
	At first glance, it might seem that the arguments in~\eqref{pf:thm:frechet:gpam:est_cm:aux} exclude amendments of our results to cover $\Phi^4_3$ where $\a \in (-\nicefrac{18}{7},-\nicefrac{5}{2})$ and, as a consequence, $\bar{\a} := \a+ 2 < 0$, cf.~\cite[Sec.~9.4]{hairer_rs}. However, this is not the case. Fundamentally, this is due to~\cite[Lem. 9.7 \& Prop. 9.8]{hairer_rs} (and their proofs) which state that, for $\gamma > \abs{2\a +4}$ and $\eta \in (-\nicefrac{2}{3},\bar{\a})$, the map $U \mapsto U^3$ is strongly locally Lipschitz from $\DD^{\gamma,\eta}_V$ to $\DD^{\gamma + 2\a +4,3\eta}_{\bar{V}}$ for sectors $V$ and $\bar{V}$ of regularity $\bar{\a}$ and $3\bar{\a}$, respectively. Hence, \cite[Prop.~6.16]{hairer_rs} implies that
	\begin{equation*}
		\threebars \CP(U^3) \threebars_{\gamma,\eta;T}
		\aac 
		\threebars \CP(U^3) \threebars_{\gamma,3\eta+2;T}
		\aac 
		\threebars \CP(U^3) \threebars_{\gamma +2\alpha +6,3\eta+2;T}
		\aac 
		\threebars U^3 \threebars_{\gamma +2\alpha +4,3\eta;T}
		\aac
		\threebars U \threebars_{\gamma,\eta;T}
	\end{equation*}
	where the first two inequalities hold because $\eta < 3\eta +2$ and $2\a + 6 > 0$.
\end{remark}

The following lemma verifies the second condition for the IFT to apply.

\begin{lemma}[IFT, $\boldsymbol{2}$nd condition] \label{lem:sec_cond_ift}
	For each $T \in (0,T_\infty^\sh \wedge T_0)$ and $(\eps,\gr{J}) \in I \x \DD_{\gr{\CU}}^{\gamma,\eta,T}(\gr{\bz})$, the map\begin{equation}
		D_2 \Psi^{\gr{\bz}}(\eps,\gr{J})
		= \operatorname{Id} - \CP^{\gr{\bz}}\del[1]{G'(\gr{J}) \bullet \sbr[1]{\eps \<wn> + \<cm>}}
	\end{equation}
	is an element of $\CL(\DD_{\gr{\CU}}^{\gamma,\eta;T}(\gr{\bz}))$ and bijective. In particular, this is true for the map~$\Theta = \Theta_T^{\gr{\bz}}$ given by
	\begin{equation}
		\Theta_T^{\gr{\bz}} := D_2 \Psi^{\gr{\bz}}(0,\gr{W})
		= \operatorname{Id} - \CP^{\gr{\bz}}\del[1]{G'(\gr{W}) \bullet \<cm>}.
	\end{equation}
	We will write~$\tilde{\Theta} := \Theta^{-1}$ for the inverse of~$\Theta$.
\end{lemma}

\begin{proof}
	It is clear that $D_2 \Psi(\eps,\gr{J}) \in \CL(\DD_{\gr{\CU}}^{\gamma,\eta}(\gr{\bz}))$: linearity is trivial and boundedness follows by exactly the same arguments as~\eqref{pf:thm:frechet:gpam:est_cm:aux},  \eqref{pf:thm:frechet:gpam:est_cm}, and the paragraph preceding them:
	\begin{equs}
		\threebars D_2 \Psi(\eps,\gr{J})(\gr{Y}) \threebars_{\gamma,\eta;T} 
		& \aac
		\threebars \gr{Y} \threebars_{\gamma,\eta;T}
		+
		\threebars \CP\del[1]{G'(\gr{J}) \gr{Y} \sbr[1]{\eps \<wn> + \<cm>}} \threebars_{\gamma+\bar{\alpha},\bar{\alpha};T} \\
		& \aac
		\threebars \gr{Y} \threebars_{\gamma,\eta;T}
		+
		\threebars G'(\gr{W}) \gr{Y} \sbr[1]{\eps \<wn> + \<cm>} \threebars_{\gamma+\alpha,\eta + \alpha;T} \\
		& \aac
		C(\threebars \gr{J} \threebars_{\gamma,\eta;T}, g) \threebars \gr{Y}  \threebars_{\gamma,\eta}
	\end{equs}
	In the last estimate, we have additionally used~\cite[Prop.~6.13]{hairer_rs}. Regarding the bijectivity of~$D_2 \Psi(\eps,\gr{J})$, we need to prove that for each $\gr{V} \in \DD_{\gr{\CU}}^{\gamma,\eta}(\gr{\bz})$ there exists a $\gr{Y} \in \DD_{\gr{\CU}}^{\gamma,\eta}(\gr{\bz})$ such that
	\begin{equation}
		\gr{Y} = \gr{V} + \CP\del[1]{G'(\gr{J}) \gr{Y} \sbr[1]{\eps \<wn> + \<cm>}}.
		\label{pf:lem:sec_cond_ift:fp_prob}
	\end{equation}
	This FP problem can again be solved by~\cite[Thm.~7.8]{hairer_rs} because $DG_\eps(\gr{J}): \gr{Y} \mapsto  G'(\gr{J}) \gr{Y} \sbr[1]{\eps\<wn> + \<cm>}$	
	\begin{itemize}
		\item maps $\DD^{\gamma,\eta;T}_{\gr{\CU}}(\gr{\bz})$ into $\DD^{\gamma+\alpha,\eta+\alpha;T}_{\gr{\CU}}(\gr{\bz})$, $\alpha := \reg{\<cm>} = -1-\kappa$,
		\item is strongly locally Lipschitz in the sense of~\thref{def:unif_strong_local_Lipschitz}
	\end{itemize}
	The arguments are the same as in the proof of~\thref{prop:a_priori_estimate}, in particular~\eqref{prop:a_priori_estimate:pf:eq}. Actually, it suffices for~$DG_\eps(\gr{J})$ to be locally Lipschitz continuous which is true because it is linear and continuous.
\end{proof}

\begin{remark}
	Later on, we will actually derive bounds for the operator norm of~$\Theta$ for models with norm of order one, see~\thref{lem:inv_theta_op} and~\thref{rmk:expl_duhamel}.
\end{remark}

With these preparations, the following theorem about $\CC^\ell$ F-differentiability of $\fI_{\gr{\bz}}$ essentially is a consequence of the IFT. In addition, we provide explicit (iterative) FP problems solved by its derivatives~$\fI^{(m)}_{\gr{\bz}}$. For the corresponding statement in the context of Malliavin calculus,~see~\cite[Prop.~4.1]{schoenbauer}.

\begin{theorem}[Differentiabiliy of $\boldsymbol{\fI}$]\label{prop:der_fp_map}
	Let $T \in (0,T_\infty^\sh \wedge T_0)$,~$\bz \in \MM$, and set $\gr{\bz} := E_\sh\bz \in \gr{\MM}$. Suppose~$g \in \CC^{\ell + 4}$ for some $\ell \geq 1$.
	Then there exists some~$\eps_\star = \eps_\star(T,\bz) > 0$ such that, for $I_\star := [0,\eps_\star)$, we have
	\begin{enumerate}[label=(\roman*)]
		\item \label{prop:der_fp_map:i} $\fI_{\gr{\bz}} \in \CC^\ell(I_\star,\DD^{\gamma,\eta;T}_{\gr{\CU}}(\gr{\bz}))$ in the Fr\'{e}chet sense.
		\item \label{prop:der_fp_map:ii} 
		For $m = 1,\ldots,\ell$ and $\eps \in I_\star$, the derivative $\fI^{(m)}_{\gr{\bz}}(\eps) \in \DD^{\gamma,\eta;T}_{\gr{\CU}}({\gr{\bz}})$ uniquely solves the FP equation 
			\begin{equation}
				\fI^{(m)}_{\gr{\bz}}(\eps) = \CM^{\gr{\bz}}_m\del[2]{\eps,\del[1]{\fI^{(k)}_{\gr{\bz}}(\eps)}_{k=0}^{m}}, \quad \eps \in I_\star,
				\label{eq:der_fp_map:fp_eq}
			\end{equation}
		for the FP map $\CM_m^{\gr{\bz}}$ given by
		\begin{align}
			\CM_m^{{\gr{\bz}}}\del[2]{\eps,\del[1]{\fI^{(k)}_{\gr{\bz}}(\eps)}_{k=0}^{m}}
			&
			:=
			\CP^{\gr{\bz}}\del[1]{G^{(m)}(\fI_{\gr{\bz}}(\eps)) \fI^{(m)}_{\gr{\bz}}(\eps) \sbr[1]{\eps \<wn> + \<cm>}} \label{eq:der_fp_map:fp_map} \\
			&
			+
			\eps \CP^{\gr{\bz}}\del[1]{B_{\gr{\bz}}^{(m-1),\<wns>}(\eps) \thinspace \<wn>} 
			+ \CP^{\gr{\bz}}\del[1]{B_{\gr{\bz}}^{(m-1),\<cms>}(\eps) \thinspace\<cm>}
			+ m \CP^{\gr{\bz}}\del[1]{A_{\gr{\bz}}^{(m-1),\<wn>}(\eps) \thinspace \<wn>}.
			\notag
		\end{align}
		Here, we have set\footnote{Note that the terms on the RHS of~\eqref{eq:A} and~\eqref{eq:B} are exactly the same, only the summation in the latter starts with~$k=2$. Note, however that~$B^{(m-1),\sigma}_{\gr{\bz}}$ only depends on~$\fI^{(n)}_{\gr{\bz}}$ for $n \in [m-1]$, whereas~$A^{(m),\<wns>}_{\gr{\bz}}$ does so for $n \in [m]$.
			The asymmetry in choice of superscript (\enquote{$m$} for $A$, but \enquote{$m-1$} for $B$) reflects that fact.
			Also, note that~$B_{\gr{\bz}}^{(m-1),\sigma}(\eps) \equiv 0$ for $m = 1$.}
		for $m \geq 1$:
		\begin{equation}
			A_{\gr{\bz}}^{(m),\<wns>}(\eps)
			:=
			m! \sum_{k=1}^{m} \frac{1}{k!} G^{(k)}(\fI_{\gr{\bz}}(\eps)) \sum_{\boldsymbol{i} \in S_k^{m}}  \prod_{\ell=1}^k \frac{1}{i_\ell!} \fI_{\gr{\bz}}^{(i_\ell)}(\eps), \quad
			A_{\gr{\bz}}^{(0),\<wns>}(\eps) := G(\fI_{\gr{\bz}}(\eps))
			\label{eq:A} 
		\end{equation}
	and
		\begin{equation}
			B_{\gr{\bz}}^{(m-1),\sigma}(\eps)
			:=
			m! \sum_{k=2}^{m}  \frac{1}{k!} G^{(k)}(\fI_{\gr{\bz}}(\eps)) \sum_{\boldsymbol{i} \in S_k^{m}} \prod_{\ell=1}^k \frac{1}{i_\ell!} \fI_{\gr{\bz}}^{(i_\ell)}(\eps), \quad \sigma \in \{\<wn>,\<cm>\}.
			\label{eq:B}
		\end{equation} 
		where
		\begin{equation*}
				S_k^m := \{\boldsymbol{i} \in \N_{\geq 1}^k: \ \abs{\boldsymbol{i}} = m\}.
		\end{equation*}
		\item \label{prop:der_fp_map:iii} 
		For $m = 1,\ldots,\ell$ and~$\eps \in I_\star$, the map $\gr{\bz} \mapsto \fI^{(m)}_{\gr{\bz}}(\eps)$ is strongly locally Lipschitz in the sense of~\thref{def:unif_strong_local_Lipschitz}.
		\item \label{prop:der_fp_map:iv} It is possible to choose $\eps_\star$ as\footnote{Recall that $T_\infty(\eps,E_\sh\bz)$ is the existence time of~$\fI_{E_\sh\bz}(\eps)$, cf.~\thref{coro:ex_time_h_tr_vs_ex} and the comments preceding it.} 
			\begin{equation}
				\eps_\star(T,\bz)
				:=
				\sup I_\infty^{\bz}(T), \quad
				I_\infty^{\bz}(T) :=
				\{\eps \in I: \ T_\infty(r,E_\sh\bz) > T \quad \text{for all} \ r \in [0,\eps)\}.
				\label{prop:der_fp_map:iv_eq_eps_star} 
			\end{equation}
		In this case, we have $\eps_\star(T,\bz) \geq \eps_0(T,\bz)$, the latter defined in~\eqref{eq:def_eps_zero}. In particular, $\fI_{\gr{\bz}} \in \CC^\ell(I_0,\DD^{\gamma,\eta;T}_{\gr{\CU}}(\gr{\bz}))$ for $I_0 = [0,\eps_0)$ as before.
	\end{enumerate}	
\end{theorem}

\begin{remark}
	Note that $\CM_m^{\gr{\bz}}$ in~\eqref{eq:der_fp_map:fp_map} is \emph{linear} in~$\fI^{(m)}_{\gr{\bz}}(\eps)$, so the explosion time of~$\fI^{(m)}_{\gr{\bz}}(\eps)$ coincides with that of~$\fI_{\gr{\bz}}(\eps)$. Also note that $\eps \<wn>$ as well as the second term on the RHS of~\eqref{eq:der_fp_map:fp_map} vanish for~$\eps = 0$.
\end{remark}

\begin{proof} 
	By lemmas~\ref{thm:frechet:gpam} and~\ref{lem:sec_cond_ift}, we can apply the IFT~(see\thinspace\cite[Thm.~4.E]{zeidler_applied_fa}) to $\Psi$. It asserts that there exists a neighbourhood~$\CO$ of $\gr{W}$ and $\eps_\star > 0$ and a unique function~$\gr{Y_\star} \in \CC^\ell(I_\star,\CO)$ such that
			\begin{equation*}
				\Psi(\eps,\gr{Y_\star}(\eps)) = 0 \qquad \text{for} \quad \eps \in I_\star = [0,\eps_\star).
			\end{equation*}
	By uniqueness of solutions to the fixed-point problem~\eqref{eq:fp_extended_eps}, we find that $\gr{Y_\star}(\eps) = \fI(\eps)$ and hence the claim~\ref{prop:der_fp_map:i} follows.
	Regarding~\ref{prop:der_fp_map:ii}, we differentiate $m$~times both sides of the FP equation
	\begin{equation*}
		\fI_{\gr{\bz}}(\eps) = 
		\CP^{\gr{\bz}}\del[1]{G(\fI_{\gr{\bz}}(\eps)) \eps \thinspace \<wn>} + \CP^{\gr{\bz}}\del[1]{G(\fI_{\gr{\bz}}(\eps)) \thinspace\<cm>} + \TT Pu_0.
	\end{equation*}
	For the expression on the RHS, this is valid thanks to the linearity of $\CP^{\gr{\bz}}$, \thref{lem:diffb_G_from_g}, and~\ref{prop:der_fp_map:i}
	By Leibniz rule, we have
	\begin{equation*}
		\partial_\eps^m [G(\fI_{\gr{\bz}}(\eps))\eps \thinspace \<wn>]
		=
		\sum_{k=0}^m \binom{m}{k} \partial_\eps^{m-k} [G(\fI_{\gr{\bz}}(\eps))] \thinspace \partial_\eps^k [\eps \thinspace\<wn>]
		= 
		\partial_\eps^m [G(\fI_{\gr{\bz}}(\eps))] \eps \thinspace\<wn> + m \partial_\eps^{m-1} [G(\fI_{\gr{\bz}}(\eps))] \thinspace \<wn>
	\end{equation*}
	which then leads to
	\begin{equation}
		\fI_{\gr{\bz}}^{(m)}(\eps) = 
		\eps \CP^{\gr{\bz}}\del[1]{\partial_\eps^m G(\fI_{\gr{\bz}}(\eps)) \thinspace \<wn>} 
		+ \CP^{\gr{\bz}}\del[1]{\partial_\eps^m G(\fI_{\gr{\bz}}(\eps)) \thinspace\<cm>}
		+ m \CP^{\gr{\bz}}\del[1]{\partial_\eps^{m-1} G(\fI_{\gr{\bz}}(\eps)) \thinspace \<wn>}. 
		\label{eq:dif_fp_eq_m}
	\end{equation}
	We calculate the derivatives by Riordan's formula~\eqref{eq:riordans_formula}, which reads
	\begin{equation}
		\partial_\eps^m (G \circ \fI_{\gr{\bz}})(\eps) 
		= m! \thinspace \sum_{k=1}^m \frac{1}{k!} G^{(k)}(\fI_{\gr{\bz}}(\eps)) \sum_{\boldsymbol{i} \in S_k^m} \prod_{\ell=1}^k \frac{1}{i_\ell!} \fI_{\gr{\bz}}^{(i_\ell)}(\eps).
		\label{eq:riordan_applied}
	\end{equation} 
	when combined with~\thref{lem:diffb_G_from_g}. Note that the last summand on the RHS of~\eqref{eq:dif_fp_eq_m} does not contribute a term containing~$\fI^{(m)}_{\gr{\bz}}$. For the first two summands, only the term corresponding to~$k=1$ in~\eqref{eq:riordan_applied} does, more precisely 
	\begin{equation*}
		k=1
		\ \rightsquigarrow \
		S_1^m = \{m\} 
		\ \rightsquigarrow \
		G^{(m)}(\fI_{\gr{\bz}}(\eps)) \fI^{(m)}_{\gr{\bz}}(\eps).
	\end{equation*}
	We single out the corresponding terms in the definition of $\CM_m^{\gr{\bz}}$ in~\eqref{eq:der_fp_map:fp_map} and, accordingly, only start the summation in~$B_{\gr{\bz}}^{(m-1),\sigma}$ in~\eqref{eq:B} at $k=2$.
	The claim in~\eqref{eq:der_fp_map:fp_eq} follows. 
	
	Regarding~\ref{prop:der_fp_map:iii}, recall that the operations of multiplication and composition with (lifts of sufficiently) regular functions are strongly locally Lipschitz continuous operations~(cf.~\cite[Prop.~$6.12$]{hairer_rs} and~\cite[Prop.~$3.11$]{hairer_pardoux}, respectively). As a consequence of~\cite[Thm.~$7.8$]{hairer_rs}, so is~$\gr{\bz} \mapsto \fI_{\gr{\bz}}(\eps)$ for $\eps \in I_\star$ fixed, cf.~\thref{coro:solvability:fp_prob} above.
	Therefore, the map
		\begin{equation*}
			\gr{V} \mapsto 
			G^{(m)}(\fI_{\gr{\bz}}(\eps)) \gr{V} \sbr[1]{\eps \<wn> + \<cm>} 
			+ B_{\gr{\bz}}^{(m-1),\<wns>}(\eps) \thinspace \eps \<wn>
			+ B_{\gr{\bz}}^{(m-1),\<cms>}(\eps) \thinspace\<cm>
			+ m A_{\gr{\bz}}^{(m-1),\<wn>}(\eps) \thinspace \<wn>
		\end{equation*}
	is strongly locally Lipschitz continuous from~$\DD^{\gamma,\eta,T}(\gr{\bz})$ to~$\DD^{\bar{\gamma},\bar{\eta},T}(\gr{\bz})$ in case~$m = 1$. Then,~\cite[Thm.~$7.8$]{hairer_rs} implies strong local Lipschitz continuity of~$\gr{\bz} \mapsto \fI^{(1)}_{\gr{\bz}}(\eps)$. 
	Iteratively, the same arguments prove the claim for~$m \in \{2,\ldots,\ell\}$.
	
	The proof of~\ref{prop:der_fp_map:iv} is almost identical to the corresponding statement in~\cite{schoenbauer}: Suppose~$\eps_\star$ is chosen as in~\eqref{prop:der_fp_map:iv_eq_eps_star} and assume there exists~$r_\star > \eps_\star$ such that $T_\infty(r,E_\sh\bz) > T$ for all $r \in [0,r_\star)$. We can then redo the arguments in~\thref{thm:frechet:gpam} with~$(0,\gr{W}) \equiv (0,\fI_{\gr{\bz}}(0))$ replaced by~$(\eps_\star,\fI_{\gr{\bz}}(\eps_\star))$ to infer that we can obtain $\theta > 0$ such that the map
		\begin{equation*}
			[0,\theta) \ni \eps \mapsto \fI_{\gr{\bz}}(\eps_\star + \eps) \in \DD^{\gamma,\eta,T}(\gr{\bz})
		\end{equation*}
	is $\CC^\ell$ in the Fr\'{e}chet sense.
	In turn, that implies that~$\fI_{\gr{\bz}} \in \CC^{\ell}([0,\eps_\star + \theta),\DD^{\gamma,\eta,T}(\gr{\bz}))$ which contradicts the maximality of~$\eps_\star$. By~\thref{coro:expl_rho} and eq.~\eqref{eq:def_eps_zero}, we immediately have that 
		\begin{equation*}
			\eps_0(T,\bz) = \sup I^{\bz}(T) \leq \sup I_\infty^{\bz}(T)	= \eps_\star(T,\bz),
		\end{equation*}
	thus finishing the proof.
\end{proof}	

\subsection{Taylor expansion in spaces of modelled distribution} \label{sec:taylor_exp_mod_distr}

Given~\thref{prop:der_fp_map}, Taylor's formula on Banach spaces~\cite[Thm.\thinspace 4.C]{zeidler_applied_fa} provides us with an abstract analogue of~\thref{thm:stoch_taylor_gpam} in the space~$\DD^{\gamma,\eta,T}_{\gr{\CU}}(E_\sh\bz)$. The following corollary is a direct consequence of the afore-mentioned theorem.

\begin{corollary} \label{thm:abstract_taylor}
	In the setting of~\thref{prop:der_fp_map}, the expansion
	\begin{equation}
		\gr{U^{\eps}}
		= \gr{W} + \sum_{m=1}^{\ell - 1} \frac{\eps^m}{m!} \gr{U^{(m)}} + \gr{\fR_{\eps}^\ell} \quad \text{in} \quad \DD^{\gamma,\eta,T}_{\gr{\CU}}(\Gamma^{\esh})
		\label{thm:abstract_taylor:eq_exp}
	\end{equation}
	 holds for each~$\eps \in I_0$ with 
	\begin{enumerate}[label=(\arabic*)]
		\item the term~$\gr{W}$ as introduced in~\thref{rmk:W_indep_model_not},~$\gr{U^{\eps}} \equiv \gr{U^{\eps}}(\bz) := \fI_{E_\sh\bz}(\eps)$, cf.~\eqref{eq:fI}, and 
		\begin{equation}
			\gr{U^{(m)}} \equiv \gr{U^{(m)}}(\bz) := \fI^{(m)}_{E_\sh\bz}(0) \quad \text{for} \quad m = 1, \ldots, \ell-1.
			\label{thm:abstract_taylor:eq_U_m}
		\end{equation}
		The latter are given iteratively as solutions to the fixed-point problems
			\begin{equation}
				\gr{U^{(m)}}(\bz)
				= 	
				\CP^{E_\sh\bz}\del[1]{G^{(m)}(\gr{W})\gr{U^{(m)}}(\bz) \thinspace \<cm>}
				+ \CP^{E_\sh\bz}\del[1]{B_{E_\sh \bz}^{(m-1),\<cms>} \thinspace\<cm>}
				+ m \CP^{E_\sh\bz}\del[1]{A_{E_\sh\bz}^{(m-1),\<wns>} \thinspace \<wn>}
				\label{thm:abstract_taylor:eq_U_m_formula_fp_eq}
			\end{equation}
		or, equivalently, by the explicit formula
			\begin{equation}
				\gr{U^{(m)}}(\bz)
				= \tilde{\Theta}^{E_\sh\bz} \sbr[1]{\CP^{E_\sh\bz}\del[1]{B_{E_\sh \bz}^{(m-1),\<cms>} \thinspace\<cm>}
				+ m \CP^{E_\sh\bz}\del[1]{A_{E_\sh\bz}^{(m-1),\<wns>} \thinspace \<wn>}}.
				\label{thm:abstract_taylor:eq_U_m_formula}
			\end{equation}
		The terms~$A_{E_\sh\bz}^{(0),\<wns>} := A_{E_\sh\bz}^{(0),\<wns>}(0) = G(\gr{W})$ are given by
			\begin{align}
				A_{E_\sh\bz}^{(m),\<wns>} 
				& := A_{E_\sh\bz}^{(m),\<wns>}(0)
				= m! \sum_{k=1}^{m} \frac{1}{k!} G^{(k)}(\gr{W}) \sum_{\boldsymbol{i} \in S_k^{m}}  \prod_{n=1}^k \frac{1}{i_n!} \gr{U^{(i_n)}}(\bz), \label{eq:A_eps_0} \\
				B_{E_\sh \bz}^{(m-1),\<cms>} 
				& := B_{E_\sh \bz}^{(m-1),\<cms>}(0)
				= m! \sum_{k=2}^{m}  \frac{1}{k!} G^{(k)}(\gr{W}) \sum_{\boldsymbol{i} \in S_k^{m}} \prod_{n=1}^k \frac{1}{i_n!} \gr{U^{(i_n)}}(\bz). \label{eq:B_eps_0}
			\end{align}
		given in~\eqref{eq:A} and~\eqref{eq:B}, respectively. They only depend on~$\gr{U^{(k)}}(\bz)$ for $k = 1,\ldots,m-1$.
		\item the remainder $\gr{\fR_\eps^{\ell}} \equiv \gr{\fR_\eps^{\ell}}(\bz)$ is defined implicitly from~\eqref{thm:abstract_taylor:eq_exp}, that is
		\begin{equation}
			\gr{\fR_{\eps}^\ell}
			:=
			\gr{U^{\eps}} - \gr{W} + \sum_{m=1}^{\ell - 1} \frac{\eps^m}{m!} \gr{U^{(m)}}. 
			\label{thm:abstract_taylor:remainder}
		\end{equation}
	\end{enumerate}
	All the expressions in the expansion~\eqref{thm:abstract_taylor:eq_exp} are continuous in the model~$\bz$.
\end{corollary}

\begin{remark}
Note that Taylor's theorem actually gives us an \emph{explicit} formula for the remainder~$\gr{\fR_{\eps}^\ell}(\bz)$, namely
	\begin{equation}
		\gr{\fR_{\eps}^\ell}(\bz)
		:= 
		\int_0^1 \frac{(1-s)^{\ell-1}}{(\ell-1)!} \fI^{(\ell)}_{E_\sh\bz}(s\eps) \dif s.
	\end{equation}
However, as we will see in subsection~\ref{sec:estimates_taylor_remainder}, the definition in~\eqref{thm:abstract_taylor:remainder} is very well-suited for an inductive argument.
\end{remark}
	
\begin{remark} \label{rmk:inahama_kawabi_faa_di_bruno}
	The formula for~$\gr{U^{(m)}}$ in~\eqref{thm:abstract_taylor:eq_U_m_formula_fp_eq} formally agrees with the corresponding ones by Inahama and Kawabi, specifically~\cite[eq.'s~(4.2) - (4.5)]{inahama_kawabi} (in the case that the SDE they consider has no drift~$b$, i.e.~$b \equiv 0$).
\end{remark}

By itself,~\thref{thm:abstract_taylor} is merely a version of Taylor's theorem. It unfolds its full merit only when complemented with an analysis of the \emph{properties} of the Taylor terms and the remainder. Most importantly in this direction, we will provide estimates for the~$\gr{U^{(m)}}$'s in subsection~\ref{sec:estimates_taylor_terms} and for $\gr{\fR_\eps^{\ell}}$ in subsection~\ref{sec:estimates_taylor_remainder}.

Before we turn to the estimates, we study further properties of the terms~$\gr{U^{(m)}}$ which we need in our analysis. To begin with, the next proposition may be understood as $m$-homegeneity of~$\gr{U^{(m)}}$ w.r.t. to dilation of the model.

\begin{proposition} \label{thm:abstract_taylor:ii} 
	With $\fd_\eps$ and $\d_\eps$ as in section~\ref{sec:models_analytic_op}, the terms~$\gr{U^{(m)}}(\bz) \in \DD^{\gamma,\eta,T}(E_\sh\bz)$ satisfy the identity
	\begin{equation}
		\eps^m \gr{U^{(m)}}(\bz) = \fd_\eps \del[1]{\gr{U^{(m)}}(\d_\eps \bz)}, \quad \eps \in I_0 = [0,\eps_0(T,\bz)). 
		\label{thm:abstract_taylor:homogenity}
	\end{equation}	
\end{proposition}
Before we can prove it, in light of~\eqref{thm:abstract_taylor:eq_U_m_formula} we first need to understand how the operator~$\tilde{\Theta}$ (see~\thref{lem:sec_cond_ift}) transforms when the underlying model is dilated.

\begin{lemma} \label{lem:interaction_theta_dilation}
	Let $\bz \in \MM$. We have 
	\begin{equation}
		\tilde{\Theta}^{E_\sh \bz} \circ \fd_\eps
		=
		\fd_\eps \circ \tilde{\Theta}^{E_\sh \d_\eps \bz}, \quad \eps \in I_0.
		\label{lem:interaction_theta_dilation:eq}
	\end{equation}
\end{lemma}

\begin{proof}
	The claimed identity is equivalent to
	\begin{equation*}
		\fd_\eps \circ \Theta^{E_\sh \d_\eps \bz}
		=
		\Theta^{E_\sh \bz} \circ \fd_\eps, \quad \eps \in I_0.
	\end{equation*}
	Let $\gr{Y_\eps} \in \DD^{\gamma,\eta;T}_{\gr{\CU}}(E_\sh \d_\eps \bz)$. By lemmas~\ref{lem:ext_dil_commute} and~\ref{lem:consistency_dilation}\ref{lem:consistency_dilation:i}, we have $\fd_\eps \gr{Y_\eps} \in \DD^{\gamma,\eta;T}_{\gr{\CU}}(E_\sh \bz)$, so we find that
	\begin{equs}
		\Theta^{E_\sh \bz}\del[0]{\fd_\eps \gr{Y_\eps}}
		&
		=
		\fd_\eps \gr{Y_\eps} - \CP_{E_\sh\bz}\del[1]{G'(\gr{W}) (\fd_\eps \gr{Y_\eps}) \thinspace \<cm>}
		=
		\fd_\eps \gr{Y_\eps} - \CP_{E_\sh\bz}\del[1]{\fd_\eps \del[1]{G'(\gr{W_\sh}) \gr{Y_\eps} \thinspace \<cm>}} \\
		&
		=
		\fd_\eps \del[1]{\gr{Y_\eps} - \CP_{E_\sh\d_\eps \bz}\del[1]{G'(\gr{W}) \gr{Y_\eps} \thinspace \<cm>}}
		=
		\fd_\eps\del[1]{\Theta^{E_\sh \d_\eps \bz}(\gr{Y_\eps})}.
	\end{equs}
	In the previous equalities, we have used the explicit form of $\Theta$ given in~\thref{lem:sec_cond_ift}, the multiplicativity and linarity of~$\fd_\eps$, \thref{lem:consistency_dilation}\ref{lem:consistency_dilation:iii} (to relate $\fd_\eps$, $\d_\eps$, and $\CP$), and \thref{lem:ext_dil_commute} (to commute~$\d_\eps$ and~$E_\sh$).
\end{proof}

\begin{proof}[of~\thref{thm:abstract_taylor:ii}]
	 The proof proceeds by induction, starting with~$m=1$: 
	\begin{equs}[][pf:thm:abstract_taylor:ii_aux1]
		\eps \gr{U^{(1)}}(\bz)
		&
		=
		\tilde{\Theta}^{E_\sh \bz}\CP^{E_\sh \bz}\del[1]{A_{E_\sh\bz}^{(0), \<wns>} \thinspace \eps\<wn>}	
		=
		\tilde{\Theta}^{E_\sh \bz}\CP^{E_\sh \bz}\del[1]{G (\gr{W}) \thinspace \eps\<wn>}	
		=
		\tilde{\Theta}^{E_\sh \bz} \CP^{E_\sh \bz}\del[1]{\fd_\eps(G(\gr{W})\thinspace\<wn>)}	\\
		&
		=
		\tilde{\Theta}^{E_\sh \bz} \fd_\eps\del[2]{\CP^{\d_\eps E_\sh \bz}\del[1]{G(\gr{W}) \thinspace \<wn>}}
		=
		\fd_\eps \tilde{\Theta}^{E_\sh \d_\eps \bz} \CP^{E_\sh \d_\eps \bz}\del[1]{G(\gr{W}) \thinspace \<wn>}
		=
		\fd_\eps \gr{U^{(1)}}(\d_\eps \bz)
	\end{equs} 
	In the previous equalities, we have used~\eqref{thm:abstract_taylor:eq_U_m_formula} for~$m=1$,~\thref{lem:interaction_theta_dilation}, lemmas~\ref{lem:consistency_dilation}\ref{lem:consistency_dilation:iii} and~\ref{lem:ext_dil_commute} as in the previous proof, and the fact that~$\fd_\eps \gr{W} = \gr{W}$ since~$\scal{\gr{W},\<1>} = 0$, cf.~\thref{rmk:W_indep_model_not}.
	
	We assume the claim is true for $m \in \{1,\ldots,\ell-2\}$ and do the induction step $m \mapsto m+1$. First, recall that 
	\begin{equation}
		\eps^{m+1} 
		= \eps^{\thinspace\abs{\bi}}
		= \prod_{n=1}^k \eps^{i_n} 
		\quad \text{for} \quad \bi \in S_{k}^{m+1} = \{\bi \in \N_{\geq 1}^k: \ \abs{\bi} = m + 1\}.
		\label{pf:thm:abstract_taylor:aux2 }
	\end{equation}
	Using this fact in combination with linearity and multiplicativity of~$\fd_\eps$, the induction hypothesis immediately implies that
	\begin{equation}
		\eps^{m+1} A_{E_\sh\bz}^{(m),\<wns>} = \fd_\eps A_{E_\sh\d_\eps\bz}^{(m),\<wns>} \thinspace \eps, \quad 
		\eps^{m+1} B_{E_\sh\bz}^{(m),\<cms>} = \fd_\eps B_{E_\sh\d_\eps\bz}^{(m),\<cms>}.
		\label{eq:A_B_dilated}
	\end{equation}
	 By~\eqref{thm:abstract_taylor:eq_U_m_formula} and the same reasoning as in~\eqref{pf:thm:abstract_taylor:ii_aux1}, we thus have
	 \begin{equs}
	 	\eps^{m+1} \gr{U^{(m+1)}}(\bz)
	 	& = 
	 	\tilde{\Theta}^{E_\sh\bz} \sbr[1]{\CP^{E_\sh\bz}\del[1]{\eps^{m+1} B_{E_\sh \bz}^{(m),\<cms>} \thinspace\<cm>}
	 	+ (m+1) \CP^{E_\sh\bz}\del[1]{\eps^{m+1} A_{E_\sh\bz}^{(m),\<wns>} \thinspace \<wn>}} \\
	 	& =
	 	\tilde{\Theta}^{E_\sh\bz} \sbr[1]{\CP^{E_\sh\bz}\del[1]{\fd_\eps B_{E_\sh\d_\eps\bz}^{(m),\<cms>}\thinspace\<cm>}
	 	+ (m+1) \CP^{E_\sh\bz}\del[1]{\fd_\eps\sbr[1]{A_{E_\sh\d_\eps\bz}^{(m),\<wns>} \thinspace \<wn>}}} \\
	 	& =
	 	\del[1]{\tilde{\Theta}^{E_\sh\bz} \circ \fd_\eps}  \sbr[1]{\CP^{E_\sh\d_\eps\bz}\del[1]{B_{E_\sh\d_\eps\bz}^{(m),\<cms>}\thinspace\<cm>}
 		+ (m+1) \CP^{E_\sh\d_\eps\bz}\del[1]{A_{E_\sh\d_\eps\bz}^{(m),\<wns>} \thinspace \<wn>}} \\
	 	& =
	 	\fd_\eps \del[1]{\tilde{\Theta}^{E_\sh\d_\eps\bz} \sbr[1]{\CP^{E_\sh\d_\eps\bz}\del[1]{B_{E_\sh\d_\eps\bz}^{(m),\<cms>}\thinspace\<cm>}
	 	+ (m+1) \CP^{E_\sh\d_\eps\bz}\del[1]{A_{E_\sh\d_\eps\bz}^{(m),\<wns>} \thinspace \<wn>}}} 
		=
	 	\fd_\eps \gr{U^{(m+1)}}(\d_\eps \bz)
	 \end{equs}
	This completes the proof.
\end{proof}

The following definition relates the quantities appearing in~\thref{thm:abstract_taylor} to those in~\thref{thm:stoch_taylor_gpam}. Note that it is consistent since the reconstruction operator~$\CR$ is linear and continuous.

\begin{definition}\label{def:terms_taylor_exp}
	As before, let $\bz \in \MM$. We set~$u^\eps_\sh(\bz) := \CR\del[1]{E_\sh\bz;\thinspace \gr{U^\eps}}$ and
	\begin{itemize}
		\item 
		$u^{(m)}_{\sh}(\bz) := \CR\del[1]{E_\sh\bz;\thinspace \gr{U^{(m)}}}$, 
		$u^{(m)}_{\xi_\d;\sh} := u^{(m)}_{\sh}(\bz^{\xi_\d})$, 
		$\hat{u}^{(m)}_{\xi_\d;\sh} := u^{(m)}_{\sh}(\hbz^{\xi_\d})$, 
		$\hat{u}^{(m)}_{\sh} := u^{(m)}_{\sh}(\hbz)$,
		\item 
		$R^{(\ell)}_{\sh,\eps}(\bz) := \CR\del[1]{E_\sh\bz;\thinspace \gr{\fR_{\eps}^{\ell}}}$, 
		$\hat{R}^{(\ell)}_{\sh,\eps} := R^{(\ell)}_{\sh,\eps}(\hbz) $.
	\end{itemize}
\end{definition}
For the sake of curiosity, explicit equations for the terms~$\hat{u}^{(m)}_{\xi_\d;\sh}$ are derived in appendix~\ref{sec:stoch_pde_taylor_terms}; they are not needed for the proof of~\thref{thm:stoch_taylor_gpam}. 

The following lemma will be needed in subsection~\ref{sec:first_order_vanish}; it states that the term~$u^{(1)}_{\xi_\d,\sh}$ needs not be renormalised. This is not surprising given that it is a linear function of the noise.

\begin{lemma} \label{coro:no_renorm_first_order_term}
	We have~$u^{(1)}_{\xi_\d;\sh} = \hat{u}^{(1)}_{\xi_\d;\sh}$. 
\end{lemma}

\begin{proof}
	From~\thref{prop:der_fp_map}\ref{prop:der_fp_map:iii}, we know that
	\begin{equation*}
		\gr{U^{(1)}}(\bz) = \CP^{E_\sh\bz} \del[1]{\gr{L^{(1)}}(\bz)}, \quad
		\gr{L^{(1)}}(\bz) 
		:= 
		A_{E_\sh\bz}^{(0),\<wns>}\thinspace \<wn> + B_{E_\sh\bz}^{(0),\<cms>}\thinspace \<cm>
		=
		G'(\gr{W}) \gr{U^{(1)}}(\bz) \thinspace\<cm> + G(\gr{W})\thinspace\<wn>
	\end{equation*}
	As a consequence of~\thref{rmk:W_indep_model_not}, eq.~\eqref{rmk:structure_W:eq}, we find~$\scal{\gr{L^{(1)}}(\bz),\<11>} = 0$ for any~$\bz \in \MM$. In turn, since the action of~$\Pi^{\xi_\d;\thinspace\esh}$ and~$\hat{\Pi}^{\xi_\d;\thinspace\esh}$ differs only on the symbol~$\<11>$, we have
	\begin{equs}
		\CR^{E_\sh \bz^{\xi_\d}}\del[1]{\gr{L^{(1)}}\del[0]{\bz^{\xi_\d}}}(z)
		& 
		=
		\Pi^{\xi_\d;\thinspace\esh}_z \sbr[1]{\gr{L^{(1)}}\del[0]{\bz^{\xi_\d}}(z)}(z) 
		= 
		\hat{\Pi}^{\xi_\d;\thinspace\esh}_z \sbr[1]{\gr{L^{(1)}}\del[0]{\bz^{\xi_\d}}(z)}(z) \\
		&
		=
		\hat{\Pi}^{\xi_\d;\thinspace\esh}_z \sbr[1]{\gr{L^{(1)}}\del[0]{\hbz^{\xi_\d}}(z)}(z)
		=
		\CR^{E_\sh \hbz^{\xi_\d}}\del[1]{\gr{L^{(1)}}\del[0]{\hbz^{\xi_\d}}}(z).
	\end{equs}
	With this observation, we then get
	\begin{equs}
		\hat{u}_{\xi_\d;\sh}^{(1)}(z)	
		&
		=
		\CR\del[1]{E_\sh \hbz^{\xi_\d};\thinspace \gr{U^{(1)}}\del[0]{\hbz^{\xi_\d}}}(z) 
		=
		\sbr[1]{P * \CR^{E_\sh \hbz^{\xi_\d}}\del[1]{\gr{L^{(1)}}\del[0]{\hbz^{\xi_\d}}}}(z) \\
		&
		=
		\sbr[1]{P * \CR^{E_\sh \bz^{\xi_\d}}\del[1]{\gr{L^{(1)}}\del[0]{\bz^{\xi_\d}}}}(z)
		=
		\CR^{E_h \bz^{\xi_\d}}\del[1]{\gr{U^{(1)}}\del[0]{\bz^{\xi_\d}}}(z) 
		=
		u_{\xi_\d;\sh}^{(1)}(z)	
	\end{equs}
	which is what we wanted to show.
\end{proof}

\paragraph*{Further preparations: Duhamel's Formula.}

Before we can estimate the terms in~\eqref{thm:abstract_taylor:eq_exp}, we need some final preparations.
In~\eqref{thm:abstract_taylor:eq_U_m_formula_fp_eq} we have seen that~$\gr{Y} \in \{\gr{U^{(m)}}: \ m \in [\ell-1]\}$ satisfies

\begin{enumerate}[label=(\arabic*)]
	\item \label{duhamel:1} a \emph{linear} fixed-point equation of type
	\begin{equation}
		\gr{Y} = \CP^{\esh} \del[1]{G'(\gr{W})\gr{Y}\thinspace\<cm>\thinspace + \gr{\tilde{V}}}
		\label{sec:duhamel:eq0}
	\end{equation}
	for some modelled distribution~$\gr{\tilde{V}} \in \DD_{\a}^{\bar{\gamma},\bar{\eta}}$ for $\bar{\gamma} := \gamma + \a$ and $\bar{\eta} := \eta + \a$. The subscript~\enquote{$\a$} indicates that $\gr{\tilde{V}}$ takes with values in a sector of regularity~$\a := \deg(\<wn>)$.
	\item \label{duhamel:2} an inhomogeneous \emph{linear} equation of the form
	\begin{equation}
		\Theta^{\esh} \gr{Y} = \gr{V}
		\Longleftrightarrow
		\gr{Y} = \tilde{\Theta}^{\esh} \gr{V}
		\label{sec:duhamel:eq1}
	\end{equation}
	where 
	\begin{equation}
		\gr{V} :=  \CP^{\esh} \gr{\tilde{V}} \in \DD_{\gr{\CU}}^{\gamma,\a+2} \embed \DD_{\gr{\CU}}^{\gamma,\eta}
		\label{sec:duhamel:eq1b_form_V}
	\end{equation}
	for~$\eta \in (0,\a+2)$  and 
	\begin{equation}
		\Theta^{\esh} := \operatorname{Id} - \CP^{\esh}\del[1]{G'(\gr{W}) \bullet \<cm>} \in \CL\del[1]{\DD^{\gamma,\eta}_{\gr{\CU}}(\Gamma^{\esh})}, 
		\quad
		\tilde{\Theta}^{\esh} := \sbr[1]{\Theta^{\esh}}^{-1},
		\label{sec:duhamel:eq2}
	\end{equation}
	introduced in~\thref{lem:sec_cond_ift} can be interpreted as a \enquote{Duhamel operator} and its inverse. 
	In the reasoning above, we have used~\cite[Prop.~6.16]{hairer_rs} which asserts that 
	\begin{equation}
		\CP^{\esh}: \DD_{\a}^{\bar{\gamma},\bar{\eta}} \to \DD_{\gr{\CU}}^{\bar{\gamma}+2,\a+2}
		\label{sec:duhamel:integr_op_bdd}
	\end{equation}
	is linear and bounded.
\end{enumerate}
In subsection~\ref{sec:estimates_taylor_remainder}, more precisely~\thref{lem:formula_R_n+1}, we will see that one can also derive an equation of form~\eqref{sec:duhamel:eq1} for $\gr{Y} = \gr{\fR_{\eps}^\ell}$. That makes a strong case for understanding the operator norm of~$\tilde{\Theta}$, so that estimating~$\gr{Y} \in \{\gr{U^{(m)}}, \gr{\fR_{\eps}^\ell}: \ m \in [\ell-1]\}$ reduces to estimating the corresponding~$\gr{V}$'s in~\eqref{sec:duhamel:eq1}. 

In this spirit, the following proposition is reminiscient of Duhamel's Formula in linear PDE theory. For its formulation, we define~$\gr{\tilde{\CV}} := \scal{\gr{\CU}\<wn> \cup \gr{\CU}\<cm>} = \scal{\tau\sigma: \ \tau \in \gr{\CU}, \sigma \in \{\<wn>,\<cm>\}}$, a sector of regularity~$\a$.

\begin{proposition}[Duhamel] \label{lem:inv_theta_op}
	As before, let~$\gamma = 1 + 2\kappa$, $\a = \deg(\<wn>) = -1 - \kappa$, and~$\eta \in (\nicefrac{1}{2},1)$.
	For each $T \in (0, T_0 \wedge T_\infty^\sh)$,  $\bz = (\Pi,\Gamma) \in \MM$, and $\gr{V} \in \DD^{\gamma,\eta}(\Gamma^{\esh})$, the equation~\eqref{sec:duhamel:eq1} admits a unique solution~$\gr{Y} \in \DD_{\gr{\CU}}^{\gamma,\eta}(\Gamma^{\esh})$~on~$(0,T)$.
	Setting~$\lambda = \lambda(\bz) := 1 + \barnorm{\bz^{\minus}}$ and defining~$\gr{V}$ from $\gr{\tilde{V}} \in \DD_{\gr{\tilde{\CV}}}^{\bar{\gamma},\bar{\eta}}(\Gamma^{\esh})$ as in~\eqref{sec:duhamel:eq1b_form_V} we also have the bound
	\begin{equation}
		\threebars \fd_\l \gr{Y} \threebars_{\gamma,\eta,T;\Gamma^{\esh;\nicefrac{1}{\l}}}
		\aac
		\threebars \fd_\l \gr{\tilde{V}} \threebars_{\bar{\gamma},\bar{\eta},T;\Gamma^{\esh;\nicefrac{1}{\l}}}
	\end{equation} 
	that holds uniformly over all~$\bz \in \MM$.
\end{proposition}

In the proof of the proposition, we combine different aspects of the equivalent viewpoints in~\ref{duhamel:1} and~\ref{duhamel:2} above. It is outsourced to section~\ref{sec:duhamel_proof} in the appendix to streamline the presentation.
Still, let us make a few remarks about the philosophy behind the previous proposition and existent literature.
\begin{remark}\label{rmk:expl_duhamel}
	In the setting of the previous proposition, we could actually prove the slightly stronger statement that
	\begin{equation*}
		\threebars \fd_\l \gr{Y} \threebars_{\gamma,\eta;T;\Gamma^{\esh;\nicefrac{1}{\l}}}
		\aac
		\threebars \fd_\l \gr{V} \threebars_{\gamma,\eta;T;\Gamma^{\esh;\nicefrac{1}{\l}}},
	\end{equation*} 
	for~$\gr{V} \in \DD^{\gamma,\eta,T}(\Gamma^{\esh})$ not necessarily of form~\eqref{sec:duhamel:eq1b_form_V}. From a functional analytic perspective, the previous inequality can then equivalently be stated as
	\begin{equation*}
		\norm[0]{\tilde{\Theta}^{E_\sh\d_{\nicefrac{1}{\lambda}}\bz}}_{\operatorname{op};T} \aac 1
		\label{lem:inv_theta_op:bound}
	\end{equation*}
	uniformly over each~$\bz \in \MM$, where~$\norm{\thinspace \cdot \thinspace}_{\operatorname{op};T}$ denotes the operator norm in~$\CL\del[1]{\DD^{\gamma,\eta;T}(\Gamma^{\esh;\nicefrac{1}{\lambda}})}$.
\end{remark}

\begin{remark} \label{rmk:philosophy_duhamel}
	The purpose of~\thref{lem:inv_theta_op} is to prepare for the proof of the estimates in~\thref{thm:stoch_taylor_gpam}, eq.~\eqref{thm:stoch_taylor_gpam:estimate}: In~\thref{def:terms_taylor_exp}, we have seen that the quantities therein are reconstructions of modelled distributions~$\gr{Y} \in \{\gr{U^{(m)}}, \gr{\fR_{\eps}^\ell}: \ m \in [\ell-1]\}$. 
	By the trivial identity~$\gr{Y} =  \fd_{\nicefrac{1}{\lambda}} \fd_\lambda \gr{Y}$ and~\thref{lem:consistency_dilation}\ref{lem:consistency_dilation:ii}, we see that 
	\begin{equation*}
		\CR^{E_\sh\bz} \gr{Y}
		= 
		\CR^{E_\sh\bz} \fd_{\nicefrac{1}{\lambda}} \fd_\lambda \gr{Y}
		= 
		\CR^{E_\sh\d_{\nicefrac{1}{\l}}\bz} \fd_\lambda \gr{Y}
		=
		\scal{\fd_\lambda \gr{Y},\1}
		\label{rmk:philosophy_duhamel:eq1}
	\end{equation*}
	where the last identity is true by~\cite[Prop.~$3.28$]{hairer_rs} and the fact that~$\gr{\CU}$ is a function-like sector, see~\cite[Def.~$2.5$]{hairer_rs}. 
	Finally, we can estimate
	\begin{equation}
		\norm[0]{\CR^{E_\sh\bz} \gr{Y}}_{\CX_T}
		=
		\norm[0]{\scal{\fd_\lambda \gr{Y},\1}}_{\CX_T}
		\leq
		\threebars \fd_\l \gr{Y} \threebars_{\gamma,\eta,T;\Gamma^{\esh;\nicefrac{1}{\l}}}.
		\label{rmk:philosophy_duhamel:eq2}
	\end{equation}	
	By~\thref{lem:inv_theta_op}, it then only remains to estimate the~$\gr{\tilde{V}}$ corresponding to the specific choice of~$\gr{Y}$. That will be done in the subsections to follow.
\end{remark}

\begin{remark}
	The previous proposition is consistent with its rough path counterpart~\cite[Lem.~2.1]{inahama_kawabi}, itself a version of Duhamel's formula. In the framework of regularity structures, Hairer and Mattingly~\cite{hairer-mattingly} established a version of Duhamel's formula that applies in our framework but did not consider the estimates provided above.
\end{remark}

We now start estimating~$\gr{Y} = \gr{U^{(m)}}$ in the spirit of~\thref{rmk:philosophy_duhamel}.

\subsubsection{Estimates on the Taylor terms}\label{sec:estimates_taylor_terms}

In this subsection, we estimate the Taylor terms $\gr{U^{(m)}}$ in the expansion~\eqref{thm:abstract_taylor:eq_exp}.
Recall from~\eqref{thm:abstract_taylor:eq_U_m_formula} that
\begin{equation}
	\gr{U^{(m)}}(\bz)
	=
	\tilde{\Theta}^{E_\sh\bz}\sbr[1]{\CP^{E_\sh\bz}\del[1]{\gr{\tilde{V}^{(m)}}(\bz)}},
	\qquad
	\gr{\tilde{V}^{(m)}}(\bz) := 
	m A_{E_\sh\bz}^{(m-1),\<wns>}\<wn> + B_{E_\sh\bz}^{(m-1),\<cms>}\<cm>.
	\label{estimates_taylor_terms:fp_eq_2}
\end{equation}
which is of type~\eqref{sec:duhamel:eq1}. 
We obtain the following estimates:

\begin{proposition}\label{prop:est_taylor_terms}
	For each $\bz \in \MM$, $\lambda = \lambda(\bz) := 1 + \barnorm{\bz^{\minus}}$, and $T \in (0,T_\infty^\sh \wedge T_0)$, we have the estimate
	\begin{equation}
		\threebars \fd_\lambda \gr{U^{(m)}}(\bz) \threebars_{\gamma,\eta;T;\Gamma^{\esh;\nicefrac{1}{\l}}} \aac \del[1]{1 + \barnorm{\bz^{\minus}}}^{m} 
		\label{prop:est_taylor_terms:eq}
	\end{equation}
	where the implicit constant differs for each~$m$. 
\end{proposition}

\begin{proof}
	Since~$A_{E_\sh\bz}^{(m-1),\<wns>}$ and~$B_{E_\sh\bz}^{(m-1),\<cms>}$ take values in~$\gr{\CU}$, a function-like sector, it is immediately clear that~$\gr{\tilde{V}^{(m)}}$ takes values in~$\gr{\tilde{\CV}}$ and thus satisfies the assumptions of~\thref{lem:inv_theta_op}.
	We now use the identity
	\begin{equation}
		\fd_\lambda \gr{\tilde{V}^{(m)}}(\bz)
		=
		m\fd_\lambda A_{E_\sh\bz}^{(m-1),\<wns>}\lambda\thinspace\<wn> 
		+ \fd_\lambda B_{E_\sh\bz}^{(m-1),\<cms>}\thinspace\<cm>
		\label{prop:est_taylor_terms:pf:eq0}
	\end{equation}
	and run an inductive argument in~$m$. Henceforth, we suppress the subscript~$E_\sh\bz$ of $A$ and $B$. All the norms~$\threebars \cdot \threebars$ are taken w.r.t. the model~$E_\sh \d_{\nicefrac{1}{\l}} \bz$.
	\begin{itemize}[wide=10pt, parsep = 0.5pt, leftmargin=15pt]
		\item $m = 1$: The previous formula reads 
			$\fd_\lambda\gr{\tilde{V}^{(1)}}(\bz) 
			=  \fd_\lambda A^{(0),\<wns>}\lambda\thinspace\<wn> 
			= \fd_\lambda G(\gr{W}) \lambda\thinspace\<wn>
			= G(\gr{W}) \lambda\thinspace\<wn>$
			since $\scal{\gr{W},\<1>} = 0$,~see~\thref{rmk:W_indep_model_not}. We then obtain the estimate
		\begin{equs}
			\threebars \fd_\lambda \gr{\tilde{V}^{(1)}}(\bz) \threebars_{\bar{\gamma},\bar{\eta};T} 
			= 
			\lambda \threebars G(\gr{W})\thinspace\<wn>\thinspace \threebars_{\bar{\gamma},\bar{\eta};T}
			\aac
			\lambda 
			= 1 + \barnorm{\bz^{\minus}}
		\end{equs}
		for $\a$, $\gamma$, and~$\eta$ (and thus~$\bar{\gamma}$ and $\bar{\eta}$) as in the formulation of~\thref{lem:inv_theta_op}. That same proposition verifies the claim.
		\item $m = 2$: In this case, the formula~\eqref{prop:est_taylor_terms:pf:eq0} is valid with 
		\begin{equation*}
			\fd_\lambda A^{(1),\<wns>} = 2 G'(\gr{W}) \fd_\lambda\gr{U^{(1)}}(\bz), \quad
			\fd_\lambda  B^{(1),\<cms>} = G''(\gr{W}) \sbr[1]{\fd_\lambda \gr{U^{(1)}}(\bz)}^2.
		\end{equation*}
		Analogously to the case~$m=1$, we find
		\begin{equs}
			\threebars \fd_\lambda A^{(1),\<wns>} \lambda \thinspace\<wn> \threebars_{\bar{\gamma},\bar{\eta};T} 
			&
			\aac 
			\threebars G'(\gr{W})\thinspace\<wn> \threebars_{\bar{\gamma},\bar{\eta};T}  \lambda \threebars \fd_\lambda \gr{U^{(1)}}(\bz) \threebars_{\gamma,\eta;T}
			\aac (1 + \barnorm{\bz^{\minus}})^2
		\end{equs}
		as well as
		\begin{equs}
			\threebars\fd_\lambda B^{(1),\<cms>}\thinspace\<cm> \threebars_{\bar{\gamma},\bar{\eta};T} 
			\aac 
			\threebars G''(\gr{W})\thinspace\<cm> \threebars_{\bar{\gamma},\bar{\eta};T} \threebars \fd_\lambda \gr{U^{(1)}}(\bz) \threebars_{\gamma,\eta;T}^2 
			\aac 
			(1 + \barnorm{\bz^{\minus}})^2.
		\end{equs}
		These estimates can be combined to
		\begin{equation*}
			\threebars \fd_\lambda \gr{\tilde{V}^{(2)}}(\bz) \threebars_{\bar{\gamma},\bar{\eta};T}
			\aac
			(1 + \barnorm{\bz^{\minus}})^2
		\end{equation*}
		which implies the claimed statement by~\thref{lem:inv_theta_op}.
	\end{itemize}
	Suppose the claim is true for some $m-1 \in [\ell-2]$. The argument in case~$m=2$ is a blueprint for the strategy in the induction step 
	\begin{itemize}[wide=10pt, parsep = 0.5pt, leftmargin=15pt]
		\item $m-1 \mapsto m$: We have
	\begin{equs}
		\threebars \fd_\lambda B^{(m-1),\<cms>}\thinspace\<cm> \threebars_{\bar{\gamma},\bar{\eta};T}
		& \aac
		m! \sum_{k=2}^{m} \frac{1}{k!} 
		\threebars G^{(k)}(\gr{W}) \<cm> \threebars_{\bar{\gamma},\bar{\eta};T} 
		\sum_{\boldsymbol{i} \in S_k^{m}} \prod_{n=1}^k \frac{1}{i_n!} 
		\threebars \fd_\l \gr{U^{(i_n)}}(\bz) \threebars_{\gamma,\eta;T} \\
		& \aac
		\sum_{k=2}^{m}  
		\sum_{\boldsymbol{i} \in S_k^{m}} (1 + \barnorm{\bz^{\minus}})^{\thinspace \abs{\bi}} 
		=
		(1 + \barnorm{\bz^{\minus}})^m \sum_{k=2}^m \#[S_k^m]
		\aac
		(1 + \barnorm{\bz^{\minus}})^m
	\end{equs} 
	where we have used that~$\sum_{n=1}^{k} i_n = m$ for all $\bi \in S_k^m$ and the induction hypothesis. Analogously, we find
	\begin{equs}
		\threebars \fd_\lambda A^{(m-1),\<wns>} \lambda \<wn> \threebars_{\bar{\gamma},\bar{\eta};T}
		& \aac
		\l (m-1)! \sum_{k=1}^{m-1} \frac{1}{k!} 
		\threebars G^{(k)}(\gr{W}) \<wn>\threebars_{\bar{\gamma},\bar{\eta};T} 
		\sum_{\boldsymbol{i} \in S_k^{m-1}}  
		\prod_{n=1}^k \frac{1}{i_n!} 
		\threebars \gr{U^{(i_n)}}(\bz) \threebars_{\gamma,\eta;T}  \\
		& \aac
		\lambda (1 + \barnorm{\bz^{\minus}})^{m-1}
		=
		(1 + \barnorm{\bz^{\minus}})^m.
	\end{equs}
	Altogether, this can be combined to give the estimate
		\begin{equation*}
			\threebars \fd_\lambda \gr{\tilde{V}^{(m)}}(\bz) \threebars_{\bar{\gamma},\bar{\eta};T}
			\aac
			(1 + \barnorm{\bz^{\minus}})^m
		\end{equation*}
	and again, the claim follows from~\thref{lem:inv_theta_op}.
	\end{itemize}
\end{proof}

\subsubsection{Estimates on the Taylor remainder} \label{sec:estimates_taylor_remainder}

In this subsection, we focus on estimating the Taylor remainder~$\gr{\fR^{\ell}_\eps}$.
Recall from~\eqref{thm:abstract_taylor:remainder} that 
\begin{equation}
	\gr{\fR_{\eps}^{n+1}}  
	=
	\gr{\fR_{\eps}^{1}} 
	- \sum_{j=1}^{n} \frac{\eps^j}{j!} \gr{U^{(j)}} 
	\
	\in \DD^{\gamma,\eta}_{\gr{\CU}}(\Gamma^{\esh})
	\label{eq:taylor_remainder}
\end{equation}
where	
\begin{equation}
	\gr{\fR_{\eps}^{1}} 
	:= \gr{U^{\eps}} - \gr{W} 
	= \CP^{\esh}\del[1]{G(\gr{U^{\eps}})\eps \<wn> + [G(\gr{U^{\eps}}) - G(\gr{W})] \<cm>} 
	\label{eq:taylor_remainder_first_order}
\end{equation}
and, as before, $\gr{U^{(m)}}$ is given in~\eqref{thm:abstract_taylor:eq_U_m_formula}.
Our aim is to prove the following proposition, the analog to last subsection's~\thref{prop:est_taylor_terms}.	
\begin{proposition} \label{prop:remainder_est}
	Let $n \in [\ell]$ and for~$\bz \in \MM$, set~$\lambda = \lambda(\bz) := 1 + \barnorm{\bz^{\minus}}$. Then, for $T \in (0,T_\infty^{\sh} \wedge T_0)$ and~$\rho \in (0,\rho_0(T))$ with $\rho_0(T) > 0$ as in~\thref{sec:explosion:cm_lem}, the estimate
	\begin{equation}
		\threebars \fd_\l \gr{\fR^{n}_\eps}(\bz) \threebars_{\gamma,\eta;T;\Gamma^{\esh,\nicefrac{1}{\l}}}
		\aac
		\eps^n \del[1]{1 + \barnorm{\bz^{\minus}}}^n
		\label{prop:remainder_est:eq}
	\end{equation}
	holds uniformly over all $(\eps,\bz) \in I \x \MM$ with~$\eps \barnorm{\bz^{\minus}} < \rho$.
\end{proposition}

The proof is based on induction. The case $n=1$ is the content of the following lemma:

\begin{lemma}[Induction basis]\label{lem:induc_basis}
	For $\bz \in \MM$, let $\lambda := 1+ \barnorm{\bz^{\minus}}$ and assume that~$\eps \barnorm{\bz^{\minus}} < \rho$. Then, the term $\gr{\fR_{\eps}^1} = \gr{U^\eps} - \gr{W} \in \DD^{\gamma,\eta,T}_{\gr{\CU}}(\Gamma^{\esh})$ 
	can be estimated by
	\begin{equation*}
		\threebars 
		\fd_{\lambda} \gr{\fR^{1}_\eps}(\bz) \threebars_{\gamma,\eta,T;\Gamma^{\nicefrac{1}{\lambda}; \esh}}
		\aac_{\rho,T}
		\eps \del[1]{1 + \barnorm{\bz^{\minus}}}.
		\label{lem:induc_basis:eq2}
	\end{equation*}
\end{lemma}

\begin{proof}
	Recall from (the proof of)~\thref{lem:dil_transl_rel} that
	\begin{equation}
		\gr{U^\eps} = \gr{\CS_{ex}^\eps}(\sh,\bz) = \fd_\eps\del[1]{\gr{U_\eps}}, \quad \gr{U_\eps} := \gr{\CS_{ex}}(\sh,\d_\eps\bz),
		\label{lem:induc_basis:pf_aux1}
	\end{equation}
	so, bearing in mind that $\gr{W} = \gr{U^0}$~(see~\thref{rmk:W_indep_model_not}), we have
	\begin{equation*}
		\fd_\lambda \gr{\fR^1_\eps} = \fd_\lambda \sbr[1]{\fd_\eps\del[1]{\gr{U_\eps}} - \fd_0\del[1]{\gr{U_0}}}.
	\end{equation*}
	Recall that~$\gr{\CU} = \{\1,\<1g>,\<1>,X\}$, among which only~$\<1>$ is \emph{not} invariant under~$\fd_\eps$.
	
	\begin{enumerate}[label=\textbf{Part \arabic*:}, align=left, leftmargin=0pt, labelindent=0pt,listparindent=0pt, itemindent=!]
	\item Let~$z \in D_T := (0,T] \x \T^2$. We have 
		\begin{equs}
			\abs[0]{\scal{\fd_\lambda \gr{\fR^{1}_\eps}(z),\<1>}}
			& 
			=
			\abs[0]{\scal{\fd_\lambda \sbr[0]{\fd_\eps\del[0]{\gr{U_\eps}(z)} - \fd_0\del[0]{\gr{U_0}(z)}},\<1>}}
			=
			\abs[0]{\scal{\fd_\lambda \fd_\eps \gr{U_\eps}(z),\<1>}} \\
			& 
			\leq
			\eps \l \abs[0]{\scal{\gr{U_\eps}(z) - \gr{U_0}(z),\<1>}} + \eps \l \abs[0]{\scal{\gr{U_0}(z),\<1>}}
		\end{equs}
	so that we can estimate the first of the two summands defining the norm~$\threebars \cdot \threebars_{\gamma,\eta,T}$ in~\thref{def:mod_distr}, eq.~\eqref{def:mod_distr:eq}:
		\begin{align}
			\thinspace
			&
			\norm[0]{\fd_\lambda \gr{\fR^{1}_\eps}}_{\gamma,\eta,T}
			=
			\sup_{\tau \in \gr{\CU}} \sup_{z \in D_T} \thinspace \abs{t}^{\frac{\text{\tiny deg}(\tau) - \eta}{2} \vee 0} \abs[0]{\scal{\fd_\lambda \gr{\fR^{1}_\eps}(z),\tau}} \notag\\
			\leq \
			&
			\eps \l \sup_{z \in D_T} \thinspace \abs{t}^{\frac{\text{\tiny deg}(\<1s>) - \eta}{2}} \del[1]{\abs[0]{\scal{\gr{U_\eps}(z) - \gr{U_0}(z),\<1>}} + \abs[0]{\scal{\gr{U_0}(z),\<1>}}}
			+
			\sup_{\tau \in \gr{\CU} \setminus \{\<1s>\}} \sup_{z \in D_T} \thinspace \abs{t}^{\frac{\text{\tiny deg}(\tau) - \eta}{2} \vee 0} \abs[0]{\scal{\gr{U_\eps}(z) - \gr{U_0}(z),\tau}}	\notag \\
			\leq \
			&
			(1 + \eps \lambda) \sup_{\tau \in \gr{\CU}} \sup_{z \in D_T} \thinspace \abs{t}^{\frac{\text{\tiny deg}(\tau) - \eta}{2} \vee 0} \abs[0]{\scal{\gr{U_\eps}(z) - \gr{U_0}(z),\tau}}
			+
			\eps \l \sup_{z \in D_T} \thinspace \abs{t}^{\frac{\text{\tiny deg}(\<1s>) - \eta}{2} \vee 0} \abs[0]{\scal{\gr{U_0}(z),\<1>}} \label{lem:induc_basis:pf_term1} \\
			= \
			&
			(1 + \eps \l) \norm[0]{\gr{U_\eps} - \gr{U_0}}_{\gamma,\eta,T} 
			+ 
			\eps \l \sup_{z \in D_T} \thinspace \abs{t}^{\frac{\text{\tiny deg}(\<1s>) - \eta}{2} \vee 0} \abs[0]{\scal{\gr{U_0}(z),\<1>}} \notag
		\end{align}
	For now, we leave this term as it is and later analyse it further. 
	\item
	We turn to the second term in~\eqref{def:mod_distr:eq}, the one containing the~$\Gamma$'s. At first, by~\thref{lem:dil_Gamma} in the appendix, we observe that 
		\begin{equs}[][lem:induc_basis:term2_aux1]
			\thinspace
			[\fd_\l \fd_\eps \gr{U_\eps}; \fd_\l \fd_0 \gr{U_0}](z,\bar{z}) 
			:= \
			&
			\fd_\lambda \fd_\eps \gr{U_\eps}(z) - \fd_\lambda \fd_0 \gr{U_0}(z)
			- \Gamma^{\esh;\nicefrac{1}{\lambda}}_{z\bar{z}} \fd_\lambda \fd_\eps \gr{U_\eps}(\bar{z})
			+ \Gamma^{\esh;\nicefrac{1}{\lambda}}_{z\bar{z}}\fd_\lambda \fd_0\gr{U_0}(\bar{z}) \\
			= \
			&
			\fd_\lambda \fd_\eps\gr{U_\eps}(z) 
			- \fd_\lambda \fd_0\gr{U_0}(z)
			- \fd_\lambda \fd_\eps \Gamma^{\esh;\eps}_{z\bar{z}} \gr{U_\eps}(\bar{z}) 
			+ \fd_\lambda \fd_0  \Gamma^{\esh;0}_{z\bar{z}} \gr{U_0}(\bar{z})
		\end{equs}
	holds for any~$z,\bar{z} \in D_T$.
	Now recall that
		\begin{equation*}
			\Gamma_{z\bar{z}}^{\esh;\eps}\<1> = \<1> + \gamma_{z\bar{z}}^{\esh;\eps}(\CJ\<wn>)\1 
			\quad \text{with} \quad 
			\gamma_{z\bar{z}}^{\esh;\eps}(\CJ\<wn>) = \eps \scal{\Pi\<wn>, K(\cdot - \bar{z}) - K(\cdot - z)}
		\end{equation*}
	so that, in particular,~$\gamma_{z\bar{z}}^{\esh;0}(\CJ\<wn>) = 0$ and thus~$\Gamma_{z\bar{z}}^{\esh;0} \<1> = \<1>$. Thus, we find that 
		\begin{equs}
			\thinspace
			&
			\fd_\l \fd_\eps \scal{\gr{U_\eps}(z),\<1>} \thinspace \<1> - \fd_\l \fd_\eps \Gamma_{z\bar{z}}^{\esh;\eps} \scal{\gr{U_\eps}(\bar{z}),\<1>} \thinspace \<1> \\
			= \
			&
			\eps \l \sbr[1]{\scal{\gr{U_\eps}(z),\<1>}  - \scal{\gr{U_\eps}(\bar{z}),\<1>}} \<1> - \scal{\gr{U_\eps}(\bar{z}),\<1>} \gamma_{z\bar{z}}^{\esh;\eps}(\CJ\<wn>) \thinspace \1 \\
			= \ 
			&
			\eps \l \scal{\gr{U_\eps}(z) - \Gamma_{z\bar{z}}^{\esh;\eps} \gr{U_\eps}(\bar{z}),\<1>} \<1>
			- \scal{\Gamma_{z\bar{z}}^{\esh;\eps} \scal{\gr{U_\eps}(\bar{z}),\<1>} \<1>, \1} \thinspace \1.
		\end{equs} 
	As a consequence, we find that
		\begin{equation}
			\scal{[\fd_\l \fd_\eps \gr{U_\eps}; \fd_\l \fd_0 \gr{U_0}](z,\bar{z}),\tau}
			=
			\scal{\gr{U_\eps}(z) - \gr{U_0}(z) 
			- \Gamma^{\esh;\eps}_{z\bar{z}} \gr{U_\eps}(\bar{z}) 
			+ \Gamma^{\esh;0}_{z\bar{z}} \gr{U_0}(\bar{z}), \tau}, 
			\label{lem:induc_basis:term2_aux2}
		\end{equation}
	 for~$\tau \in \gr{\CU} \setminus \{\<1>\}$. For~$\tau = \<1>$, the situation is different:
		\begin{equs}[][lem:induc_basis:term2_aux3]
			\thinspace 
			&
			\scal{[\fd_\l \fd_\eps \gr{U_\eps}; \fd_\l \fd_0 \gr{U_0}](z,\bar{z}),\<1>} 
			=
			\eps \l \scal{\gr{U_\eps}(z) - \Gamma_{z\bar{z}}^{\esh;\eps} \gr{U_\eps}(\bar{z}),\<1>} \\  
			= \
			&
			\eps \l \scal{\gr{U_\eps}(z) - \gr{U_0}(z) - \Gamma_{z\bar{z}}^{\esh;\eps} \gr{U_\eps}(\bar{z}) + \Gamma_{z\bar{z}}^{\esh;0} \gr{U_0}(\bar{z}),\<1>} 
			+
			\eps \l \scal{\gr{U_0}(z) - \Gamma_{z\bar{z}}^{\esh;0} \gr{U_0}(\bar{z}),\<1>} \\  
			= \
			&
			\eps \l \scal{\gr{U_\eps}(z) - \gr{U_0}(z) - \Gamma_{z\bar{z}}^{\esh;\eps} \gr{U_\eps}(\bar{z}) + \Gamma_{z\bar{z}}^{\esh;0} \gr{U_0}(\bar{z}),\<1>} 
			+
			\eps \l \scal{\gr{U_0}(z) - \gr{U_0}(\bar{z}),\<1>}
		\end{equs}
	We can now combine~\eqref{lem:induc_basis:term2_aux1},~\eqref{lem:induc_basis:term2_aux2}, and~\eqref{lem:induc_basis:term2_aux3} to obtain the following estimate:
		\begin{align}
			\thinspace
			&
			\sup_{\substack{z,\bar{z} \in D_T, \\ \abs{z - \bar{z}} \leq 1}} \sup_{\tau \in \gr{\CU}} \del[1]{\abs[0]{t} \wedge \abs[0]{\bar{t}}}^{\frac{\text{\tiny deg}(\tau)-\eta}{2} \vee 0} \frac{\abs[0]{\thinspace \scal{[\fd_\l \fd_\eps \gr{U_\eps}; \fd_\l \fd_0 \gr{U_0}](z,\bar{z}),\tau}}}{\abs{z - \bar{z}}^{\gamma - \text{\tiny deg}(\tau)}} \notag\\ 
			\leq \
			&
			(1 + \eps \l)
			\sup_{\substack{z,\bar{z} \in D_T, \\ \abs{z - \bar{z}} \leq 1}} \sup_{\tau \in \gr{\CU}} \del[1]{\abs[0]{t} \wedge \abs[0]{\bar{t}}}^{\frac{\text{\tiny deg}(\tau)-\eta}{2} \vee 0} 
			\frac{\abs[0]{\thinspace \scal{\gr{U_\eps}(z) - \gr{U_0}(z) - \Gamma^{\esh;\eps}_{z\bar{z}} \gr{U_\eps}(\bar{z}) 
			+ \Gamma^{\esh;0}_{z\bar{z}} \gr{U_0}(\bar{z}), \tau}}}{\abs{z - \bar{z}}^{\gamma - \text{\tiny deg}(\tau)}} \label{lem:induc_basis:pf_term2} \\ 
			+ \
			&
			\eps \l 
			\sup_{\substack{z,\bar{z} \in D_T, \\ \abs{z - \bar{z}} \leq 1}} 
			\del[1]{\abs[0]{t} \wedge \abs[0]{\bar{t}}}^{\frac{\text{\tiny deg}(\<1s>)-\eta}{2} \vee 0} 
			\frac{\abs[0]{\thinspace \scal{\gr{U_0}(z)
			- \gr{U_0}(\bar{z}), \<1>}}}{\abs{z - \bar{z}}^{\gamma - \text{\tiny deg}(\<1s>)}} 
			\notag
		\end{align}
	\item Observe that $\gr{U_0}$ solves the fixed-point equation
	\begin{equation*}
		\gr{U_0} = \CP^{E_\sh\d_0\bz} G(\gr{U_0})[\<wn> + \<cm>] + \TT P u_0 
		=
		\CI\del[1]{G(\gr{U_0})[\<wn> + \<cm>]} + \bar{\CT}
		\quad \text{in} \ \DD^{\gamma,\eta}_{\gr{\CU}}(\Gamma^{\esh;0}),
	\end{equation*}
	so a quick comparison of coefficients shows that~$\scal{\gr{U_0},\<1>} = \scal{G(\gr{U_0}),\1} = g\del[0]{\scal{\gr{U_0},\1}}$. We remind the reader of the identity
	\begin{equation*}
		\scal{\gr{U_0},\1}
		=
		\CR^{E_\sh \d_0\bz} \gr{U_0}
		=
		\CR^{E_\sh \bz} \fd_0\gr{U_0}
		=
		\CR^{E_\sh \bz} \gr{W}
		=
		\scal{\gr{W},\1},
	\end{equation*} 
	so that~$\scal{\gr{U_0},\<1>} = \scal{G(\gr{W}),\1}$. Next, we emphasise that for~$z = (t,x), \bar{z} = (\bar{t},\bar{x}) \in D_T$ with $\abs{z - \bar{z}} \leq 1$ we have
	\begin{equation*}
		\abs{z - \bar{z}}^\gamma \leq \abs{z - \bar{z}}^{\gamma - \text{\tiny deg}(\<1s>)}, 
		\qquad
		\del[1]{\thinspace \abs{t} \vee \abs{\bar{t}}}^{\frac{\text{\tiny deg}(\<1s>)-\eta}{2} \vee 0} 
		\leq T^{\frac{\text{\tiny deg}(\<1s>)-\eta}{2} \vee 0}. 
	\end{equation*}
	Since~$\deg(\1) = 0$ and $\Gamma_{z\bar{z}}^{\esh;0} \1 = \1$, we then infer that 
	\begin{equs}[][lem:induc_basis:pf_term_U_0]
		\thinspace
		&
		\eps \l 
		\sbr[3]{
		\sup_{z \in D_T} \thinspace \abs{t}^{\frac{\text{\tiny deg}(\<1s>) - \eta}{2} \vee 0} \abs[0]{\scal{\gr{U_0}(z),\<1>}}
		+
		\sup_{\substack{z,\bar{z} \in D_T, \\ \abs{z - \bar{z}} \leq 1}} 
		\del[1]{\abs[0]{t} \wedge \abs[0]{\bar{t}}}^{\frac{\text{\tiny deg}(\<1s>)-\eta}{2} \vee 0} 
		\frac{\abs[0]{\thinspace \scal{\gr{U_0}(z)
		- \gr{U_0}(\bar{z}), \<1>}}}{\abs{z - \bar{z}}^{\gamma - \text{\tiny deg}(\<1s>)}}} \\
		\aac_T \
		&
		\eps \l 
		\sbr[3]{
		\sup_{z \in D_T} \abs[0]{\scal{G(\gr{W})(z),\1}}
		+
		\sup_{\substack{z,\bar{z} \in D_T, \\ \abs{z - \bar{z}} \leq 1}} 
		\frac{\abs[0]{\thinspace \scal{G(\gr{W})(z)
		- \Gamma_{z\bar{z}}^{\esh;0} G(\gr{W})(\bar{z}), \1}}}{\abs{z - \bar{z}}^{\gamma}}}
		\leq
		\eps \l \thinspace \threebars G(\gr{W}) \threebars_{\gamma,\eta,T}.
	\end{equs}
	\item 
	The estimates in~\eqref{lem:induc_basis:pf_term1},~\eqref{lem:induc_basis:pf_term2}, and~\eqref{lem:induc_basis:pf_term_U_0} then lead to the inequality
	\begin{equs}[][lem:induc_basis:pf_estimate]
		\threebars \fd_{\lambda} \gr{\fR^{1}_\eps}(\bz) \threebars_{\gamma,\eta,T;\Gamma^{\nicefrac{1}{\lambda}; \esh}}
		& 
		\leq 
		(1 + \eps \l) \threebars \gr{U_\eps} ; \gr{U_0} \threebars_{\gamma,\eta,T} 
		+ \eps \l \thinspace \threebars G(\gr{W}) \threebars_{\gamma,\eta,T} \\
		&
		=
		(1 + \eps \l) \threebars \gr{\CS_{ex}}(\sh,\d_\eps\bz) ; \gr{\CS_{ex}}(\sh,\d_0\bz) \threebars_{\gamma,\eta,T} 
		+ \eps \l \thinspace \threebars G(\gr{W}) \threebars_{\gamma,\eta,T}.
	\end{equs}
	where the equality is true by definition of~$\gr{U_\eps}$ in~\eqref{lem:induc_basis:pf_aux1}. 
	By joint local Lipschitz continuity of~$\gr{\CS_{ex}}$ as established in~\cite[Prop.~3.25]{cfg},\footnote{In~\cite[Prop.~3.25]{cfg}, the map~$\gr{\CS_{ex}}$ is denoted by~$\CS_{ex}^H$.} we can further estimate
	\begin{equs}
		\threebars \gr{\CS_{ex}}(\sh,\d_\eps\bz); \gr{\CS_{ex}}(\sh,\d_0\bz) \threebars_{\gamma,\eta,T} 
		& \aac 
		\threebars  \d_\eps \bz - \d_0 \bz \threebars
		\aac
		\threebars \d_\eps \bz^{\minus} - \boldsymbol{0} \threebars \\
		=
		& \
		\eps\barnorm{\Pi^{\minus}}_{\<wns>} \vee \eps^2\barnorm{\Pi^{\minus}}^2_{\<11s>} 
		\leq
		\eps (1 + \rho) \barnorm{\bz^{\minus}}.
	\end{equs}
	By assumption, we have~$\eps \l \leq 1 + \rho$ since $\eps \leq 1$, so the previous estimate may be combined with~\eqref{lem:induc_basis:pf_estimate} to give
		\begin{equation*}
		\threebars \fd_{\lambda} \gr{\fR^{1}_\eps}(\bz) \threebars_{\gamma,\eta,T;\Gamma^{\nicefrac{1}{\lambda}; \esh}}
		\aac 
		\eps \barnorm{\bz^{\minus}} + \eps \lambda
		\aac
		\eps (1 + \barnorm{\bz^{\minus}})
		\end{equation*}
	where the implicit constant depends on~$\rho$, $T$, and the norm of~$G(\gr{W})$, a deterministic quantity.
	\end{enumerate}
\end{proof}

We proved the case $n=1$ by leveraging the continuity of the involved solution maps. The general strategy, however, is to invoke Duhamel's formula,~\thref{lem:inv_theta_op}, again. As alluded to earlier, our first step towards this goal is to derive an equation of type~\eqref{sec:duhamel:eq1} for~$\gr{\fR_{\eps}^{n+1}}$.

\begin{proposition}\label{lem:formula_R_n+1}
	For $n \geq 1$ we have 
	\begin{equation}
		\Theta^{\esh} \gr{\fR_{\eps}^{n+1}}  
		= 
		\CP^{\esh}\del[1]{[\gr{\fR_{\eps}^{n+1}}; \<wn>]\thinspace\eps \thinspace \<wn>
		+
		[\gr{\fR_{\eps}^{n+1}};\<cm>]\thinspace\<cm>},
		\label{lem:formula_R_n+1:eq}
	\end{equation}
	where 
	\begin{align}
		[\gr{\fR_{\eps}^{n+1}}; \<wn>] 
		& := 
		G(\gr{U^\eps}) - G(\gr{W})
		- \sum_{j=1}^{n-1} \sum_{\boldsymbol{i} \in \cup_{l=1}^{n-1} S_j^l} \frac{1}{j!} 
		G^{(j)}(\gr{W}) \prod_{m=1}^j \frac{\eps^{i_m}}{i_m!} \gr{U^{(i_m)}}, 
		\label{lem:formula_R_n+1_wn} \\
		[\gr{\fR_{\eps}^{n+1}}; \<cm>] 
		& := 
		G(\gr{U^\eps}) - G(\gr{W}) - G'(\gr{W})(\gr{U^\eps} - \gr{W})
		- \sum_{j=2}^n \sum_{\boldsymbol{i} \in \cup_{l=1}^n S_j^l} \frac{1}{j!} 
		G^{(j)}(\gr{W}) \prod_{m=1}^j \frac{\eps^{i_m}}{i_m!} \gr{U^{(i_m)}}, \label{lem:formula_R_n+1_cm}
	\end{align}	
	and $S_k^n := \{\boldsymbol{i} \in \N_{\geq 1}^k: \ \abs{\boldsymbol{i}} = n\}$.
\end{proposition}

\begin{proof}
	We start with the case~$n=1$.
	Recall from~\thref{thm:abstract_taylor} that $A^{(0),\<wns>} = G(\gr{W})$, $B^{(0),\<cms>} = 0$, thus~$\gr{U^{(1)}} = \CP^{\esh}(G'(\gr{W}) \gr{U^{(1)}} \<cm>) + \CP^{\esh}(G(\gr{W})\thinspace \<wn>)$. Combined with~\eqref{eq:taylor_remainder_first_order} we find
	\begin{equs}
		\gr{\fR_{\eps}^{2}}  
		&
		=
		\gr{\fR_{\eps}^{1}} 
		- \eps \gr{U^{(1)}}
		=
		\CP^{\esh}\del[1]{[G(\gr{U^\eps}) - G(\gr{W})]\eps \<wn>} 
		+
		\CP^{\esh}\del[1]{G(\gr{U^\eps}) - G(\gr{W}) - G'(\gr{W})\eps\gr{U^{(1)}}\<cm>}
	\end{equs}
	and then use the identity 
		$\eps \gr{U^{(1)}} 
		= \gr{\fR_{\eps}^{1}} - \gr{\fR_{\eps}^{2}} 
		= \gr{U^\eps} - \gr{W} - \gr{\fR_{\eps}^{2}}$ in the last term of the second summand.
	This leads to the claimed identity
	\begin{equs}[][pf:lem:formula_R_n+1:eq1]
		\gr{\fR_{\eps}^{2}}  
		&
		=
		\CP^{\esh}\del[1]{[G(\gr{U^\eps}) - G(\gr{W})]\eps \thinspace \<wn>} 
		+
		\CP^{\esh}\del[1]{[G(\gr{U^\eps}) - G(\gr{W}) - G'(\gr{W})[\gr{U^\eps} - \gr{W}]] \thinspace \<cm>} \\
		&
		+
		\CP^{\esh}\del[1]{G'(\gr{W}) \gr{\fR_{\eps}^{2}}  \thinspace \<cm>}
	\end{equs}
	For $n \geq 2$, we use the fact that
	\begin{equation*}
		\gr{\fR_{\eps}^{n+1}}
		=
		\gr{\fR_{\eps}^{2}} - \sum_{j=2}^{n} \frac{\eps^j}{j!} \gr{U^{(j)}}
	\end{equation*}
	combined with~\eqref{thm:abstract_taylor:eq_U_m_formula_fp_eq} and~\eqref{pf:lem:formula_R_n+1:eq1} to obtain the identity
	\begin{equs}[][lem:formula_R_n+1:eq_induction]
		\thinspace
		&
		\gr{\fR_{\eps}^{n+1}}  
		-
		\CP^{\esh}\del[1]{[G(\gr{U^\eps}) - G(\gr{W})]\eps \thinspace \<wn>} 
		-
		\CP^{\esh}\del[1]{[G(\gr{U^\eps}) - G(\gr{W}) - G'(\gr{W})[\gr{U^\eps} - \gr{W}]] \thinspace\<cm>}
		\\
		= 
		&
		\
		\gr{\fR_{\eps}^{2}}
		-
		\CP^{\esh}\del[1]{[G(\gr{U^\eps}) - G(\gr{W})]\eps \<wn>}
		-
		\CP^{\esh}\del[1]{[G(\gr{U^\eps}) - G(\gr{W}) - G'(\gr{W})[\gr{U^\eps} - \gr{W}]]\<cm>}
		- 
		\sum_{j=2}^{n} \frac{\eps^j}{j!} \gr{U^{(j)}} 
		\\
		=
		&
		\
		\CP^{\esh}\del[1]{G'(\gr{W})\gr{\fR_{\eps}^{2}}\thinspace\<cm>}
		-
		\sum_{j=2}^{n}\frac{\eps^j}{j!}
		\sbr[2]{ 
			\CP^{\esh}(G'(\gr{W}) \gr{U^{(m)}} \<cm>)
			+
			\CP^{\esh}\del[1]{A^{(j-1),\<wns>}\thinspace\<wn>} + \CP^{\esh}\del[1]{B^{(j-1),\<cms>}\thinspace\<cm>}} \\
		=
		&
		\
		\CP^{\esh} \del[1]{G'(\gr{W})\gr{\fR_{\eps}^{n+1}} \thinspace \<cm>}
		-
		\sum_{j=2}^{n}\frac{\eps^j}{j!}
		\sbr[2]{
			\CP^{\esh}\del[1]{A^{(j-1),\<wns>}\thinspace\<wn>} + \CP^{\esh}\del[1]{B^{(j-1),\<cms>}\thinspace\<cm>}}.
	\end{equs}
	Next, observe that $S_k^j = \emptyset$ for $k > j$ $(*)$. By definition of~$B^{(j-1),\<cms>}$ in~\eqref{eq:B_eps_0}, we have
	\begin{equs}
		\sum_{j=2}^n \frac{\eps^j}{j!} B^{(j-1),\<cms>}
		&
		=
		\sum_{j=2}^n \eps^{j} \sum_{k=2}^{j} \frac{1}{k!} G^{(k)}(\gr{W}) \sum_{\boldsymbol{i} \in S_k^j}  \prod_{m=1}^k \frac{1}{i_m!} \gr{U^{(i_m)}}
		\overset{(*)}{=}
		\sum_{k=2}^{n} \frac{1}{k!} G^{(k)}(\gr{W}) \sum_{j=1}^n \eps^{j} \sum_{\boldsymbol{i} \in S_k^j}  \prod_{m=1}^k \frac{1}{i_m!} \gr{U^{(i_m)}} \\
		&
		=
		\sum_{k=2}^{n} \frac{1}{k!} G^{(k)}(\gr{W})  \sum_{\boldsymbol{i} \in \cup_{j=1}^n S_k^j}  \prod_{m=1}^k \frac{\eps^{i_m}}{i_m!} \gr{U^{(i_m)}} 
	\end{equs}
	where the last equality uses that $\eps^j = \eps^{\thinspace \abs{\bi}} = \prod_{m=1}^k \eps^{i_m}$ for $\bi \in S_k^j$. Analogously, by definition of~$A^{(j),\<wns>}$ in~\eqref{eq:A_eps_0} we have
		\begin{equation*}
			\sum_{j=2}^n \frac{\eps^j}{j!} j A^{(j-1),\<wns>}
			=
			\eps \sum_{j=1}^{n-1} \frac{\eps^j}{j!} A^{(j),\<wns>}
			=
			\eps \sum_{k=1}^{n-1} \frac{1}{k!} G^{(k)}(\gr{W}) \sum_{\boldsymbol{i} \in \cup_{j=1}^{n-1}  S_k^j}  \prod_{m=1}^k \frac{\eps^{i_m}}{i_m!} \gr{U^{(i_m)}}.
		\end{equation*}
	Once we update~\eqref{lem:formula_R_n+1:eq_induction} by the previous two identities, the claim follows.
\end{proof}

\begin{remark}
	Formally, the expression for $\gr{\fR^{n+1}_\eps}$ from~\thref{lem:formula_R_n+1} agrees with the one by~Inahama and Kawabi, see~the penultimate equation on page~303 in~\cite{inahama_kawabi}. 
\end{remark}

For $\gr{Y} := \gr{\fR_{\eps}^{n+1}}$, we learn from the previous proposition that the corresponding~$\gr{\tilde{V}}$ and~$\gr{V}$ in~\eqref{sec:duhamel:eq1b_form_V} resp.~\eqref{sec:duhamel:eq1} are given by
	\begin{equation}
		\gr{\tilde{V}^{[n+1]}} 
		:= 
		[\gr{\fR_{\eps}^{n+1}}; \<wn>]\thinspace\eps \thinspace \<wn>
		+
		[\gr{\fR_{\eps}^{n+1}};\<cm>]\thinspace\<cm>, \quad
		\gr{V^{[n+1]}} := \CP^{\esh}(\gr{\tilde{V}}).
	\end{equation}
In the next two lemmas, we derive more explicit expressions for~$[\gr{\fR_{\eps}^{n+1}}; \sigma]$, $\sigma \in \{\thinspace\<wn>, \thinspace \<cm>\thinspace\}$. 

\begin{lemma}\label{lem:remainder_noise_part}
	The identity 
	\begin{equs}[][lem:remainder_noise_part:eq]
		\thinspace
		[\gr{\fR_{\eps}^{n+1}}; \<wn>] 
		& = 
		\sum_{j=1}^{n-1} \frac{1}{j!} G^{(j)}(\gr{W}) 
		\del[4]{\thinspace
			\sum_{\substack{\boldsymbol{i} \in [n-1]^j \\ \abs{\boldsymbol{i}} \geq n}} \prod_{m=1}^j \frac{\eps^{i_m}}{i_m!} \gr{U^{(i_m)}}
			+ \sum_{m=1}^j \binom{j}{m} \del[1]{\gr{\fS^{n-1}_\eps}}^{j-m} \del[1]{\gr{\fR_{\eps}^{n}}}^m} \\
		&
		+
		\int_0^1 \frac{(1-s)^{n-1}}{(n-1)!} G^{(n)}(\gr{W} + s \gr{\fR_{\eps}^{1}}) \sbr[1]{\gr{\fR_{\eps}^{1}}}^n \dif s
	\end{equs}
	holds with $\gr{\fS^{n-1}_\eps} := \sum_{i=1}^{n-1} \frac{\eps^i}{i!} \gr{U^{(i)}}$.
\end{lemma}

\begin{proof}
	By definition,~$\gr{\fR_{\eps}^{1}} = \gr{U^\eps} - \gr{W}$.
	We expand $G(\gr{U^\eps})$ at $\gr{W} \equiv \gr{U^0}$ to write 
	\begin{equs}[][taylor_exp_G_U_eps]
		\thinspace
		&
		G(\gr{U^\eps})
		-
		G(\gr{W}) \\
		=
		& \
		\sum_{j=1}^{n-1} \frac{1}{j!} D^{(j)}G(\gr{W})\sbr[1]{\gr{\fR_{\eps}^{1}} ,\ldots,\gr{\fR_{\eps}^{1}}} + \int_0^1 \frac{(1-s)^{n-1}}{(n-1)!} D^{(n)}G(\gr{W} + s \gr{\fR_{\eps}^{1}})\sbr[1]{\gr{\fR_{\eps}^{1}} ,\ldots,\gr{\fR_{\eps}^{1}}} \dif s \\
		=
		& \
		\sum_{j=1}^{n-1} \frac{1}{j!} G^{(j)}(\gr{W})\sbr[1]{\gr{\fR_{\eps}^{1}}}^{j} 
		+ \int_0^1 \frac{(1-s)^{n-1}}{(n-1)!} G^{(n)}(\gr{W} + s \gr{\fR_{\eps}^{1}}) \sbr[1]{\gr{\fR_{\eps}^{1}}}^n \dif s
	\end{equs}
	where we used~\thref{lem:diffb_G_from_g} in the last equality. At the same time, by the identity $\gr{\fR_{\eps}^{1}} 	= \gr{\fS^{n-1}_\eps} + \gr{\fR_{\eps}^{n}}$ and the binomial theorem, the first term in the expansion~\eqref{taylor_exp_G_U_eps} reads
	\begin{equs}[][taylor_exp_G_U_eps_plugin]
		\sum_{j=1}^{n-1} \frac{1}{j!} G^{(j)}(\gr{W})\sbr[1]{\gr{\fR_{\eps}^{1}}}^{j}
		& =
		\sum_{j=1}^{n-1} \frac{1}{j!} G^{(j)}(\gr{W}) \sbr[3]{\del[1]{\gr{\fS^{n-1}_\eps}}^j + \sum_{m=1}^j \binom{j}{m} \del[1]{\gr{\fS^{n-1}_\eps}}^{j-m} \del[1]{\gr{\fR_{\eps}^{n}}}^m}. 
	\end{equs}
	For~$[\gr{\fR_{\eps}^{n+1}}; \<wn>]$ as in~\eqref{lem:formula_R_n+1_wn}, we have thus far established that
	\begin{align}
		[\gr{\fR_{\eps}^{n+1}}; \<wn>]
		=
		& \
		\sum_{j=1}^{n-1} \frac{1}{j!} G^{(j)}(\gr{W}) \sbr[3]{
			\del[1]{\gr{\fS^{n-1}_\eps}}^j + \sum_{m=1}^j \binom{j}{m} \del[1]{\gr{\fS^{n-1}_\eps}}^{j-m} \del[1]{\gr{\fR_{\eps}^{n}}}^m
			-
			\sum_{\boldsymbol{i} \in \cup_{l=1}^{n-1} S_j^l} \prod_{m=1}^j \frac{\eps^{i_m}}{i_m!} \gr{U^{(i_m)}}
		} \notag \\
		+
		& \
		\int_0^1 \frac{(1-s)^{n-1}}{(n-1)!} G^{(n)}(\gr{W} + s \gr{\fR_{\eps}^{1}}) \sbr[1]{\gr{\fR_{\eps}^{1}}}^n \dif s.
		\label{taylor_exp_G_U_eps_update}
	\end{align}
	We further rewrite\footnote{Recall that~$[n] := \{1,\ldots,n\}$ for~$n \in \N$.}
	\begin{equs}[][eq:fS_eps_n-1_j]
		\del[1]{\gr{\fS^{n-1}_\eps}}^j
		& = 
		\del[4]{\sum_{i=1}^{n-1} \frac{\eps^i}{i!} \gr{U^{(i)}}}^j
		=
		\prod_{r=1}^j \del[4]{\sum_{\ell_r=1}^{n-1} \frac{\eps^{\ell_r}}{\ell_r!} \gr{U^{(\ell_r)}}}
		=
		\sum_{\boldsymbol{\ell} \in [n-1]^j} \prod_{m=1}^j \frac{\eps^{\ell_m}}{\ell_m!} \gr{U^{(\ell_m)}}
	\end{equs}
	and observe that
	\begin{equation*}
		\bigcup_{l=1}^{n-1} S_j^l 
		=
		\bigcup_{l=1}^{n-1} \{\boldsymbol{i} \in \N_{\geq 1}^j: \ \abs{\boldsymbol{i}} = l\}
		=
		\{\boldsymbol{i} \in \N_{\geq 1}^j: \ \abs{\boldsymbol{i}}  \in [n-1] \}
		=
		\{\boldsymbol{i} \in [n-1]^j: \ \abs{\boldsymbol{i}}  \in [n-1] \}.
	\end{equation*}
	Combining this identity with the formula~\eqref{eq:fS_eps_n-1_j} for $\del[0]{\gr{\fS^{n-1}_\eps}}^j$, we see that 
	\begin{equs}
		\del[1]{\gr{\fS^{n-1}_\eps}}^j 
		-
		\sum_{\boldsymbol{i} \in \cup_{l=1}^{n-1} S_j^l} \prod_{m=1}^j \frac{\eps^{i_m}}{i_m!} \gr{U^{(i_m)}} 
		=
		& \
		\sum_{\bi \in [n-1]^j} \prod_{m=1}^j \frac{\eps^{i_m}}{i_m!} \gr{U^{(i_m)}}
		-
		\sum_{\substack{\boldsymbol{i} \in [n-1]^j \\ \abs{\boldsymbol{i}} \in [n-1]}} \prod_{m=1}^j \frac{\eps^{i_m}}{i_m!} \gr{U^{(i_m)}} \\
		=
		& \
		\sum_{\substack{\boldsymbol{i} \in [n-1]^j \\ \abs{\boldsymbol{i}} \geq n}} \prod_{m=1}^j \frac{\eps^{i_m}}{i_m!} \gr{U^{(i_m)}}
	\end{equs}
	Plugging this identity into~\eqref{taylor_exp_G_U_eps_update} finishes the proof.
\end{proof}

A similar identity holds for~$[\gr{\fR_{\eps}^{n+1}}; \<cm>]$ and can be proved analogously.

\begin{lemma} \label{lem:remainder_cameron_martin_part}
	For~$\gr{\fS^{n-1}_\eps}$ as in~\thref{lem:remainder_noise_part}, the following identity holds:
	\begin{equs}[][lem:remainder_cameron_martin_part:eq]
		\thinspace
		[\gr{\fR_{\eps}^{n+1}}; \<cm>] 
		& = 
		\sum_{j=2}^{n} \frac{1}{j!} G^{(j)}(\gr{W}) 
		\del[4]{\thinspace
			\sum_{\substack{\boldsymbol{i} \in [n-1]^j \\ \abs{\boldsymbol{i}} \geq n+1}} \prod_{m=1}^j \frac{\eps^{i_m}}{i_m!} \gr{U^{(i_m)}}
			+ \sum_{m=1}^j \binom{j}{m} \del[1]{\gr{\fS^{n-1}_\eps}}^{j-m} \del[1]{\gr{\fR_{\eps}^{n}}}^m} \\
		&
		+
		\int_0^1 \frac{(1-s)^{n}}{n!} G^{(n+1)}(\gr{W} + s \gr{\fR_{\eps}^{1}}) \sbr[1]{\gr{\fR_{\eps}^{1}}}^{n+1} \dif s
	\end{equs}
\end{lemma}

Finally, we have gathered all the tools to prove~\thref{prop:remainder_est}.

\begin{proof}[of~\thref{prop:remainder_est}]
	The statement for~$n=1$ is the content of~\thref{lem:induc_basis}, so we assume the claim is true for~$n \in [\ell-1]$ and consider the induction step~$n \mapsto n+1$.
	Using~\thref{lem:formula_R_n+1} in conjunction with lemmas~\ref{lem:interaction_theta_dilation} and~\ref{lem:consistency_dilation}, we obtain the identity
	\begin{equs}
		\fd_\lambda \gr{\fR_{\eps}^{n+1}}(\bz)  
		& =
		(\fd_\lambda \circ\tilde{\Theta}^{E_\sh\bz}) 
		\sbr[1]{\CP^{E_\sh\bz}([\gr{\fR_{\eps}^{n+1}}; \thinspace\<wn>\thinspace]\thinspace\eps\thinspace\<wn>
			+
			[\gr{\fR_{\eps}^{n+1}};\thinspace\<cm>]\thinspace\<cm>\thinspace)} \\
		& =
		\tilde{\Theta}^{E_\sh \d_{\nicefrac{1}{\lambda}}\bz} 
		\sbr[1]{\CP^{E_\sh\d_{\nicefrac{1}{\lambda}}\bz}(\fd_\lambda[\gr{\fR_{\eps}^{n+1}};\thinspace \<wn>\thinspace]\eps\thinspace\<wn>\thinspace\lambda
			+
			\fd_\lambda[\gr{\fR_{\eps}^{n+1}};\thinspace\<cm>\thinspace]\thinspace\<cm>)},
	\end{equs}
	so by Duhamel's principle presented in~\thref{lem:inv_theta_op}, we have the estimate
	\begin{equs}[][prop:remainder_est:pf_aux1] 
		\threebars 
		\fd_{\lambda} \gr{\fR^{n+1}_\eps}(\bz) \threebars_{\gamma,\eta;T;\Gamma^{\esh;\nicefrac{1}{\lambda};}}
		& \aac \
		\eps \lambda \threebars \fd_\lambda[\gr{\fR_{\eps}^{n+1}}; \<wn>] \threebars_{\gamma,\eta;T;\Gamma^{\esh; \nicefrac{1}{\lambda}}} 
		+ \
		\threebars \fd_\lambda[\gr{\fR_{\eps}^{n+1}}; \<cm>] \threebars_{\gamma,\eta;T;\Gamma^{\esh; \nicefrac{1}{\lambda}}}.
	\end{equs}
	We have to estimate the terms~$\threebars \fd_\lambda[\gr{\fR_{\eps}^{n+1}}; \sigma] \threebars_{\gamma,\eta;\Gamma^{\nicefrac{1}{\lambda}; \esh}}$ on the right hand side and start with~$\sigma = \<wn>$\thinspace. In this case, the factor~$\eps\lambda = \eps(1 + \barnorm{\bz^{\minus}})$ already contributes one power towards the exponent in the claim.
	By~\thref{lem:remainder_noise_part}, terms of three types are to be estimated:\footnote{Recall that the operator~$\fd_\lambda$ is linear, multiplicative, and commutes with the application of lifts~$G$ of functions~$g$, see eq.~\eqref{def:frac_dil_symbols} and~\thref{lem:consistency_dilation}\ref{lem:consistency_dilation:i}.}
	\begin{enumerate}[label=(\arabic*)]
		\item \label{item:terms_1} Those only containing Taylor terms $\fd_\lambda \gr{U^{(i)}}$ with $\bi \in [n-1]$, for which~\thref{prop:est_taylor_terms} implies the estimate
		\begin{equs}
			\thinspace
			&
			\sum_{\substack{\boldsymbol{i} \in [n-1]^j \\ \abs{\boldsymbol{i}} \geq n}} \prod_{m=1}^j \frac{\eps^{i_m}}{i_m!} \threebars \fd_\lambda \gr{U^{(i_m)}}(\bz) \threebars_{\gamma,\eta;T;\Gamma^{\esh; \nicefrac{1}{\lambda}}}
			\aac
			\sum_{\substack{\boldsymbol{i} \in [n-1]^j \\ \abs{\boldsymbol{i}} \geq n}} \prod_{m=1}^j \frac{\eps^{i_m}}{i_m!} (1 + \barnorm{\bz^{\minus}})^{i_m} \\
			\aac
			& \
			\sum_{\substack{\boldsymbol{i} \in [n-1]^j \\ \abs{\boldsymbol{i}} \geq n}}
			\sbr[1]{\eps (1 + \barnorm{\bz^{\minus}})}^{\abs{\boldsymbol{i}}}
			\aac_\rho
			\eps^n (1+\barnorm{\bz^{\minus}})^n,
		\end{equs}
		where we have used $\eps \barnorm{\bz^{\minus}} < \rho$ to estimate terms with powers bigger that~$n$.
		\item \label{item:terms_2} Those that contain at least one factor of~$\fd_\lambda\gr{\fR_{\eps}^{n}}$, for which we have
		\begin{equation}
			\sum_{m=1}^j \binom{j}{m} \threebars \fd_\lambda\gr{\fS^{n-1}_\eps} \threebars_{\gamma,\eta;T;\Gamma^{\esh; \nicefrac{1}{\lambda}}}^{j-m} \threebars \fd_\lambda\gr{\fR_{\eps}^{n}} \threebars_{\gamma,\eta;T;\Gamma^{\esh; \nicefrac{1}{\lambda}}}^m
			\aac
			\eps^n (1+\barnorm{\bz^{\minus}})^n
			\label{item:terms_2:eq}
		\end{equation}
		Here, we have used the induction hypothesis,~\thref{prop:est_taylor_terms} to estimate the terms $\fd_\lambda \gr{U^{(k)}}$ (hidden in $\fd_\lambda\gr{\mathfrak{S}_{\eps}^{n-1}}$), and the fact that $\eps \barnorm{\bz^{\minus}} < \rho$ to estimate powers higher than~$n$ by a $\rho$-dependent constant.
		\item \label{item:terms_3} The Bochner integral in~\eqref{lem:remainder_noise_part:eq}. In order to control this term, we observe that
			\begin{equation*}
				\threebars 
				\gr{W} + s \fd_\lambda \gr{\fR_{\eps}^{1}} 
				\threebars_{\gamma,\eta,T;\Gamma^{\esh; \nicefrac{1}{\lambda}}}
				\aac
				\threebars 
				\gr{W} 
				\threebars_{\gamma,\eta,T;\Gamma^{\esh; \nicefrac{1}{\lambda}}}
				+
				(1 + \rho)
			\end{equation*}
		holds uniformly over~$s \in [0,1]$. In this estimate, we have used the induction basis~(\thref{lem:induc_basis}) and the assumption that~$\eps \barnorm{\bz^{\minus}} < \rho$. 
		Hence, we can estimate
		\begin{equs}
			\thinspace
			&
			\int_0^1 \frac{(1-s)^{n-1}}{(n-1)!} \threebars G^{(n)}(\gr{W} + s \fd_\lambda \gr{\fR_{\eps}^{1}}) \threebars_{\gamma,\eta;T;\Gamma^{\esh; \nicefrac{1}{\lambda}}} \dif s \ \threebars \fd_\lambda\gr{\fR_{\eps}^{1}}\threebars_{\gamma,\eta,T;\Gamma^{\esh; \nicefrac{1}{\lambda}}}^n \\
			\aac \
			&
			\int_0^1 \frac{(1-s)^{n-1}}{(n-1)!} 
			\threebars \gr{W} + s \fd_\lambda \gr{\fR_{\eps}^{1}} \threebars_{\gamma,\eta;T;\Gamma^{\esh; \nicefrac{1}{\lambda}}} \dif s \ \threebars \fd_\lambda\gr{\fR_{\eps}^{1}}\threebars_{\gamma,\eta,T;\Gamma^{\esh; \nicefrac{1}{\lambda}}}^n
			\aac
			\eps^n (1 + \barnorm{\bz^{\minus}})^n
		\end{equs}
		again by the induction basis. 
		Note that the first estimate, depends on the norms of~$\norm[0]{g^{(r)}}_{\infty}$, $r \in \{n,n+1\}$, as well as~$\norm[0]{\Gamma^{\esh;\nicefrac{1}{\l}}}$; the latter, however, satisfies a deterministic bound that depends on~$\norm[0]{\sh}_\CH$, see~\thref{app:lemma:deter_bound_Gamma} in the appendix.
	\end{enumerate}
	By definition of $\lambda$, these estimates can be combined to
	\begin{equation*}
		\eps \lambda \thinspace \threebars \fd_\lambda[\gr{\fR_{\eps}^{n+1}}; \sigma] \threebars_{\gamma,\eta;T;\Gamma^{\nicefrac{1}{\lambda}; \esh}}
		\aac_\rho
		\eps^{n+1} (1 + \barnorm{\bz^{\minus}})^{n+1},
	\end{equation*}
	in case~$\sigma = \<wn>$. For $\sigma = \<cm>$, we argue analogously using~\thref{lem:remainder_cameron_martin_part} instead of~\thref{lem:remainder_noise_part}. However, note that no additional factor~$\eps\lambda$ appears in
		$\threebars \fd_\lambda[\gr{\fR_{\eps}^{n+1}}; \<cm>] \threebars_{\gamma,\eta;T;\Gamma^{\esh; \nicefrac{1}{\lambda}}}$
	in~\eqref{prop:remainder_est:pf_aux1}, so all estimates have to be of order $n+1$ rather than~$n$. In fact, they are: 
	For terms of types~\ref{item:terms_1} and~\ref{item:terms_3} appearing in~$[\gr{\fR_{\eps}^{n+1}}; \<cm>]$, this can be seen immediately from their definition.
	For terms of type~\ref{item:terms_2} note that the first sum in~\eqref{lem:remainder_cameron_martin_part:eq} starts with $j=2$, so $\del[1]{\gr{\fS^{n-1}_\eps}}^{j-m}$ always comes with an exponent $j-m \geq 1$, contributing one additional power to the exponent of $\eps \del[0]{1 + \barnorm{\bz^{\minus}}}$ in~\eqref{item:terms_2:eq}.
\end{proof}

\subsection{Proof of \thref{thm:stoch_taylor_gpam}} \label{sec:pf_stoch_taylor}
	
The whole section is occupied by the proof. 
\begin{proof}
Note that 
	\begin{equation*}
		\gr{\CU} = \scal{\1,\<1>,\<1g>,X} \subseteq \bar{\CT} \oplus \CT_{\geq \a + 2}, \quad  1-\kappa = \a +2 < \gamma = 1 + 2\kappa
	\end{equation*}	
and recall that $\CX_T := \CC\del[1]{[0,T],\CC^{\eta}(\T^2)}$ for fixed~$\eta \in (\nicefrac{1}{2},1)$. 
As a consequence of~\cite[Prop.~$3.28$]{hairer_rs}, it thus follows that
	\begin{equation*}
		\gr{V} \in \DD_{\gr{\CU}}^{\gamma,\eta,T}(\gr{\bz})
		\implies
		\CR^{\gr{\bz}} \gr{V} \in 
		\CC\del[1]{(0,T],\CC^{\a+2}(\T^2)} \cap \CC\del[1]{[0,T],\CC^{\eta}(\T^2)} 
		\subseteq \CX_T.
	\end{equation*}
More generally, the reconstruction operators
	\begin{equation*}
		\CR: \gr{\MM} \ltimes \DD_{\gr{\CU}}^{\gamma,\eta;T} \to \CX_T, \quad
		\CR^{E_\sh\bullet}: \MM \ltimes \DD_{\gr{\CU}}^{\gamma,\eta;T}(E_\sh\bullet) \to \CX_T.
	\end{equation*}	
are jointly locally Lipschitz continuous by the reconstruction theorem~\cite[Thm.~3.10]{hairer_rs} and~\thref{app:prop:extension_op}. Since~$\CR^{E_\sh\bz}$ is linear and bounded for any fixed~$\bz \in \MM$ (and thus Fr\'{e}chet~$\CC^\infty$), it immediately follows from~\thref{prop:der_fp_map} that
	\begin{equation*}
		u_\sh^{\bullet}(\bz) = \CR^{E_\sh \bz} \circ \fI_{E_\sh\bz} \in \CC^{(\ell)}(I_0(T,\bz),\CX_T).
	\end{equation*}

\vspace{0.5em}
\noindent
\textbf{$\bullet$ Property~\ref{thm:stoch_taylor_gpam:i}: Continuous dependence on the model and estimates.}
\vspace{0.5em}

\noindent
By~\thref{def:terms_taylor_exp}, continuity of $u_\sh^{(m)}(\bz)$, $m \in [\ell-1]$, and~$R_{\sh,\eps}^{(\ell)}(\bz)$ in~$\bz$ is a direct consequence of joint continuity of~$\CR^{E_\sh\bullet}$ and~\thref{thm:abstract_taylor}.

In view of~\thref{rmk:philosophy_duhamel}, the estimates in~\eqref{thm:stoch_taylor_gpam:estimate} are straightforward consequences of~\thref{prop:est_taylor_terms} and~\thref{prop:remainder_est}, respectively. We have
	\begin{equs}
		\norm[0]{u_{\sh}^{(m)}(\bz)}_{\CX_T}
		& =
		\norm[0]{\CR^{E_\sh\bz} \gr{U^{(m)}}(\bz)}_{\CX_T}
		\leq
		\threebars \fd_\l \gr{U^{(m)}}(\bz) \threebars_{\gamma,\eta;T;\Gamma^{\esh;\nicefrac{1}{\l}}}
		\aac
		(1 + \barnorm{\bz^{\minus}})^{m}
	\end{equs}
	and
	\begin{equs}
		\norm[0]{R_{\sh,\eps}^{(\ell)}(\bz)}_{\CX_T}
		& = 
		\norm[0]{\CR^{E_\sh\bz} \gr{\fR^{\ell}_\eps}(\bz)}_{\CX_T} 	
		\leq 
		\threebars \fd_{\l} \gr{\fR^{\ell}_\eps}(\bz) \threebars_{\gamma,\eta;T;\Gamma^{\esh,\nicefrac{1}{\l}}}
		\aac
		\eps^{\ell} \del[1]{1 + \barnorm{\bz^{\minus}}}^{\ell}.
	\end{equs}

\vspace{0.5em}
\noindent
\textbf{$\bullet$ Property~\ref{thm:stoch_taylor_gpam:ii}: Homogeneity w.r.t. model dilation.}
\vspace{0.5em}

\noindent
Recall that we want to prove the equality~$\eps^m u^{(m)}_{\sh}(\bz) = u^{(m)}_{\sh}(\d_\eps \bz)$ for $\eps \in I_0(T,\bz)$.
By now, this is an easy consequence of~\thref{thm:abstract_taylor:ii}: We have
\begin{equation*}
	\eps^m u^{(m)}_{\sh}(\bz) 
	=
	\CR^{E_\sh \bz}\del[1]{\eps^m \gr{U^{(m)}}(\bz)} 
	=
	\CR^{E_\sh \bz}\del[1]{\fd_\eps \gr{U^{(m)}}(\d_\eps \bz)}
	=
	\CR^{E_\sh \d_\eps \bz}\del[1]{\gr{U^{(m)}}(\d_\eps \bz)} 
	=
	u^{(m)}_{\sh}(\d_\eps \bz)
\end{equation*}
where we have additionally used~\thref{lem:consistency_dilation}\ref{lem:consistency_dilation:ii} and~\thref{lem:ext_dil_commute}). 
\end{proof}

\subsection{Stochastic PDEs for the Taylor terms} \label{sec:stoch_pde_taylor_terms}

In this section, we derive stochastic PDEs that are satisfied by the Taylor terms~$\hat{u}^{(m)}_{\xi_\d;\sh}$ in the expansion~\eqref{thm:stoch_taylor_gpam:exp} in~\thref{thm:stoch_taylor_gpam} (with the notational convention introduced in~\thref{def:terms_taylor_exp}). Even though these equations are not used elsewhere in this article, we feel that they are still interesting to the reader. For a corresponding statement in the framework of rough paths theory, see Inahama and Kawabi~\cite[Def.~4.1]{inahama_kawabi}. In what follows, we will use the following convention:

\begin{notation}\label{def:fp_prob_uhat_notation}
For $\gr{U^{(m)}}(\bz)$, $A^{(m),\<wns>}_{E_\sh \bz}$, and~$B^{(m-1),\<cms>}_{E_\sh\bz}$ as in~\thref{thm:abstract_taylor}, we set
	\begin{equation*}
	\gr{\hat{U}^{(m)}_{\xi_\d,\sh}} := \gr{U^{(m)}}(\hbz^{\xi_\d}), \quad
	\hat{A}^{(m),\<wns>}_{\xi_\d,\sh} := A^{(m),\<wns>}_{E_\sh\hbz^{\xi_\d}}, \quad
	\hat{B}^{(m-),\<cms>}_{\xi_\d,\sh} := B^{(m-),\<cms>}_{E_\sh\hbz^{\xi_\d}}.
	\label{def:fp_prob_uhat_notation:eq}
	\end{equation*}
\end{notation}
With this convention, the fixed-point equation~\eqref{thm:abstract_taylor:eq_U_m_formula_fp_eq} for~$\gr{\hat{U}^{(m)}_{\xi_\d,\sh}}$ reads

\begin{equation}
	\gr{\hat{U}_{\xi_\d,\sh}^{(m)}}
	=
	\CP^{E_\sh\hbz^{\xi_\d}}(G'(\gr{W}) \gr{\hat{U}_{\xi_\d,\sh}^{(m)}} \thinspace \<cm>)
	+
	m \CP^{E_\sh\hbz^{\xi_\d}}\del[1]{\hat{A}_{\xi_\d,\sh}^{(m-1),\<wns>}\thinspace \<wn>} + \CP^{E_\sh\hbz^{\xi_\d}}\del[1]{\hat{B}_{\xi_\d,\sh}^{(m-1),\<cms>}\thinspace \<cm>}
	\label{def:fp_prob_uhat_notation:fp_eq_U_m}
\end{equation}	

The announced equations are:

\begin{proposition}[Explicit equations for $\boldsymbol{\hat{u}^{(m)}_{\xi_\d;\sh}}$]\label{prop:expl_eq}
	The term $\hat{u}^{(m)}_{\xi_\d;\sh}$ satisfies the \emph{linear} stochastic PDE given by
	\begin{equs}[][prop:expl_eq:eq1]
		(\partial_t - \Delta) \hat{u}^{(m)}_{\xi_\d;\sh}
		&
		=
		\hat{a}_{\xi_\d;\sh}^{(m-1),\<wns>}  \xi_\d 
		- \hat{a}_{\xi_\d;\sh}^{(m-1),\fc} \fc_{\d} 
		+ \sbr[1]{\hat{b}_{\xi_\d;\sh}^{(m-1),\<cms>}
			+ \hat{u}^{(m)}_{\xi_\d;\sh} g'(w_{\sh})} \sh,
		\quad
		\hat{u}^{(m)}_{\xi_\d;\sh}(0,\cdot) = 0.
	\end{equs}
	The coefficient functions of $\xi_\d$ and $\sh$ are defined by\footnote{The second equalities with $\scal{\ldots,\1}$ are due to an analogue of~\cite[Prop.~3.28]{hairer_rs} to the spaces $\DD^{\gamma,\eta}_{\gr{\CU}}(\hbz^{\xi_\d;\esh})$.}
	\begin{equs}[][prop:expl_eq:coeff_fct_xi_h]
		\hat{a}_{\xi_\d;\sh}^{(m-1),\<wns>} 
		& := m \CR^{\esh}\del[1]{\hat{A}_{\xi_\d;\sh}^{(m-1),\<wns>}}
		= m \scal{\hat{A}_{\xi_\d;\sh}^{(m-1),\<wns>},\mathbf{1}}, \\ 
		\hat{b}_{\xi_\d;\sh}^{(m-1),\<cms>} 
		& := \CR^{\esh}\del[1]{\hat{B}_{\xi_\d;\sh}^{(m-1),\<cms>}}
		= \scal{\hat{B}_{\xi_\d;\sh}^{(m-1),\<cms>},\mathbf{1}}. 
	\end{equs}
	They calculate to the same expressions as the corresponding terms in~\eqref{eq:A_eps_0}, multiplied by~$m$, and~\eqref{eq:B_eps_0}, respectively, only with the replacements~$G \rightsquigarrow g$, $\gr{W} \rightsquigarrow w_\sh$, and $\gr{U^{(k)}} \rightsquigarrow \hat{u}_{\xi_\d;\sh}^{(k)}$.
	For the coefficient of~$\fc_\d$, we have
	\begin{equs}[][prop:expl_eq:coeff_fct_ren_const]
		\hat{a}_{\xi_\d;\sh}^{(m-1),\fc}
		&
		:=  
		m \scal{\hat{A}_{\xi_\d;\sh}^{(m-1),\<wns>},\<1>}
		=
		m! \sum_{k=1}^{m-1} \frac{1}{k!} g^{(k)}(w_\sh) \sum_{\bi \in S_k^{m-1}} \sum_{r=1}^k \frac{1}{i_r!} \hat{a}_{\xi_\d;\sh}^{(i_r-1),\<wns>} \prod_{\substack{n=1, \\ n \neq r}}^k \frac{1}{i_n!} \hat{u}_{\xi_\d;\sh}^{(i_n)}.
	\end{equs}
\end{proposition}

\begin{remark}\label{rmk:heuristic_comp}
	As a sanity check, one can verify that the equation for $\hat{u}^{(m)}_{\xi_\d;\sh}$ in~\thref{prop:expl_eq} can also be obtained by naively applying $\partial_\eps^{m}\sVert_{\eps=0}$ to
	\begin{equation*}
		(\partial_t - \Delta) \hat{u}^{\eps}_{\xi_\d;\sh} 
		= g\del[1]{\hat{u}^{\eps}_{\xi_\d;\sh}}\del[1]{\eps \xi_\d  + \sh - \eps^2 \fc_\d g'\del[1]{\hat{u}^{(\eps)}_{\xi_\d;\sh}}}, 
		\quad 
		\hat{u}^{\eps}_{\xi_\d;\sh}(0,\cdot) = u_0.
	\end{equation*}
	In particular, we find	
	\begin{equs}
		(m & = 1) \quad 
		\hat{a}_{\xi_\d;\sh}^{(0),\<wns>} = g(w_{\sh}), \quad \hat{a}_{\xi_\d;\sh}^{(0),\fc} = 0 = \hat{b}_{\xi_\d;\sh}^{(0),\<cms>} \\ 
		(m & = 2) \quad 
		\hat{a}_{\xi_\d;\sh}^{(1),\<wns>} = 2 g'(w_{\sh})\hat{u}^{(1)}_{\xi_\d;\sh}, \quad \hat{a}_{\xi_\d;\sh}^{(1),\fc} = 2g'(w_{\sh})g(w_{\sh}), \quad
		\hat{b}_{\xi_\d;\sh}^{(1),\<cms>} = g''(w_{\sh})\sbr[1]{\hat{u}^{(1)}_{\xi_\d;\sh}}^2
	\end{equs}
	such that
	\begin{equs}
		(m & = 1) \quad
		(\partial_t - \Delta) \hat{u}^{(1)}_{\xi_\d;\sh}
		=
		g(w_{\sh})  \xi_\d 
		+ \hat{u}^{(1)}_{\xi_\d;\sh} g'(w_{\sh}) \sh \\ 
		(m & = 2) \quad 
		(\partial_t - \Delta) \hat{u}^{(2)}_{\xi_\d;\sh}
		=
		2 g'(w_{\sh})\hat{u}^{(1)}_{\xi_\d;\sh}  \xi_\d 
		- 2g'(w_{\sh})g(w_{\sh}) \fc_{\d} 
		+ \sbr[2]{
			g''(w_{\sh})\sbr[1]{\hat{u}^{(1)}_{\xi_\d;\sh}}^2
			+ \hat{u}^{(2)}_{\xi_\d;\sh} g'(w_{\sh})
		} \sh
	\end{equs}
	with $\hat{u}^{(i)}_{\xi_\d;\sh}(0,\cdot) = 0$ for $i=1,2$.
\end{remark}

\begin{proof}[of~\thref{prop:expl_eq}]
	All the operators $\CR$ and $\CP$ in this proof depend on the model~$E_\sh \hbz^{\xi_\d}$.
	Recall that $\CR \circ \CP = P * \CR$, $(\partial_t - \Delta) [P*v] = v$, and 
	\begin{equation*}
		(\CR\gr{V})(z)
		\equiv
		\del[1]{\CR^{E_\sh\hbz^{\xi_\d}}\gr{V}}(z)
		=
		\hat{\Pi}_z^{\xi_\d;\esh}\del[1]{\gr{V}(z)}(z), \qquad
		\gr{V} \in \DD^{\gamma+\a,\eta+\a}_{\gr{\CU}}(E_\sh\hbz^{\xi_\d})
	\end{equation*}
	where~$\a = \deg(\<wn>)$.
	By definition of $\hat{u}_{\xi_\d,\sh}^{(m)}$ and the fixed-point equation~\eqref{def:fp_prob_uhat_notation:fp_eq_U_m} for $\gr{\hat{U}_{\xi_\d,\sh}^{(m)}}$, we find
	\begin{equs}
		(\partial_t - \Delta) \hat{u}_{\xi_\d,\sh}^{(m)}(z)
		&
		=
		(\partial_t - \Delta) \del[1]{\CR\gr{\hat{U}_{\xi_\d,\sh}^{(m)}}}(z) \\
		&
		=
		\CR\del[2]{
			m\hat{A}_{\xi_\d,\sh}^{(m-1),\<wns>}\thinspace\<wn> 
			+
			\sbr[2]{
				G'(\gr{W}) \gr{\hat{U}_{\xi_\d,\sh}^{(m)}} 
				+ 
				\hat{B}_{\xi_\d,\sh}^{(m-1),\<cms>}
			} \<cm>
		}(z) \\
		&
		=
		\hat{\Pi}_z^{\xi_\d;\esh}\del[2]{
			m \hat{A}_{\xi_\d,\sh}^{(m-1),\<wns>}(z)\<wn> 
			+
			\sbr[2]{
				G'(\gr{W}(z)) \gr{\hat{U}_{\xi_\d,\sh}^{(m)}}(z) 
				+ 
				\hat{B}_{\xi_\d,\sh}^{(m-1),\<cms>}(z)
			} \<cm>
		}(z) \\
		&
		=
		\Pi_z^{\xi_\d;\esh}\del[2]{
			M(\fc_\d)
			\del[2]{
				m \hat{A}_{\xi_\d,\sh}^{(m-1),\<wns>}(z)\<wn> 
				+
				\sbr[2]{ 
					\hat{B}_{\xi_\d,\sh}^{(m-1),\<cms>}(z)
					+
					G'(\gr{W}(z)) \gr{\hat{U}_{\xi_\d,\sh}^{(m)}}(z) 
				} \<cm>
			}
		}(z) \\
		&
		=
		\hat{a}_{\xi_\d;\sh}^{(m-1),\<wns>}(z)  \xi_\d(z) 
		- \hat{a}_{\xi_\d;\sh}^{(m-1),\fc}(z) \fc_\d
		+ \sbr[1]{
			\hat{b}_{\xi_\d;\sh}^{(m-1),\<cms>}(z)
			+
			g'(w_{\sh}(z)) \hat{u}^{(m)}_{\xi_\d;\sh}(z) 
		} \sh(z).
	\end{equs}
	In the previous computation, we have used the fact that only the summand $\hat{A}_{\xi_\d;\sh}^{(m-1),\<wns>}(z) \thinspace \<wn>$ \ contributes a term that contains the symbol $\<11> \equiv \<1> \thinspace \<wn>$ on which~$M(\fc_\d)$ acts other than trivially; 
	in particular,
	\begin{equation*}
		M(\fc_\d) \del[2]{m \hat{A}_{\xi_\d;\sh}^{(m-1),\<wns>}(z) \thinspace \<wn> + (\ldots)}
		=
		m\hat{A}_{\xi_\d;\sh}^{(m-1),\<wns>}(z) \thinspace \<wn> 
		- m\scal{\hat{A}_{\xi_\d;\sh}^{(m-1),\<wns>}(z), \<1>} \fc_\d \1
		+ (\ldots).
	\end{equation*}
	The claimed explicit form of~$\hat{a}_{\xi_\d;\sh}^{(m-1),\<wns>}$ and~$\hat{b}_{\xi_\d;\sh}^{(m-1),\<cms>}$, follows directly from the equalities
	\begin{equation*}
		\scal{G^{(k)}(\gr{W}),\1} = 
		g^{(k)}(w_\sh), 
		\quad 
		\scal{\gr{\hat{U}_{\xi_\d;\sh}^{(j)}},\1} = \hat{u}_{\xi_\d;\sh}^{(j)}, 
		\quad
		j = 1, \ldots, m-1.
	\end{equation*}
	Hence, we are left to establish the explicit expression for~$\hat{a}_{\xi_\d;\sh}^{(m-1),\fc}$ in~\eqref{prop:expl_eq:coeff_fct_ren_const}. For that purpose, recall from~\eqref{eq:A_eps_0} that 
	\begin{equation*}
		m\hat{A}_{\xi_\d;\sh}^{(m-1),\<wn>} 
		= 
		m! \sum_{k=1}^{m-1} \frac{1}{k!} G^{(k)}(\gr{W}) \sum_{\boldsymbol{i} \in S_k^{m-1}}  \prod_{n=1}^k \frac{1}{i_n!} \gr{U_{\xi_\d,\sh}}^{(i_n)}
	\end{equation*}
	and note that $\scal{G^{(k)}(\gr{W}),\<1>} = 0$. The calculation of $\scal{\hat{A}_{\xi_\d;\sh}^{(m-1),\<wns>},\<1>}$ then turns out to be a simple combinatorial problem: we need to take the coefficient function $\scal{\thinspace \cdot \thinspace,\<1>}$ of \emph{exactly} one of the factors
	\begin{equation*}
		\gr{U^{(i_n)}_{\xi_\d,\sh}}, \quad \bi \in S_k^{m-1}, \quad n = 1,\ldots,k,
	\end{equation*}
	take the coefficient functions $\scal{\thinspace \cdot \thinspace,\1}$ for the rest of them, and sum over all possibilities to do so.
	To calculate $\scal{\gr{\hat{U}_{\xi_\d;\sh}^{(j)}},\<1>}$ for $j = 1,\ldots,m-1$, we rewrite \eqref{def:fp_prob_uhat_notation:fp_eq_U_m} as\footnote{The term \enquote{$+ \bar{\CT}$} collects all summands that take values in the polynomial regularity structure $\bar{\CT}$.}
	\begin{equs}
		\gr{\hat{U}_{\xi_\d;\sh}^{(j)}}
		&
		=
		\CI\del[2]{j\hat{A}_{\xi_\d;\sh}^{(j-1),\<wns>} \thinspace \<wn> + \sbr[1]{\hat{B}_{\xi_\d;\sh}^{(j-1),\<cms>} + G'(\gr{W}) \gr{\hat{U}_{\xi_\d;\sh}^{(j)}}} \thinspace \<cm>} + \bar{\CT} \\
	\end{equs}
	and, since $\<1> = \CI \thinspace \<wn>$, observe that
	\begin{equation*}
		\scal{\gr{\hat{U}_{\xi_\d;\sh}^{(j)}},\<1>}
		=
		\left\langle \CI\del[2]{j\hat{A}_{\xi_\d;\sh}^{(j-1),\<wns>} \thinspace \<wn> + \sbr[1]{\hat{B}_{\xi_\d;\sh}^{(j-1),\<cms>}+ G'(\gr{W}) \gr{\hat{U}_{\xi_\d;\sh}^{(j)}}} \thinspace \<cm>}, \<1> \right\rangle
		=
		j \scal{\hat{A}_{\xi_\d;\sh}^{(j-1),\<wns>},\1}
		=
		\hat{a}_{\xi_\d;\sh}^{(j-1),\<wns>}.
	\end{equation*}
	In case $j=1$, note that $\scal{\gr{\hat{U}_{\xi_\d;\sh}^{(1)}},\<1>} = \hat{a}_{\xi_\d;\sh}^{(0),\<wns>} = g(w_\sh)$.
\end{proof}

\section{Local analysis in the vicinity of the minimiser}\label{sec:local_analysis}

Recall that the application of the Cameron-Martin theorem in \thref{prop:cameron_martin} has lead to the expression
	\begin{equation}
		J_\rho(\eps) 
		=                                                  
		\exp\del[3]{-\frac{\FF(\sh)}{\eps^2}} \E\sbr[3]{\exp\del[3]{-\frac{\tilde{F}^{\minus}_\Phi(\sh,\eps)}{\eps^2}}; \thinspace \eps \barnorm{\hbz^{\minus}} < \rho},
		\label{sec:local_analysis:eq1}
	\end{equation}
with~$\tilde{F}^{\minus}_\Phi(\sh,\eps)$	given in~\eqref{eq:prop:cameron_martin:FPhi}. By plugging in the expansion obtained in~\thref{coro:stoch_taylor_gpam_functional}, we obtain the identity
	\begin{equation}
		\tilde{F}^{\minus}_\Phi(\sh,\eps) 
		= 
		\eps \sbr[1]{DF\sVert_{w_\sh}\del[1]{\hat{u}_\sh^{(1)}} + \xi(\sh)}
		+
		\frac{\eps^2}{2} \hat{Q}_\sh
		+
		\sum_{m=3}^{N + 2} \frac{\eps^m}{m!} \hat{F}^{(m)}_{\sh} 
		+ \hat{R}^{F;\eps;N + 3}_{\sh}
		\label{sec:local_analysis:eq2}
	\end{equation}
for each $\eps \in I_0 := [0,\eps_0)$ with $\eps_0 = \eps_0(T,\hbz)$ as in~\eqref{eq:def_eps_zero}.
	
\subsection{Vanishing of the linear term by first-order optimality} \label{sec:first_order_vanish}
In this subsection, we want to use first-order optimality implied by assumption~\ref{ass:h2} to prove the proposition that follows. Recall our convention from~\thref{def:terms_taylor_exp} that $\hat{u}_\sh^{(1)}(\omega) = u_\sh^{(1)}(\hbz(\omega))$ for $\omega \in \Omega$.

\begin{proposition}\label{prop:first_order_optimality}
	We have 
		\begin{equation}
			DF\sVert_{w_\sh}\del[1]{\hat{u}_\sh^{(1)}(\omega)} + \xi(\omega)(\sh) = 0.
			\label{prop:first_order_optimality:eq1}
		\end{equation}
	for $\P$-a.e. $\omega \in \Omega$. Hence, 
		\begin{equation}
			J_\rho(\eps) 
			=                                                  
			\exp\del[3]{-\frac{\FF(\sh)}{\eps^2}} 
			\E\sbr[2]{\exp\del[2]{
					- \frac{1}{2} \hat{Q}_\sh
					- \hat{S}_\sh^{F;\eps;N} 
					- \frac{1}{\eps^2}\hat{R}^{F;\eps;N + 3}_{\sh}
			}; \thinspace \eps\barnorm{\hbz^{\minus}} < \rho}
			\label{prop:first_order_optimality:eq2}
		\end{equation}
	where we have set
		\begin{equation}
			\hat{S}_\sh^{F;\eps;N}
			:= 
			\sum_{m=3}^{N + 2} \frac{\eps^{m-2}}{m!} \hat{F}^{(m)}_{\sh}. 
		\end{equation}
\end{proposition}
	
\begin{proof} 
	By assumption~\ref{ass:h2}, the Cameron-Martin function $\sh$ is the unique minimiser of $\FF = (F_\Phi \circ \LL) + \II$, so $D\FF\sVert_{\sh} \equiv 0 \in \CH'$. We calculate the Fr\'{e}chet derivatives of the two summands of~$\FF$. Since $\xi_\d$ is smooth~a.s., we have $\xi_\d(\omega) \in \CH$ for $\P$-a.e. $\omega \in \Omega$.
		\begin{enumerate}[label=(\arabic*)]
			\item  For Schilder's rate function $\II$, we obtain
				\begin{equation}
					D\II\sVert_{\sh} (\xi_\d(\omega))
					=
					\frac{1}{2} \frac{\dif}{\dif \l}\sVert[2]_{\l=0} \scal{\sh + \l \xi_\d(\omega), \sh + \l \xi_\d(\omega)}_{\CH}
					=
					\scal{\xi_\d(\omega),\sh}_{\CH}.
					\label{eq:der_schilder}
				\end{equation}
			\item Recall that~$(\Phi \circ \LL)(\sh) = w_\sh$ and, by~\thref{prop:fp_choice} (with $\bz := \bz^{\xi_\d}(\omega)$), 
				\begin{equs}
					(\Phi \circ \LL)(\eps \xi_\d(\omega) + \sh)
					& =
					\CR\del[1]{\CS\del[1]{T_\sh \d_\eps \bz^{\xi_\d}(\omega)}}
					=
					\CR\del[1]{\CS_{tr}(u_0,\sh,\d_\eps \bz^{\xi_\d}(\omega))} \\
					&
					=
					\CR\del[1]{\gr{\CS_{ex}^{\eps}}(u_0,\sh,\bz^{\xi_\d}(\omega))}
					=
					\CR\del[1]{\fI_{E_\sh \bz^{\xi_\d}(\omega)}(\eps)}
					=
					u_\sh^{(\eps)}\del[1]{\bz^{\xi_\d}(\omega)},
				\end{equs}
			which in turn implies that
				\begin{equation*}
					D(\Phi \circ \LL)\sVert[1]_{\sh}(\xi_\d(\omega))
					=
					\frac{\dif}{\dif \eps}\sVert[2]_{\eps = 0} u_\sh^{(\eps)}\del[1]{\bz^{\xi_\d}(\omega)}
					=
					u^{(1)}_\sh(\bz^{\xi_\d}(\omega))
					=
					u^{(1)}_\sh\del[1]{\hbz^{\xi_\d}(\omega)}.
				\end{equation*}
			where $\hbz^{\xi_\d}(\omega)$ denotes the BPHZ model associated to~$\xi_\d(\omega)$ and the last equality is due to~\thref{coro:no_renorm_first_order_term}. We combine these observations to obtain
				\begin{equation}
					D(F \circ \Phi \circ \LL)\sVert_{\sh}(\xi_\d(\omega))
					=
					DF\sVert[0]_{(\Phi \circ \LL)(\sh)} \del[2]{D(\Phi \circ \LL)\sVert[1]_{\sh}(\xi_\d(\omega))}
					=
					DF\sVert[0]_{w_\sh}\del[1]{u^{(1)}_\sh\del[1]{\hbz^{\xi_\d}(\omega)}}.
					\label{eq:der_FPhi}
				\end{equation}
		\end{enumerate}
		Altogether, \eqref{eq:der_schilder} and \eqref{eq:der_FPhi} lead to the equality
			\begin{equation}
				DF\sVert[0]_{w_\sh}\del[1]{u^{(1)}_\sh\del[1]{\hbz^{\xi_\d}(\omega)}}
				+
				\scal{\xi_\d(\omega),\sh}_{\CH}
				= D\FF\sVert_\sh(\xi_\d(\omega)) = 0.
				\label{eq:foo_cm}
			\end{equation}
		We recall two more results:
			\begin{enumerate}[label=(\arabic*)]
				\item From \cite[Rmk.~$10.6$]{hairer_rs} and the comments thereafter, we know that 
					\begin{equation}
						\lim_{\d \to 0} \ \scal{\xi_\d,\sh}_{\CH} = \xi(\sh) \quad \text{in} \ L^2(\P).
					\end{equation}
				\item From \cite[Thm.~$10.19$]{hairer_rs} (or, more generally, \cite[Thm.~$2.33$]{chandra-hairer}), we know that $\hbz^{\xi_\d}$ converges to $\hbz$ in probability in $\MM$ as $\d \to 0$.
			\end{enumerate}
		Therefore, we may pass to a common subsequence to obtain $\P$-a.s. limits in the preceding two statements. As $DF\sVert_{w_\sh}$ is continuous by assumption~\ref{ass:h3} and $u^{(1)}_\sh$ is continuous in the model by~\thref{thm:stoch_taylor_gpam}\ref{thm:stoch_taylor_gpam:i}, taking $\d \to 0$ in~\eqref{eq:foo_cm} establishes the claim in~\eqref{prop:first_order_optimality:eq1}.
\end{proof}

By the previous proposition, we now need to analyse the behaviour as $\eps \to 0$ of
\begin{equation}
	\CE_\rho(\eps)
	:=
	\E\sbr[2]{\exp\del[2]{
			- \frac{1}{2} \hat{Q}_\sh 
			- \hat{S}_\sh^{F;\eps;N} 
			- \frac{1}{\eps^2}\hat{R}^{F;\eps;N + 3}_{\sh}
		}; \thinspace \eps\barnorm{\hbz^{\minus}} < \rho}.
	\label{eq:cal_E}
\end{equation}

\subsection{Exponential integrability of the quadratic term} \label{sec:exp_integr_quadr}

In a first step, we study the exponential integrability of the quadratic term $-\nicefrac{1}{2} \thinspace \hat{Q}_\sh$ in \eqref{eq:cal_E}, as it will give us the leading order term $a_0$ in the expansion~\eqref{thm:laplace_asymp:expansion}, cf.~section~\ref{sec:asymp_coeff}.

\begin{proposition} \label{prop:exp_integr_Q}
There exists some $\beta > 0$ such that $\exp\del[2]{-\frac{1}{2} \hat{Q}_\sh} \in L^{p}(\P)$ with $p := 1 +\beta$.
\end{proposition}	
The proof of the preceding proposition crucially relies on the non-degeneracy assumption~\ref{ass:h4} on the minimiser~$\sh$. It will be combined with the following elementary lemma.

\begin{lemma}\label{lem:exp_tails}
	Let $X: \Omega \to \R$ be a random variable and $\lambda \in \R$. Suppose there exists some $C > \lambda$ and $k, 
	K > 0$ such that for all $r \in (k,\infty)$ we have the estimate $\P(X \geq r) \leq K\exp(-C r)$.
	Then, there exists some $\beta > 0$ such that $\exp(\lambda X) \in L^{1 + \beta}(\P)$.
\end{lemma}

Recall that $\hat{Q}_\sh = Q_\sh(\hbz)$.
In \thref{coro:stoch_taylor_gpam_functional:properties}\ref{coro:stoch_taylor_gpam_functional:ii}, we have established that~$\eps^2 Q_\sh(\hbz) = Q_\sh(\d_\eps \hbz)$. Since $Q_\sh$ is continuous in the model (by~\thref{coro:stoch_taylor_gpam_functional:properties}\ref{coro:stoch_taylor_gpam_functional:i}) and the family $\del[1]{\d_\eps\hbz: \eps > 0}$ satisfies a LDP on model space $\MM$ with good rate function (RF)
	\begin{equation}
		\JJ_{\MM}(\bz)
		=
		\inf\{\II(h): \ h \in \CH \ \ \text{with} \ \ \LL(h) = \bz\}
		,
	\end{equation}
cf. \thref{app:thm_ldp_models}, we can apply the \emph{contraction principle} to infer that the family $\del[1]{-\frac{1}{2} Q_\sh(\d_\eps \hbz): \eps > 0}$ satisfies a LDP on $\R$ with good RF
	\begin{equs}
		\Lambda(y)
		:=
		& \thinspace
		\inf\{\JJ_{\MM}(\bz): \ \bz \in \MM \ \ \text{with} \  -\nicefrac{1}{2} \thinspace Q_\sh(\bz) = y\} \\
		=
		& 
		\thinspace 
		\inf\{\II(h): h \in \CH \ \ \text{with} \  -\nicefrac{1}{2} \thinspace Q_\sh(\LL(h)) = y\}.
	\end{equs}
Recall that $\II$ denotes Schilder's RF, that is: $\II(h) = \frac{1}{2}\norm{h}^2_\CH$ for $h \in \CH$.

\begin{proof}[of \thref{prop:exp_integr_Q}]
We want to apply \thref{lem:exp_tails} with $X := -\frac{1}{2} Q_\sh(\hbz)$ and $\lambda = 1$. 
We set
	\begin{equation}
		C_*
		:=
		\inf\{\JJ_{\MM}(\bz): \ \bz \in \MM, \ -\nicefrac{1}{2} \thinspace Q_\sh(\bz) \geq 1\} 
		\label{eq:c_star}
	\end{equation}
and immediately find\footnote{We remind the reader that~assumption~\ref{ass:h4} reads~$D^2\FF\sVert_\sh > 0$ for $\FF = F \circ \Phi \circ \LL + \II$. As we will see shortly, it implies~$Q_\sh \circ \LL > -\operatorname{Id}$ (and \emph{not} $Q_\sh \circ \LL > 0$), so the condition~$-\nicefrac{1}{2} \thinspace Q_\sh(\LL(h)) \geq 1$ for~$h \in \CH$ in~\eqref{eq:HH_star} is indeed meaningful. In particular, it is consistent with earlier work by Ben Arous~\cite{ben_arous_laplace}.}
\begin{equation}
	C_* 
	= 
	\inf\{\nicefrac{1}{2} \thinspace \norm{h}_\CH^2: \ h \in \CH_* \}, \quad
	\CH_* := \{h \in \CH: \ -\nicefrac{1}{2} \thinspace Q_\sh(\LL(h)) \geq 1\}.
	\label{eq:HH_star}
\end{equation}	
The LDP upper bound for $\del[1]{-\frac{1}{2}Q_\sh(\d_\eps \hbz): \eps > 0}$ now implies that for each~$C \in (0,C_*)$ there exists an~$\eps_0 = \eps_0(C) > 0$ such that 
	\begin{equs}
		\thinspace
		&
		\P\del[2]{-\frac{1}{2}Q_\sh(\hbz) \geq \eps^{-2}}
		=
		\P\del[2]{-\frac{1}{2} Q_\sh(\d_\eps\hbz)  \geq 1}
		\leq
		\exp\del[1]{-\eps^{-2} C},
	\end{equs}	
holds for all $\eps \leq \eps_0$. 
Hence, all we need to show is $C_* > \lambda = 1$, or equivalently
	\begin{equation}
		C_* - 1 > 0.
		\label{pf:prop:exp_integr_Q:aux1}
	\end{equation}
If $\CH_* = \emptyset$, then $C_* = +\infty$ and there is nothing to prove, so we assume~$\CH_* \neq \emptyset$. By a compactness argument, oursourced to~\thref{lem:pf_exp_int:compactness_argument} below, we find $h_* \in \CH_*$ with $h_* \neq 0$ and $C_* = \frac{1}{2}\norm{h_*}^2_{\CH}$.
Now take $v,w \in \CH$ and observe the following:
	\begin{enumerate}[label=(\arabic*)]
		\item From \eqref{eq:der_schilder}, we know that $D\II\sVert_{\sh}(v) = \scal{\sh,v}_\CH$. For the second derivative, we find
			\begin{equation*}
				D^2 \II\sVert_{\sh} [v,w] 
				= 
				\frac{\dif}{\dif t}\sVert[2]_{t = 0} D\II\sVert_{\sh + t v}(w)
				= 
				\frac{\dif}{\dif t}\sVert[2]_{t = 0} \scal{\sh + t v,w}_\CH
				=
				\scal{v,w}_\CH.
			\end{equation*}
		In particular, $C_* = \frac{1}{2}\norm{h_*}_\CH^2 = \frac{1}{2}D^2\II\sVert_\sh[h_*,h_*]$.
		\item 
		The definitions in~\thref{coro:stoch_taylor_gpam_functional} immediately imply
			\begin{equs}[][eq:quad_non_degen]
				\frac{1}{2} Q_\sh(\LL(h_*))
				&
				=
				\frac{1}{2} \partial^2_\eps\sVert[1]_{\eps = 0} F \del[1]{u^{(\eps)}_\sh(\LL(h_*))}
				=
				\frac{1}{2} \partial^2_\eps\sVert[1]_{\eps = 0} F \circ \Phi\del[1]{T_\sh \d_\eps \LL(h_*)}\\
				&
				=
				\frac{1}{2} \partial^2_\eps\sVert[1]_{\eps = 0} F \circ \Phi \circ \LL \del[0]{\sh + \eps h_*}
				=
				\frac{1}{2} D^2 (F \circ \Phi \circ \LL)\sVert_\sh[h_*,h_*].
			\end{equs}
		where we additionally use~\thref{rmk:lift_h} and~\thref{app:def:dilation} in the third equality.
	\end{enumerate}
Altogether, we obtain \eqref{pf:prop:exp_integr_Q:aux1} from
	\begin{equation*}
		-1 + C_*
		\geq
		\frac{1}{2} D^2 \del[1]{(F \circ \Phi \circ \LL) + \II}\sVert[1]_\sh [h_*,h_*]
		=
		\frac{1}{2} D^2 \FF\sVert_\sh [h_*,h_*]
		> 0,
	\end{equation*}
where we have used that $h_* \in \CH_*$ in the first estimate and assumption~\ref{ass:h4} in the last.
\end{proof}

\begin{lemma}\label{lem:pf_exp_int:compactness_argument}
	In the setting of the previous proof, assume~$\CH_* \neq \emptyset$. Then, $C_* < \infty$ and there exists $h_* \in \CH_*$ with $h_* \neq 0$ and~$C_* = \frac{1}{2}\norm{h_*}^2_{\CH}$.
\end{lemma}

\begin{proof}
	Notice that $C_* \leq \JJ_{\MM}(\LL(h)) < \infty$ for any~$h \in \CH_* \neq \emptyset$ and recall that $\JJ_{\MM}$ is a good RF, so that the set $L_k := \{\bz \in \MM: \thinspace \JJ_{\MM}(\bz) \leq k\}$ is compact in~$\MM$ for each~$k \in [0,\infty)$. Also, for any such $k$ the set
		\begin{equation*}
			A_k := L_k \cap \{\bz \in \MM: \thinspace -\nicefrac{1}{2} \thinspace Q_\sh(\bz) \geq 1\} \subseteq \MM
		\end{equation*}
	is compact again because~$(-\nicefrac{1}{2} \thinspace Q_\sh)^{-1}([1,\infty])$ is closed in~$\MM$ by continuity of~$Q_\sh$. Now choose $k_*$ sufficiently large such that $A_{k_*} \neq \emptyset$. Then,   
		\begin{equation*}
			C_*
			\equiv 
			\inf\{\JJ_{\MM}(\bz): \ \bz \in \MM, \ -\nicefrac{1}{2} \thinspace Q_\sh(\bz) \geq 1\} 
			=
			\inf\{\JJ_{\MM}(\bz): \thinspace \bz \in A_{k_*} \}
		\end{equation*}
	and since $\JJ_{\MM}$ is l.s.c. as a RF, it attains its minimum on compact sets, so there exists a $\bz_* \in A_{k_*}$ with $C_* = \JJ_{\MM}(\bz_*)$. By definition of~$\JJ_{\MM}$, we may thus find $h_* \in \CH_*$ with $C_* = \frac{1}{2}\norm{h_*}^2_{\CH}$.
	
	Suppose $h_* = 0$. By~\eqref{eq:quad_non_degen}  and the fact that~$h_* \in \CH_*$, the latter defined in~\eqref{eq:HH_star}, we find that
		\begin{equation*}
			-1 \geq \frac{1}{2} Q_\sh(\LL(0)) = \frac{1}{2} D^2 (F \circ \Phi \circ \LL)\sVert_\sh[0,0] = 0,
		\end{equation*}
	because~$D^2 (F \circ \Phi \circ \LL)\sVert_\sh$ is a bilinear form on~$\CH$. Therefore, we clearly have~$h_* \neq 0$.
\end{proof}

\begin{remark}\label{rmk:ben_arous_ldp}
		Ben Arous has also employed a large deviations argument to prove~\cite[Lem.~1.51]{ben_arous_laplace}, the equivalent of~\thref{prop:exp_integr_Q} in his setting. He uses~\cite[Lem.~1.11]{ben_arous_laplace} to close that argument and the proof of the latter, in turn, is based on~\cite[Lem. 1.9,~1.23, and~1.26]{ben_arous_laplace}. These lemmas crucially rely on the correspondence between trace-class operators, Carleman-Fredholm determinants, and exponential integrability for elements in the second Wiener chaos that we alluded to in the introduction.
		
		This is the point where our line of reasoning diverts: We use the continuity of the It\^{o}-Lyons solution map (resp. its counterpart in regularity structures), one of the main results of rough paths theory only developed ten years after Ben Arous's work. More precisely, this idea entered our argument via continuity of $Q_\sh$ that we used to identify the set $\del[1]{-\nicefrac{1}{2}Q_\sh}^{-1}([1,\infty]) \subseteq \MM$ as closed.
		
		The connection to Ben Arous's line of reasoning is elaborated upon in the recent follow-up article~\cite{klose_arxiv} by the second author, see also~\thref{rmk:constant_coeff}.
\end{remark}

\subsection{The coefficients in the asymptotic expansion} \label{sec:asymp_coeff}

We will now systematically decompose the quantity~$\CE_\rho(\eps)$ as introduced in~\eqref{eq:cal_E} to calculate the coefficients~$(a_k)_{k=0}^N \subseteq [0,\infty)$ in the expansion~\eqref{thm:laplace_asymp:expansion}. The terms $I_\rho^{(j)}(\eps)$, $j=1,\ldots,4$, that pop up in this procedure will be estimated later.

We stress that the arguments in this section are completely analogous to those of Inahama and Kawabi~\cite[sec.\thinspace6]{inahama_kawabi}. Since the formal structure of our problem resembles theirs, this is not surprising: once the expansion in~\thref{coro:stoch_taylor_gpam_functional} and the corresponding estimates in~\thref{coro:stoch_taylor_gpam_functional:properties} are obtained, we can recycle their arguments to obtain the afore-mentioned coefficients. 
Since the notational conventions of Inahama and Kawabi are quite different from ours, we still provide the full argument for the sake of completeness. 

\paragraph*{Step 1.} We start by the trivial equality
	\begin{equs}[][coeff_err_1]
	\CE_\rho(\eps)
	& 
	=
	\E\sbr[2]{\exp\del[2]{
			- \frac{1}{2}\hat{Q}_\sh 
			- \hat{S}_\sh^{F;\eps;N}}
			\left\{\exp\del[2]{ 
			- \frac{1}{\eps^2} \hat{R}^{F;\eps;N + 3}_{\sh}} - 1 
		\right\}; \thinspace \eps \barnorm{\hbz^{\minus}} < \rho} \\
	&
	+
	\E\sbr[2]{\exp\del[2]{
			- \frac{1}{2}\hat{Q}_\sh 
			- \hat{S}_\sh^{F;\eps;N}
		}; \eps \barnorm{\hbz^{\minus}}  < \rho}
	=:
	I_\rho^{(1)}(\eps) + \CE_\rho^{(1)}(\eps).
	\end{equs}

\paragraph*{Step 2.} We expand the exponential of~$- \hat{S}_\sh^{F;\eps;N}$ to decompose~$\CE_\rho^{(1)}(\eps) = \CE_\rho^{(2)}(\eps) + I_\rho^{(2)}(\eps)$ with
	\begin{equs}[][coeff_err_2]
		\CE_\rho^{(2)}(\eps)
		& 
		:=
		\E\sbr[4]{\exp\del[3]{
				- \frac{1}{2}\hat{Q}_\sh} 
				\sbr[4]{1 + \sum_{k=1}^N \frac{(-1)^k}{k!} \del[4]{\sum_{m=3}^{N + 2} \frac{\eps^{m-2}}{m!}\hat{F}^{(m)}_{\sh}}^k\thinspace}
			; \thinspace \eps \barnorm{\hbz^{\minus}} < \rho}, \\
		I_\rho^{(2)}(\eps) 
		&
		:=
		\E\sbr[4]{\exp\del[3]{
				- \frac{1}{2}\hat{Q}_\sh}
				\sum_{k=N+1}^\infty \frac{(-1)^k}{k!} \del[4]{\sum_{m=3}^{N + 2} \frac{\eps^{m-2}}{m!}\hat{F}^{(m)}_{\sh}
			 	}^k
			; \thinspace \eps \barnorm{\hbz^{\minus}} < \rho}. 
	\end{equs}
For $k \in \{1,\ldots,N\}$, we now introduce the set $\CG(k,N)$ of all maps~$\pi: \{1,\ldots,k\} \to \{3,\ldots,N+2\}$. For $\pi \in \CG(k,N)$ and~$i \in \{1,\ldots,k\}$, we also set 
	\begin{equation*}
		\pi_i := \pi(i), \quad
		\ell(\pi) := \sum_{i=1}^k (\pi_i-2), \quad
		\abs{\pi} := \sum_{i=1}^k \pi_i,
	\end{equation*} 
to express~$\CE_\rho^{(2)}(\eps)$ as
	\begin{equation*}
	\CE_\rho^{(2)}(\eps)
	=
	\E\sbr[4]{\exp\del[3]{
			- \frac{1}{2}\hat{Q}_\sh}
		\sbr[4]{1 + \sum_{k=1}^N \frac{(-1)^k}{k!} \sum_{\pi \in \CG(k,N)} \eps^{\ell(\pi)} \prod_{i=1}^k \frac{\hat{F}^{(\pi_i)}_{\sh}}{\pi_i!} \thinspace}
		; \thinspace \eps \barnorm{\hbz^{\minus}} < \rho}.
	\end{equation*} 

\paragraph*{Step 3.} The quantity~$\CE_\rho^{(2)}(\eps)$ still contains powers of $\eps$ larger than~$N$, so we observe that we can write~$\CG(k,N) = \CG_{\leq}(k,N) \sqcup \CG_{>}(k,N)$ with
	\begin{equation}
	\CG_{\leq}(k,N) := \{\pi \in \CG(k,N): \ \ell(\pi) \leq N\}
	\label{eq:maps_k_N}
	\end{equation}
and $\CG_{>}(k,N)$ defined analogously, with \enquote{$\leq$} replaced by \enquote{$>$} in the previous definition. Thus, we can isolate all the powers of $\eps$ up to order $N$ in~$\CE_\rho^{(3)}(\eps)$, writing~$\CE_\rho^{(2)}(\eps) = \CE_\rho^{(3)}(\eps) + I_\rho^{(3)}(\eps)$ with 
	\begin{equs}[][coeff_err_3]
	\CE_\rho^{(3)}(\eps)
	&
	:=
	\E\sbr[4]{\exp\del[3]{
			- \frac{1}{2}\hat{Q}_\sh}
		\sbr[4]{1 + \sum_{k=1}^N \frac{(-1)^k}{k!} \sum_{\pi \in \CG_{\leq}(k,N)} \eps^{\ell(\pi)} \prod_{i=1}^k \frac{\hat{F}^{(\pi_i)}_{\sh}}{\pi_i!} \thinspace}
		; \thinspace \eps \barnorm{\hbz^{\minus}} < \rho}, \\
	I_\rho^{(3)}(\eps)  
	&
	:=
	\E\sbr[4]{\exp\del[3]{
			- \frac{1}{2}\hat{Q}_\sh}
		\sbr[4]{\sum_{k=1}^N \frac{(-1)^k}{k!} \sum_{\pi \in \CG_{>}(k,N)} \eps^{\ell(\pi)} \prod_{i=1}^k \frac{\hat{F}^{(\pi_i)}_{\sh}}{\pi_i!} \thinspace}
		; \thinspace \eps \barnorm{\hbz^{\minus}} < \rho}.
	\end{equs}

\paragraph*{Step 4.} In a last step, we get rid of the indicator function of the set $\{\eps \barnorm{\hbz^{\minus}} < \rho\}$ in~$\CE_\rho^{(3)}(\eps)$: We write~$\CE_\rho^{(3)}(\eps) = \CE_\rho^{(4)}(\eps) + I_\rho^{(4)}(\eps)$ with

\begin{equs}[][coeff_err_4]
	\CE_\rho^{(4)}(\eps)
	&
	:=
	\E\sbr[4]{\exp\del[3]{
			- \frac{1}{2}\hat{Q}_\sh}
		\sbr[4]{1 + \sum_{k=1}^N \frac{(-1)^k}{k!} \sum_{\pi \in \CG_{\leq}(k,N)} \eps^{\ell(\pi)} \prod_{i=1}^k \frac{\hat{F}^{(\pi_i)}_{\sh}}{\pi_i!} \thinspace}
		}, \\
	I_\rho^{(4)}(\eps)  
	&
	:=
	- \E\sbr[4]{\exp\del[3]{
			- \frac{1}{2}\hat{Q}_\sh}
		\sbr[4]{1 + \sum_{k=1}^N \frac{(-1)^k}{k!} \sum_{\pi \in \CG_{\leq}(k,N)} \eps^{\ell(\pi)} \prod_{i=1}^k \frac{\hat{F}^{(\pi_i)}_{\sh}}{\pi_i!} \thinspace}
		; \thinspace \eps \barnorm{\hbz^{\minus}} \geq \rho}.
\end{equs}

Finally, it is clear what the coefficients~$a_m$, $m \in \{0,\ldots,N\}$, should be.
Recall that the terms~$\hat{F}_\sh^{(i)}$ have been introduced in~\thref{coro:stoch_taylor_gpam_functional} above.

\begin{proposition}[The coefficients~$a_m$]\label{prop:coeff_asymp_exp}
	The equality $\CE_\rho^{(4)}(\eps) = \sum_{m=0}^N a_m \eps^m$ holds with
		\begin{equation*}
			a_m
			\equiv
			a_m(\sh,\hbz,F)
			:= \E\sbr[3]{\exp\del[3]{- \frac{1}{2}\hat{Q}_\sh}W_m} < \infty
		\end{equation*}
	for~$W_0 := 1$ and 
		\begin{equation*}
			W_m
			:=
			\sum_{k=1}^N \frac{(-1)^k}{k!} \sum_{\pi \in \CG_{=}(k,m)} \prod_{i=1}^k \frac{\hat{F}^{(\pi_i)}_{\sh}}{\pi_i!}, \qquad
			\CG_{=}(k,m) := \{\pi \in \CG(k,m): \ \ell(\pi) = m\}
		\end{equation*}	
	with~$m \in [N] := \{1,\ldots,N\}$.
\end{proposition}

\begin{remark} \label{rmk:coeff_asymp_exp}
	Note that the terms~$W_m$ can be calculated~\emph{explicitly} by~\thref{lem:coro_stoch_taylor} and~\thref{prop:expl_eq}.
\end{remark}

\begin{proof}
	Observe that for~$m \in [N]$,~\enquote{$\pi \in \CG_{\leq}(k,N)$ and $\ell(\pi) = m$} is equivalent to $\pi \in \CG_=(k,m)$.
	Therefore, we may rewrite~$\CE_\rho^{(4)}(\eps)$ from~\eqref{coeff_err_4} as
		\begin{equation*}
			\CE_\rho^{(4)}(\eps)
			=
			\E\sbr[4]{\exp\del[3]{
					- \frac{1}{2}\hat{Q}_\sh}
				\sum_{m=0}^N \eps^m W_m
			}
			= \sum_{m=0}^N a_m \eps^m
		\end{equation*}
	with~$W_m$ and~$a_m$ as claimed in the assertion.
	It remains to prove the~$a_m$'s are finite. By~\thref{prop:exp_integr_Q}, we know that $\exp\del[1]{-\nicefrac{1}{2}\hat{Q}_\sh} \in L^{p}(\P)$ for some $p > 1$. We denote the conjugate H\"older exponent to $p$ by $q$ to infer that
		\begin{equation*}
			\norm[3]{\exp\del[3]{- \frac{1}{2}\hat{Q}_\sh}W_m}_{L^1(\P)}
			\leq
			\norm[3]{\exp\del[3]{- \frac{1}{2}\hat{Q}_\sh}}_{L^p(\P)} \norm[0]{W_m}_{L^q(\P)}
			< \infty,
		\end{equation*} 
	Here we have used that~$W_m \in L^q(\P)$ because it belongs to a finite inhomogeneous Wiener chaos and thus has moments of all orders~\cite[Thm.~3.50]{janson}.
\end{proof}

We will now prove that $I^{(j)}_\rho(\eps) = o(\eps^N)$ as $\eps \to 0$ for each $j=1,\ldots,4$. For that purpose, recall that $\eps \in I_0$ is such that~$\eps \barnorm{\hbz^{\minus}} < \rho$; in particular, the estimates  
	\begin{equation*}
		\abs[1]{\hat{F}^{(m)}_{\sh}} \aac \del[1]{1 + \barnorm{\hbz^{\minus}}}^m, \quad
		\abs[1]{\hat{R}^{F;\eps;N + 3}_{\sh}} \aac_\rho \eps^{N+3}\del[1]{1 + \barnorm{\hbz^{\minus}}}^{N+3}
	\end{equation*}
from~\thref{coro:stoch_taylor_gpam_functional:properties} hold. In the remainder of this section, we will frequently use them in the following form:
	\begin{equs}[][eq:est_terms_and_remainder]
		\eps^m \abs[1]{\hat{F}^{(m+2)}_{\sh}} \1_{\eps \barnorm{\hbz^{\minus}} < \rho} 
		& \aac \eps^m \del[1]{1 + \barnorm{\hbz^{\minus}}}^{m+2} \1_{\eps \barnorm{\hbz^{\minus}} < \rho} 
		\aac (\eps + \rho)^m \del[1]{1 + \barnorm{\hbz^{\minus}}}^2 \\
		\frac{1}{\eps^2} \abs[1]{\hat{R}^{F;\eps;N + 3}_{\sh}} \1_{\eps \barnorm{\hbz^{\minus}} < \rho} 
		&
		\aac_\rho \eps^{N+1}\del[1]{1 + \barnorm{\hbz^{\minus}}}^{N+3} \1_{\eps \barnorm{\hbz^{\minus}} < \rho}
		\aac (\eps + \rho)^{N+1} \del[1]{1 + \barnorm{\hbz^{\minus}}}^2
	\end{equs}	
\\

\noindent
\textbf{Estimates for $\boldsymbol{I^{(1)}_\rho(\eps)}$ from Step $\boldsymbol{1}$.} 
Recall that
	\begin{equation*}
		I^{(1)}_\rho(\eps)
		=
		\E\sbr[3]{\exp\del[3]{
				- \frac{1}{2}\hat{Q}_\sh
				- \sum_{m=1}^{N} \frac{\eps^m}{(m+2)!}\hat{F}^{(m+2)}_{\sh}}
			\left\{\exp\del[2]{ 
				- \frac{1}{\eps^2} \hat{R}^{F;\eps;N + 3}_{\sh}} - 1 
			\right\}; \thinspace \eps \barnorm{\hbz^{\minus}} < \rho}
	\end{equation*}
and the mean value theorem provides the estimate
	\begin{equation*}
		\abs[2]{\exp\del[2]{- \frac{1}{\eps^2} \hat{R}^{F;\eps;N + 3}_{\sh}} - 1}
		\leq
		\exp\del[2]{\frac{1}{\eps^2} \abs{\hat{R}^{F;\eps;N + 3}_{\sh}}\thinspace} 
		\frac{1}{\eps^2} \abs{\hat{R}^{F;\eps;N + 3}_{\sh}}.
	\end{equation*}
Combined with the estimates in~\eqref{eq:est_terms_and_remainder}, we obtain the inequality
	\begin{equs}
		\abs[1]{I^{(1)}_\rho(\eps)}
		&
		\leq
		C \eps^{N+1}
		\E\sbr[3]{\exp\del[3]{
				- \frac{1}{2}\hat{Q}_\sh 
				+ C_2\del[3]{\thinspace\sum_{m=1}^{N} \frac{(\eps +\rho)^m}{(m+2)!} + (\eps +\rho)^{N+1}}\del[1]{1 + \barnorm{\hbz^{\minus}}}^{2}}
			\del[1]{1 + \barnorm{\hbz^{\minus}}}^{N+3}
			} \\
		&
		\leq
		C \eps^{N+1}
		\E\sbr[3]{\exp\del[3]{
					- \frac{1}{2}\hat{Q}_\sh 
				+ C_3 \del[1]{1 + \barnorm{\hbz^{\minus}}}^{2}}
			\del[1]{1 + \barnorm{\hbz^{\minus}}}^{N+3}
		}
	\end{equs}
where $C_3 = C_3(N,\sh,\eps,\rho) > 0$ and $C_3 \to 0$ as $\eps,\rho \to 0$. 
Finally, we first apply Hölder's and then the Cauchy-Schwarz inequality to obtain
	\begin{equs}
		\thinspace
		&
		\E\sbr[3]{\exp\del[3]{
				- \frac{1}{2}\hat{Q}_\sh
				+ C_3 \del[1]{1 +\barnorm{\hbz^{\minus}}}^{2}}
			\del[1]{1 +\barnorm{\hbz^{\minus}}}^{N+3}} \\
		\leq
		\thinspace
		&
		\norm[3]{\exp\del[3]{-\frac{1}{2}\hat{Q}_\sh}}_{L^p(\P)}
		\norm[1]{\exp\del[1]{C_3 q \del[1]{1 + \barnorm{\hbz^{\minus}}}^{2}}\del[1]{1 +\barnorm{\hbz^{\minus}}}^{q(N+3)}}_{L^1(\P)}^{\frac{1}{q}} \\
		\leq
		\thinspace 
		&
		\norm[3]{\exp\del[3]{-\frac{1}{2}\hat{Q}_\sh}}_{L^p(\P)}
		\norm[2]{\exp\del[2]{2C_3 q \del[1]{1 +\barnorm{\hbz^{\minus}}}^{2}}}_{L^1(\P)}^{\frac{1}{2q}}
		\norm[1]{\del[1]{1 +\barnorm{\hbz^{\minus}}}^{2q(N+3)}}_{L^1(\P)}^{\frac{1}{2q}}
	\end{equs}
The first factor is finite by~\thref{prop:exp_integr_Q}, where $p := 1 + \beta  > 1$ and $q = \nicefrac{p}{(p - 1)}$ its conjugate H\"older exponent. For the  second and third factor, we choose $\eps$ and $\rho$ so small that $2 C_3 q < \chi$, where $\chi$ is the constant from the Fernique-type~\thref{thm:fernique}. This choice renders the two factors finite. 

In total, we find $I^{(1)}_\rho(\eps) = o(\eps^N)$ as $\eps \to 0$. \\

\noindent
\textbf{Estimates for $\boldsymbol{I^{(2)}_\rho(\eps)}$ from Step $\boldsymbol{2}$.} Like Inahama and Kawabi \cite{inahama_kawabi}, we use the elementary inequality
	\begin{equation*}
		\abs{\sum_{k=n}^{\infty} \frac{1}{k!} x^k} \leq \frac{\thinspace \abs{x}^n}{n!} \exp(\thinspace\abs{x}), \quad x \in \R, \quad n \in \N_0,
	\end{equation*}
for $n := N+1$ and $x :=  \hat{S}_\sh^{F;\eps;N} = - \sum_{m=1}^{N} \eps^m ((m+2)!)^{-1} \hat{F}^{(m+2)}_{\sh}$ to obtain the inequality
	\begin{equs}
		\abs[1]{I_\rho^{(2)}(\eps)}
		&
		\leq 
		\frac{\eps^{N+1}}{(N+1)!}
		\E\sbr[3]{
			\exp\del[3]{- \frac{\hat{Q}_\sh}{2} + \sum_{m=1}^{N} \frac{\eps^m \abs[1]{\hat{F}^{(m+2)}_{\sh}}}{(m+2)!} }
			\del[3]{\sum_{m=1}^{N} \frac{\eps^{m-1} \abs[1]{\hat{F}^{(m + 2)}_{\sh}}}{(m + 2)!}}^{N+1}
			; \thinspace \eps \barnorm{\hbz^{\minus}} < \rho}
	\end{equs}
By~\eqref{eq:est_terms_and_remainder}, we can further estimate~$I_\rho^{(2)}(\eps)$:
	\begin{align}
		\abs[1]{I_\rho^{(2)}(\eps)}
		& \leq
		C 
		\frac{\eps^{N+1}}{(N+1)!}
		\E\sbr[3]{
			\exp\del[3]{- \frac{\hat{Q}_\sh}{2} + C_2 \sum_{m=1}^{N} \frac{(\eps + \rho)^m}{(m+2)!} \del[1]{1 + \barnorm{\hbz^{\minus}}}^2} \del[1]{1 + \barnorm{\hbz^{\minus}}}^{3(N+1)}} \notag\\
		& \leq
		C 
		\frac{\eps^{N+1}}{(N+1)!}
		\E\sbr[3]{
			\exp\del[3]{- \frac{\hat{Q}_\sh}{2} + C_3  \del[1]{1 + \barnorm{\hbz^{\minus}}}^2}
		 \del[1]{1 + \barnorm{\hbz^{\minus}}}^{3(N+1)}},
		 \label{eq:est_I_2}
	\end{align}	
where $C_3 = C_3(N,\sh,\eps,\rho) > 0$ is the same as in the estimates for~$I_\rho^{(1)}$. Choosing $\eps, \rho$ small and applying Hölder's and the Cauchy-Schwarz inequality as for~$I_\rho^{(1)}$, we find that the expected value in~\eqref{eq:est_I_2} is~finite. 

Thus, $I^{(2)}_\rho(\eps) = o(\eps^N)$ as $\eps \to 0$. \\

\noindent
\textbf{Estimates for $\boldsymbol{I^{(3)}_\rho(\eps)}$ from Step $\boldsymbol{3}$.} 
To begin with, observe that $\pi \in \CG_>(k,N)$, i.e. $\ell(\pi) \geq (N+1)$. By~\eqref{eq:est_terms_and_remainder}, we find
	\begin{equs}
		\eps^{\ell(\pi)} \prod_{i=1}^k \frac{\abs[0]{\hat{F}^{(\pi_i)}_{\sh}}}{\pi_i!} \1_{\eps \barnorm{\hbz^{\minus}} < \rho}
		& \aac
		\eps^{\ell(\pi)} \prod_{i=1}^k \frac{\del[1]{1 + \barnorm{\hbz^{\minus}}}^{\pi_i}}{\pi_i!} \1_{\eps \barnorm{\hbz^{\minus}} < \rho}
		\leq
		\eps^{\ell(\pi)} \del[1]{1 + \barnorm{\hbz^{\minus}}}^{\abs[0]{\pi}} \\
		&
		\leq
		\eps^{N + 1} \del[1]{1 + \barnorm{\hbz^{\minus}}}^{\abs[0]{\pi}}
	\end{equs}
We use these inequalities on~\eqref{coeff_err_3} to find
	\begin{equation}
		\abs[1]{I_\rho^{(3)}(\eps)} 
		\aac \eps^{N+1} \E\sbr[3]{\exp\del[3]{-\frac{\hat{Q}_\sh}{2}}\sum_{k=1}^N \frac{1}{k!} \sum_{\pi \in \CG_>(k,N)} \del[1]{1 + \barnorm{\hbz^{\minus}}}^{\abs[0]{\pi}}}.
		\label{eq:est_I_3}
	\end{equation}
Using the inequalities
	\begin{equation*}
		\#[\CG_{S}(k,N)] \leq \#[\CG(k,N)] \leq N^k, 
		\quad \abs[0]{\pi} \leq k (N+3) \quad \del[1]{\pi \in \CG_>(k,N), \ S \in \{\leq,>\}}
	\end{equation*}
we can further estimate
	\begin{equation*}
		\text{RHS of } \eqref{eq:est_I_3}
		\aac \eps^{N+1} \exp(N) \E\sbr[3]{\exp\del[3]{-\frac{\hat{Q}_\sh}{2}} (1 + \barnorm{\hbz^{\minus}})^{N(N+3)}}.
	\end{equation*}
The expected value in the previous expression is finite by the same arguments as before; altogether, we thus find that $I^{(3)}_\rho(\eps) = o(\eps^N)$ as $\eps \to 0$. \\
	
\noindent
\textbf{Estimates for $\boldsymbol{I^{(4)}_\rho(\eps)}$ from Step $\boldsymbol{4}$.} Recall that
	\begin{equation}
		I_\rho^{(4)}(\eps)
		=
		- \E\sbr[3]{\exp\del[3]{-\frac{\hat{Q}_\sh}{2}} 
			\sbr[3]{ 
				1 +
				\sum_{k=1}^N \frac{(-1)^k}{k!} \sum_{\pi \in \CG_{\leq}(k,N)} \eps^{\ell(\pi)}\prod_{i=1}^m \frac{\hat{F}^{(\pi_i)}_{\sh}}{\pi_i!}
			}; \thinspace \eps \barnorm{\hbz^{\minus}} \geq \rho}.
		\label{eq_I_4_reminder}
	\end{equation}
First, observe that since $\pi \in \CG_{\leq}(k,N)$, we have $\ell(\pi) \leq N$ and thus $\abs[0]{\pi} = \ell(\pi) + 2k \leq N + 2k$.
We now argue analogously as in the estimate for~$I_\rho^{(3)}$ to get
	\begin{equs}
		\sum_{k=1}^N \frac{1}{k!} \sum_{\pi \in \CG_{\leq}(k,N)} \prod_{i=1}^k \frac{\abs[0]{\hat{F}^{(\pi_i)}_{\sh}}}{\pi_i!} 
		& \aac 
		\sum_{k=1}^N \frac{1}{k!} \sum_{\pi \in \CG_{\leq}(k,N)} \del[1]{1 + \barnorm{\hbz^{\minus}}}^{\abs[0]{\pi}}
		\leq
		\sum_{k=1}^N \frac{1}{k!} \sum_{\pi \in \CG_{\leq}(k,N)} \del[1]{1 + \barnorm{\hbz^{\minus}}}^{N + 2k} \\
		&
		\leq 
		\sum_{k=1}^N \frac{N^k}{k!} \del[1]{1 + \barnorm{\hbz^{\minus}}}^{3N} 
		\leq
		\exp(N) \del[1]{1 + \barnorm{\hbz^{\minus}}}^{3N}
	\end{equs}
and then, brutally estimating $\eps^{\ell(\pi)} \leq 1$, we have	
	\begin{equation}
		\abs[1]{I_\rho^{(4)}(\eps)}
		\aac
		\E\sbr[3]{\exp\del[3]{-\frac{\hat{Q}_\sh}{2}}\sbr[1]{1 + \exp(N) (1 + \barnorm{\hbz^{\minus}})^{3N}} \1_{\eps \barnorm{\hbz^{\minus}} \geq \rho}}.
		\label{eq_I_4_est_1}
	\end{equation}
Next, with $\chi > 0$ as in~\thref{thm:fernique}, we estimate
	\begin{equs}[][eq_I_4_est_2]
		\P\del[1]{\barnorm{\hbz^{\minus}} \geq \nicefrac{\rho}{\eps}}	
		=
		\P\del[1]{\exp\del[1]{\chi \barnorm{\hbz^{\minus}}^2} \geq \exp\del[1]{\chi \sbr[1]{\nicefrac{\rho}{\eps}}^2}}
		\leq
		\exp\del[1]{- \chi \sbr[1]{\nicefrac{\rho}{\eps}}^2} \E\sbr[1]{\exp\del[1]{\chi \barnorm{\hbz^{\minus}}^2}}
	\end{equs}
using Markov's inequality. For $\nicefrac{1}{p_1} + \nicefrac{1}{p_2} = 1$ with $1 < p_1 < p$ and $p = 1 + \beta$ as in~\thref{prop:exp_integr_Q}, H\"{o}lder's inequality then implies the estimate
	\begin{equs}[][eq_I_4_est_3]
		\text{RHS of } \eqref{eq_I_4_est_1}
		\leq
		\P\del[1]{\barnorm{\hbz^{\minus}} \geq \nicefrac{\rho}{\eps}}^{\nicefrac{1}{p_2}} 
		\norm[3]{\exp\del[3]{-\frac{\hat{Q}_\sh}{2}}\sbr[1]{1 + \exp(N) \del[1]{1 + \barnorm{\hbz^{\minus}}}^{3N}}}_{L^{p_1}(\P)}
	\end{equs}
so that we are left to estimate the term~$\norm[0]{\cdot}_{L^{p_1}(\P)}$. Note that, since $p_1 < p$, we have
	\begin{equation}
		\norm[3]{\exp\del[3]{-\frac{\hat{Q}_\sh}{2}}}_{L^{p_1}(\P)} < \infty
		\label{eq_I_4_est_4}
	\end{equation}
by~\thref{prop:exp_integr_Q}. For the other part, we choose $r_1 > 1$ and $r_2$ such that $r_1 p_1 \leq p$ and~$\nicefrac{1}{r_1} + \nicefrac{1}{r_2} = 1$ and apply H\"{o}lder's inequality again to obtain
	\begin{equation}
		\norm[3]{\exp\del[3]{-\frac{\hat{Q}_\sh}{2}}\sbr[1]{\exp(N) (1 + \barnorm{\hbz^{\minus}})^{3N}}}_{L^{p_1}(\P)}
		\aac
		\norm[3]{\exp\del[3]{-\frac{p_1}{2}\hat{Q}_\sh}}_{L^{r_1}(\P)}^{\frac{1}{p_1}}
		\norm[1]{\del[1]{1 + \barnorm{\hbz^{\minus}}}^{3Np_1}}_{L^{r_2}(\P)}^{\frac{1}{p_1}}
		\label{eq_I_4_est_5}
	\end{equation}
which is finite.
Finally, we combine~\eqref{eq_I_4_est_2}, \eqref{eq_I_4_est_4}, and~\eqref{eq_I_4_est_5} to estimate
	\begin{equation*}
		\text{RHS of } \eqref{eq_I_4_est_3}
		\aac 
		\exp\del[1]{- \chi \sbr[1]{\nicefrac{\rho}{\eps}}^2}
	\end{equation*}
which implies that $I^{(4)}_\rho(\eps) = o(\eps^m)$ as $\eps \to 0$ for any~$m \in \N_0$.

\appendix

\section{Regularity structures -- background knowledge} \label{app:background_pam}

In this section, we provide the basis on regularity structures needed in this article. We try to keep the presentation self-contained but the sheer scope of the material forces us to refer the reader to the literature where appropriate. For a short introduction on regularity structures, we refer the reader to Hairer~\cite{hairer_intro} as well as Chandra and Weber~\cite{chandra_weber}.

\subsection{The original and the extended regularity structure} \label{app:rs_background}

In this section, we briefly repeat the algebraic structures needed in this article. More precisely, we introduce the (extended) \eqref{gpam} regularity structure.

\subsubsection*{$\bullet$ The original gPAM regularity structure}

The main ingredient of the theory is a \emph{regularity structure} $\TT = (\CA,\CT,\CG)$. Its first part is a graded vector space $\CT = \bigoplus_{\a \in \CA} \CT_\a$, the so-called \emph{model space}, where the second part $\CA \subseteq \R$ is a locally finite \emph{index set} that is bounded from below. The space $\CT$ is endowed with a group $\CG$ of linear bounded operators, the so-called \emph{structure group}, such that for every $\alpha \in \CA$, $\tau \in \CT_\a$, and every $\Gamma \in \CG$, one has
\begin{equation*}
	\Gamma \tau - \tau \in \CT_{< \a} := \bigoplus_{\substack{\b \in \CA, \thinspace \b < \a}} \CT_\b.
\end{equation*}
The specific regularity structure for \eqref{gpam} has been constructed frequently, see Hairer~\cite{hairer_rs} and~Hairer and Pardoux~\cite{hairer_pardoux}:  The model space is given by $\CT = \CT_{\CW} \oplus \CT_\CU$ where $\CT_{\CW} := \scal{\CW}$, $\CT_{\CU} := \scal{\CU}$, and 
\begin{align*}
	\CW & := \{\Xi, \CI(\Xi)\Xi,X_i \Xi: \ i=1,2\} = \{\<wn>, \<11>, X_i\<wn>: \ i=1,2\}, \\
	\CU	 & := \{\mathbf{1}, \CI(\Xi), X_i: \ i=1,2\} = \{\mathbf{1}, \<1>, X_i: \ i=1,2\}.
\end{align*}
We also set $\CF := \CU \cup \CW = \{\<wn>, \<11>, X_i\<wn>, \mathbf{1}, \<1>, X_i: \ i = 1,2\}$ such that $\CT = \scal{\CF}$. The index set $\CA$ specifies the \emph{homogeneities} associated to the symbols in $\CF$. We fix $0 < \kappa \ll 1$ and set $\CA := \{\reg{\tau}: \ \tau \in \CF\}$, where for $i = 1,2$ we set\footnote{As the graphical formalism is classical by now (and self-explanatory in our context), we refer to Hairer and Pardoux~\cite[Sec.~4.1]{hairer_pardoux} for a more detailed explanation.}
{\small
	\begin{table}[h]
		\captionsetup{width=.8\linewidth}
		\centering
		\renewcommand{\arraystretch}{1.5}
		\begin{tabular}{ccccccc}
			\toprule
			$\tau$ & $\<wn>$& $\<11>$ & $X_i\<wn>$& $\mathbf{1}$& $\<1>$& $X_i$ 
			\\
			\midrule
			$ \ \reg{\tau} \ $ & $\ -1 - \kappa \ $ & $ \ - 2\kappa \ $ & $ \ -\kappa \ $ & $ \ 0 \ $ & $ \ 1 - \kappa\ $ & $ \ 1 \ $
			\\
			\bottomrule
		\end{tabular}
		\label{app:rs:table_reg}
	\end{table}
}

\noindent
for fixed $\kappa > 0$ with $\kappa \ll 1$.
The construction of the structure group $\CG$ will be presented separately below.

\subsubsection*{$\bullet$ The extended gPAM regularity structure}

As exemplified by eq.~\eqref{eq:approx_renorm_shifted}, we want to consider noises shifted by Cameron-Martin functions~$h \in \CH$ during our analysis. Hence, we need to extend~$\TT$ by introducing another noise symbol $H \equiv \<cm>$ that represents $h$. This procedure has already been carried out in the context of Malliavin calculus, for \eqref{gpam} by Cannizzaro, Friz, and Gassiat~\cite{cfg}, and more recently by Sch\"{o}nbauer~\cite{schoenbauer} in a much more general framework. We adopt the setting of the former and consider the extended regularity structure $\gr{\TT} = (\gr{\CA}, \gr{\CT}, \gr{\CG})$: the extended model space $\gr{\CT} = \gr{\CT_{\CW}} \oplus \gr{\CT_{\CU}}$ is given by $\gr{\CT_{\CW}} := \scal{\gr{\CW}}$ and $\gr{\CT_{\CU}} := \scal{\gr{\CU}}$, where
\begin{align*}
	\gr{\CW}
	& 
	: = \{\Xi, H, \CI(\Xi)\Xi, \CI(H)\Xi, \CI(\Xi)H, \CI(H)H, X_i \Xi, X_i H: \ i=1,2\} \\
	& 
	= \{\<wn>, \<cm>, \<11>, \<1g1>, \<11g>, \<1g1g>, X_i \<wn>, X_i\<cm>: \ i=1,2\} \\
	\gr{\CU}	 
	& 
	:= \{\mathbf{1}, \CI(\Xi), \CI(H), X_i: \ i=1,2\} = \{\mathbf{1}, \<1>, \<1g>, X_i: \ i=1,2\}
\end{align*}
Sometimes, we will also refer to the extended list of symbols as $\gr{\CF} := \gr{\CU} \cup \gr{\CW}$, so that $\gr{\CT} = \scal{\gr{\CF}}$. Further, we introduce $\gr{\CA} := \{\reg{\tau}: \ \tau \in \gr{\CF}\}$, where for $i = 1,2$ the additional symbols have homogeneities
{\small
	\begin{table}[h]
		\captionsetup{width=.8\linewidth}
		\centering
		\renewcommand{\arraystretch}{1.5}
		\begin{tabular}{ccccccc}
			\toprule
			$\tau$ & $\<cm>$& $\<11g>$ & $\<1g1>$ & $\<1g1g>$ & $X_i\<cm>$ & $\<1g>$ 
			\\
			\midrule
			$ \ \reg{\tau} \ $ & $\ -1 - \kappa \ $ & $ \ - 2\kappa \ $ & $ \ -2\kappa \ $ & $ \ -2\kappa \ $ & $ \ - \kappa\ $ & $ \ 1 - \kappa \ $
			\\
			\bottomrule
		\end{tabular}
		\label{app:rs:table_reg_ext}
	\end{table}
} 

\subsubsection*{$\bullet$ The structure groups $\boldsymbol{\CG}$ and $\boldsymbol{\gr{\CG}}$.}

For the sake of brevity, we focus on the construction of $\gr{\CG}$ as in \cite[Sec.~3.2]{cfg}. This is in analogy with~\cite[Sec.~8.1]{hairer_rs} where the structure group~$\CG$ is constructed. The latter also serves as a reference for further details on our construction; for a far-reaching generalisation, see Bruned, Hairer, and Zambotti~\cite{bhz}.

We introduce the vector space $\gr{\CT_+}$ built from finite linear combinations of the basis vectors
\begin{equation*}
	X^k \prod_{j} \CJ_{\ell_j} \tau_j, \qquad j,k \in \N^2, \ \tau_j \in \gr{\CT}, \ \reg{\tau_j} + 2 - \abs{j} > 0, 
\end{equation*} 
on which we act by a co-module $\gr{\Delta}$, i.e. a linear map $\gr{\Delta}: \gr{\CT} \to \gr{\CT} \tp \gr{\CT_+}$ given by
\begin{equation*}
	\gr{\Delta} \tau := \tau \tp \1 \ \text{for} \ \tau \in \{\1,\<wn>,\<cm>\}, \quad \gr{\Delta} X_i := X_i \tp \1 + \1 \tp X_i,
\end{equation*}
as well as, for $\tau, \sigma \in \gr{\CT}$ and $B_\tau := \{m \in \N^2: \reg{\tau} + 2 - \abs{m} > 0\}$, recursively by
\begin{equation*}
	\gr{\Delta} (\tau\sigma) := (\gr{\Delta}\tau) (\gr{\Delta}\sigma), \quad 
	\gr{\Delta} (\CI \tau) := (\CI \tp \operatorname{Id}) \gr{\Delta} \tau + \sum_{k + \ell \in B_\tau} \frac{1}{k! \ell!} X^k \tp X^\ell\CJ_{k+\ell}(\tau)
\end{equation*}
In the previous expression, note that~$X^k \in \gr{\CT}$ and~$X^\ell \in \gr{\CT_+}$: By abuse of notation, we do not distinguish between these symbols on the left and right hand side of the tensor product.

One then defines $\gr{\CG_+}$ as the set of \emph{characters} of~$\gr{\CT_+}$, that is: linear functionals $\gr{f}: \gr{\CT_+} \to \R$ such that $\gr{f}(\tau\sigma) = \gr{f}(\tau) \gr{f}(\sigma)$. The structure group $\gr{\CG}$ is then given by all $\Gamma_{\gr{f}}$, $\gr{f} \in \gr{\CG_+}$, where
\begin{equation}
	\Gamma_{\gr{f}} \tau := (\operatorname{Id} \tp \gr{f}) \Delta \tau, \quad \tau \in \gr{\CT}.
	\label{app:rs:str_gr}
\end{equation}
A matrix representation for the structure groups $\CG$ and $\gr{\CG}$ is given in \cite[eq.~(3.3)]{cfg} and \cite[eq.~(3.4)]{cfg}, respectively.

\begin{remark}
	For convenience, the constructions that follow are presented for the regularity structure~$\TT$. Mutatis mutandis, they also hold for~$\gr{\TT}$.
\end{remark}

\subsection{Admissible models} \label{app:sec:adm_models}

In practice, the cornerstone of Hairer's theory are the so-called \emph{models}: a family of analytical objects taylored to the equation at hand to mimic the role of classical Taylor polynomials. 

In particular, \emph{admissible} models relate last paragraph's abstract integration operator $\CI$ to a concrete integral kernel~$K$. We follow Hairer~\cite[Sec.~10.4]{hairer_rs} in choosing
\begin{equation}
	K(x) = -\frac{1}{2\pi} \log \thinspace \norm{x}
	\label{app:sec:adm_models:def_K}
\end{equation} 
for $x \in \R^2$ in a sufficiently small neighbourhood of $0$. Outside of that neighbourhood, we choose~$K$ smooth, compactly supported, and such that~$\int_{\R^2} x^k K(x) \dif x = 0$ for~$\abs{k} \leq r$ and some fixed degree $r$ that is sufficiently large. The definitions that follow are in adapted to this choice (and thus our equation at hand) but they work in much greater generality.

\begin{definition}[Smooth admissible model]\label{def:smooth_admissible_model}
	Let $\Pi: \R^2 \to \CL(\CT,\CC^{\8}(\R^2))$ be a smooth map that satisfies the following four conditions:
	\begin{enumerate}[label=($\Pi$\arabic*)]
		\item For each $\fK \subseteq \R^2$ compact, the bound $\abs[0]{(\Pi_y\tau)(\bar{y})} \aac_{\fK} \abs{\bar{y}-y}^{\thinspace\reg{\tau}}$ holds uniformly over $y,\bar{y} \in \fK$ and $\tau \in \CF$. \label{def:smooth_admissible_model:Pi1}
		\item For each $y,\bar{y} \in \R^2$, we have $\Pi_y\symbol{\1}(\bar{y}) \equiv 1$. \label{def:smooth_admissible_model:Pi2}
		\item For any $k \in \N^2$, $y,\bar{y} \in \R^2$ and $\tau \in \CF$ with $\symbol{X^k\tau} \in \CT$ we have \label{def:smooth_admissible_model:Pi3}
		\begin{equation}
			(\Pi_y[\symbol{X^k\tau}])(\bar{y}) = (\bar{y} - y)^k(\Pi_y \tau)(\bar{y}).
		\end{equation}	
		\item For any $\tau \in \CF$ with $\CI \tau \in \CT$ and all $y,\bar{y} \in \R^2$ we have \label{def:smooth_admissible_model:Pi4}
		\begin{equation}
			(\Pi_y\CI \tau)(\bar{y}) = (K * \Pi_y\tau)(\bar{y}) + \sum_{\abs{\ell} < \thinspace\reg{\tau} +2} \frac{(\bar{y}-y)^l}{\ell!} f_y(\CJ_\ell \tau).
			\label{def:smooth_admissible_model:Pi4:eq}
		\end{equation}	
	\end{enumerate}
	where the map $f: \R^2 \to \CG_+$, $y \mapsto f_y$, is uniquely specified from~$\Pi$ by imposing multiplicativity and 
	\begin{enumerate}[label=($\Pi$5)]
		\item \qquad $f_y(X^\ell) = (-y)^\ell$, 
		\qquad $f_y(\CJ_\ell \tau) = - (D^\ell K * \Pi_y \tau)(y) \quad \text{for} \ \abs{\ell} < \reg{\tau} +2$.
		\label{def:smooth_admissible_model:Pi5}
	\end{enumerate}
	We then define $F: \R^2 \to \CG$ from $f$ by $F_y := \Gamma_{f_y}$ according to \eqref{app:rs:str_gr} and impose the algebraic relation
	\begin{enumerate}[label=($\Pi$6)]
		\item \qquad $\Pi_y \circ F_y^{-1} = \Pi_{\bar{y}} \circ F_{\bar{y}}^{-1}$ for every $y,\bar{y} \in \R^2$. \label{def:smooth_admissible_model:Pi6}
	\end{enumerate} 
	The pair $\bz := (\Pi,F)$ is then called a \emph{smooth admissible model}  for $\TT$, the collection of which we denote by~$\MM_{\8}$.
\end{definition}

\begin{remark}
	A priori, one would have to account for time-dependence of~\eqref{gpam} in the models. However, given that SWN~$\xi$ does not depend on time, one would immediately find that
		\begin{equation*}
			\Pi_z \tau(\bar{z}), \quad z = (t,x), \quad \bar{z} = (\bar{t},\bar{x}) \in \R^3
		\end{equation*}
	is independent both of~$t$ and~$\bar{t}$, cf.~\cite[sec.~$10.4$]{hairer_rs}; this fact is reflected in our definitions.
	
	In particular, it has informed our choice of the integral kernel~$K$ in~\eqref{app:sec:adm_models:def_K}: While initially, we need to take a suitable \enquote{cut-off} version of the heat kernel~$P$ that complies with \cite[Lem.~5.5]{hairer_rs}, we can actually integrate out the temporal variable and then arrive at the Green's function~$K$ of the Laplacian.
\end{remark}

\begin{remark}
	Notice that the map $\Pi$ uniquely specifies $f$ (and a fortiori $F$ through \eqref{app:rs:str_gr} and $\Gamma$ by $\Gamma_{y\bar{y}} = F_y^{-1} \circ F_{\bar{y}}$) via \ref{def:smooth_admissible_model:Pi5}. Hence, we will interchangeably write and speak of an admissible model
	\begin{equation*}
		\Pi \longleftrightarrow \bz= (\Pi,f) \longleftrightarrow \bz=(\Pi,F) \longleftrightarrow \bz=(\Pi,\Gamma).
	\end{equation*}
\end{remark}

In addition, we consider a set of test functions
\begin{equation}
	\CB \equiv \CB_2 := \{\eta \in \CC^2(\R^2): \ \norm{\eta}_{\CC^2} \leq 1, \ \supp \eta \subseteq B(0,1)\},
	\label{def:minimal_admissible_model:spatial_test_functions}
\end{equation}
whose rescaled versions are given by $\eta_{y}^{\l}(\bar{y}) = \l^{-2}\eta(\l^{-1}(\bar{y} - y))$ for $\l \in (0,1]$ and $y,\bar{y} \in \R^2$. Since we are studying \eqref{gpam} with periodic boundary conditions, it is then natural to enforce periodicity: 

\begin{definition}[Periodicity]\label{def:periodic_models}
	Let $\bz = (\Pi,F) \in \MM_\infty$. We denote by $\{\se_1, \se_2\}$ the standard basis of $\R^2$ and introduce the \emph{translation maps} 
	\begin{equation}
		T_i: \R^2 \to \R^2, \quad T_i(y) =  y + \se_i \qquad i = 1,2.
		\label{def:periodic_models:translation}
	\end{equation}	
	These maps $T_i$ act on test functions $\phi$ by $T_i\phi := \phi \circ T_i^{-1}$. The model $\bz = (\Pi,F)$ is said to be \emph{periodic} if
	\begin{equation}
		(\Pi_{T_i y} \stau)(T_i \phi) = (\Pi_y \stau)(\phi), \qquad
		F_{T_i y} = F_y 
		\label{def:periodic_models:eq}
	\end{equation}
	holds for every $y \in \R^2$, every test function $\phi$, every $\tau \in \CT$ and every $i = 1,2$. 
	The set of periodic admissible smooth models is denoted by $\MM_{\infty,p}$.
\end{definition}

\begin{definition}[Admissible model]\label{def:admissible_model}
	Let $\bz = (\Pi, F), \bbz = (\bar{\Pi},\bar{F}) \in \MM_{\infty,p}$, and $\fK \subseteq \R^2$ a compact domain. We set
	\begin{equs}[][def:admissible_model:pseudometric]
		\threebars \Pi - \bar{\Pi} \threebars_{\fK}
		:=
		\sup_{\l,\phi, \tau, z} \l^{-\thinspace\reg{\tau}} \abs[1]{(\Pi_z\tau - \bar{\Pi}_z \tau)(\phi_z^{\l})}
	\end{equs}
	where the suprema are taken over $\l \in (0,1]$, $\phi \in \CB$, $\tau \in \CF$, and $z \in \fK$.
	The set $(\MM_\infty,\threebars \cdot \threebars)$ then forms a pseudometric space, from which we define the (separable) space
	\begin{equation}
		\MM := \cl_{\threebars \cdot \threebars} \thinspace \MM_{\8,p}.
		\label{def:admissible_model:closure}
	\end{equation}
	The elements of $\MM$ are called (periodic) \emph{admissible models} for $\TT$. Another metric on $\MM$, in the literature (e.g. \cite[Sec.~A.3]{friz_gassiat_pigato}) by abuse of terminology referred to as the \emph{homogeneous} model norm~$\barnorm{\cdot}$, is given by\footnote{The reader might wonder why the metric~$\barnorm{\cdot\thinspace; \cdot}$ does not involve~$\Gamma$. This fact has extensively been elaborated on in the literature: We refer to~\cite[Rmk.~$3.5$]{cfg} for further details and references.}
	\begin{equation}
		\barnorm{\Pi - \bar{\Pi}}_{\fK}
		:=
		\sup_{\l,\phi, \tau, z} \del[2]{\l^{-\thinspace\reg{\tau}} \abs[1]{(\Pi_z\tau - \bar{\Pi}_z \tau)(\phi_z^{\l})}}^{\frac{1}{[\tau]}}
		\label{def:admissible_model:pseudometric_hom}
	\end{equation}
	where the supremum runs over the same set as in~\eqref{def:admissible_model:pseudometric}. In addition, $[\tau] := \#[\<wn> \ \text{in} \ \tau] \vee 1$ denotes the number of noises $\<wn>$ in the symbol $\tau$, e.g.~$[\<11>] = 2$. 
	Analogously to~$\MM$, we define the space~$\gr{\MM}$ of models for~$\gr{\TT}$.
\end{definition}

\begin{remark}
	Periodicity implies that we can always work on the spatial domain~$\fK = \T^2$, so we will omit the subscript~$\fK$ in~$\threebars \cdot \threebars$ and~$\barnorm{\cdot}$.
\end{remark}

\subsubsection*{$\bullet$ Minimal admissible models.}

We want to characterise a model given as minimal information as possible. For that purpose, we fix 
	\begin{itemize}
		\item $\CW_- := \{\<wn>,\<11>\} \subseteq \CW$ and~$\CT_- := \scal{\CW_-} \subseteq \CT$,
		\item $\gr{\CW_-} := \{\<wn>,\<cm>,\<11>,\<1g1>,\<11g>,\<1g1g>\} \subseteq \gr{\CW}$ and~$\gr{\CT_-} := \scal{\gr{\CW_-}} \subseteq \gr{\CT}$, 
	\end{itemize}
i.e. the symbols of negative homogeneity without factors of $\symbol{X_i}$ (in~$\CF$ or $\gr{\CF}$, respectively) and the linear spaces they span.

\begin{definition}[Minimal admissible model] \label{def:minimal_admissible_model}
	Let $\Pi^{\minus}: \R^2 \to \CL(\CT_-,\CC^{\8}(\R^2))$ be a smooth periodic map such that $\norm[0]{\Pi^{\minus}} < \8$
	for 
	\begin{equs}[][def:minimal_admissible_model:norm]
		\threebars \Pi^{\minus}\threebars
		& := 
		\max_{\tau \in \CW_-} \thinspace \threebars\Pi^{\minus}\threebars_\tau, \quad 
		\barnorm{\Pi^{\minus}}
		:= 
		\max_{\tau \in \CW_-} \thinspace \barnorm{\Pi^{\minus}}_\tau
	\end{equs}
	with 
	\begin{equation}
		\threebars \Pi^{\minus}\threebars_\tau := \sup_{\l, \phi, z} \l^{-\thinspace\reg{\tau}} \abs[1]{\scal{\Pi^{\minus}_z \tau,\phi_z^{\l}}}, \quad
		\barnorm{\Pi^{\minus}}_\tau :=
		\sup_{\l, \phi, z} \del[2]{\l^{-\thinspace\reg{\tau}} \abs[1]{\scal{\Pi^{\minus}_z \tau,\phi_z^{\l}}}}^{\frac{1}{[\tau]}},
		\label{def:minimal_admissible_model:norm:suprema}
	\end{equation}
	where
	\begin{itemize}
		\item $\scal{\cdot,\cdot}$ denotes the inner product in $L^2$ and
		\item the suprema are taken over $\l \in (0,1]$, $\phi \in \CB$, $z \in D$ as in~\eqref{def:admissible_model:pseudometric}.
	\end{itemize}
	In addition, we impose the constraints~\ref{def:smooth_admissible_model:Pi3} to \ref{def:smooth_admissible_model:Pi6} from~\thref{def:smooth_admissible_model}. The closure of the set of all such maps $\Pi^{\minus}$ under $\threebars \cdot \threebars$ is denoted by ~$\MM_-$.  Its elements are called \emph{minimal admissible models}. 
\end{definition}
The consistency of the preceding definition is checked in~\cite[Sec.~3.3]{cfg} and \cite[Sec.~2.4]{hairer_weber_ldp}. We emphasise that changing~$\CW_- \rightsquigarrow \gr{\CW_-}$ leads to the space~$\gr{\MM_-}$ of minimal admissible models w.r.t.~$\gr{\TT}$.

The following lemma is straightforward but helpful.

\begin{lemma}\label{lem:estimate_min_norm_vs_hom_norm}
	Convergence w.r.t. $\threebars \cdot \threebars$ and $\barnorm{\cdot}$ in $\MM_-$ are equivalent.
\end{lemma}

\begin{proof}
	This follows immediately from the continuity of the maps~$x \mapsto x^k$ and~$x \mapsto x^{\nicefrac{1}{k}}$, $k \in \N$, at the origin.
\end{proof}

The introduction of minimal admissible model is justified by the following theorem, see~\cite[Thm.~3.9]{cfg} in case of~\eqref{gpam};~the general case is considered in~\cite[Thm.~$2.10$]{hairer_weber_ldp}.

\begin{theorem}[Extension theorem]\label{app:thm:extension}
	For every $\Pi^{\minus} \in \MM_{-}$, there exists a unique admissible model $\bz=(\Pi,F) \in \MM$ such that $\Pi_z\stau = \Pi^{\minus}_z\tau$ for all $\tau \in \CW_-$ and $z \in \R^2$. In addition, the extension map 
	\begin{equation}
		\EE: \MM_- \to \MM, \ \Pi^{\minus} \mapsto (\Pi,F)
		\label{app:thm:extension:eq}
	\end{equation}
	is locally Lipschitz continuous on $(\MM_-, \threebars \cdot \threebars)$\footnote{Note that $\EE$ is \emph{not} locally Lipschitz continuous on~$(\MM_-,\barnorm{\cdot})$ simply because the square root is not Lipschitz. However, $\EE$ is still continuous w.r.t $\barnorm{\cdot}$ by~\thref{lem:estimate_min_norm_vs_hom_norm}.}	
	and bijective.
\end{theorem}

\begin{notation}
	Given a model $\bz = (\Pi,f) \in \MM$, by the previous theorem there always exists a $\Pi^{\minus} \in \MM_-$ (simply given by restriction of $\Pi$ to $\CT_-$) such that $\bz = \EE(\Pi^{\minus})$. We will also write~$\bz^{\minus} := \Pi^{\minus}$ and sometimes also write~$\bz^{\minus}$ for a generic element in~$\MM_-$.
	
	With a superscript~\enquote{$-$}-sign on an operation we signify that it acts on~$\MM_-$ (rather than~$\MM$), that is: it is composed with~$\EE$. In this article, this convention for example concerns the solution map~$\Phi^{\minus} := \Phi \circ \EE$, the functional~$F_\Phi^{\minus} := F \circ \Phi^{\minus}$, the explosion time~$T_\infty^{\minus} := T_\infty \circ \EE$ to be introduced in section~\ref{sec:explosion} below, etc. 
\end{notation}

Finally, we present how to lift smooth functions~$\zeta$ to~$\MM_-$ and then to~$\MM$. The example the reader should have in mind is~$\zeta = \xi_\d$, a realisation of smoothened SWN~$\xi$.

\begin{definition}[Canonical lift] \label{def:canonical_lift}
	Let~$\zeta \in \C^\infty(\T^2)$. We define~$\LL_-(\zeta) := \Pi^\zeta \in \MM_-$ by 
		\begin{equation}
			\Pi^{\zeta}_y \<wn>(\bar{y}) := \zeta, \quad \Pi_y^\zeta \<11>(\bar{y}) := \Pi_y^\zeta \<1>(\bar{y})\Pi_y^\zeta \<wn>(\bar{y})
			\label{def:canonical_lift:eq}
		\end{equation}
	where~$\Pi_y^\zeta \<1>(\bar{y})$ is defined by~\ref{def:smooth_admissible_model:Pi4} and~\ref{def:smooth_admissible_model:Pi5} above.
	In addition, we impose~\ref{def:smooth_admissible_model:Pi2} and~\ref{def:smooth_admissible_model:Pi3} and then define~$\LL(\zeta) = (\Pi^\zeta,F^\zeta) := \EE(\LL_-(\zeta))$.
\end{definition}
The fact that~$\LL(\zeta) \in \MM_\infty$ is commonplace in the literature by now; we refer to Hairer~\cite[Prop.~$8.27$]{hairer_rs} for details.

\subsubsection{Analytical operations for admissible models} \label{sec:models_analytic_op}

In this section, we introduce analytical operations on~$\MM$.
Let us emphasise that all these operations have obvious counterparts on~$\MM_-$: actually, given~\thref{app:thm:extension}, we could even \emph{define} them on~$\MM_-$ and simply \enquote{push them forward} onto~$\MM$ via~$\EE$.

We start with a recap of \emph{extension}\footnote{This is not to be confused with the extension operator~$\EE$ from~\thref{app:thm:extension}. Here, we will focus on extending models for~$\TT$ to models for~$\gr{\TT}$, that is: extend~$\bz \in \MM$ to~$\gr{\MM}$.} and \emph{translation} studied by Cannizzaro, Friz, and Gassiat~\cite{cfg}. 

\subsubsection*{$\bullet$ Extension and translation.}

On an algebraic level, we have introduced the symbol~$\<cm>$ to encode shifts of the driving noise~$\xi$ into Cameron-Martin direction~$h \in \CH$, thus extending the regularity structure~$\TT$ to~$\gr{\TT}$.
The following proposition provides an analytical counterpart on the level of models,~see~\cite[Prop.~3.10]{cfg}. 

\begin{proposition}[Extension]\label{app:prop:extension_op}
	Let $\bz = (\Pi,f) \in \MM$ and $h \in \CH$. Then, there exists a unique admissible model $\bz^{\eh} = (\Pi^{\eh},f^{\eh}) \in \gr{\MM}$ such that for all $z \in \R^3$, 
	\begin{enumerate}[label=(\arabic*)]
		\item $\Pi^{\eh}\sVert_{\CT} = \Pi$, $f^{\eh}\sVert_{\CT_+} = f$, $\Pi^{\eh}_z \<cm> \equiv h$, and
		\item $\Pi^{\eh}_z \tau = (\Pi^{\eh}_z \tau_1)(\Pi^{\eh}_z \tau_2)$ for all $\tau = \tau_1 \tau_2 \in \gr{\CW} \setminus \CW$ with $\tau_1 \in \gr{\CU}$ and $\tau_2 \in \{\<wn>,\<cm>\}$. 
		\item the conditions~\ref{def:smooth_admissible_model:Pi3} to \ref{def:smooth_admissible_model:Pi6} hold.
	\end{enumerate}
	In addition, the \emph{extension map} $E$ given by
	\begin{equation}
		E: \CH \x \MM \to \gr{\MM}, \ (h,\bz) \mapsto E_h \bz := \bz^{\eh}, 
	\end{equation}
	is jointly locally Lipschitz continuous.
\end{proposition}

Eventually, we want to build a \emph{translated} model $T_h \bz \in \MM$ for the original regularity structure~$\TT$, so we introduce abstract linear translation operations $\ft_{\<cms>}: \CT \to \CT^{\<cms>}$ and $\ft_{\<cms>}^+: \CT_+ \to \CT^{\<cms>}_+$ by
\begin{equs}
	\ft_{\<cms>} \<wn> := \<wn> + \<cm>, \quad 
	\ft_{\<cms>} X^k & := X^k, \quad
	\ft_{\<cms>} (\tau \sigma) := (\ft_{\<cms>} \tau) (\ft_{\<cms>} \sigma), \quad
	\ft_{\<cms>} (\CI \tau) := \CI(\ft_{\<cms>} \tau), \\
	\ft_{\<cms>}^+ X^k & := X^k, \quad
	\ft_{\<cms>}^+ (\tau \sigma) := (\ft^+_{\<cms>} \tau) (\ft^+_{\<cms>} \sigma), \quad
	\ft_{\<cms>}^+ (\CJ_\ell \tau) := \CJ_\ell(\ft_{\<cms>} \tau)
\end{equs}
where the very last relations holds for all~$\tau \in \CT$ such that $\CI\tau \in \CT$ and $\reg{\tau} + 2 - \abs{\ell} > 0$. We then have the following result, see~\cite[Prop.~3.12]{cfg}:

\begin{proposition}[Translation] \label{app:prop:translation}
	Let $\bz = (\Pi,f) \in \MM$ and $h \in \CH$. Then, $\bz^{h} = (\Pi^h,f^h)$ given by $\Pi^h := \Pi^{\eh} \circ \ft_{\<cms>}$ and $ f^h := f^{\eh} \circ \ft_{\<cms>}^+$ is an element of $\MM$, i.e. an admissible model for $\TT$. In addition, the \emph{translation map} $T$ given by
	\begin{equation}
		T: \CH \x \MM \to \MM, \ (h,\bz) \mapsto T_h \bz := \bz^h, 
	\end{equation}
	is jointly locally Lipschitz continuous.
\end{proposition} 

\begin{remark} \label{rmk:lift_h}
	From the previous proposition, we also see that~$\LL(h)$ is well-defined for~$h \in \CH$. Indeed, because~$T_h \boldsymbol{0}$ for~$\boldsymbol{0} := \LL(0)$ is a model, we know that the analytical and algebraic relations for constructing~$\LL(h)$ by~\thref{def:canonical_lift} are true. 
	In fact, one simply has~$\LL(h) = T_h \boldsymbol{0}$, where the RHS may even serve as a \emph{definition} of the LHS. In the most general setting, this relation also holds, see the first equation of~p.~$28$ in the work of Sch\"onbauer~\cite{schoenbauer}.
	
	In addition, it is easy to see that for~$h, k \in \CH$, we have~$T_h \LL(k) = \LL(k + h)$.
\end{remark}

The following lemma will be crucial in sec.~\ref{app:fernique} below when we investigate Gaussian concentration on~$(\MM_-,\barnorm{\cdot})$.

\begin{lemma}\label{lem:ineq_translation_barnorm}
	Let $h \in \CH$ and $\bz^- \in \MM_-$.	Then, we have the inequality 
	\begin{equation*}
		\barnorm{T_h \bz^-} \aac \barnorm{\bz^-} + \norm{h}_\CH.
	\end{equation*}
\end{lemma}

\begin{remark}
	Note that the inequality in the previous lemma is uniform in~$h \in \CH$. This is to be compared to~\thref{app:prop:translation}, where joint Lipschitz continuity w.r.t.~$\threebars \cdot \threebars$ holds \emph{locally}.
\end{remark}

\begin{proof}
	Recall that $[\tau] = \#[\<wn> \ \text{in} \ \tau]$ and set~$[\tau]_{\<cms>} := \#[\<cm> \ \text{in} \ \tau]$. By the identities 
		\begin{equation}
		\CW_- = \{\<wn>,\<11>\}, \quad
		\gr{\CW_-} = \{\<cm>,\<wn>,\<11>,\<1g1>,\<11g>,\<1g1g>\},  \quad
		\ft_{\<cms>} \<wn> = \<wn> + \<cm>, \quad
		\ft_{\<cms>} \<11> = \<11> + \<1g1> + \<11g> + \<1g1g>,
		\end{equation}
	we trivially have 
	\begin{equs}[][lem:ineq_translation_barnorm:pf:eq1]
		\barnorm{\bz^{h,\minus}} 
		&
		= 
		\barnorm{\bz^{\eh,\minus} \circ \ft_{\<cms>}}
		= 
		\sup_{\tau \in \CW_-} \sup_{\l, \phi, y}  \del[2]{\l^{-\thinspace\reg{\tau}} \abs[1]{\scal{(\Pi^{\eh}_y \circ \ft_{\<cms>}) \tau,\phi_y^{\l}}}}^{\frac{1}{[\tau]}} \\
		&
		\leq  
		\sup_{\l, \phi, y}  \l^{-\thinspace\reg{\<wns>}} \abs[1]{\scal{(\Pi^{\eh}_y \circ \ft_{\<cms>}) \<wn>,\phi_y^{\l}}}
		+ 
		\sup_{\l, \phi, y}  \del[2]{\l^{-\thinspace\reg{\<11s>}} \abs[1]{\scal{(\Pi^{\eh}_y \circ \ft_{\<cms>}) \<11>,\phi_y^{\l}}}}^{\frac{1}{2}} \\
		&
		\aac
		\sup_{\tau \in \gr{\CW_-}} \sup_{\l,\phi,y} \del[2]{\l^{-\thinspace\reg{\tau}} \abs[1]{\scal{\Pi^{\eh}_y \tau,\phi_y^{\l}}}}^{\frac{1}{\llbracket \tau \rrbracket}}
		=
		\sup_{\tau \in \gr{\CW_-}} \threebars \Pi^{\eh} \threebars_\tau^{\frac{1}{\llbracket \tau \rrbracket}},
	\end{equs}
	where $\llbracket \tau \rrbracket$ denotes the total number of \enquote{noises} $\<wn>$ and $\<cm>$ in the symbol~$\tau$, that is: $\llbracket \tau \rrbracket := [\tau] + [\tau]_{\<cms>}$ and the suprema over $\l$, $\phi$, and $y$ are taken as in~\eqref{def:minimal_admissible_model:norm:suprema}. 
	
	Changing~$\gr{\CW_-} \rightsquigarrow \CW_-$ in the last term of~\eqref{lem:ineq_translation_barnorm:pf:eq1} simply gives~$\barnorm{\bz^{\minus}}$.
	Upon carefully inspecting the proof of~\cite[Lem.~A.2]{cfg}, one finds that the authors actually show the following bound for~$\tau \in \gr{\CW_-} \setminus \CW_-$:
	\begin{equation*}
		\threebars \Pi^{\eh} \threebars_\tau
		\aac
		\norm{h}_\CH^{[\tau]_{\<cms>}} \threebars\Pi^{\minus}\threebars_{\<wns>}^{[\tau]_{\<wns>}}
		=
		\norm{h}_\CH^{[\tau]_{\<cms>}} \barnorm{\Pi^{\minus}}_{\<wns>}^{[\tau]_{\<wns>}}
	\end{equation*}
	In particular, using the trivial bound $\sqrt{ab} \leq a+b$ for $a,b \geq 0$ in case~$\tau \in \{\<1g1>,\<11g>\}$, this implies the estimate
	\begin{equation*}
		\sup_{\tau \in \gr{\CW_-} \setminus \CW_-} \threebars\Pi^{\eh}\threebars_\tau^{\frac{1}{\llbracket \tau \rrbracket}}
		\aac
		\norm{h}_\CH + \barnorm{\Pi^{\minus}}_{\<wns>}
		\aac
		\norm{h}_\CH + \barnorm{\bz^{\minus}}.
	\end{equation*}
	Combining all the previous observations leads to
	\begin{equs}
		\sup_{\tau \in \gr{\CW_-}} \threebars \Pi^{\eh} \threebars_\tau^{\frac{1}{\llbracket \tau \rrbracket}}
		\leq
		\sup_{\tau \in \CW_-} \threebars\Pi^{\eh}\threebars_\tau^{\frac{1}{\llbracket \tau \rrbracket}}  
		+ 
		\sup_{\tau \in \gr{\CW_-} \setminus \CW_-} \threebars\Pi^{\eh}\threebars_\tau^{\frac{1}{\llbracket \tau \rrbracket}} 
		\aac
		\barnorm{\bz^-} + \norm{h}_\CH,
	\end{equs}
	and thereby proves the claim.
\end{proof}

\subsubsection*{$\bullet$ Dilation.}

The operator $T_h$ from~\thref{app:prop:translation} accounts for translations $\xi \rightsquigarrow \xi + h$ in Cameron-Martin directions~$h \in \CH$. Reviewing \thref{gpam}, we also need to analytically encode the operation~$\xi \rightsquigarrow \eps \xi$ which goes by the name~\emph{dilation} by $\eps \in I$. 
In analogy with translation, we first introduce some abstract linear dilation operations $\fd_{\eps}: \gr{\CT} \to \gr{\CT}$ and $\fd_{\eps}^+: \gr{\CT_+} \to \gr{\CT_+}$ by
\begin{equs}[][def:frac_dil_symbols]
	\fd_{\eps} \<wn> := \eps\<wn>, \quad 
	\fd_{\eps} \<cm> := \<cm>, \quad
	\fd_{\eps} X^k & := X^k, \quad
	\fd_{\eps} (\tau \sigma) := (\fd_{\eps} \tau) (\fd_{\eps} \sigma), \quad
	\fd_{\eps} (\CI \tau) := \CI(\fd_{\eps} \tau), \\
	\fd_{\eps}^+ X^k & := X^k, \quad
	\fd_{\eps}^+ (\tau \sigma) := (\fd_{\eps}^+ \tau) (\fd_{\eps}^+ \sigma), \quad
	\fd_{\eps}^+ (\CJ_\ell \tau) := \CJ_\ell(\fd_{\eps} \tau)
\end{equs}
where the very last relations holds for all~$\tau \in \gr{\CT}$ such that $\CI\tau \in \gr{\CT}$ and $\reg{\tau} + 2 - \abs{\ell} > 0$.

\begin{proposition}[Dilation] \label{app:def:dilation}
	Let $\gr{\bz} = (\gr{\Pi},\gr{f}) \in \gr{\MM}$ be an admissible model. Then, $\d_\eps \gr{\bz} = (\gr{\Pi}^\eps,\gr{f}^\eps)$ given by $\gr{\Pi}^\eps := \gr{\Pi} \circ \fd_{\eps}$ and $\gr{f}^\eps := \gr{f} \circ \fd_{\eps}^+$ is an element of $\gr{\MM}$. In addition, the \emph{dilation map} $\d$ given by
	\begin{equation}
		\d: \R \x \gr{\MM} \to \gr{\MM}, \ (\eps,\gr{\bz}) \mapsto \d_\eps \gr{\bz}, 
	\end{equation}
	is jointly continuous and $\gr{\Pi}^\eps \tau = \eps^{[\tau]} \gr{\Pi} \tau$ for any~$\tau \in \gr{\CT}$.
	For~$h \in \CH$ and the canonical lift~$\LL$ introduced in~\thref{def:canonical_lift} (cf. also~\thref{rmk:lift_h}), we have~$\d_\eps\LL(h) = \LL(\eps h)$. 
\end{proposition}

\begin{remark} \label{rmk:dilation_min_models}
	Observe that the dilation operator~$\d$ also maps~$\MM$ resp.~$\MM_-$ onto themselves.
\end{remark}

In order to prove that~$\d$ is well-defined and thus the preceding proposition, we need the following lemma which investigates how the structure group transforms unter dilations.

\begin{lemma} \label{lem:dil_Gamma}
	In the setting of~\thref{app:def:dilation}, consider the operators $\gr{\Gamma} := \Gamma_{\gr{f}} \in \gr{\CG}$ and $\gr{\Gamma}^\eps := \Gamma_{\gr{f}^\eps} \in \gr{\CG}$ built via~\eqref{app:rs:str_gr}. Then, $\fd_\eps \circ \gr{\Gamma}^\eps = \gr{\Gamma} \circ \fd_\eps$.
\end{lemma}

\begin{proof}
	Observe that 
	\begin{equation*}
		\fd_\eps \circ \gr{\Gamma}^\eps 
		= \fd_\eps (\operatorname{id} \tp \gr{f}^\eps) \gr{\Delta}
		= (\fd_\eps \tp \gr{f}) (\operatorname{id} \tp \fd_\eps^+) \gr{\Delta}
		= (\operatorname{id} \tp \gr{f}) (\fd_\eps \tp \fd_\eps^+) \gr{\Delta}
	\end{equation*}	
	and
	\begin{equation*}
		\gr{\Gamma} \circ \fd_\eps
		=
		(\operatorname{id} \tp \gr{f}) (\gr{\Delta} \circ \fd_\eps),
	\end{equation*}
	so the proof is finished if we can establish that 
	\begin{equation}
		(\fd_\eps \tp \fd_\eps^+) \gr{\Delta} = \gr{\Delta} \circ \fd_\eps.
		\label{pf:lem:dil_Gamma:eq1}
	\end{equation}
	This is true for $\tau \in \{\<wn>,\<cm>,\1\}$ since $\gr{\Delta} \tau = \tau \tp \1$ is true for those symbols. We now proceed recursively. Suppose \eqref{pf:lem:dil_Gamma:eq1} is true for $\tau \in \gr{\CT}$ with $\CI \tau \in \gr{\CT}$. Then,
	\begin{equs}
		\gr{\Delta} (\fd_\eps \CI \tau) 
		& = 
		\gr{\Delta} (\CI \fd_\eps \tau)
		=
		(\CI \tp \operatorname{id}) \gr{\Delta} (\fd_\eps \tau) + \sum_{k + \ell \in B_\tau} \frac{1}{k! \ell!} X^k \tp X^\ell\CJ_{k+\ell}(\fd_\eps \tau) \\
		&
		=
		(\CI \tp \operatorname{id}) (\fd_\eps \tp \fd_\eps^+) \gr{\Delta} \tau + \sum_{k + \ell \in B_\tau} \frac{1}{k! \ell!} X^k \tp X^\ell\fd_\eps^+\del[1]{\CJ_{k+\ell}(\tau)} \\
		&
		=
		(\fd_\eps \tp \fd_\eps^+) \del[3]{(\CI \tp \operatorname{id}) \gr{\Delta} \tau + \sum_{k + \ell \in B_\tau} \frac{1}{k! \ell!} X^k \tp X^\ell\CJ_{k+\ell}(\tau)}
		=
		(\fd_\eps \tp \fd_\eps^+) \gr{\Delta} (\CI \tau).
	\end{equs}
	Similarly, suppose \eqref{pf:lem:dil_Gamma:eq1} is true for $\tau,\sigma \in \gr{\CT}$ such that $\tau \sigma \in \gr{\CT}$. We then find that
	\begin{equs}
		\gr{\Delta}\del[0]{\fd_\eps(\tau \sigma)}
		& =
		\gr{\Delta}(\fd_\eps\tau \fd_\eps\sigma)
		=
		\gr{\Delta}(\fd_\eps\tau) \gr{\Delta}(\fd_\eps\sigma)
		=
		(\fd_\eps \tp \fd_\eps^+)(\gr{\Delta} \tau) (\fd_\eps \tp \fd_\eps^+)(\gr{\Delta} \sigma) \\
		&
		=
		(\fd_\eps \tp \fd_\eps^+) (\gr{\Delta} \tau \gr{\Delta} \sigma)
		=
		(\fd_\eps \tp \fd_\eps^+) \del[1]{\gr{\Delta} (\tau\sigma)}.
	\end{equs}
	Thus, \eqref{pf:lem:dil_Gamma:eq1} is true for all $\tau \in \gr{\CT}$ and the proof is finished.
\end{proof}

\begin{proof}[of~\thref{app:def:dilation}]
	As dilation does not change the homogeneity of symbols, it is clear that the analytical constraint~\ref{def:smooth_admissible_model:Pi1} is true for~$\d_\eps \gr{\bz}$. 
	Since~$\fd_\eps$ is linear, multiplicative, and commutes with~$\CI$, so are the conditions~\ref{def:smooth_admissible_model:Pi2} to~\ref{def:smooth_admissible_model:Pi5}.
	The algebraic condition~\ref{def:smooth_admissible_model:Pi6} holds as well: For~$y,\by \in \T^2$ and~$\tau \in \gr{\CT}$, by~\thref{lem:dil_Gamma} we have
		\begin{equation*}
			\gr{\Pi}^\eps_y \gr{\Gamma}^{\eps}_{y\by} \tau
			=
			(\gr{\Pi}_y \circ \fd_\eps \circ \fd_{\eps^{-1}} \circ \gr{\Gamma}_{y\by} \circ \fd_\eps) \tau
			=
			(\gr{\Pi}_y \circ \gr{\Gamma}_{y\by} \circ \fd_\eps) \tau
			=
			(\gr{\Pi}_{\by} \circ \fd_\eps) \tau
			=
			\gr{\Pi}^\eps_{\by} \tau.
		\end{equation*}
	The other claims are immediately seen to be true.
\end{proof}

As is easily seen, the operations of extension and dilation commute.

\begin{lemma} \label{lem:ext_dil_commute}
	Let $h \in \CH$ and $\eps > 0$. Then, 
	$E_h \circ \d_\eps = \d_\eps \circ E_h$.
\end{lemma}

\begin{proof}
	For $\tau \in \gr{\CT}$ and $\sigma \in \gr{\CT_+}$, we have $\fd_\eps \tau = \eps^{[\tau]} \tau$ and $\fd_\eps^+ \sigma = \eps^{[\sigma]} \sigma$.
	Combined with the multiplicativity of $\fd_\eps$ and $\fd_\eps^+$ as well as the identities $\fd_\eps \circ \CI = \CI \circ \fd_\eps$ and $\fd_\eps^+ \circ \CJ_\ell = \CJ_\ell \circ \fd_\eps$, the assertion is a straightforward consequence of the definitions of $E_h$ and $\d_\eps$.
\end{proof}

On the other hand, the operations of translation and dilation do \emph{not} commute, but instead satisfy a relationship that is easy to state.

\begin{lemma} \label{lem:rel_transl_dil}
	For each $h \in \CH$ and $\eps > 0$, we have $T_h \circ \d_\eps = \d_\eps \circ T_{\nicefrac{h}{\eps}}$.
\end{lemma}

\begin{proof}
	The assertion follows by $\fd_\eps \circ \ft_{\<cms>} = \ft_{\nicefrac{\<cms>}{\eps}} \circ \fd_\eps$ and $\fd_\eps^+ \circ \ft_{\<cms>}^+ = \ft_{\nicefrac{\<cms>}{\eps}}^+ \circ \fd_\eps^+$, both of which are straightforward consequences of the definitions of $\fd_\eps$, $\fd_\eps^+$, $\ft_{\<cms>}$, and $\ft_{\<cms>}^+$.
\end{proof}

The following lemma allows us to normalise models to have norm of order one.

\begin{lemma}\label{lem:rescaling_lambda}
	Let $h \in \CH$,~$\bz \in \MM$, and~$\lambda = \lambda(\bz) := 1 + \barnorm{\bz^{\minus}}$. Then, the estimate
	\begin{equation}
		\threebars E_h \d_{\nicefrac{1}{\l}} \bz \threebars \aac 1
		\label{lem:rescaling_lambda:eq}
	\end{equation}
	holds with an implicit constant that depends on~$\norm[0]{h}_{\CH}$.
\end{lemma}

\begin{proof}
	By the triangle inequality, we have
		\begin{equation}
			\threebars E_h \d_{\nicefrac{1}{\l}} \bz\threebars
			=
			\threebars \EE\del[0]{E_h \d_{\nicefrac{1}{\l}} \bz^{\minus}}\threebars
			\leq
			\threebars\EE\del[0]{E_h \d_{\nicefrac{1}{\l}} \bz^{\minus}} - \EE\del[0]{E_h \d_{\nicefrac{1}{\l}} \boldsymbol{0}} \threebars
			+
			\threebars\EE\del[0]{E_h \d_{\nicefrac{1}{\l}} \boldsymbol{0}}\threebars
			\label{lem:rescaling_lambda:pf_aux_1}
		\end{equation}
	where~$\boldsymbol{0} \in \MM_-$ is the lift of $0 \in \C^{\infty}(\T^2)$ and~$\EE$ the extension map from~\thref{app:thm:extension}. 
	By (joint) local Lipschitz continuity of~$\EE$ and~$E$, we estimate the first summand by
		\begin{equation*}
			\threebars \EE\del[0]{E_h \d_{\nicefrac{1}{\l}} \bz^{\minus}} - \EE\del[0]{E_h \d_{\nicefrac{1}{\l}} \boldsymbol{0}} \threebars
			\aac
			\threebars E_h \d_{\nicefrac{1}{\l}} \bz^{\minus} - E_h \d_{\nicefrac{1}{\l}} \boldsymbol{0} \threebars
			\aac
			\threebars \d_{\nicefrac{1}{\l}} \bz^{\minus} - \d_{\nicefrac{1}{\l}} \boldsymbol{0} \threebars
			=
			\threebars\d_{\nicefrac{1}{\l}} \bz^{\minus} \threebars.
		\end{equation*}
	By the definitions in~\thref{def:minimal_admissible_model:norm:suprema}, the RHS can further be estimated by
		\begin{equation*}
			\threebars \d_{\nicefrac{1}{\l}} \bz^{\minus}\threebars
			=
			\frac{1}{\l} \threebars\Pi^{\minus}\threebars_{\<wns>} + \frac{1}{\l^2} \barnorm{\Pi^{\minus}}^2_{\<11s>}
			\leq 
			1
		\end{equation*}
		using the fact that $\fd_{\nicefrac{1}{\l}} \tau = \l^{-[\tau]} \tau$ and the definition of~$\l$.
		For the second summand in~\eqref{lem:rescaling_lambda:pf_aux_1}, we have 
		\begin{equation*}
			\threebars\EE\del[0]{E_h \d_{\nicefrac{1}{\l}} \boldsymbol{0}}\threebars
			=
			\threebars \Pi^{0,\eh}\threebars
			=
			\threebars\Pi^{h}\threebars
		\end{equation*}
		where~$\Pi^h$ denotes the lift of~$h \in \CH$. The RHS can be bounded by a constant (for the polynomial terms) and (powers of)~$\norm[0]{h}$ by the estimates in~\cite[Lem.~A.$2$]{cfg}.
\end{proof}

As a corollary, this allows us to control the extended and appropriately normalised operators~$\Gamma \in \CG$.

\begin{lemma} \label{app:lemma:deter_bound_Gamma}
	Let~$h \in \CH$, $\bz = (\Pi,\Gamma) \in \MM$, and set~$\l := 1+\barnorm{\bz^{\minus}}$. We have the bound 
	\begin{equation*}
	\norm[0]{\Gamma^{\eh,\nicefrac{1}{\lambda}}}
	:=
	\sup_{\tau \in \gr{\CT}} \sup_{\b < \text{deg}(\tau)} \sup_{z \neq \bar{z}} 
	\frac{\norm[0]{\Gamma_{z\bar{z}}^{\eh,\nicefrac{1}{\lambda}}\tau}_{\b}}{\abs[0]{z - \bar{z}}^{\text{deg}(\tau) - \beta}}
	\aac_{\norm[0]{h}}
	1.
	\end{equation*}
\end{lemma}

\begin{proof}
	In the case of gPAM, this can be proved directly, see~\cite[Rmk.~$3.5$]{cfg} and its proof in their appendix~B. 
	However, as pointed out in~\cite[Rmk.~$2.4$]{hairer_weber_ldp} and~\cite[Rmk.~$3.5$]{hairer_pardoux}, it is also a direct consequence of~\cite[Thm.~$5.14$]{hairer_rs} in the general case.
\end{proof}

The merit of the following lemma lies in the fact that it explains why the homogeneous norm~$\barnorm{\cdot}$ from~eq.'s~\eqref{def:minimal_admissible_model:norm} and~\eqref{def:minimal_admissible_model:norm:suprema} is called~\emph{homogeneous} in the first place.
It also explains the very reason for working with the~\emph{minimal} rather than the full model: The latter includes the polynomials and therefore is \emph{not homogeneous} (w.r.t. dilation) under the homogeneous norm.

\begin{lemma}~\label{lem:expl_hom}
	For~$\eps \geq 0$ and~$\bz^{\minus} \in \MM_-$, we have~$\barnorm{\d_\eps\bz^{\minus}} = \eps \barnorm{\bz^{\minus}}$.
\end{lemma}

\begin{proof}
	The claim follows directly from the definition of~$\barnorm{\cdot}$ and the fact that~$\Pi^{\eps} \tau = \eps^{[\tau]} \Pi \tau$, cf.~\thref{app:def:dilation}.
\end{proof}

\subsubsection{Renormalisation} \label{sec:renormalisation}

Renormalisation in regularity structures is encoded on the level of models. 
In recent years, it has been set up in great generality in the series~\cite{bhz, chandra-hairer,rs_renorm} of seminal papers; for the relatively simple gPAM, however, it can be done \enquote{by hand}, as carried out in Hairer's foundational article~\cite{hairer_rs}.

We just review the results we need; for details, see~\cite[Sec.'s~$8.3$ and~$9.1$]{hairer_rs}. We fix~$\fc_\d \in \R$ and suitably amend~\thref{def:canonical_lift} of the canonical lift~$\LL(\xi_\d)$: For~$y, \by \in \T^2$, we define
	\begin{equation}
		\hat{\Pi}_y^{\xi_\d}\<wn>(\by) := \Pi_y^{\xi_\d}\<wn>(\by), \quad
		\hat{\Pi}_y^{\xi_\d}\<1>(\by) := \Pi_y^{\xi_\d}\<1>(\by), \quad
		\hat{\Pi}_y^{\xi_\d}\<11>(\by) 
		:= \hat{\Pi}_y^{\xi_\d}\<1>(\by) \hat{\Pi}_y^{\xi_\d}\<wn>(\by) - \fc_{\d},
		\label{app:ldp_rv_chaos_1}
	\end{equation}
impose~\ref{def:smooth_admissible_model:Pi2} to~\ref{def:smooth_admissible_model:Pi5} above, set~$\hbz^{\xi_\d,-} := \hat{\Pi}^{\xi_\d} \in \MM_-$ and then~$\hbz^{\xi_\d} := \EE\del[0]{\hbz^{\xi_\d,-}} \in \MM$.	 
We then have the following convergence result~\cite[Thm.~$10.19$]{hairer_rs}:

\begin{proposition}
	Let~$\fc_\d := \scal{K_\d,\rho_\d}$. Then, there exists a (random) model~$\hbz \in \MM$ independent of the specific choice of mollifier~$\rho$ such that
		\begin{equation*}
			\lim_{\d \to 0} \hbz^{\xi_\d} = \hbz \quad \text{in prob. in} \ \MM.
		\end{equation*} 
	We call~$\hbz$ the \emph{BPHZ model}.
\end{proposition}

\begin{remark} \label{rmk:bphz_model_ren_constant}
	For~$\hbz$ to be the BPHZ model, one defines~$\fc_{\d} := \E\sbr[1]{\Pib^{\xi_\d}\<11>(0)}$. That is consistent with our choice, as the following short calculation using the covariance structure of SWN~$\xi$ shows:
		\begin{align} 
			\E\sbr[1]{\Pib^{\xi_\d}\<11>(0)}
			&
			= \E\sbr[1]{\Pib^{\xi_\d}\<1>(0)\Pib^{\xi_\d}\<wn>(0)}
			= \E\sbr[1]{(K * \xi_\d)(0)\xi_\d(0)}
			= \E\sbr[1]{\scal{\xi,K_\d(-\cdot)}\scal{\xi,\rho_\d(-\cdot)}} \label{calculation_ren_constant} \\
			&
			= \E\sbr[1]{\scal{\xi,K_\d(-\cdot)}\scal{\xi,\rho_\d(-\cdot)}}
			= \int K_\d(-x) \rho_\d(-x) \dif x
			= \int K_\d(x) \rho_\d(x) \dif x
			= \scal{K_\d,\rho_\d}. \notag
		\end{align}	
	Furthermore, the statement of the previous proposition is also true on minimal model space: We denote the limiting model by~$\hbz^{\minus} \in \MM_-$ and call it the \emph{minimal BPHZ model}.
\end{remark}

We recall another result from Cannizzaro, Friz, and Gassiat, namely \cite[Lem.~$3.20$]{cfg}, in slightly updated terminology. 

\begin{lemma}\label{lem:trans_prob_cfg}
	Let $\hbz \in \MM$ be the BPHZ model. There exists a null set $N$ such that, for every $h \in \CH$ and $\omega \in N^{\operatorname{c}}$, we have
	\begin{equation}
		T_h \sbr[1]{\hbz(\omega)} = \hbz(\omega + h).
		\label{lem:trans_prob_cfg:eq}
	\end{equation}
\end{lemma}

The following lemma is a consequence of~\cite[Lem.~$3.20$]{cfg} and~\thref{lem:rel_transl_dil}:
\begin{lemma}\label{lem:transop_cm_shift}
	There exists a nullset $N$ such that, for all $\omega \in N^c$, $h \in \CH$, and $\eps > 0$, we have 
	\begin{equation*}
		(T_h \circ \d_\eps) \sbr[1]{\hbz(\omega)} 
		=
		\d_\eps \sbr[1]{\hbz (\omega + \nicefrac{h}{\eps})}.
		\label{lem:transop_cm_shift:eq}
	\end{equation*}
\end{lemma}

\begin{proof}
	Let $N$ be the null set from \thref{lem:trans_prob_cfg}. Now consider eq.~\eqref{lem:trans_prob_cfg:eq}, replace $h \rightsquigarrow \nicefrac{h}{\eps}$, and apply $\d_\eps$ on both sides. We then conclude by~\thref{lem:rel_transl_dil}. 
\end{proof}
Trivially, both of the previous lemmas also hold for~$\hbz^{\minus}$.

\subsection{Modelled distributions} \label{sec:modelled_distributions}

Spaces of modelled distributions play a key role in Hairer's theory: In case of the polynomial regularity structure, they are simply H\"older continuous functions and for the rough paths regularity structure, they coincide with controlled paths, see~\cite[sec.~$13.3$]{friz-hairer}.

\begin{definition}
	For $s \in [0,\infty)$, we define the \emph{hyperplane at time}~$s$ by
		\begin{equation*}
			P_s := \{z = (t,x) \in \R^3: \ t =s \}.
		\end{equation*}
	We also set~$P := P_0$.
\end{definition}

\begin{definition}[Modelled distribution] \label{def:mod_distr}
	Let $\bz = (\Pi,\Gamma) \in \MM$ be a model for the regularity structure~$\TT$.
	Fix~$\gamma > 0$,~$\eta \in \R$, and~$s \in [0,\infty)$. 
	\begin{enumerate}[label=(\roman*)]
	\item \label{def:mod_distr:i} We say that 
		\begin{equation*}
			U: (0,\infty) \x \R^2 \to \CT_{<\gamma} := \bigoplus_{\beta < \gamma} \CT_\b
		\end{equation*}
	belongs to~$\DD_s^{\gamma,\eta}(\Gamma)$ if, for every compact set~$D \subseteq (0,\infty) \x \R^2$, the quantity 
		\begin{equs}[][def:mod_distr:eq]
			\threebars U \threebars_{\gamma,\eta;\Gamma;D;P_s}
			& 
			=
			\sup_{z \in D \setminus P_s} \sup_{\beta < \gamma} \thinspace \abs{t-s}^{\frac{\beta-\eta}{2} \vee 0} \norm[0]{U(z)}_{\beta} \\ 
			& 
			+
			\sup_{\substack{z,\bar{z} \in D \setminus P_s, \\ \abs{z - \bar{z}} \leq 1}} \sup_{\beta < \gamma} \del[1]{\abs[0]{t - s} \wedge \abs[0]{\bar{t} - s}}^{\frac{\beta-\eta}{2} \vee 0} \frac{\norm[0]{U(z) - \Gamma_{z\bar{z}} U(\bar{z})}_\beta}{\abs{z - \bar{z}}^{\gamma - \beta}}
		\end{equs}
	is finite, where $z = (t,x)$ and $\bar{z} = (\bar{t},\bar{x})$. We denote the first summand on the RHS of~\eqref{def:mod_distr:eq} by~$\norm[0]{U}_{\gamma,\eta;D;P_s}$.
		\item \label{def:mod_distr:ii}  Let~$\bbz = (\bar{\Pi},\bar{\Gamma}) \in \MM$ be another model. For~$U \in \DD_s^{\gamma,\eta}(\Gamma)$ and $\bar{U} \in \DD_s^{\gamma,\eta}(\bar{\Gamma})$, we write $\threebars U;\bar{U} \threebars_{\gamma,\eta;D;P_s}$ for the quantity in~\eqref{def:mod_distr:eq} where we change $U(z) \rightsquigarrow U(z) - \bar{U}(z)$ in the first summand and 
			\begin{equation*}
				U(z) - \Gamma_{z\bar{z}} U(\bar{z})
				\rightsquigarrow
				U(z) - \bar{U}(z) - \Gamma_{z\bar{z}} U(\bar{z}) + \bar{\Gamma}_{z\bar{z}} \bar{U}(\bar{z})
			\end{equation*}  
		in the second.
	\end{enumerate}
\end{definition}

\begin{remark} \label{rmk:domain_D}
	In practice, when the model~$\bz$ is clear from the context, we refrain from indicating~$\Gamma$ in the norm in~\eqref{def:mod_distr:eq}. 
	In case~$s = 0$, we also omit the subscript~$P_s$.
	In addition, recall that we are only considering \emph{periodic} models, cf.~\thref{def:periodic_models}, complying with our assumption of periodic boundary conditions. Hence, the domain~$D$ can always be taken as $D = I \x \T^2$ for some interval~$I \subseteq \R$. We will use the following notation:
		\begin{equation*}
			\threebars \cdot \threebars_{\gamma,\eta,(s,T]} \quad \text{when } D := (s,T] \x \T^2, \quad
			\threebars \cdot \threebars_{\gamma,\eta,T} \quad \text{when } D = D_T := (0,T] \x \T^2.
		\end{equation*}
	Finally, observe that the notation~$\threebars U ; \bar{U} \threebars$ in~\ref{def:mod_distr:ii} above indicates that it is \emph{not} a function of~$U - \bar{U}$. In fact, $U$ and~$\bar{U}$ do not even live in the same space, a consequence of the \enquote{fibred} structure of~$\MM \ltimes \DD^{\gamma,\eta}$, cf.~figure~\ref{figure:model_space_fibres} on p.~\pageref{figure:model_space_fibres}. For any fixed domain~$D$, the space~$(\DD_{P_s}^{\gamma,\eta}(\Gamma),\threebars_{\gamma,\eta;D;P_s})$ is a bona fide Banach space.
\end{remark}

\subsubsection*{$\bullet$ Fr\'{e}chet differentiability and Taylor's formula for modelled distributions}

We can always write a modelled distribution~$\gr{Y} \in \DD^{\gamma,\eta}_{\gr{\CU}}(\gr{\bz})$ that takes values in the sector~$\gr{\CU} = \scal{\1,\<1>,\<1g>,X}$ as
\begin{equation}
	\gr{Y}(z) = \phi_{\tiny \1}(z) \1 + \phi_{\<1s>}(z)\<1> + \phi_{\<1gs>}(z)\<1g> + \scal{\phi_X(z),X}
	\label{eq:ansatz_U}
\end{equation}
for some coefficient functions~$\phi_{\tiny \1}$, $\phi_{\<1s>}$, $\phi_{\<1gs>}$, and $\phi_{X}$. 
Further layed out in \cite[sec.~3.4]{hairer_pardoux}, the function $g$ then induces a function $G$ that acts on such modelled distributions by
\begin{equation}
	G(\gr{Y})(z) := 
	g(\phi_{\tiny \1}(z)) \1 
	+ 
	g'(\phi_{\tiny \1}(z)) \phi_{\<1s>}(z)\<1>
	+ 
	g'(\phi_{\tiny \1}(z)) \phi_{\<1gs>}(z)\<1g>
	+ 
	g'(\phi_{\tiny \1}(z)) \scal{\phi_X(z),X}.
	\label{eq:G_ext_rs}
\end{equation}

\begin{remark}[Projection onto~$\boldsymbol{\CT_{<\gamma}}$] \label{rmk:projection}
	In following Hairer~\cite[sec.~$4.2$]{hairer_rs}, the function~$G$ should actually be denoted~$G_\gamma$ to keep track of the (here implicit) projection onto~$T_{< \gamma}$. In the whole article, we will keep $\gamma = 1 + 2\kappa$ for some $0 < \kappa \ll 1$ \emph{fixed} and thus suppress the subindex. In the same spirit, we will not explicitly indicate said projection when multiplying two elements in~$\DD^{\gamma,\eta}_{\gr{\CU}}(\gr{\bz})$ and when expanding~$Pu_0$ into its Taylor jet
		\begin{equation*}
			\TT P u_0 := \sum_{\abs[0]{k} < \gamma} \frac{X^k}{k!} \del[1]{D^k(Pu_0)} \in \DD_{\CU}^{\gamma,\eta}
		\end{equation*}  
	of order~$\gamma$.	 
\end{remark}

In a next step, we will quantify the Fr\'{e}chet differentiability that $G$ inherits from $g$ and, as a natural byproduct, obtain Taylor's formula for~$G$.

\begin{lemma} \label{prop:taylor_modelled_distrib}
	Fix $\gr{\bz} \in \gr{\MM}$ and consider two modelled distributions $\gr{Y}, \gr{\tilde{Y}} \in \DD^{\gamma,\eta}_{\gr{\CU}}(\gr{\bz})$ with coefficient functions $\phi_\tau$ and $\tilde{\phi}_\tau$, $\tau \in \{\mathbf{1}, \<1>, \<1g>, X\}$, respectively. Given a function $f \in \CC^4(\R;\R)$ with $F$ defined as in~\eqref{eq:G_ext_rs}, we then have the identity
	\begin{equation}
		F(\gr{Y} + \gr{\tilde{Y}}) - F(\gr{Y}) - F'(\gr{Y})\gr{\tilde{Y}}
		=
		\int_0^1 (1-s) F''(\gr{Y} + s\gr{\tilde{Y}}) \gr{\tilde{Y}}^2 \dif s.
		\label{prop:taylor_modelled_distrib_eq}
	\end{equation}
	For $f \in \CC^3(\R,\R)$, we have
	\begin{equation}
		F(\gr{Y} + \gr{\tilde{Y}}) - F(\gr{Y}) 
		=
		\int_0^1 F'(\gr{Y} + s\gr{\tilde{Y}})\gr{\tilde{Y}} \dif s.
		\label{prop:taylor_modelled_distrib_eq2}
	\end{equation}
\end{lemma}
This lemma is true in greater generality than stated here, see~\cite[Prop.~B.$2$]{schoenbauer} and its proof. 
It is instrumental in the proof of the following proposition.

\begin{proposition} \label{lem:diffb_G_from_g} 
	Fix $\gr{\bz} \in \gr{\MM}$, let $\DD := \DD^{\gamma,\eta}_{\gr{\CU}}(\gr{\bz})$, and suppose $g \in \CC^{\ell + 4}$, $\ell \geq 1$. Then, $G \in \CC^\ell(\DD,\DD)$ (in the Fr\'{e}chet sense) and
	\begin{equation}
		D^{(k)} G(\gr{Y})\sbr[1]{(\gr{Y_m})_{m=1}^k}
		:=
		D^{(k)} G(\gr{Y})[\gr{Y_1},\ldots,\gr{Y_k}] 
		= G^{(k)}(\gr{Y}) \prod_{m=1}^k \gr{Y_m}, \quad k \leq \ell,
		\label{lem:diffb_G_from_g:eq}
	\end{equation} 
	where $G^{(k)}$ is built from $g^{(k)}$ by means of~\eqref{eq:G_ext_rs}.
\end{proposition}

\begin{remark}[Regularity assumptions on $\boldsymbol{g}$] \label{rmk:order_of_diffb}
	Let us comment on the order of differentiability imposed on~$g$.   
	First of all, recall that $\gamma = 1 + 2\kappa$ and that $\gr{\CU}$ is a function-like sector, the lowest non-zero homogeneity of which is 
	\begin{equation*}
		\chi := \reg{\<1>} = \reg{\<1g>} = 1 - \kappa.
	\end{equation*}
	As detailed in \cite[Prop.~6.13]{hairer_rs}, for $G^{(\ell)}$ to be well-defined we actually only need $g^{(\ell)} \in \C^n$ with $n \in \N$ and $n \geq \nicefrac{\gamma}{\chi} \vee 1$, i.e. $n \geq 2$ and $g \in \CC^{\ell+2}$. However, in our proof we will make use of~\thref{prop:taylor_modelled_distrib} with $F' := G^{(\ell)}$, so that $G^{(\ell+1)}$ appears on the RHS of~\eqref{prop:taylor_modelled_distrib_eq} and, by the same token, necessitates~$g \in \CC^{(\ell+3)}$.
	However, we will also want to use (strong) local Lipschitz continuity of~$G^{(\ell+1)}$, so \cite[Prop.~$6.13$]{hairer_rs} and \cite[Prop.~3.11]{hairer_pardoux} imply the need for an additional~\enquote{$+1$} in regularity, resulting in~$g \in \CC^{(\ell+4)}$.
	
	In the setting of~\thref{thm:laplace_asymp}, we have~$\ell = N+3$ and so the above reasoning requires~$g \in \CC^{(N+7)}$.
\end{remark}

The following proof borrows from that of \cite[Prop.~4.7]{cfg}.

\begin{proof}
	Suppose we had already established that $G \in \CC^k$ for some $k \in \{1,\ldots,\ell-1\}$ with $D^{(m)}G$, $m = 1,\ldots,k$, given by~\eqref{lem:diffb_G_from_g:eq}. We want to do the induction step $k \mapsto k+1$ and prove that $G \in \CC^{k+1}$, so let $Y \in \DD$ and consider
	\begin{equation}
		R^{(k)}\sbr[1]{\gr{Y};(\gr{Y_i})_{i=1}^{k+1}}
		:=
		D^{(k)}G(\gr{Y} + \gr{Y_{k+1}})\sbr[1]{(\gr{Y_i})_{i=1}^k}
		-
		D^{(k)}G(\gr{Y})\sbr[1]{(\gr{Y_i})_{i=1}^k}
		-
		G^{(k+1)}(\gr{Y}) \prod_{m=1}^{k+1} \gr{Y_m}.
		\label{lem:diffb_G_from_g:pf:remainder_k}
	\end{equation}
	We want to prove that
	\begin{equation}
		\sup_{\substack{\threebars \gr{Y_m} \threebars_{\gamma,\eta} \leq 1, \\ 1 \leq m  \leq k}} \threebars R^{(k)}\sbr[1]{\gr{Y};(\gr{Y_i})_{i=1}^{k+1}} \threebars_{\gamma,\eta} = o\del[1]{\threebars \gr{Y_{k+1}} \threebars_{\gamma,\eta}}, \quad \gr{Y_{k+1}} \to 0.
		\label{lem:diffb_G_from_g:pf:remainder_k_est}
	\end{equation}
	Setting $P_k := \prod_{n=1}^k \gr{Y_n}$, an element of~$\DD$ by~\cite[Prop.~6.12]{hairer_rs}, we easily see that	
	\begin{equation}
		R^{(k)}\sbr[1]{\gr{Y};(\gr{Y_i})_{i=1}^{k+1}}
		=
		\del[1]{G^{(k)}(\gr{Y} + \gr{Y_{k+1}}) - G^{(k)}(\gr{Y}) - G^{(k+1)}(\gr{Y})\gr{Y_{k+1}}} P_k.
		\label{lem:diffb_G_from_g:pf:remainder_k_est_prod}
	\end{equation} 
	Since $k \leq \ell - 1$, we have $g^{(k)} \in \CC^m$ with $m \geq 5$, so we may invoke~\thref{prop:taylor_modelled_distrib}, more specifically~\eqref{prop:taylor_modelled_distrib_eq}, to obtain 
	\begin{equation*}
		R^{(k)}\sbr[1]{\gr{Y};(\gr{Y_i})_{i=1}^{k+1}}
		=
		\int_0^1 (1-s) G^{(k+2)}(\gr{Y} + s \gr{Y_{k+1}}) \gr{Y_{k+1}}^2 \dif s \thinspace P_k =: I_{k+2}(\gr{Y},\gr{Y_{k+1}}) P_k
	\end{equation*} 
	From \cite[Prop.~6.13]{hairer_rs}, we know that $G^{(k+2)}(\gr{Y} + s \gr{Y_{k+1}}) \in \DD$ for each $s \in [0,1]$; \cite[Prop.~6.12]{hairer_rs} then implies that $G^{(k+2)}(\gr{Y} + s \gr{Y_{k+1}}) \sbr[0]{\gr{Y_{k+1}}}^2$ is also in $\DD$. 
	Hence, so is the Bochner integral $I_{k+2}(\gr{Y},\gr{Y_{k+1}})$ and, by the same token as before, the product $I_{k+2}(\gr{Y},\gr{Y_{k+1}}) P_k$. The quoted propositions then give
	\begin{equs}[][lem:diffb_G_from_g:pf:remainder_k_est_prod_est]
		\threebars R^{(k)}\sbr[1]{\gr{Y};(\gr{Y_i})_{i=1}^{k+1}} \threebars_{\gamma,\eta}
		& 
		\aac
		\sup_{s \in [0,1]} \threebars G^{(k+2)}(\gr{Y} + s \gr{Y_{k+1}})\threebars_{\gamma,\eta} \threebars \gr{Y_{k+1}} \threebars_{\gamma,\eta}^2 \prod_{n=1}^k \threebars \gr{Y_n} \threebars_{\gamma,\eta} \\
		&
		\aac
		\del[1]{\threebars \gr{Y} \threebars_{\gamma,\eta} + \threebars \gr{Y_{k+1}} \threebars_{\gamma,\eta}} \threebars \gr{Y_{k+1}} \threebars_{\gamma,\eta}^2 \prod_{n=1}^k \threebars \gr{Y_n} \threebars_{\gamma,\eta}, 
	\end{equs} 
	using, in particular, that $G^{(k+2)}$ is locally Lipschitz continuous. 
	The previous equation implies that 
	\begin{equation}
		\sup_{\substack{\threebars \gr{Y_m} \threebars_{\gamma,\eta} \leq 1, \\ 1 \leq m  \leq k}} \threebars R^{(k)}\sbr[1]{\gr{Y};(\gr{Y_i})_{i=1}^{k+1}} \threebars_{\gamma,\eta} = O\del[1]{\threebars \gr{Y_{k+1}} \threebars_{\gamma,\eta}^2}, \quad \gr{Y_{k+1}} \to 0,
		\label{lem:diffb_G_from_g:pf:remainder_k_est:O}
	\end{equation}
	and so~\eqref{lem:diffb_G_from_g:pf:remainder_k_est} follows. Since the previous argument also works for $k=0$, where
	\begin{equation*}
		R^{(0)}\sbr[1]{\gr{Y};\gr{Y_1}}
		= G(\gr{Y}+\gr{Y_1}) - G(\gr{Y}) - G'(\gr{Y})\gr{Y_1},
	\end{equation*}
	the induction argument can be closed. 
\end{proof}

As a corollary of the last proposition, we may lift Taylor's theorem from $g$ to $G$.

\begin{corollary}
	In the setting of~\thref{lem:diffb_G_from_g}, the identity
	\begin{equation*}
		G(\gr{Y} + \gr{\tilde{Y}}) + G(\gr{Y}) + \sum_{m=1}^{\ell-1} \frac{1}{m!} G^{(m)}(\gr{Y})\gr{\tilde{Y}}^m + R_\ell(\gr{Y},\gr{\tilde{Y}}), \quad \gr{Y},\gr{\tilde{Y}} \in \DD,
	\end{equation*}
	holds true with
	\begin{equation*}
		R_\ell(\gr{Y},\gr{\tilde{Y}}) = \int_0^1 \frac{(1-s)^{\ell-1}}{(\ell-1)!} G^{(\ell)}(\gr{Y} + s \gr{\tilde{Y}}) \gr{\tilde{Y}}^\ell \dif s.
	\end{equation*}
\end{corollary}

\subsubsection*{$\bullet$ Dilation of modelled distributions}

The following lemma clarifies the relation between modelled distributions based on an admissible model and its dilated counterpart. It is the analogue of \cite[Lem.~3.22]{cfg} and proved similarly.

\begin{lemma} \label{lem:consistency_dilation}
	Let $\bz \in \MM$ be an admissible model for a regularity structure~$\TT$, $\eps \in I = [0,1]$, and $\d_\eps \bz = (\Pi^\eps,\Gamma^\eps)$.	For each $\gamma > 0$, $\eta \in [0,\gamma]$, and $U^\eps \in \DD^{\gamma,\eta}(\Gamma^\eps)$, we have:
	\begin{enumerate}[label=(\roman*)]
		\item \label{lem:consistency_dilation:i} $\fd_\eps U^\eps \in \DD^{\gamma,\eta}(\Gamma)$, i.e. $\fd_\eps\del[1]{\DD^{\gamma,\eta}(\Gamma^\eps)} \subseteq \DD^{\gamma,\eta}(\Gamma)$. In addition, $\fd_\eps$ commutes with the operations of composition with smooth functions and product between modelled distributions.
		\item \label{lem:consistency_dilation:ii} $\CR^{\d_\eps \bz} U^\eps = \CR^{\bz} \fd_\eps U^\eps$.
		\item \label{lem:consistency_dilation:iii} $\fd_\eps\del[1]{\CP^{\d_\eps \bz} U^\eps} = \CP^{\bz} \fd_\eps U^\eps$.
	\end{enumerate}
\end{lemma}

\begin{proof}
	\begin{enumerate}[label=(\roman*)]
		\item Observe that $\norm{\fd_\eps \tau}_\b = \eps^{[\tau]} \norm[0]{\tau}_\b \leq \norm{\tau}_\b$ for any $\tau \in \CT_\b$. For each $\b < \gamma$, we thus have $\norm[0]{(\fd_\eps U^\eps)(z)}_{\b} \leq \norm[0]{U^\eps(z)}_{\b}$ and
		\begin{equation*}
			\norm[0]{(\fd_\eps U^\eps)(z) - \Gamma_{z\bar{z}} (\fd_\eps U^\eps)(\bar{z})}_\b
			=
			\norm[1]{\fd_\eps\del[1]{U^\eps(z) - \Gamma^\eps_{z\bar{z}} U^\eps(\bar{z})}}_\b
			\leq
			\norm[0]{U^\eps(z) - \Gamma^\eps_{z\bar{z}} U^\eps(\bar{z})}_\b.
		\end{equation*}
		where the equality is due to~\thref{lem:dil_Gamma}.
		In summary, 
		\begin{equation*}
			\threebars \fd_\eps U^\eps \threebars_{\gamma,\eta;\Gamma} 
			\leq
			\threebars U^\eps \threebars_{\gamma,\eta;\Gamma^\eps} < \infty. 
		\end{equation*}
		The other claims are true because $\fd_\eps$ is linear and multiplicative.
		\item We set $\CR^\eps := \CR^{\d_\eps \bz}$ and $\CR := \CR^{\bz}$. Recall that, by definition, $\Pi^\eps = \Pi \circ \fd_\eps$. For $V^\eps \in \DD^\gamma(\Gamma^\eps)$ and $V \in \DD^\gamma(\Gamma)$, the  bounds for $\CR^\eps$ and $\CR$ in the reconstruction theorem~\cite[Thm.~3.10]{hairer_rs} read: For any compact $\fK \subseteq \R^+ \x \R^2$,
		\begin{equation}
			\abs[1]{(\CR^\eps V^\eps - \Pi_z [\fd_\eps V^\eps(z))](\phi^\l_z)} \aac \l^\g, \qquad
			\abs[1]{(\CR V - \Pi_z V(z))(\phi^\l_z)} \aac \l^\g
			\label{lem:consistency_dilation:pf_ii_eq}
		\end{equation} 
		hold uniformly over $\phi \in \CB$, $\l \in (0,1]$, and $z \in \fK$. Since $\g > 0$, these bounds characterise $\CR^\eps$ and $\CR$ uniquely. As $\fd_\eps V^\eps \in \DD^\gamma(\Gamma)$ by~\ref{lem:consistency_dilation:i}, we may plug it into the second bound in~\eqref{lem:consistency_dilation:pf_ii_eq} in place of~$V$, so the statement follows for~$V^\eps \in \DD^\gamma(\Gamma^\eps)$ by uniqueness. 
		The corresponding statement for $U^\eps \in \DD^{\gamma,\eta}(\Gamma^\eps)$ follows from~\cite[Prop.~6.9]{hairer_rs} for $\eta \in (-2,\gamma]$.
		\item For $U \in \DD^{\gamma,\eta}(\bz)$, the operator $\CP^{\bz}$ is defined by
			\begin{equation}
			(\CP^{\bz} U)(z) = \CI [U(z)] + \CJ^{\bz}(z) [U(z)] + \CN^{\bz} [U(z)] + R\del[1]{\CR^{\bz} U(z)},
			\end{equation}
		cf.~\cite[Eq.'s~(5.15) and (7.18)]{hairer_rs}. The operators $\CJ^{\bz}(z)$, $\CN^{\bz}$, and $R$ are defined in~\cite[Eq.'s~(5.11), (5.16), and (7.7)]{hairer_rs}, respectively, and take values in the polynomial regularity structure. 
		
		Since $\CI$ and $\fd_\eps$ commute and $\fd_\eps$ leaves the polynomials~$X^k$ invariant, we need only check that the coefficients of the polynomials $(k!)^{-1} X^k$ in $\CJ^{\d_\eps\bz}(z) [U^\eps(z)]$, $\CN^{\d_\eps\bz} [U^\eps(z)]$, and $R\del[1]{\CR^{\d_\eps\bz} U^\eps(z)}$ match those in $\CJ^{\bz}(z) [\fd_\eps U^\eps(z)]$, $\CN^{\bz} [\fd_\eps U^\eps(z)]$, and $R\del[1]{\CR^{\bz} \fd_\eps U^\eps(z)}$, respectively. Indeed, this is the case:
		\begin{itemize}
			\item $\CJ$: \qquad \hspace{-0.3em} $\scal{\Pi_z^\eps [U^\eps(z)], D^{(k)} K(z - \cdot)} = \scal{\Pi_z [\fd_\eps U^\eps(z)], D^{(k)} K(z - \cdot)}$.
			\item $\CN$: \qquad \hspace{-0.3em} $\scal{\CR^\eps U^\eps - \Pi^\eps_z [U^\eps(z)], D^{(k)} K(z - \cdot)} = \scal{\CR [\fd_\eps U^\eps] - \Pi_z [\fd_\eps U^\eps(z)], D^{(k)} K(z - \cdot)}$.
			\item $R \circ \CR$: \ $\scal{\CR^\eps U^\eps, D^{(k)}R(z - \cdot)} = \scal{\CR [\fd_\eps U^\eps], D^{(k)}R(z - \cdot)}$.
		\end{itemize}
		In these identities, we have used that~$\Pi_z^\eps = \Pi_z \circ \fd_\eps$ and~\ref{lem:consistency_dilation:ii}.
	\end{enumerate}
\end{proof}

\begin{remark}
	We emphasise that the preceding lemma is valid for \emph{any} regularity structure, in particular for those introduced in section~\ref{app:rs_background} above.
\end{remark}

\section{Deterministic gPAM and explosion times} \label{sec:explosion}

In this section, we summarise facts about explosion times of the (stochastic) PDEs under investigation. 

We start with gPAM driven by a Cameron-Martin function. 
The following proposition, in essence, is a version of~\cite[Prop.~A.1]{chouk_friz}, with some appropriate changes to cover our case of interest. 
Recall that~$\CX_T$ is the space of functions~$u \in \CC([0,T];\CC^{\eta}(\T^2))$ with~$u(0,\cdot) = u_0$ and~$\eta \in (\nicefrac{1}{2},1)$.

\begin{proposition}\label{prop:ex_det_gpam}
	Let $T > 0$. Given~$g \in \CC_{b}^1(\R)$, $u_0 \in \CC^\eta(\T^2)$ for $\eta \in (0,1)$, and $h \in \CH \equiv L^2(\T^2)$, there exists a unique global solution~$w_h \in \CX_T$ to the equation
		\begin{equation*}
			(\partial_t - \Delta) w_h = g(w_h)h, \quad w_h(0,\cdot) = u_0.
			\label{prop:ex_det_gpam:eq}
		\end{equation*}
\end{proposition}

\begin{proof}
	Let $a, b \in \C^\eta(\T^2)$, set $v := g(a) - g(b)$, and note that we have $v \in \CC^\eta(\T^2)$ because $g \in \CC_b^1$. 
	We have~$\CH = B_{2,2}^0(\T^2) \embed B_{2,2}^{-\eps}(\T^2)$ for any~$\eps > 0$. Choosing~$\eps \in (0,\eta)$, the paraproduct estimates conveniently summarised by Mourrat and Weber~\cite[Prop.~A.7]{mourrat_weber_infinity} then imply that
		\begin{equation*}
			vh \in B_{2,\infty}^{-\eps}(\T^2) \embed \CC^{-1-\eps}(\T^2).
		\end{equation*}
	In addition, for any~$\gamma < -1$ we have the estimate
		\begin{equation*}
			\norm[0]{\sbr[0]{g(a) - g(b)} h}_{\gamma} \aac \norm[0]{g(a) - g(b)}_{\eta} \norm{h}_{\CH} \aac \norm[0]{g'}_{L^\infty} \norm[0]{a-b}_{\eta} \norm[0]{h}_\CH.
		\end{equation*}
	By the regularising properties of the heat semigroup~$P_t = e^{t\Delta}$~\cite[Prop.~A.13]{mourrat_weber_infinity},  we find for each~$f \in \CX_T$ and~$\gamma > \eta -2$ that
		\begin{equs}
			\norm[3]{\int_0^t P_{t-s} f_s \dif s}_{\eta} 
			& \aac 
			\int_0^t \norm[0]{P_{t-s} f_s}_{\eta} \dif s 
			\aac
			\int_0^t (t-s)^{\frac{1}{2}(\gamma - \eta)} \norm[0]{f_s}_\gamma \dif s
			\leq
			T^{1 + \frac{1}{2}(\gamma-\eta)} \norm[0]{f}_{\CX_T}. 
		\end{equs}
	We then introduce the map
		\begin{equation*}
			\Gamma_T: \CX_T \to \CX_T, \quad
			\Gamma_T(w)(t) := P_t u_0 + \int_0^t P_{t-s}\sbr[0]{g(w_s)h} \dif s
		\end{equation*}
	which is well-defined by the previous calculations (choose~$f = g(w)h$). Then, one easily finds that 
		\begin{equs}
			\norm[0]{\Gamma_T(w) - \Gamma_T(v)}_{\CX_T}
			&
			=
			\sup_{t \in [0,T]} \norm[3]{\int_0^t P_{t-s} \sbr[1]{(g(w_s) - g(v_s))h} \dif s}_{\eta} \\
			&
			\aac
			T^{1 + \frac{1}{2}(\gamma-\eta)} \norm[0]{(g(w_s) - g(v_s))h}_{\CX_T}
			\aac
			T^{1 + \frac{1}{2}(\gamma-\eta)} \norm[0]{g'}_{L^\infty} \norm[0]{h}_\CH \norm[0]{w-v}_{\CX_T}. 
		\end{equs}
	The exponent of~$T$ is positive, so we may choose $T_*$ small enough such that $\Gamma_{T_*}$ becomes a contraction, in turn admitting a unique fixed-point~$w_h$. Since $T_*$ does not depend on~$u_0$, this procedure can be iterated, giving rise to a global solution.
\end{proof}

\subsubsection*{$\bullet$ Explosion times.} 

Next, recall that the explosion time $T_\infty$ of the solution to eq.~\eqref{eq:fp_normal} -- that is: the smallest time for which the equation
	\begin{equation}
		U = \CP^{\bz} \del[1]{G(U)\thinspace\<wn>} + \TT Pu_0 \quad \text{in} \quad \DD^{\gamma,\eta}_\CU(\bz),
		\label{sec:explosion:abstract_gpam}
	\end{equation}
does \emph{not} have a solution~(cf.~\cite[Coro.~$9.3$]{hairer_rs}) -- is lower semicontinuous (abbr. \enquote{l.s.c.}) as a function of $\bz \in \MM$ ($u_0$ is fixed), see~\cite[Prop.~3.23]{cfg}. Since the extension operator $\EE: \MM_- \to \MM$ is continous~(cf.~\thref{app:thm:extension}), the map
	\begin{equation*}
		T_\infty^{\minus} := T_\infty(u_0,\cdot) \circ \EE: \ \MM_- \to [0,+\infty]
	\end{equation*}
is l.s.c. both on $(\MM_-,\norm[0]{\cdot})$ and~$(\MM_-,\barnorm{\thinspace\cdot\thinspace})$. We introduce
		\begin{itemize}
			\item the explosion time $T_0$ of the \emph{deterministic} PDE when the noise is set to zero ($\eps = 0$). More precisely, we set 
				\begin{equation}
					T_0 := T_\infty^{\minus}(\boldsymbol{0})
					\label{sec:explosion:time_0}
				\end{equation}
			where $\boldsymbol{0} \in \MM_-$ is the lift of the function that is constant~$0$.
			\item For $\eps \in [0,1]$, let $T^\eps := T_\infty^{\minus}(\d_\eps\hbz^{\minus})$. In other words, $T^\eps$ is the explosion time of~$\hat{u}^\eps$ from~\thref{intro:ex_sol_hairer} in the introduction.
			\item For $h \in \CH$, let $T_\infty^h := T_\infty^{\minus}(\LL_-(h))$ denote the \emph{deterministic} explosion time of the PDE driven by the non-random noise~$h$. 
		\end{itemize}

\begin{remark}\label{rmk:explosion_gpam}
	In the specific case at hand, $T_0$ is the explosion time of the homogeneous heat equation with initial condition $u_0 \in \CC^{\eta}(\T^2)$. Hence, $T_0 = +\infty$ by linearity of the equation. Furthermore, for gPAM we have~$T_\infty^h = + \infty$ for any $h \in \CH$ as~\thref{prop:ex_det_gpam} above shows.
\end{remark}

The following two lemmas are general statements for stochastic PDEs amenable to analysis via regularity structures and not specific to~gPAM.

\begin{lemma}\label{sec:explosion:time_0_lem}
	Let $T < T_0$. Then $\P(T^\eps > T) \to 1$ as $\eps \to 0$.
\end{lemma}

\begin{proof}
	We assume that $T_0 < +\infty$, with obvious modifications if~$T_0 = +\infty$. By l.s.c. of $T_\infty^{\minus}$, we know that for each $\eta > 0$ there exists $\d = \d(\eta) > 0$ such that
		\begin{equation}
			\norm[0]{\bz^{\minus} - \boldsymbol{0}}
			=
			\norm[0]{\bz^{\minus}} < \d \implies T_{\infty}^{\minus}(\bz^{\minus}) > T_\infty^{\minus}(\boldsymbol{0}) - \eta \quad (\equiv T_0 - \eta).
			\label{sec:explosion:time_0_lem:pf_eq1}
		\end{equation}
	We choose~$\tilde{\eta} := T_0 - T > 0$. Since $\d_\eps \hbz^{\minus} \to \boldsymbol{0}$ in prob. as $\eps \to 0$, we know that 
		\begin{equation*}
			1 \geq \P(T^\eps > T) 
			= \P(T_\infty^{\minus}(\d_\eps \hbz^{\minus}) > T)
			\overset{\eqref{sec:explosion:time_0_lem:pf_eq1}}{\geq}
			\P(\norm[0]{\d_\eps \hbz^{\minus}} < \d(\tilde{\eta}))
			\to 1 \quad \text{as} \ \eps \to 0. 
			\label{sec:explosion:time_0_lem:pf_eq2}
		\end{equation*}
	The claim follows.
\end{proof}

Loosely speaking, the next lemma states that if one controls the noise (in the form of the model~$\d_\eps\hbz^{\minus}$), then one can ensure that the stochastic PDE driven by~$h + \eps \xi$ lives as long as that driven by~$h$. 
Its proof is informed by~\cite[Cor.~$3.27$]{cfg}.

\begin{lemma}\label{sec:explosion:cm_lem}
	Fix $h \in \CH$ and let $T < T_\infty^h$. Then, there exists $\rho = \rho(T) > 0$ such that, for all~$\eps \in [0,1]$ and $\bz^{\minus} \in \MM_-$ with~$\eps \barnorm{\bz^{\minus}} < \rho$, we have~$T_\infty^{\minus}(T_h \d_\eps \bz^{\minus}) > T$.
\end{lemma}

\begin{proof}
	By rescaling, we may assume that~$\eps = 1$.
	In addition, we assume that $T_\infty^h < +\infty$ with obvious modifications of our arguments otherwise. The explosion time~$T_\infty^{\minus}$ is l.s.c., so for each $\eta > 0$ there exists $\d = \d(\eta) > 0$ such that
	\begin{equation}
		\norm[0]{\tbz^{\minus} - \LL_-(h)} < \d \implies T_{\infty}^{\minus}(\tbz^{\minus}) > T_\infty^{\minus}(\LL_-(h)) - \eta \quad (\equiv T_\infty^h - \eta).
		\label{sec:explosion:cm_lem:pf_eq1}
	\end{equation}
	for any~$\tbz^{\minus} \in \MM_-$.
	We choose~$\tilde{\eta} := T_\infty^h - T > 0$. Since $\LL_-(h) = T_h\boldsymbol{0}$ and $T_h$ is locally Lipschitz continuous w.r.t. $\threebars \thinspace \cdot \thinspace \threebars$~(see~\thref{app:prop:translation}), we have
		\begin{equs}
			\threebars T_h \bz^{\minus} - \LL_-(h) \threebars
			&
			=
			\threebars T_h \bz^{\minus} - T_h\boldsymbol{0} \threebars
			\aac
			\threebars \bz^{\minus} \threebars
			\leq
			\max_{\tau \in \CW_-} \barnorm{\bz^{\minus}}^{[\tau]}
		\end{equs}
	Upon choosing $\rho := \d(\tilde{\eta}) \wedge 1$, the claim follows by~\eqref{sec:explosion:cm_lem:pf_eq1} for $\tbz^{\minus} := T_h \bz^{\minus}$.
\end{proof}

\section{Large deviations} \label{app:ldp_pam}

In this section, we will give a precise formulation of~\thref{thm:ldp} and characterise the large deviation behaviour of the solutions~$(\hat{u}^{(\eps)}: \eps \in I)$~from~\thref{intro:ex_sol_hairer}. In doing so, we will leverage the results of Hairer and Weber: our work amounts to proving applicability of~\cite[thm.~3.5]{hairer_weber_ldp}.
Recall that
	\begin{equation}
		\CW_- = \{\<wn>,\<11>\}, \quad
		\reg{\<wn>} = -1-\kappa, \quad 
		\reg{\<11>} = -2\kappa. 
	\end{equation}
Let us emphasise again that the methodology here is entirely contained in Hairer's and Weber's~\cite{hairer_weber_ldp} work. The backbone of their setting is the \emph{abstract Wiener space} (introduced by Gross~\cite{gross})~$(B,\CH,\mu)$ defined by
	\begin{itemize}
		\item the Hilbert space~$\CH := L^2(\T^2)$,
		\item the centred Gaussian measure~$\mu$ with Cameron-Martin space~$\CH$, and
		\item the Banach space $B$, the closure of~$\CC^{\infty}(\T^2)$ in the H\"older-Besov space~$\CC^{-1-\kappa}(\T^2)$ for any~$\kappa > 0$.
	\end{itemize}
Note that under~$\mu$, the space white noise~$\xi$ is the canonical process.

\subsubsection*{$\bullet$ The Banach spaces.}
With~$\norm[0]{\cdot}_{\tau}$ for $\tau \in \CW_-$~introduced in~\eqref{def:minimal_admissible_model:norm} and~\eqref{def:minimal_admissible_model:norm:suprema},the Banach space~$\BE := E_{\<wns>} \oplus E_{\<11s>}$ is constructed as follows:
	\begin{enumerate}[label=(\arabic*)]
		\item $E_{\<wns>}$ is the closure of smooth functions $\{\Pi\<wn>: \thinspace \tz \mapsto (\Pi \<wn>)(\tz) \in \CC^{\infty}(\T^2;\R)\}$ under the norm~$\norm[0]{\Pi}_{\<wns>}$. In other words,~$E_{\<wns>} = B$.
		\item $E_{\<11s>}$ is the closure of smooth functions $\{\Pi\<11>: (z,\tz) \mapsto (\Pi_z\<11>)(\tz) \in \CC^{\infty}(\T^2\x\T^2;\R)\}$ under the norm~$\norm[0]{\cdot}_{\<11s>}$.
	\end{enumerate}
Recall that there are also the homogeneous \enquote{norms}~$\barnorm{\cdot}_\tau$,~$\tau \in \CW_-$, in~\eqref{def:minimal_admissible_model:norm}. The same arguments as in the proof of~\thref{lem:estimate_min_norm_vs_hom_norm} show that we can also equip~$\BE$ with~$\barnorm{\cdot}$. Unless otherwise noted, however, we will always consider~$(\BE,\norm[0]{\cdot})$.

We emphasise that the Banach spaces~$E_\tau$, $\tau \in \CW_-$, are set up to dovetail nicely with the definition of the space~$\MM_-$ in~\thref{def:minimal_admissible_model}. In fact, $\MM_-$ is a closed subset (w.r.t. $\norm[0]{\cdot}$ and $\barnorm{\cdot}$) of~$\BE$.

\subsubsection*{$\bullet$ The random variables and their chaos decomposition.} 
We define the random variables $\Psib_\d: B \to \BE$ by
	\begin{equation}
		\Psib_\d(\xi) := \bigoplus_{\tau \in \CW_-}\Psi_{\d,\tau}(\xi), \quad \Psi_{\d,\tau}(\xi) := \hat{\Pi}^{\xi_\d} \tau, \quad \tau \in \CW_- = \{\<wn>,\<11>\}.
		\label{app:ldp:def_psib_delta}
	\end{equation}
where~$\hat{\Pi}^{\xi_\d}$ has been introduced in subsection~\ref{sec:renormalisation} above.
For $k \in \N_0$ and a Banach space~$E$, we denote by $\CH^{(k)}(\mu;E)$ the $E$-valued homogeneous Wiener-It\^{o} chaos (w.r.t.~$\mu$) of order $k$.	
By the definitions in~\eqref{app:ldp_rv_chaos_1}, one easily sees that the previously defined random variables are elements of the following (in-)homogeneous chaoses:
	\begin{equation*}
		\Psi_{\d,\<wns>} \equiv \Psi_{\d,\<wns>,1} \in \CH^{(1)}(\mu;E_{\<wns>}), \quad
		\Psi_{\d,\<11s>} = \Psi_{\d,\<11s>,2} + \Psi_{\d,\<11s>,0} \in \CH^{(2)}(\mu;E_{\<11s>}) \oplus \CH^{(0)}(\mu;E_{\<11s>})
	\end{equation*}
For $y,\by \in \T^2$, we want to calculate the respective integral kernels
	\begin{equation*}
		\hat{\CW}_{\d;\<wns>}(y,\by;\bullet_1) \in L_s^2\del[1]{\T^2}, \quad
		\hat{\CW}_{\d;\<11s>,2}(y,\by;\bullet_1,\bullet_2) \in L_s^2\del[1]{(\T^2)^2}, \quad 
		\hat{\CW}_{\d;\<11s>,0}(y,\by) \in \R 
	\end{equation*}
such that
	\begin{equation*}
		\Psi_{\d,\<wns>}[y,\phi] = I_1(\scal{\hat{\CW}_{\d,\<wns>}(y,\cdot,\bullet_1),\phi}), \quad 
		\Psi_{\d,\<11s>}[y,\phi] = I_2(\scal{\hat{\CW}_{\d,\<11s>,2}(y,\cdot,\bullet_1,\bullet_2),\phi}) 
		+ \scal{\hat{\CW}_{\d;\<11s>,0}(y,\cdot),\phi}
	\end{equation*}
where the bullets indicate the input variables for the \emph{Wiener-It\^{o} isometries}
	\begin{equation*}
		I_k: \CH^{\tp_s k} \simeq L_s^2\del[1]{(\T^2)^k} \to \CH^{(k)}(\mu;\R), \quad k = 1,2.
	\end{equation*}
Here, we have denoted by~$L_s^2$ the \emph{symmetric} $L^2$ space.
For more information on these isometries, we refer the reader to Nualart~\cite{nualart} and for the Banach-valued Wiener chaoses to Ledoux~\cite{ledoux}.	
\begin{enumerate}[label=(\arabic*)]
	\item For the symbol $\tau = \<wn>$, we find
	\begin{equation}
	\Psi_{\d,\<wns>}(\xi)[y,\by]
	= \hat{\Pi}^{\xi_\d}_y\<wn>(\by) 
	= \xi_\d(\by) 
	= (\xi * \rho_\d)(\by)
	= \scal{\xi,\rho_\d(\by - \cdot)}
	= I_1\del[1]{\rho_\d(\by - \cdot)},
	\end{equation}
	so that $\hat{\CW}_{\d;\<wns>}(y,\by;z_1) = \rho_\d(\by-z_1)$, independently of $y \in \T^2$.
	\item For the symbol $\tau = \<11>$, with $\wp$ denoting the Wick product we find
	\begin{equs}
		\Psi_{\d,\<11s>,2}(\xi)[y,\by] 
		& 
		=
		\Pi_y^{\xi_\d}\<1>(\by) \wp \Pi_y^{\xi_\d}\<wn>(\by) 
		=
		\sbr[1]{I_1\del[1]{\rho_\d(\by - \bullet)} \wp I_1\del[1]{K_\d(\by - \bullet) - K_\d(y - \bullet)}}(\xi) \\
		&
		=
		I_2\del[2]{\rho_\d(\by - \bullet_2) \tp \sbr[1]{K_\d(\by - \bullet_1) - K_\d(y - \bullet_1)}}(\xi),
	\end{equs}
	whereby it follows that $\hat{\CW}_{\d,\<11s>,2}(y,\by;z_1,z_2) = \rho_\d(\by - z_2) \sbr[1]{K_\d(\by - z_1) - K_\d(y - z_1)}$. 
	Using that $K(z) = K(-z)$, a calculation similar to~\eqref{calculation_ren_constant} leads to the equality
	\begin{equs}
		\Psi_{\d,\<11s>,0}(\xi)[y,\by] 
		& =
		E\sbr[1]{\Psi_{\d,\<11s>}[y,\by]} - \fc_{\d}
		=
		- \int_{\T^2} \rho_\d^{\star 2}(\by - z) K(y-z) \dif z
	\end{equs}
	independently of~$\xi$ where
	\begin{equation*}
		\rho_\d^{\star 2}(\by - z)
		:=
		\int_{\T^2} \rho_\d(\by - x) \rho_\d(z - x) \dif x
		=
		\E[\xi_\d(z)\xi_\d(\by)]
	\end{equation*}
	Hence,~$\hat{\CW}_{\d,\<11s>,0}(y,\by) = - (\rho_\d^{\star 2} * K)(\by - y)$.
	The previous computations can also be found in the original article of Hairer~\cite[Pf. of~Thm.~$10.19$]{hairer_rs}.
\end{enumerate}

\subsubsection*{$\bullet$ The homogeneous parts.}	

When establishing the large deviation principle, we need the concept of the \emph{homogeneous part}, in the literature also known as the \emph{S-transform}, see~\cite[Def.~$4.5$]{hu_yan_wick}. It is well-known that for $n \in \N$ and $h \in \CH$, we have
	\begin{equation}
		\sbr[1]{I_n(f_n)}_{\hom}(h) := \E\sbr[1]{I_n(f_n)(\cdot + h)} = \scal{f_n,h^{\tp_s n}}_{\CH^{\tp_s n}}.
		\label{eq:s_trafo}
	\end{equation}
For random variables with components in different inhomogeneous chaoses, the homogeneous part is defined to only take into account the highest chaos component:
	\begin{equation*}
		(\Psib_\d)_{\hom}(h) 
		= \bigoplus_{\tau \in \CW_-} \del[0]{\Psi_{\d,\tau,[\tau]}}_{\hom}(h)
		= \del[0]{\Psi_{\d,\<wn>}}_{\hom}(h) \oplus \del[0]{\Psi_{\d,\<11s>,2}}_{\hom}(h).
	\end{equation*}
Combining~\eqref{eq:s_trafo} with last paragraph's calculations, we find for $h \in \CH$ and
	\begin{enumerate}[label=(\arabic*)]
		\item the symbol $\tau = \<wn>$:
			\begin{equation*}
				(\Psi_{\d,\<wns>})_{\hom}(h)[y,\by]
				=
				\scal{\hat{\CW}_{\d,\<wns>,1}(y,\by;\bullet_1),h}_{\CW}
				=
				h_\d(\by)
				=
				\Pi_y^{h_\d}\<wns>(\by).
			\end{equation*}
		\item the symbol $\tau = \<11>$:
			\begin{equs}
				\del[0]{\Psi_{\d,\<11s>,2}}_{\hom}(h)[y,\by]
				=
				\scal{\hat{\CW}_{\d,\<11s>,2}(y,\by;\bullet_1,\bullet_2),h^{\tp_s 2}}_{\CH^{\tp_s 2}}
				=
				h_\d(y) \del[1]{K * h_\d(\by) - K * h_\d(y)}
				=
				\Pi_y^{h_\d} \<11>(\by).
			\end{equs}
		\end{enumerate}
By continuity of the translation operator~$T_h$ (cf.~\thref{app:prop:translation}) in~$h \in \CH$, we also find that
	\begin{equation*}
		\Pi^{h_\d} \tau = \Pi^{0,h_\d} \tau \to \Pi^{0,h} \tau = \Pi^{h} \tau \quad \text{in } E_\tau, \quad \tau \in \CW_-,
	\end{equation*}	
where $\Pi^{0,h}$ denotes the canonical lift of the $0$ function, translated into direction~$h$~(cf.~\thref{rmk:lift_h}). 
In summary, we have thus found that 
	\begin{equation}
		(\Psib_\d)_{\hom}(h) 
		= \bigoplus_{\tau \in \CW_-} \Pi^{h_\d} \tau 
		\quad \text{and} \quad 
		\lim_{\d \to 0} \thinspace (\Psib_\d)_{\hom}(h) 
		=
		\bigoplus_{\tau \in \CW_-} \Pi^{h} \tau 
		\quad \text{in } \BE.
		\label{app:ldp_hom_part_summary}
	\end{equation}
	
\subsubsection*{$\bullet$ Convergence of renormalised models.} 

We need to prove the existence of a random variable $\Psib = \bigoplus_{\tau \in \CW_-} \Psi_{\tau}: B \to \BE$ such that
	\begin{equation}
		\Psi_\tau = \lim_{\d \to 0} \Psi_{\d,\tau} \quad \text{in} \quad L^2(\mu;E_{\tau}).
		\label{app:ldp:eq_conv_renorm_models}
	\end{equation}
is true for all~$\tau \in \CW_-$. The limiting random variable~$\Psib$ is constructed similarly to~$\Psib_\d$ by its Wiener-It\^{o} chaos decomposition: the respective components~$\Psi_{\tau,k_\tau}$ for~$k_\tau \in \{0,\ldots,[\tau]\}$ and~$\tau \in \CW_-$ are identified in the proof of~\cite[Thm.~$10.19$]{hairer_rs}, more precisely in~~\cite[eq.~$10.23$]{hairer_rs} for~$\tau = \<11>$.

Let us recall the following result:

\begin{theorem}\label{app:thm:wic_top}
Let $E$ be a real Banach space, $n \in \N_0$, and denote by $\CH^{(n)}(\mu;E)$ the $E$-valued Wiener-It\^{o} chaos of order $n$. Then, for all $p \in (0,\infty)$, the $L^p(\mu;E)$-topologies coincide on $\CH^{(n)}(\mu;E)$.
\end{theorem}
In case $E = \R$, the result is classical, see for example~\cite[Thm.~3.50]{janson}. Concerning the general case, we refer the reader to Ledoux~\cite{ledoux} and the references therein. We then find:
	\begin{enumerate}[label=(\arabic*)]
		\item 
		For the symbol $\tau = \<wn>$, \cite[Prop.~9.5]{hairer_rs} asserts that $\xi_\d$ converges to $\xi$ in $L^1(\mu,E_{\<wns>})$. Since $\xi_\d,\xi \in \CH^{(1)}(\mu;E_{\<wns>})$, this implies \eqref{app:ldp:eq_conv_renorm_models} due to~\thref{app:thm:wic_top}.\
		\item 
		In case $\tau = \<11>$, \eqref{app:ldp:eq_conv_renorm_models} is the content of~\cite[Thm. 10.7 \& 10.19]{hairer_rs}. By definition of~$\norm[0]{\cdot}_{\tau}$ and~$\barnorm{\cdot}_{\tau}$,  we immediately obtain
			\begin{equation*}
				\norm[0]{\Psi_{\d,\<11s>,2} - \Psi_{\<11s>,2}}_{L^2(\mu; (E_{\<11s>},\barnorm{\cdot}_{\<11s>}))}^2
				=
				\norm[0]{\Psi_{\d,\<11s>,2} - \Psi_{\<11s>,2}}_{L^1(\mu;(E_{\<11s>},\norm[0]{\cdot}_{\<11s>}))}.
			\end{equation*}
		Since $\Psi_{\d,\<11s>,2}, \Psi_{\<11s>,2} \in \CH^{(2)}(\mu;E_{\<11s>})$, by \thref{app:thm:wic_top}~$\Psi_{\d,\<11s>,2} \to \Psi_{\<11s>,2}$ in $L^2(\mu;(E_{\<11s>},\barnorm{\cdot}_{\<11s>}))$ as $\d \to 0$. 
		
		The convergence~$\Psi_{\d,\<11s>,0} \to \Psi_{\<11s>,0}$ in~$(E_{\<11s>},\norm[0]{\cdot})$ and~$(E_{\<11s>},\barnorm{\cdot})$ follows from Hairer's result combined with~\thref{lem:estimate_min_norm_vs_hom_norm}. 
	\end{enumerate}
Summing up, we have established~\eqref{app:ldp:eq_conv_renorm_models}, both for~$\BE$ equipped with~$\norm[0]{\cdot}$ and~$\barnorm{\cdot}$. It can be verified by a direct calculation that $\Psib_{\hom}(h) = \bigoplus_{\tau \in \CW_-} \Pi^h \tau$, so that~\eqref{app:ldp_hom_part_summary} implies that $(\Psi_\d)_{\hom}(h) \to \Psi_{\hom}(h)$ in~$\BE$ as~$\d \to 0$.

\subsubsection*{$\bullet$ Rescaling.} 
It is easy to verify the last ingredient to get an LDP on model space, namely the rescaling behaviour assumed in~\cite[Thm.~3.5]{hairer_weber_ldp}:
	\begin{equation}
		\Psib_\d^{(\eps)}(\xi) 
		= \bigoplus_{\tau \in \CW_-} \eps^{[\tau]} \Psi_{\d,\tau}(\xi)
		= \eps \hat{\Pi}^{\xi_\d} \<wn> \oplus \eps^2 \hat{\Pi}^{\xi_\d} \<11>
		= \hat{\Pi}^{\eps\xi_\d} \<wn> \oplus \hat{\Pi}^{\eps\xi_\d} \<11>, 
		\quad \xi \in B,
		\label{eq:rescaling}
	\end{equation}
In these identities, the first equality is by definition, see~\cite[eq.~(3.10)]{hairer_weber_ldp}, and the last equality follows directly from the definitions in~\eqref{app:ldp_rv_chaos_1}.

\begin{remark}
	Note that the rescaling in~\eqref{eq:rescaling} does not simply amout to keeping~$\mu$ as it is while rescaling~$\xi \rightsquigarrow \eps \xi$. Instead, it amounts to both rescaling~$\xi \rightsquigarrow \eps \xi$ \emph{and} $\mu \rightsquigarrow \mu_\eps := \mu \circ \eps^{-1}$.
	This is entirely consistent with the fact that the generalised contraction principle~\cite[Lem.~$3.3$]{hairer_weber_ldp}, to be essential below, pushes the LDP for $(\mu_\eps: \eps \in I)$ forward onto the probability measures~$(\Psib^{(\eps)}_*\mu_\eps: \eps \in I)$.
	
	In particular, given that~$\Psi_{\<11s>,2}$ may be represented as~$I_2(\CW_{\<11s>,2})$, this is consistent with the fact that it is in general \emph{not} true that $I_n(f_n)(\eps \xi) = \eps^n I_n(f_n)(\xi)$ for $\xi \in B$ and~$n \geq 2$. Instead, it is true that~
		\begin{equation*}
			I_n^{\eps}(f_n)(\eps \xi) = \eps^n I_n(f_n)(\xi), \quad I_n^\eps(f_n): \CH^{\tp_s n} \to \CH^{(n)}(\mu_\eps;\R)
		\end{equation*}
	for~$\mu$-a.e.~$\xi \in B$, where~$I_n^\eps$ is the $n$-th Wiener-It\^{o} associated with~$\mu_\eps$. Note that this also explains the rescaling behaviour of the renormalisation constant~$\fc_{\d}$, namely
		\begin{equation*}
			\eps^2 \fc_{\d} 
			= \E_{\mu_\eps}\sbr[1]{\Pib^{\eps\xi_\d}\<11>(0)} 
			= \fc_{\d,\eps}.
		\end{equation*}
\end{remark}

\subsubsection*{$\bullet$ The large deviations principle(s).} 

We have verified all the prerequisities for~\cite[Thm.~3.5]{hairer_weber_ldp} to apply. The first LDP concerns the space of minimal models~$\MM_-$. Recall that $\II(h) = \frac{1}{2}\norm{h}_\CH^2$ denotes Schilder's RF.

	\begin{theorem}\label{app:thm_ldp_models}
		The family $(\d_\eps\hbz^{\minus}: \eps \in I)$ of dilated, minimal BPHZ models satisfies a large deviation principle on $\MM_-$ (both equipped with $\norm[0]{\cdot}$ and $\barnorm{\cdot}$) under $\mu$ with rate $\eps^2$ and good rate function
			\begin{equation}
				\JJ_{\MM_-}: \MM_- \to [0,+\infty], \quad
				\JJ_{\MM_-}(\bz^{\minus})
				=
				\inf\{\II(h): \ h \in \CH, \ \LL_-(h) = \bz^{\minus}\}.
				\label{app:thm_ldp_models:rf}
			\end{equation} 
	\end{theorem}
The following proof is virtually the same as the one of~\cite[Thm.~4.3]{hairer_weber_ldp}.
\begin{proof}
By~\cite[Thm.~3.5]{hairer_weber_ldp}, the large deviations of~$\Psib^{(\eps)}$ on $\BE$ are governed by the rate function $\Lambda: \BE \to [0,+\infty]$ given by 
	\begin{equation*}
		\Lambda(\bs) 
		= \inf\bigg\{ \frac{1}{2} \norm{h}_\CH^2: \Psib_{\operatorname{hom}}(h) = \bs \bigg\}
		= \inf\bigg\{ \frac{1}{2} \norm{h}_\CH^2: \bigoplus_{\tau \in \CW_-} \Pi^h \tau = \bs \bigg\}.
	\end{equation*}
By construction, the random variables~$\Psib^{(\eps)}$ take values in~$\MM_-$. The latter is a closed subset of $\BE$, so by~\cite[Lem.~$4.1.5$]{dembo-zeitouni} the LDP also holds on $\MM_-$ (both w.r.t~$\norm[0]{\cdot}$ and~$\barnorm{\cdot}$) with appropriately restricted RF and w.r.t. the relative topology.
\end{proof}

Our strategy is to \enquote{push forward} this LDP via the generalised contraction principle~\cite[Lem.~$3.3$]{hairer_weber_ldp}.
In a first instance, we easily obtain the following corollary. 

\begin{corollary} \label{app:coro_ldp_models}
The family $(\d_\eps\hbz: \eps \in I)$ of dilated BPHZ models satisfies a large deviation principle on~$\MM$~under $\mu$ with rate $\eps^2$ and good rate function
	\begin{equation}
	\JJ_{\MM}: \MM \to [0,+\infty], \quad
	\JJ_{\MM}(\bz)
	=
	\inf\{\II(h): \ h \in \CH, \ \LL(h) = \bz\}.
	\label{app:coro_ldp_models:rf}
	\end{equation} 	
\end{corollary}

\begin{proof}
	Since the extension map~$\EE: \MM_- \to \MM$ from~\thref{app:thm:extension}	is continuous, the generalised contraction principle \cite[Lem.~$3.3$]{hairer_weber_ldp} yields an LDP on~$\MM$ for~$(\d_\eps\hbz = \EE\del[1]{\Psib^{(\eps)}}: \eps \in I)$. The corresponding RF is given by
	\begin{equs}
		\JJ_{\MM}(\bz) 
		& 
		= \inf\{\JJ_{\MM_-}(\bs): \bs \in \MM_-, \ \EE(\bs) = \bz\} \\
		&
		=
		\inf\bigg\{ 
		\frac{1}{2} \norm{h}_\CH^2: \ h \in \CH, \
		\Pi^h \tau = s_\tau \ \text{for} \ \bs = \bigoplus_{\tau \in \CW_-} s_\tau \in \MM_- \ \text{and} \ \EE(\bs) = \bz
		\bigg\} \\
		&
		=
		\inf\bigg\{ 
		\frac{1}{2} \norm{h}_\CH^2: \ h \in \CH, \ \LL(h) = \bz
		\bigg\}.
	\end{equs}
\end{proof}

Finally, we can give a precise statement of~\thref{thm:ldp} announced in the introduction. Its proof entirely mirrors that of~\cite[Thm.~$4.4$]{hairer_weber_ldp}. 
Recall that~$\CX_T := \{u \in \CC([0,T],\CC^\eta(\T^2)): u(0,\cdot) = u_0\}$, where~$u_0 \in \CC^{\eta}(\T^2)$ is the \emph{fixed} initial condition, and that~\enquote{$u^{\deathsmall}$} denotes a \enquote{graveyard path} which we postulate to have distance~$1$ from every element~$u \in \CX_T$.

\begin{theorem}\label{app:thm_ldp_gpam}
	Let~$T > 0$, ~$\eta \in (\nicefrac{1}{2},1)$, and~$\bar{\CX}_T := \CX_T \cup \{u^{\deathsmall}\}$. Then, the family~$(\hat{u}^\eps: \eps \in I)$ in~\thref{intro:ex_sol_hairer} satisfies a LDP in $\bar{\CX}_T$ with RF
		\begin{equation*}
			\JJ(u) = \inf\{\II(h): \ h \in \CH, \ \Phi(\LL(h)) = u\}, \quad \JJ(u^{\deathsmall}) := \infty
		\end{equation*}
	Here, we have set~$\hat{u}^{\eps} := u^{\deathsmall}$ in case~$T > T^\eps$, the latter introduced in~sec.~\ref{sec:explosion} above.	
\end{theorem}

\begin{proof}[of~\thref{thm:ldp}]
	Recall that~$\hat{u}^{\eps} = \Phi(\d_\eps \hbz)$ where~$\Phi = \CR \circ \CS$. Let $\MM_\CH$ be the set of models that arise as lifts of functions in~$\CH$, that is
		\begin{equation*}
			\MM_\CH 
			:= \{\bz \in \MM: \ \exists \thinspace h \in \CH \ \text{s.t.} \ \bz = \LL(h)\}.
		\end{equation*}
	By~\thref{prop:ex_det_gpam}, solutions to gPAM driven by~$h \in \CH$ are global (in particular, $\Phi(\LL(h)) = w_h$), so that~$\Phi(\bz) \neq +\infty$ for~$\bz \in \MM_\CH$. 
	As a consequence,~$\Phi$ is continuous on an open neighbourhood~$\CO$ of~$\MM_\CH$. 
	Finally, since~$\MM_\CH = \{\JJ_{\MM} < +\infty\}$ by~\eqref{app:coro_ldp_models:rf}, the generalised contraction principle~\cite[Lem.~$3.3$]{hairer_weber_ldp} implies the claim.
\end{proof}

\section{Fernique's theorem on model space} \label{app:fernique}

Let~$(B,\CH,\mu)$ be the abstract Wiener space from sec.~\ref{app:ldp_pam} which is also the setting for the generalised Fernique theorem proved by Friz and Oberhauser~\cite{generalised_fernique}.

By~$\hbz \in \MM$, we denote the BPHZ model and let~$\hbz^{\minus} \in \MM_-$ be its minimal part, that is~$\hbz = \EE(\hbz^{\minus})$. The following theorem establishes Gaussian concentration.

\begin{theorem}[Fernique]\label{thm:fernique}
	There exists an $\chi > 0$ such that 
	\begin{equation}
	\E_\mu\sbr[1]{\exp\del[1]{\chi \barnorm{\hbz^{\minus}}^2}}
	\equiv
	\int_B \exp\del[1]{\chi \barnorm{\hbz^{\minus}(\omega)}^2} \mu(\dif \omega) 
	<
	\infty.
	\label{thm:fernique:eq}
	\end{equation}
\end{theorem}

\begin{proof}
	Let $h \in \CH$ and $N \subseteq B$ be the nullset from~\thref{lem:trans_prob_cfg}. For each $\omega \in N^{\operatorname{c}}$, we have
	\begin{equs}
		\barnorm{\hbz^{\minus}(\omega)}
		&
		=
		\barnorm{T_h\sbr[1]{T_{-h} \del[1]{\hbz^{\minus}(\omega)}}\hspace{-0.4em}}
		\aac
		\barnorm{T_{-h} \del[1]{\hbz^{\minus}(\omega)}\hspace{-0.2em}} + \norm{h}_{\CH}
		=
		\barnorm{\hbz^{\minus}(\omega - h)} + \norm{h}_{\CH},
	\end{equs}
	where the inequality is due to~\thref{lem:ineq_translation_barnorm} and the equalities due to~\thref{lem:trans_prob_cfg}, which trivially also holds for the minimal admissible BPHZ model~$\hbz_-$. The claim now follows from the generalised Fernique theorem~\cite[Thm.~$2$]{generalised_fernique}.
\end{proof}

\section{Proof of Duhamel's Formula} \label{sec:duhamel_proof}

For~$s \geq 0$, we denote the indicator function of the set~$P_{> s} := \{z=(t,x) \in \R^3: t > s\}$~by~$\1_s^+$. 

Let us emphasise that in the previous parts of the paper, we have abused notation and simply wrote~$\CP^{\esh}$ when really referring to~$\CQ_{<\gamma} \CP^{\esh} \1_0^+$. Regarding the projection $\CQ_{<\g}$~onto~$\CT_{< \g}$, this is in line with~\thref{rmk:projection}.
We introduce a few sectors relevant to the proof of~\thref{lem:inv_theta_op}. 
	\begin{definition}\label{def:sectors}	
		Recall that~$\gr{\CU} = \scal{\1, \<1>, \<1g>, X}$. We define the sectors
		\begin{align*}
			\gr{\CV}[\<cm>] 
			& 
			:= \scal{\gr{\CU}\<cm>} 
			:= \scal{\gr{\tau}\<cm>: \thinspace \gr{\tau} \in \gr{\CU}}
			\equiv \scal{\<cm>,\<11g>,\<1g1g>,X\<cm>}, \quad
			\gr{\CV_0}
			:= \gr{\CV}[\<cm>]\setminus\{\<11g>\}
			\equiv \scal{\<cm>,\<1g1g>,X\<cm>}, \\
			\gr{\CV}[\<wn>] 
			& := \scal{\gr{\CU}\<wn>} 
			\equiv \scal{\<wn>,\<11>,\<1g1>,X\<wn>}, \quad
			\gr{\tilde{\CV}}
			:=
			\gr{\CV}[\<wn>] \cup \gr{\CV}[\<cm>] 
			\equiv \scal{\<wn>,\<11>,\<1g1>,\<11g>,\<1g1g>,X\<wn>}, \quad
			\gr{\bar{\CV}} := \gr{\CV} \cup \{\1,X\}.
		\end{align*}	
	\end{definition}	

\begin{proof}[of~\thref{lem:inv_theta_op}]	
	The proof will be delivered in various steps, with the following figure as a guideline.
	
	\begin{figure}[h]
		\centering
		\includegraphics[scale=1]{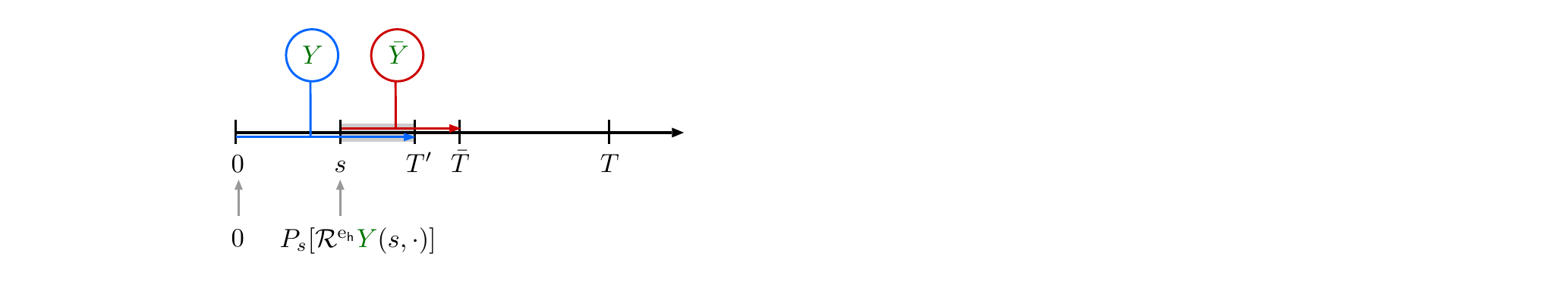}
		\captionsetup{font=small,format=plain}
		\caption{A schematic overview of the flow of time in the proof. At the bottom are the initial conditions: $0$ at time $0$, resulting in a solution~$\gr{Y}$ on~$(0,T')$, and~$P_s\sbr[1]{\CR^{\esh}\gr{Y}(s,\cdot)}$ for restart at time~$s$ (cf.~step~$3$) giving a solution~$\gr{\bar{Y}}$~on~$(s,\bar{T})$. Between~$s$ and $T'$ (shaded), the solutions $\gr{Y}$ and~$\gr{\bar{Y}}$ coincide~(cf.~step~$4$), rendering the procedure consistent.} 
		\label{figure:lem:inv_theta_op:pf}
	\end{figure}
	
	\begin{enumerate}[label=\textbf{Step \arabic*:}, align=left, leftmargin=0pt, labelindent=0pt,listparindent=0pt, itemindent=!]
		\item \textbf{The fixed-point problem~\textbf{\eqref{sec:duhamel:eq0}} is solvable.}	 \label{lem:inv_theta_op:pf_step_1}
		We verify the criteria of~\cite[Thm.~7.8]{hairer_rs}. First, note that~$G'(\gr{W})\thinspace\<cm> \in \DD^{\bar{\gamma},\bar{\eta}}_{\gr{\CV_0}}(\Gamma^{\esh})$ is invariant under the action of~$\fd_\lambda$.
		We introduce the maps~$\gr{\Lambda}$ and~$\gr{\tilde{\Lambda}}$ as
		\begin{equation*}
			\gr{\Lambda}: \DD^{\gamma,\eta}_{\gr{\CU}}(\Gamma^{\esh}) \to \DD^{\bar{\gamma},\bar{\eta}}_{\gr{\CV}[\<cms>]}(\Gamma^{\esh}), \quad
			\gr{Y} \mapsto  \gr{\Lambda Y} := G'(\gr{W})\thinspace\<cm>\thinspace\gr{Y}, \qquad
			\gr{\tilde{\Lambda} Y} := \gr{\Lambda Y + \tilde{V}}\thinspace;
		\end{equation*}
		they are well-defined by~\cite[Prop.~6.12]{hairer_rs}. For $\gr{\tilde{V}} \in \DD^{\bar{\gamma},\bar{\eta}}_{\gr{\tilde{\CV}}}(\Gamma^{\esh})$ it is easily seen that the pair of sectors~$(\gr{\CU},\gr{\bar{\CV}})$ satisfies the conditions
		\begin{equation*}
			\bar{\CT} \subseteq \gr{\CU} \subseteq \bar{\CT} + \CT_{\geq \zeta}, \quad
			\bar{\CT} \subseteq \gr{\bar{\CV}} \subseteq \bar{\CT} + \CT_{\geq \bar{\zeta}}, \quad
			\CQ_{< \gamma} \CI \gr{\bar{\CV}}_{< \bar{\gamma}} \subseteq \gr{\CU}_{< \gamma}
		\end{equation*} 
		with
		\begin{equation*}
			0 < \zeta < \bar{\zeta} +2, \quad 
			0 < \bar{\gamma} \leq \gamma \leq \bar{\gamma} +2, \quad
			\bar{\gamma} \leq \zeta
		\end{equation*}
		for
		\begin{equation*}
			\zeta := \a + 2, \quad \bar{\zeta} := \a, \quad \gamma := - \a + \kappa, \quad \bar{\gamma} := \gamma + \a.
		\end{equation*}
		Since $\eta \in (0,\a +2)$, cf.~\cite[Lem.~9.1]{hairer_rs}, we have $\bar{\eta} \in (\a,2\a+2]$ and thus the conditions
		\begin{equation*}
			\eta < (\bar{\eta} \wedge \bar{\zeta}) + 2, \quad
			\bar{\eta} \wedge \bar{\zeta} > -2
		\end{equation*}
		are also satisfied. Finally, the map~$\gr{\Lambda}$ (resp.~$\gr{\tilde{\Lambda}}$) is (affine) linear and continuous, hence Lipschitz continuous. Actually, we even have the estimate
		\begin{equation}
			\threebars \gr{\Lambda Y} \threebars_{\bar{\gamma},\bar{\eta};[s,T]}
			\leq
			C \threebars \gr{Y} \threebars_{\gamma,\eta;[s,T]}, 
			\quad
			C := \threebars G'(\gr{W})\thinspace\<cm> \thinspace \threebars_{\bar{\gamma},\bar{\eta};T}
			\label{lem:inv_theta_op:pf_estimate}
		\end{equation} 
		uniformly over all~$s \in [0,T]$ and $\gr{\Lambda}$ as well as~$\gr{\tilde{\Lambda}}$ map $\DD^{\gamma,\eta}_{s}(\Gamma^{\esh})$ into~$\DD^{\bar{\gamma},\bar{\eta}}_{s}(\Gamma^{\esh})$. In short, \cite[Ass.~12]{hairer-mattingly} is satisfied, the FP problem~\eqref{sec:duhamel:eq0} is solvable, and step~$1$ thus completed.
		\item \textbf{Neumann series bounds on short time intervals.} \label{lem:inv_theta_op:pf_step_2}
		For a linear, bounded operator~$A$ with $\norm[0]{A} < 1$, it is well-known under the name \emph{Neumann series} that
		\begin{equation*}
			\norm[0]{\sbr[0]{\operatorname{Id} - A}^{-1}} \leq \del[1]{1 - \norm[0]{A}}^{-1}.		
		\end{equation*} 
		For $s, t \in [0,T]$ such that $t > s$, we analyse the norm of the operator 
		\begin{equation}
			A_s := \CQ_{< \gamma}
			\CP^{E_\sh\d_{\nicefrac{1}{\lambda}}\bz}\1_s^+\del[1]{\gr{\Lambda}\thinspace\bullet}
			\in 
			\CL\del[1]{
				\DD^{\gamma,\eta;t}(\Gamma^{\esh;\nicefrac{1}{\l}})}
		\end{equation}
		where the last statement is true by~\cite[Prop.~$6.16$]{hairer_rs} combined with~\eqref{lem:inv_theta_op:pf_estimate}. 
		
		For $\tilde{z} = (\tilde{t},\tilde{x})$ and $\bar{z} = (\bar{t},\bar{x})$ with $\tilde{z},\bar{z} \in (s,t] \x \R^{d}$, observe that $d_{\fs}(z,P_s) \leq (t-s)^{\nicefrac{1}{2}}$ implies that
		\begin{equation*}
			\norm[0]{\tilde{z}}_{P_s} = 1 \wedge  d_{\fs}(\tilde{z},P_s) \leq (t-s)^{\nicefrac{1}{2}}, \quad
			\norm[0]{\tilde{z},\bar{z}}_{P_s} = \norm[0]{\tilde{z}}_{P_s} \wedge \norm[0]{\bar{z}}_{P_s} \leq (t-s)^{\nicefrac{1}{2}}
		\end{equation*}
		for $t > s$ small enough. The proof of~\cite[Thm.~7.1]{hairer_rs} can thus be suitably amended to imply	the bound
		\begin{equs}
			\threebars A_s \gr{Y} \threebars_{\gamma,\eta,t;\Gamma^{\esh,\nicefrac{1}{\l}}}
			&
			\aac
			(t-s)^{\beta} \norm[0]{\Pi^{\esh;\nicefrac{1}{\lambda}}} \threebars \gr{\Lambda Y} \threebars_{\bar{\gamma},\bar{\eta};(s,t]}
			\aac (t-s)^{\beta}
			\leq t^\beta
		\end{equs}
		for some~$\beta > 0$, uniformly over all $\gr{Y}$ with $\threebars \gr{Y} \threebars_{\gamma,\eta;t} \leq 1$ and all $s,t \in [0,T]$ such that $t > s$ and $t - s < 1$. In the previous estimate, we have used~\eqref{lem:inv_theta_op:pf_estimate} and~\thref{lem:rescaling_lambda}. Hence, by choosing $t$ small enough we can always ensure that~$\norm[0]{A_s}_{\operatorname{op};t;\Gamma^{\esh;\nicefrac{1}{\lambda}}} \leq \nicefrac{1}{2}$ and that
		\begin{equation}
			\norm[0]{\sbr[0]{\operatorname{Id} - A_s}^{-1}}_{\operatorname{op};t;\Gamma^{\esh;\nicefrac{1}{\lambda}}} \leq \del[1]{1 - \norm[0]{A_s}_{\operatorname{op};t;\Gamma^{\esh;\nicefrac{1}{\lambda}}}}^{-1}
			\leq
			2
			\label{lem:inv_theta_op:pf_neumann}
		\end{equation}
		uniformly over all $s \in [0,T]$.
		\item \textbf{Restarting the equation after short time.} \label{lem:inv_theta_op:pf_step_3}	
		By step~$1$, there exists $T' < T$ small such that the FP problem~\eqref{sec:duhamel:eq0} admits a unique solution $\gr{Y} \in \DD^{\gamma,\eta;T'}_{\gr{\CU}}(\Gamma^{\esh})$. In particular, step~$2$ (with $A := A_0$) implies that, for each~$s \in (0,T')$, we have
		\begin{align}
			\thinspace 
			&
			\threebars \fd_\lambda \gr{Y} \threebars_{\gamma,\eta;s;\Gamma^{\esh;\nicefrac{1}{\l}}}
			=
			\threebars [\operatorname{Id} - A]^{-1} \del[0]{\fd_\lambda \gr{V}} \threebars_{\gamma,\eta;s;\Gamma^{\esh;\nicefrac{1}{\l}}}
			\aac
			\threebars \fd_{\lambda} \gr{V} \threebars_{\gamma,\eta;s,\Gamma^{\esh;\nicefrac{1}{\l}}} \label{lem:inv_theta_op:pf_est_step3}\\
			= \ 
			&
			\threebars \CQ_{<\gamma}\CP^{E_\sh\d_{\nicefrac{1}{\l}}\bz}\1_0^+(\fd_\l\gr{\tilde{V}}) \threebars_{\gamma,\eta;s,\Gamma^{\esh;\nicefrac{1}{\l}}}
			\aac
			s^{\beta} \norm[0]{\Pi^{\esh;\nicefrac{1}{\l}}}
			\threebars \1_0^+(\fd_\l\gr{\tilde{V}}) \threebars_{\bar{\gamma},\bar{\eta};s;\Gamma^{\esh;\nicefrac{1}{\l}}} 
			\aac 
			\threebars \fd_\l\gr{\tilde{V}} \threebars_{\bar{\gamma},\bar{\eta};T;\Gamma^{\esh;\nicefrac{1}{\l}}} \notag
		\end{align}
		In the previous computation, we have again used~\thref{lem:rescaling_lambda} and, in addition, eq.~\eqref{lem:inv_theta_op:pf_neumann} as well as the transformative behaviour of~$\CP^{E_\sh\bz}$ w.r.t.~$\fd_\l$, see~\thref{lem:consistency_dilation}\ref{lem:consistency_dilation:iii}.
		
		We now consider
		\begin{equation}
			\gr{\bar{Y}} = \CQ_{< \gamma}\CP^{\esh}\1_s^+\del[1]{G'(\gr{W})\gr{\bar{Y}}\thinspace\<cm>\thinspace + \gr{\tilde{V}}} + P_s\sbr[1]{\CR^{\esh}\gr{Y}(s,\cdot)},
			\label{lem:inv_theta_op:pf_rest_fp_prob}
		\end{equation}
		with $P y_s$ the \enquote{harmonic extension} of $y_s := \CR^{\esh}\gr{Y}(s,\cdot)$ defined by
		\begin{equation*}
			P_sy_s(t,x) := \sbr[1]{P(t-s,\cdot) *_{\operatorname{s}} y_s}(x)
		\end{equation*}
		where $*_{\operatorname{s}}$ denotes spatial convolution. In conjunction with~\cite[Prop.~3.28]{hairer_rs}, our assumptions on the sector~$\gr{\CU}$ detailed in step~$1$ guarantee that $y_s \in \mathcal{C}^{\zeta} \embed \mathcal{C}^{\eta}$. A suitable modification of~\cite[Lem.~7.5]{hairer_rs} then shows that (the lift of) $P_s y_s$ (to $\bar{\CT}$) belongs to $\DD_s^{\gamma,\eta}(\Gamma^{\esh;\nicefrac{1}{\l}})$ and can be estimated~by
		\begin{equs}[][lem:inv_theta_op:pf_est_harm_ext]
			\threebars P_s y_s \threebars_{\gamma,\eta,T;\Gamma^{\esh;\nicefrac{1}{\l}}}
			& \aac
			\norm[0]{\CR^{E_\sh\d_{\nicefrac{1}{\l}}\bz} \fd_\l \gr{Y}(s,\cdot)}_{\eta}
			\aac
			\norm[0]{\Pi^{\esh;\nicefrac{1}{\l}}} \threebars \fd_\lambda \gr{Y} \threebars_{\gamma,\eta;s;\Gamma^{\esh;\nicefrac{1}{\l}}} \\
			& \aac
			\threebars \fd_\l\gr{\tilde{V}} \threebars_{\bar{\gamma},\bar{\eta};T;\Gamma^{\esh;\nicefrac{1}{\l}}}
		\end{equs}
		using~\eqref{lem:inv_theta_op:pf_est_step3} and~\thref{lem:rescaling_lambda}. This observation qualifies $P_s y_s$ as an admissible initial condition specified in~\cite[Thm.~7.8]{hairer_rs}, so the same argument as in step~$1$ renders the restarted FP problem in~\eqref{lem:inv_theta_op:pf_rest_fp_prob} solvable with~$\gr{\bar{Y}} \in \DD^{\gamma,\eta;(s,\bar{T})}_s(\Gamma^{\esh})$ for some $\bar{T} \in (T',T)$. 
		Step~$2$ combined with the estimate~\eqref{lem:inv_theta_op:pf_est_harm_ext} and~\thref{lem:rescaling_lambda} then implies that
		\begin{equs}
			\threebars \fd_\lambda \gr{\bar{Y}} \threebars_{\gamma,\eta;(s,\bar{T});\Gamma^{\esh;\nicefrac{1}{\l}}}
			& =
			\threebars [\operatorname{Id} - A_s]^{-1}  
			\del[1]{\CQ_{< \gamma}\CP^{E_\sh\d_{\nicefrac{1}{\l}}\bz}\thinspace (\1_s^+ \fd_\l \gr{\tilde{V})} 
				+ 
				P_s\sbr[1]{\CR^{\esh}\gr{Y}(s,\cdot)}}	
			\threebars_{\gamma,\eta;\bar{T};\Gamma^{\esh;\nicefrac{1}{\l}}} \\
			& \aac \ 
			\threebars 
			\CP^{E_\sh\d_{\nicefrac{1}{\l}}\bz}\thinspace (\1_s^+ \fd_\l \gr{\tilde{V})} 
			\threebars_{\gamma,\eta;\bar{T};\Gamma^{\esh;\nicefrac{1}{\l}}}
			+ 
			\threebars
			P_s\sbr[1]{\CR^{\esh}\gr{Y}(s,\cdot)}
			\threebars_{\gamma,\eta;\bar{T};\Gamma^{\esh;\nicefrac{1}{\l}}} \\
			& \aac \
			\bar{T}^\beta \thinspace \norm[0]{\Pi^{\esh;\nicefrac{1}{\l}}} \threebars \1_s^+ \fd_\l \gr{\tilde{V}} \threebars_{\gamma,\eta;\bar{T};\Gamma^{\esh;\nicefrac{1}{\l}}}
			+
			\threebars \fd_\l \gr{\tilde{V}} \threebars_{\bar{\gamma},\bar{\eta};T} 
			\aac 
			\threebars \fd_\l \gr{\tilde{V}} \threebars_{\bar{\gamma},\bar{\eta};T}
		\end{equs}
		\item \textbf{Patching together and iteration.} \label{lem:inv_theta_op:pf_step_4}	
		By~\cite[Prop.~7.11]{hairer_rs}, $\gr{\bar{Y}}$ agrees with $\gr{Y}$ on $(s,T']$, so we may patch solutions together: We define~
		\begin{equation*}
			\gr{\tilde{Y}}(t,x) := \gr{Y}(t,x) \1_{t \leq s} + \gr{\bar{Y}}(t,x) \1_{t \in (s,\bar{T})}
		\end{equation*}
		which we know solves~\eqref{sec:duhamel:eq0} on~$(0,\bar{T})$. In addition, the estimates in step~$3$ imply the bound
		\begin{equation}
			\threebars \fd_\l \gr{\tilde{Y}} \threebars_{\gamma,\eta;\bar{T};\Gamma^{\esh;\nicefrac{1}{\l}}}	
			\leq
			\threebars \fd_\l \gr{Y} \threebars_{\gamma,\eta;s;\Gamma^{\esh;\nicefrac{1}{\l}}}
			+
			\threebars \fd_\l \gr{\bar{Y}} \threebars_{\gamma,\eta;(s,\bar{T});\Gamma^{\esh;\nicefrac{1}{\l}}}
			\aac
			\threebars \fd_\l \gr{\tilde{V}} \threebars_{\bar{\gamma},\bar{\eta};T}.
			\label{lem:inv_theta_op:pf_step_4_eq}
		\end{equation}
		The procedure described in steps~$1$ to~$3$ that culminated in the estimate~\eqref{lem:inv_theta_op:pf_step_4_eq} can be repeated ad infinitum until the (finite!) time~$T$ is reached. Thereby, we iteratively obtain a solution~$\gr{Y}$ to~\eqref{sec:duhamel:eq0} on~$(0,T)$ with the desired estimate
		\begin{equation*}
			\threebars \fd_\l \gr{Y} \threebars_{\gamma,\eta;T;\Gamma^{\esh;\nicefrac{1}{\l}}}
			\aac
			\threebars \fd_\l \gr{\tilde{V}} \threebars_{\bar{\gamma},\bar{\eta};T;\Gamma^{\esh;\nicefrac{1}{\l}}}.
		\end{equation*}
	\end{enumerate}
\end{proof}

\bibstyle{alpha}
\bibliography{bibfile.bib}

\end{document}